\documentclass[11pt,reqno]{amsart}
\usepackage{geometry} 
\usepackage{amsmath,amssymb,amsthm, mathrsfs}
\usepackage{xcolor}
\usepackage{graphicx} 

\usepackage{enumerate}
\usepackage{indentfirst}
\usepackage[colorlinks,
            linkcolor=black,
            anchorcolor=black,
            citecolor=black
            ]{hyperref}

\DeclareMathOperator{\diver}{\mathrm{div}}

\newtheorem{mydef}{Definition}
\newtheorem{thm}{Theorem}[section]
\newtheorem{lem}{Lemma}[section]

\newtheorem{rk}{Remark}
\newtheorem{cor}{Corollary}[section]
\numberwithin{rk}{section}
\numberwithin{prop}{section}
\numberwithin{mydef}{section}
\numberwithin{lem}{section}
\numberwithin{equation}{section}
\numberwithin{thm}{section}
\allowdisplaybreaks[4]

\title[Degenerate compressible Navier--Stokes equations]{Global Regular Solutions \\[0.5mm]
of the Compressible Navier-Stokes Equations \\[0.5mm]
with Nonlinear Density-Dependent Viscosities\\[0.5mm] 
and Large Initial Data of Spherical Symmetry}

\date{\today}

\author{Gui-Qiang G. Chen}
\address[Gui-Qiang G. Chen]{Mathematical Institute, University of Oxford, Oxford, OX2 6GG, UK.} \email{\tt gui-qiang.chen@maths.ox.ac.uk}

\author{Jiawen Zhang}
\address[Jiawen Zhang]{School of Mathematical Sciences, Shanghai Jiao Tong University,
Shanghai 200240, P. R. China} \email{\tt zhangjiawen317@sjtu.edu.cn}

\author{Shengguo Zhu }
\address[Shengguo Zhu]{School of Mathematical Sciences,  CMA-Shanghai,   and MOE-LSC,   Shanghai Jiao Tong University, Shanghai 200240, P. R. China.}
 \email{\tt  zhushengguo@sjtu.edu.cn}

\begin{document}
\begin{abstract}
For the physically important case in which the viscosity coefficients depend on the density $\rho$ through a power law  ({\it i.e.}, $\rho^\delta$ with some exponent $\delta \in (\frac{1}{2},1)$), we establish the global well-posedness of regular solutions of the compressible Navier-Stokes equations for barotropic flow with large initial data of spherical symmetry in two and three spatial dimensions. The initial density considered here is positive everywhere but vanishes in the far field, ensuring that the resulting solutions satisfy the conservation laws of total mass and momentum. 
The most crucial step in our analysis is to obtain a uniform upper bound for the density, which is challenging due to the combined difficulties of degeneracy near the far-field vacuum, coordinate singularity at the origin, and nonlinearity of viscosity coefficients. 
Furthermore, the methodology developed here can also be applied to the corresponding problem in which the density remains strictly away from the vacuum.

\end{abstract}



\maketitle

\tableofcontents

\section{Introduction}

The motion of viscous and Newtonian barotropic 
fluids in $\mathbb{R}^n$ ($n=2$ or $3$) is governed by the following compressible Navier-Stokes system (\textbf{CNS}):
\begin{equation}\label{eq:1.1}
\begin{cases}
\displaystyle 
\rho_t+\diver(\rho \boldsymbol{u})=0,\\[4pt]
\displaystyle
(\rho \boldsymbol{u})_t+\diver(\rho \boldsymbol{u}\otimes \boldsymbol{u})
+\nabla P =\diver \mathbb{T},
\end{cases}
\end{equation}
where $t\geq 0$ denotes  the time coordinate, $\boldsymbol{x}=(x_1,\cdots\!,x_n)^\top\in \mathbb{R}^n$  the Eulerian spatial coordinate, $\rho\geq 0$  the  mass density of the fluid,  $\boldsymbol{u}=(u_1,\cdots\!,u_n)^\top\in \mathbb{R}^n$   the fluid velocity, and  $P$ the pressure of the fluid.  For the polytropic gases, the constitutive relation is given by
\begin{equation*}
P=A\rho^{\gamma},
\end{equation*}
where $A>0$ is  an entropy  constant and  $\gamma>1$ is the adiabatic exponent.  The viscous stress tensor $\mathbb{T}$ is of the form: 
\begin{equation}\label{eq:1.1t}
\mathbb{T}=2\mu(\rho)D(\boldsymbol{u})+\lambda(\rho)\diver\boldsymbol{u}\,\mathbb{I}_n,
\end{equation} 
where $D(\boldsymbol{u})=\frac{1}{2}(\nabla \boldsymbol{u}+(\nabla \boldsymbol{u})^\top)$ is the deformation tensor, $\mathbb{I}_n$ is the $n\times n$ identity matrix,
\begin{equation}\label{fandan}
\mu(\rho)=a_1  \rho^\delta,\quad \lambda(\rho)=a_2\rho^\delta,
\end{equation}
for the viscosity exponent  $\delta\geq 0$, $\mu(\rho)$ is the shear viscosity coefficient, $\lambda(\rho)+\frac{2}{n}\mu(\rho)$ is the bulk viscosity coefficient,  and $a_1$ and $a_2$ are both constants satisfying
\begin{equation}\label{kelaoxiusi}
a_1>0 ,\qquad \ \  2a_1+n a_2\geq 0.
\end{equation}

For rarefied gases, the full \textbf{CNS}  can be formally derived from the Boltzmann equation 
via the Chapman--Enskog expansion (see Chapman--Cowling \cite{chap}). 
Under proper physical assumptions,  the viscosity coefficients  and the thermal conductivity coefficient  are all functions 
of the absolute temperature $\theta$. In fact,  
for the cutoff inverse power force models, if the intermolecular potential varies as $\ell^{-\varkappa}$,
where $\ell$ is the intermolecular distance and $\varkappa$ is a positive constant, then   
\begin{equation}\label{eq:1.6g}
\mu(\theta)=b_1 \theta^{\frac{1}{2}+ b},\quad \lambda(\theta)=b_2 \theta^{\frac{1}{2}+ b}, \quad  \kappa(\theta)=b_3 \theta^{\frac{1}{2}+ b}\qquad\,\, 
\text{with $\,\,b=\frac{2}{\varkappa} \in [0,\infty)$},
\end{equation}
for some constants $b_i$ ($i=1,2,3$), where $\kappa(\theta)$ is the thermal conductivity coefficient.
In particular, $\varkappa=1$ for the ionized gas, $\varkappa=4$ for Maxwellian molecules, and $\varkappa=\infty$ for rigid elastic spherical molecules (see \S 10 of \cite{chap}). For polytropic fluids, such a dependence is inherited through the laws of Boyle and Gay-Lussac{\rm:}
\begin{equation*}
P=\hat{A}\rho \vartheta=A\rho^\gamma \qquad \text{for a constant } \hat{A}>0,
\end{equation*}
\textit{i.e.}, $\vartheta=A\hat{A}^{-1}\rho^{\gamma-1}$, 
and  the viscosity coefficients become functions of $\rho$ taking form $\eqref{fandan}$.  
Furthermore, there are other physical models satisfying the density-dependent viscosities  assumption \eqref{fandan}, including the Korteweg system, the shallow water equations,  the quantum Navier-Stokes system, among others; see \cite{bd3,Dunn,gpm} and the references therein.

In this  paper,  we   establish   the global well-posedness   of spherically symmetric (regular) solutions taking the form 
\begin{equation*}
\displaystyle
(\rho,\boldsymbol{u})(t,\boldsymbol{x})
=(\rho(t,|\boldsymbol{x}|),u(t,|\boldsymbol{x}|) \frac{\boldsymbol{x}}{|\boldsymbol{x}|})
\end{equation*}
of system 
\eqref{eq:1.1}--\eqref{kelaoxiusi} in $[0,T]\times \mathbb{R}^n$ ($n=2$ or $3$) with the initial  data:
\begin{equation}\label{eqs:CauchyInit}
(\rho,\boldsymbol{u})(0,\boldsymbol{x})=(\rho_0,\boldsymbol{u}_0)(\boldsymbol{x})=(\rho_0(|\boldsymbol{x}|),u_0(|\boldsymbol{x}|)\frac{\boldsymbol{x}}{|\boldsymbol{x}|}) 
\qquad \text{for $\,\boldsymbol{x} \in \mathbb{R}^n$},
\end{equation}
and the   far-field behavior:
\begin{equation}\label{e1.3}
\begin{aligned}
\displaystyle
(\rho,\boldsymbol{u})(t,\boldsymbol{x})
\to (\bar \rho,\boldsymbol{0})
\qquad \text{as $|\boldsymbol{x}|\to \infty\,$  for $\,t\ge 0$},
\end{aligned}
\end{equation}
where $\bar \rho \geq 0$ is a constant.
Our results hold for all adiabatic exponents $\gamma\in (1,\infty)$ and  physical  viscosity exponent $ \delta\in (\frac{1}{2},1)$ in $\mathbb{R}^2$, and for $\gamma\in (1,6\delta-3)$ and $\delta 
\in (7-2\sqrt{10},1)$ in $\mathbb{R}^3$, without restriction on the size of the initial data. Moreover, the viscosity coefficients $(\mu(\rho),\lambda(\rho))$ satisfy the Bresch--Desjardins (BD) relation (see \cite{bd3}):
\begin{equation}\label{coeff}
\lambda(\rho)=2(\rho \mu'(\rho)-\mu(\rho)).
\end{equation}
In particular, when $\bar{\rho}=0$ in \eqref{e1.3}, our  solutions have the conserved total mass and momentum, and the fluid density keeps positive in $\mathbb{R}^n$  but decays to zero in the far-field, which is consistent with the facts that the total mass and momentum are conserved, and  \textbf{CNS} is a model of non-dilute fluids.

When $(\mu, \lambda,\kappa)$  are all constants, there is a vast literature on the well-posedness theory for  \textbf{CNS}. In the absence of vacuum, the one-dimensional (1-D) problem   has been studied extensively (see \cite{Kanel, KN,KS, ZA1} and the references therein), and   the  local well-posedness  of the multi-dimensional (M-D) classical solutions  are known in \cite{nash, serrin}. 
However,   when vacuum appears,  the approaches used in the references  mentioned above do not work directly, owing to the degeneracy of the time evolution in the momentum equations, which makes it difficult to study the dynamics of the fluid velocity in the domain where the density vanishes. 
In general, a vacuum is required for the far-field 
under some physical requirements
such as finite total mass and total energy in  $\mathbb{R}^n$. 
By introducing some  compatibility conditions, the local well-posedness of  3-D regular solutions with vacuum was established successfully 
in Salvi--Stra\v skraba \cite{Salvi},  Cho--Choe--Kim  \cite{CK3},  and the references therein. 
On the other hand,  it  has been shown that,  in the small data regime (near constant density),  the M-D solutions  remain globally regular if the initial data are smooth; see \cite{danchin1,  MN} and the references therein.
For  M-D problems with general data, though many important results on the global existence of weak solutions have been obtained in \cite{BJ,fu3,H,HJ,JZ,lions}, the uniqueness problem  
It is widely open due to the fairly low regularity of the solutions.

In fact, when $(\mu, \lambda,\kappa)$  are all constants, no matter for the barotropic flow or the non-isentropic flow, 
the global regularity of M-D solutions with large initial data remains open until now, even when the initial data exhibit some form of symmetry (if the domain considered contains the origin). 
The main obstacle lies in obtaining uniform {\it a priori} bounds on the density, both above and below; see Lions \cite{lions},  Feireisl \cite{fu3}, Sun--Wang--Zhang \cite{zhangzhifei},  and the references therein. 
For the M-D problem, as far as we know,  such kind of bounds can only be obtained in the region excluding the origin for spherically symmetric flow; see \cite{ck,H,HJ,J} and the references therein.
Recently, for 3-D spherically symmetric flow, Merle--Rapha\"el--Rodnianski--Szeftel \cite{MPI}  proved that there exists a set of finite-energy smooth initial data with far-field vacuum for which the corresponding solutions to the barotropic   \textbf{CNS} implode  in a finite time $T$, \textit{i.e.},
for any $\varepsilon>0$,
\begin{equation}\label{densityimplosion}
\quad \ \ \ \ \lim_{t\to T}\rho(t,\boldsymbol{0})=\infty,
\qquad  \lim_{t\to T} \sup_{|\boldsymbol{x}|\leq \varepsilon}|\boldsymbol{u}(t,\boldsymbol{x})|=\infty.
\end{equation}  
It is worth noting that the case $\gamma = \frac{5}{3}$, which corresponds to a monatomic gas,  was excluded in \cite{MPI} 
due to a triple-point degeneracy in the underlying analysis. This gap was later filled   by Shao--Wang--Wei--Zhang \cite{shao} by introducing a novel renormalization
for some autonomous ODEs. 
Moreover, for 3-D barotropic flow,  in  \cite{buckmaster}, Buckmaster--Cao-Labora--Gómez-Serrano  proved the finite-time blow-up of \textbf{CNS} with strictly positive initial density and spherical symmetry when $\gamma=\frac{7}{5}$ (corresponding to a diatomic gas).
On the other hand,  for 3-D barotropic \textbf{CNS} without symmetry assumption, 
 Cao-Labora--Gómez-Serrano--Shi--Staffilani \cite{shijia}  constructed some smooth solutions 
 that remain strictly away from vacuum and nevertheless develop an imploding finite-time 
 singularity in $\mathbb{T}^3$ or $\mathbb{R}^3$.

When vacuum appears, 
some  singular and counterintuitive behaviors of solutions to \textbf{CNS} with constant viscosity and thermal conductivity coefficients have been observed.
In particular,  Hoff--Serre \cite{hoffserre} showed that  1-D weak solutions need not depend continuously on their initial data when the initial density contains an interval of vacuum states,  and Duan-Xin--Zhu \cite{dxz}  proved that the classical solutions with far-field vacuum  cannot preserve the conservation of the momentum in $\mathbb{R}^3$.
Such pathological phenomena may be traced back to the unphysical assumption that the viscosity coefficients are constants when modeling viscous fluids in the presence of vacuum, under which the vacuum exerts an artificial force on the fluid across the fluid--vacuum interface. Therefore, from  a physical standpoint, compressible viscous flows near vacuum are therefore more appropriately modeled by the degenerate  \textbf{CNS} with $(\mu(\theta),\lambda(\theta),\kappa(\theta))$ shown  in \eqref{eq:1.6g} for the nonisentropic flow or $(\mu(\rho),\lambda(\rho))$ shown in \eqref{fandan} for the barotropic flow, respectively.

In fact, the degenerate \textbf{CNS} \eqref{eq:1.1}--\eqref{kelaoxiusi} (with $\delta>0$ in \eqref{fandan})  has received extensive attention in recent years. 
Significant progress on the global existence of smooth solutions has been achieved when the initial density is strictly positive; see \cite{cons,HB,kv,vassu2} for one-dimensional flows with general initial data, and \cite{Sundbye2,weike} for two-dimensional flows with small initial data. Some significant progress has been made on the global existence of smooth solutions when the initial density is strictly positive; see \cite{cons,HB,kv,vassu2} for 1-D flow with general data, and \cite{Sundbye2,weike} for 2-D flow with small data.  
However, when  $\inf_{\boldsymbol{x}\in\mathbb{R}^n} {\rho_0(\boldsymbol{x})}=0$, the momentum equations become 
degenerate in both the 
time evolution and the spatial dissipation: 
\begin{equation}\label{doubled}
\displaystyle
 \underbrace{\rho(\boldsymbol{u}_t+\boldsymbol{u}\cdot \nabla \boldsymbol{u})}_{\circledast}+\nabla P= \underbrace{\diver(2\mu(\rho)D(\boldsymbol{u})+\lambda(\rho)\diver \boldsymbol{u} \mathbb{I}_n)}_{\Diamond},
\end{equation}
where $\circledast$ denotes the degenerate time evolution, and $\Diamond$ denotes the degenerate spatial dissipation.
Such a double degenerate structure shown in \eqref{doubled} makes it formidable to establish the propagation and mollification mechanisms of the regularity of solutions. A new mathematical entropy
function was first introduced in Bresch--Desjardins \cite{bd3}  for
 $(\mu(\rho),\lambda(\rho))$ satisfying \eqref{coeff}, which offers the well-known BD entropy estimate
\begin{equation}\label{BDentropye}
\mu'(\rho)\nabla \rho/\sqrt{\rho}\in L^\infty ([0,T];L^2(\mathbb{R}^n))
\end{equation}
provided that    $
\mu'(\rho_0)\nabla \rho_0/\sqrt{\rho_0}\in L^2(\mathbb{R}^n)$.
This observation plays a key role in the development of the global existence of M-D weak solutions with finite energy; see   Bresch--Vasseur--Yu \cite{bvy},  Li--Xin \cite{lz}, Vasseur--Yu \cite{vayu}, 
and the references therein. 

Recently, based on some elaborate analysis of the intrinsic degenerate-singular structures of the degenerate \textbf{CNS} \eqref{eq:1.1}--\eqref{kelaoxiusi}, the local well-posedness of classical solutions with far-field vacuum is known in  \cite{GG, sz3,sz333,zhuthesis} and the references therein. In stark contrast with the case that the viscosity coefficients are constants, some solid progress have been obtained on the global regularity of solutions of the degenerate \textbf{CNS} with large initial data of spherical symmetry.
In particular, when $\delta=1$, by making full use of the linear dependence of the viscosity coefficient on the density,
Chen--Zhang--Zhu \cite{CZZ1} have established global well-posedness of spherically symmetric smooth solutions  in $\mathbb{R}^n$ ($n=2$ or $3$) for the following two classes of initial density profiles{\rm:}
 \begin{enumerate}
\item[\rm{(i)}] the initial density $\rho_0(\boldsymbol{x})$ is strictly positive in the whole space $\mathbb{R}^n$: 
\begin{equation*}
\inf_{\boldsymbol{x}\in \mathbb{R}^n} \rho_0(\boldsymbol{x})>0;
\end{equation*}
\item[\rm{(ii)}] the initial density remains positive for all $\boldsymbol{x}\in \mathbb{R}^n$  but decays to zero in the far-field:
\begin{equation*}
\rho_0(\boldsymbol{x}) >0 \quad \text{for $\boldsymbol{x}\in \mathbb{R}^n$},  \qquad\quad\rho_0(\boldsymbol{x}) \to 0 \quad \text{as $|\boldsymbol{x}|\to \infty$}.      
\end{equation*}
\end{enumerate}
Furthermore, for the case  $\delta=1$, 
even when the initial density is compactly supported, 
under the setting of the vacuum free boundary problem, 
it is proved in  Chen--Zhang--Zhu \cite{CZZ2} that 
the solutions with large initial data of spherical symmetry remain globally regular and unique in two and three spatial dimensions. Moreover,  the physical vacuum boundary condition is allowed for the initial density profile considered in \cite{CZZ2}, and the solutions are smooth all the way up to the moving vacuum boundary.

On the other hand, for the degenerate \textbf{CNS} \eqref{eq:1.1}--\eqref{kelaoxiusi}, when $\delta>1$, it is shown in \cite{sz333} that,  for certain classes of 3-D initial data with vacuum in some open set, one can construct corresponding local smooth solutions in the inhomogeneous Sobolev spaces, which break down in finite time, regardless of the size of  the initial data. Moreover, in a very recent paper, for   $\gamma \in (1,\, 1+\frac{2}{\sqrt{3}})$ except for at most countably many points,  
Chen--Liu--Zhu \cite{GLZ}  identified  a value $\delta^*(\gamma)<\frac12$, \textit{i.e.},  
  \begin{equation}\label{def:delta*}
\begin{aligned}
\delta^*(\gamma)=\begin{cases}
\displaystyle
     \frac{\gamma+1}{4}-\frac{\sqrt{2(\gamma-1)}}{2}\quad &\displaystyle\text{for}\,\,\, 1 < \gamma <\frac{5}{3},\\[8pt]
     \displaystyle
     \frac{1-(2\sqrt{3}-3)\gamma}{2(3-\sqrt{3})}\quad &\displaystyle\text{for}\,\,\, \frac{5}{3} \leq \gamma <1+\frac{2}{\sqrt{3}},
\end{cases}
\end{aligned}
\end{equation}
such that, 
for every $0<\delta< \delta^*(\gamma)$ in \eqref{fandan}, 
there exists a class of smooth initial data for which the corresponding  smooth solutions 
 implode  in finite time at the origin (as shown in \eqref{densityimplosion}) in $\mathbb{T}^3$ (torus) or $\mathbb{R}^3$. 
For  ruling out 
the possibility that the corresponding implosion is an artifact of the vacuum, the  initial density profile considered in \cite{GLZ} is assumed to be  strictly positive.

These observations shown in \cite{buckmaster,shijia, GLZ,CZZ1,CZZ2,sz333, MPI,shao} for the global regularity of solutions with large data or finite time singularity formation indicate that the behavior   of   M-D  regular solutions to \textbf{CNS} \eqref{eq:1.1}--\eqref{kelaoxiusi} depends sensitively on the  viscosity exponent.
This strong dependence makes the analysis of global well-posedness for M-D regular solutions with general initial data highly intricate. 

In this paper, for the physically important case  $\delta \in (\frac{1}{2},1)$, we study  the global well-posedness of regular solutions of the degenerate \textbf{CNS} \eqref{eq:1.1}--\eqref{kelaoxiusi} with large initial data of spherical symmetry in two and three spatial dimensions. 
For this purpose, under assumption \eqref{coeff},
we first reformulate the Cauchy problem \eqref{eq:1.1}--\eqref{kelaoxiusi} with \eqref{eqs:CauchyInit}--\eqref{e1.3} in $[0,T]\times \mathbb{R}^n$ for some time $T>0$  into the following initial-boundary value problem ({\rm \textbf{IBVP}}) in $\Omega_T:=[0,T]\times [0,\infty)$:
\begin{equation}\label{e1.5hh}
\begin{cases}
\displaystyle 
\rho_t+u\rho_r+\rho\big(u_r+\frac{m u}{r}\big)=0,\\[3pt]
\displaystyle
\rho u_t+\rho uu_r+A(\rho^\gamma)_r=2a_1\delta\big(\rho^\delta u_r+\frac{m\rho^\delta u}{r} \big)_r-\frac{2a_1 m(\rho^\delta)_r u}{r},\\[7pt]
\displaystyle
(\rho, u)|_{t=0}=(\rho_0, u_0 )\qquad\qquad \qquad \qquad \qquad \text{for  $r\in I:=[0,\infty)$},\\[7pt]
\displaystyle
u|_{r=0}=0\qquad\qquad\qquad \qquad \qquad\qquad \
\qquad  \text{for $t\in(0,T]$},\\[7pt]
\displaystyle
(\rho,u)\to \left(0,0\right) \qquad\qquad \, \qquad \qquad  \qquad 
\qquad \text{as $r\to \infty$ \ for $t\in (0,T]$},
\end{cases}
\end{equation}
where $a_1>0$ is a constant,  $r=|\boldsymbol{x}|$, $m=n-1$, and the boundary condition $\eqref{e1.5hh}_4$ is derived from the continuity of $\boldsymbol{x}\mapsto \boldsymbol{u}(t,\boldsymbol{x})$ at $\boldsymbol{x}=\boldsymbol{0}$. Drawing upon the progress presented in \cite{GLZ,CZZ1,CZZ2,sz333}, our results will further refine  the global well-posedness theory of M-D smooth  solutions with general data for the degenerate \textbf{CNS}.

The corresponding study on the global regularity of solutions with general data  of \eqref{e1.5hh} is extremely difficult, since the structure of momentum equation $\eqref{e1.5hh}_2$ is full of degeneracy, singularity, and nonlinearity, \textit{i.e.},
\begin{equation}\label{cosingu}
\displaystyle
\underbrace{\rho (u_t+uu_r)}_{\circledast}+P_r-\underbrace{2a_1\delta(\rho^\delta u_r)_r}_{\Diamond}-\underbrace{2 a_1 m \Big(\delta \big(\frac{\rho^\delta u}{r} \big)_r-\frac{ (\rho^\delta)_r u}{r}\Big)}_{\star} =0,
\end{equation}
where $\star$ denotes the  coordinates singularity.
In fact, for the 1-D case ($m=n-1=0$), 
the so-called coordinate singularity vanishes. 
Hence, the analysis for M-D spherically symmetric flows is essentially different from that for 1-D flow.
Due to the compressibility of fluids, 
we encounter the following  two major obstacles:
\begin{itemize}
\item possible cavitation, {\it i.e.}, $\rho(t,r)\to 0$ for some $(t,r)\in (0,T]\times [0,\infty)$;
\vspace{3pt}
\item possible implosion, {\it i.e.}, $\rho(t,r)\to \infty$ for some $(t,r)\in (0,T]\times [0,\infty)$.
\end{itemize}
Overcoming these difficulties is highly nontrivial  because of three inherent issues:
\begin{itemize}
\item the coordinate singularity at the origin, manifested by the singular factor $\tfrac{1}{r}$ in equations $\eqref{e1.5hh}_1$--$\eqref{e1.5hh}_2$;

\vspace{3pt}
\item the degeneracies in both the time evolution ($\circledast$) and the spatial dissipation ($\Diamond$), arising  from the far-field vacuum 
in \eqref{e1.3};

\vspace{3pt}
\item the strong nonlinearity in the spatial dissipation caused by the nonlinear density-dependent viscosity coefficients.
\end{itemize}

Thus, new ideas and techniques are required to establish the global well-posedness of spherical symmetric regular solutions 
of the Cauchy problem \eqref{eq:1.1}--\eqref{kelaoxiusi} with \eqref{eqs:CauchyInit}--\eqref{e1.3}, especially when far-field vacuum appears.
Fortunately, in this paper,  based on some elaborate analysis of the intrinsic degenerate-singular structures of \eqref{eq:1.1} in the radial coordinates, we prove the global well-posedness of regular solutions of the degenerate   \textbf{CNS} for barotropic flow
with large initial data of spherical symmetry and nonlinear density-dependent viscosities in two and three spatial dimensions.  Our result holds without restrictions on the size of the initial data, and the solutions we obtained satisfy the conservation laws of total mass and momentum.
In our analysis, the most crucial step is to obtain the uniform upper bound of $\rho$,  which is challenging due to  the combined difficulties of degeneracy near the far-field vacuum, coordinate singularity at the origin, and nonlinearity of viscosity coefficients. Owing to the nonlinear dependence of the viscosity coefficients on the density (\textit{i.e.}, $\delta\in (0,1)$ in \eqref{fandan}) as well as the appearance of far-field vacuum,  the dynamics of the fluid velocity is governed by an intrinsic {\it singular parabolic} system. To handle such a singular structure, we introduce well-designed singular weighted estimates for the fluid velocity. Besides, in the simpler case where the density is strictly positive, our method can be applied  directly without major modification.

We now outline the organization of the rest of this paper.
In \S \ref{maintheorem}, the main theorems on the global well-posedness for spherically symmetric regular solutions of the Cauchy problem \eqref{eq:1.1}--\eqref{kelaoxiusi} with \eqref{eqs:CauchyInit}--\eqref{e1.3} are presented.
In \S \ref{Section2}, we introduce the notations, present an enlarged reformulation of \eqref{eq:1.1}, and then outline the main strategies of our analysis. 
In \S \ref{section-upper-density}--\S \ref{se46}, we provide a detailed proof of the global well-posedness under the far-field vacuum case stated in \S \ref{maintheorem}. 
On one hand, in \S \ref{section-upper-density}--\S \ref{section-global2}, 
we establish the global uniform estimates for the regular solutions in carefully designed  function spaces, which consist of:
\begin{enumerate}
\item[(i)] The global {\it a priori} upper bound for $\rho$ (\S \ref{section-upper-density});
\vspace{3pt}
\item[(ii)] The global  $L^\infty(\mathbb{R}^n)$-estimate for the effective velocity (\S \ref{section-effective});
\vspace{3pt}
\item[(iii)] The non-formation of vacuum  inside the fluids  in finite time (\S \ref{section-nonformation});
\vspace{3pt}
\item[(iv)] The global singular weighted  estimates for regular solutions (\S \ref{section-global2}).
\end{enumerate}
On the other hand, based on these uniform estimates, in \S \ref{se46}, we obtain the desired  global well-posedness 
of   regular solutions,
by using the method of continuity. In \S \ref{nonvacuumfarfield}, we sketch the proof for the global well-posedness of regular solutions with strictly positive initial density. Furthermore, we briefly outline the proof for the desired local well-posedness in \S \ref{section-local-regular}. Finally, we list some auxiliary lemmas and new Sobolev embedding theorems 
for spherically symmetric functions that are used frequently throughout this paper 
in  Appendices \ref{appA}--\ref{improve-sobolev}.

\section{Main Theorems}\label{maintheorem}
This section is devoted to stating our main theorems on the global well-posedness of the spherically symmetric regular solutions 
of the Cauchy problem \eqref{eq:1.1}--\eqref{kelaoxiusi} with \eqref{eqs:CauchyInit}--\eqref{e1.3}  for general data in $\mathbb{R}^n$ for $n=2,3$.  For simplicity, throughout this paper,
for any function space defined on $\mathbb{R}^n$, 
the following conventions are used for any $k\in \mathbb{N}$,
unless otherwise specified: 
\begin{equation} \label{eulerspace}
\begin{aligned}
&\|f\|_{L^p}=\|f\|_{L^p(\mathbb{R}^n)},\quad \|f\|_{H^k}=\|f\|_{H^k(\mathbb{R}^n)},\quad  \|f\|_{W^{k,p}}=\|f\|_{W^{k,p}(\mathbb{R}^n)},\\[2pt]
&D^{k,p}(\mathbb{R}^n)
=\big\{f\in L^1_{\mathrm{loc}}(\mathbb{R}^n):\,\|f\|_{D^{k,p}(\mathbb{R}^n)}=\|\nabla^k f\|_{L^p(\mathbb{R}^n)}<\infty\big\},\\[2pt] & D^k(\mathbb{R}^n)=D^{k,2}(\mathbb{R}^n),   \quad \|f\|_{D^{k,p}}=\|f\|_{D^{k,p}(\mathbb{R}^n)},\quad  \|f\|_{D^{k}}=\|f\|_{D^{k}(\mathbb{R}^n)},\\[2pt]
&  H^{-k}(\mathbb{R}^n)=(H^k(\mathbb{R}^n))^*,\quad \|(f,g)\|_X=\|f\|_X+\|g\|_X.\end{aligned}
\end{equation}

\subsection{Global Spherically Symmetric Solutions of the Degenerate \textbf{CNS} with Far-Field Vacuum}
We first address the case that  $\bar\rho=0$ in \eqref{e1.3}. Consider the following physical range of $(\delta,\gamma)$ in system \eqref{eq:1.1}:
\begin{equation}\label{cd1}
\begin{aligned}
\delta&\in (\frac{1}{2},1),\quad &&\gamma\in (1,\infty) \ \quad &&&\text{if $n=2$},\\
\delta&\in (7-2\sqrt{10},1),\quad &&\gamma\in (1,6\delta-3) \ \quad &&&\text{if $n=3$}.    
\end{aligned}
\end{equation}

Notice that,  if  $\rho>0$, the momentum equations $\eqref{eq:1.1}_2$ can be formally rewritten as
\begin{equation}\label{qiyi}
\begin{aligned}
\boldsymbol{u}_t+\boldsymbol{u}\cdot\nabla \boldsymbol{u} +\frac{A\gamma}{\gamma-1}\nabla\rho^{\gamma-1}+ \rho^{\delta-1} L\boldsymbol{u}=\frac{\delta}{\delta-1}\nabla \rho^{\delta-1} \cdot  Q(\boldsymbol{u}),
\end{aligned}
\end{equation}
where $L \boldsymbol{u}$ and $Q(\boldsymbol{u})$ are  given by
\begin{equation}\label{operatordefinition}
L\boldsymbol{u}=-a_1\Delta \boldsymbol{u}-(a_1+a_2)\nabla \diver\boldsymbol{u},\qquad Q(\boldsymbol{u})=2a_1 D(\boldsymbol{u})+a_2\diver\boldsymbol{u}\,\mathbb{I}_n.
\end{equation}
We find that in  \eqref{qiyi}, the degeneracies both in the time evolution and spatial dissipation have been transferred to the possible singularities of $(\nabla  \rho^{\delta-1},\rho^{\delta-1}L\boldsymbol{u})$.

It is worth pointing out that, when vacuum appears in the far field,   the dynamics of the velocity field $u$ is governed by the quasilinear singular parabolic system \eqref{qiyi}--\eqref{operatordefinition}.  Compared with  the case  $\delta= 1$ studied in \cite{clz1,CZZ1,sz3},
some new obstacles for the global well-posedness theory with large initial data and far field vacuum arise:
\begin{enumerate}
\item[(i)]
The source term on the right-hand side of \eqref{qiyi} contains some singularity as
\begin{equation*}
\nabla \rho^{\delta-1}=(\delta-1)\rho^{\delta-1}\nabla \log \rho,
\end{equation*}
whose behavior is more singular than that of $\nabla \log \rho $ in \cite{clz1,CZZ1,sz3} due to $\delta-1<0$ when $\rho \to 0$ as $|\boldsymbol{x}|\to \infty$;
\smallskip
\item[(ii)]
The coefficient $\rho^{\delta-1}$ in front of the Lam\'e operator $ L\boldsymbol{u}$  tends to $\infty$ as $\rho\to 0$ in the far filed instead of equaling to $1$  in \cite{clz1,CZZ1,sz3}. Then it is necessary to show that the term $\rho^{\delta-1} Lu$ is well defined in some smooth function space.
\end{enumerate}

Therefore, the two quantities
\begin{equation*}
(\rho^{\gamma-1}, \nabla \rho^{\delta-1})
\end{equation*}
will play significant roles in our analysis on the high-order regularity of the fluid velocity 
$\boldsymbol{u}$. 
Based on the above observation, we first introduce a proper class of solutions, called regular solutions, of the Cauchy problem \eqref{eq:1.1}--\eqref{kelaoxiusi} with \eqref{eqs:CauchyInit}--\eqref{e1.3}.

\begin{mydef}\label{cjk}
The vector function $(\rho, \boldsymbol{u})$ is called a regular solution of the Cauchy problem \eqref{eq:1.1}--\eqref{kelaoxiusi} with \eqref{eqs:CauchyInit}--\eqref{e1.3}  in $[0,T]\times \mathbb{R}^n$ $(n=2$ or $3)$ for $T>0$ if $(\rho,\boldsymbol{u})$ satisfy the following properties{\rm:}
\begin{equation*}\begin{aligned}
\mathrm{(i)}& \ (\rho, \boldsymbol{u}) \ \text{satisfies this problem in the sense of distributions};\\
\mathrm{(ii)}& \ 0<\rho\in C([0,T];L^1(\mathbb{R}^n))\cap L^\infty([0,T]\times\mathbb{R}^n), \quad \nabla \rho^{\gamma-1} \in C([0,T]; H^2(\mathbb{R}^n),\\
& \ \nabla\rho^{\delta-1}\in C([0,T];L^\infty(\mathbb{R}^n)\cap D^{1,n}(\mathbb{R}^n)\cap D^2(\mathbb{R}^n)),\ \ \nabla\rho^{\frac{\delta-1}{2}}\in C([0,T];L^{2n}(\mathbb{R}^n));\\[2pt]
\mathrm{(iii)}& \ \boldsymbol{u}\in C([0,T]; H^3(\mathbb{R}^n))\cap L^2([0,T]; D^{4}(\mathbb{R}^n)),\  \ \rho^\frac{\delta-1}{2}\nabla \boldsymbol{u} \in C([0,T]; L^2(\mathbb{R}^n)),\\[2pt]
&\ \boldsymbol{u}_t\in C([0,T]; H^1(\mathbb{R}^n))\cap L^2([0,T]; D^2(\mathbb{R}^n)),\ \ \boldsymbol{u}_{tt}\in L^2([0,T]; L^2(\mathbb{R}^n)),\\[2pt]  
& \ \rho^\frac{\delta-1}{2}\nabla \boldsymbol{u}_t \in  L^\infty([0,T];L^2(\mathbb{R}^n)),\quad  \rho^{\delta-1}\nabla^2 \boldsymbol{u} \in C([0,T]; L^2(\mathbb{R}^n)),\\[2pt] 
& \ \rho^{\delta-1}\nabla^3 \boldsymbol{u} \in C([0,T]; L^2(\mathbb{R}^n)),\quad \rho^{\delta-1}\nabla^4 \boldsymbol{u} \in L^2([0,T]; L^2(\mathbb{R}^n)).
\end{aligned}
\end{equation*}
\end{mydef}

\begin{rk}\label{regularsolution}
We first introduce some physical quantities to be used in this paper{\rm :}
\begin{equation*}
\begin{aligned}
\mathcal{M}(t)&=\int_{\mathbb{R}^n} \rho(t,\boldsymbol{x})\,\mathrm{d}\boldsymbol{x}&&\qquad (\textrm{total mass}),\\[2pt]
\mathcal{P}(t)&=\int_{\mathbb{R}^n} (\rho \boldsymbol{u})(t,\boldsymbol{x})\,\mathrm{d}\boldsymbol{x} &&\qquad (\textrm{momentum}). 
\end{aligned}
\end{equation*}
It will be shown later that the regular solutions defined here satisfy the conservation of $\mathcal{M}(t)$ and $\mathcal{P}(t)$. Next, the regularity of $\rho$ shown in {\rm Definition \ref{cjk}} 
implies that $\rho>0$ in $\mathbb{R}^n$ but decays to zero 
in the far-field, 
which is consistent with the facts 
that $\mathcal{M}(t)$ and $\mathcal{P}(t)$ are both conserved 
and {\rm\textbf{CNS}} is a model of non-dilute fluids. 
Thus, the definition of regular solutions is consistent with the physical background of {\rm\textbf{CNS}}.
\end{rk}

Our main theorem is on the global well-posedness of the spherically symmetric regular solutions 
of the Cauchy problem \eqref{eq:1.1}--\eqref{kelaoxiusi} with \eqref{eqs:CauchyInit}--\eqref{e1.3}  for general initial data. The regularity of this solution ensures uniqueness in both the 2-D and 3-D cases and demonstrates that it is a classical solution of this problem.

\begin{thm}\label{th1-high}
Let $n=2$ or $3$, $\bar\rho=0$ in \eqref{e1.3}, and \eqref{coeff} and \eqref{cd1} hold. Assume that the initial data  $(\rho_0, \boldsymbol{u}_0)(\boldsymbol{x})$ are spherically symmetric and satisfy 
\begin{equation}\label{id1-high}
\begin{aligned}
&0<\rho_0\in L^1(\mathbb{R}^n),\qquad \nabla \rho_0^{\delta-1}\in  D^{1,n}(\mathbb{R}^n)\cap D^{2}(\mathbb{R}^n),\qquad \nabla\rho_0^\frac{\delta-1}{2}\in L^{2n}(\mathbb{R}^n),\\
& \nabla\rho_0^{\gamma-1}\in H^2(\mathbb{R}^n),\qquad\boldsymbol{u}_0 \in H^3(\mathbb{R}^n), 
\end{aligned}
\end{equation}
and the initial compatibility conditions 
\begin{equation}\label{th78zx}
\nabla \boldsymbol{u}_0=\rho^{\frac{1-\delta}{2}}_0 \mathcal{G}_1,\qquad
\nabla (\rho^{\delta-1}_0L\boldsymbol{u}_0)=\rho^{\frac{1-\delta}{2}}_0\mathcal{G}_2,\qquad  L\boldsymbol{u}_0= \rho^{1-\delta}_0\boldsymbol{g}_*,
\end{equation}
for some matrices $(\mathcal{G}_1,\mathcal{G}_2) \in L^2(\mathbb{R}^n)$ and a vector function  $\boldsymbol{g}_*\in L^2(\mathbb{R}^n)$. Then, for any $T>0$,  the Cauchy problem \eqref{eq:1.1}{\rm--}\eqref{kelaoxiusi} with \eqref{eqs:CauchyInit}{\rm--}\eqref{e1.3}  admits a unique global regular solution $(\rho,\boldsymbol{u})(t,\boldsymbol{x})$
in $[0,T]\times\mathbb{R}^n$ that satisfies
\begin{equation}\label{er2-high}
\begin{aligned}
&\sqrt{t}\boldsymbol{u}\in L^\infty([0,T];D^4(\mathbb{R}^n)),\quad \sqrt{t}\boldsymbol{u}_t\in L^\infty([0,T];D^2(\mathbb{R}^n)),\\
&\sqrt{t}\boldsymbol{u}_{tt}\in L^\infty([0,T];L^2(\mathbb{R}^n))\cap L^2([0,T];D^1(\mathbb{R}^n)).
\end{aligned}
\end{equation}
Moreover,  $(\rho, \boldsymbol{u})$ is spherically symmetric with the form{\rm:}
\begin{equation}\label{2.8}
(\rho,\boldsymbol{u})(t,\boldsymbol{x})=(\rho(t,r),u(t,r)\frac{\boldsymbol{x}}{r}) \qquad\text{for $r=|\boldsymbol{x}|$}
\end{equation}
and satisfies the following properties{\rm:}
\begin{itemize}
\item[$\mathrm{(i)}$] The solution we obtain here 
is a classical solution{\rm :}
\begin{equation}\label{3dclassical2}
\begin{aligned}
(\rho,\nabla \rho, \rho_t, \boldsymbol{u},\nabla \boldsymbol{u})\in C([0,T]\times \mathbb{R}^n),\qquad ( \nabla^2 \boldsymbol{u},\boldsymbol{u}_t)\in C((0,T]\times \mathbb{R}^n).
\end{aligned}
\end{equation}
\item[$\mathrm{(ii)}$] The conservations of total mass and total momentum {\rm(}that remains zero{\rm)} hold{\rm :}
\begin{equation}\label{massconservation}
\mathcal{M}(t)=\mathcal{M}(0), \quad\, \mathcal{P}(t)\equiv \boldsymbol{0} \qquad \ 
\text{for $t\in [0,T]$};
\end{equation}
\item[$\mathrm{(iii)}$] For any $T>0$ and $(t,r)\in (0,T]\times [0,\infty)$,
\begin{equation}\label{decay-est}
\begin{gathered}
\frac{C(T)^{-1}\underline{\rho}(r)}{(r^{\frac{1}{2-2\delta}}+1)(\underline{\rho}(r)+1)}\leq \rho(t,r) \leq \min\big\{C(T),Cr^{-n+1}\big\},
\end{gathered}
\end{equation}
where $\underline{\rho}(r):=\inf_{z\in [0,r]} \rho_0(z)$,
$C\geq 1$ is a constant depending only on $(\rho_0,\boldsymbol{u}_0)$ and $(n,a_1,\delta,\gamma,A)$, and $C(T)\geq 1$ is a constant depending only on $(C,T)$. 
\end{itemize}
\end{thm}

As a direct consequence of Lemma \ref{initial3} in Appendix \ref{appA}, the initial assumption \eqref{id1-high} in Theorem \ref{th1-high} can be simplified when $\gamma\geq \delta+\frac{1}{2}$.
\begin{cor}
If the constraint on $(\delta,\gamma)$ in \eqref{cd1} is replaced by 
\begin{equation*} 
\begin{aligned}
\delta&\in (\frac{1}{2},1),\quad &&\gamma\in \big[\delta+\frac{1}{2},\infty\big) \ \quad &&&\text{if $n=2$},\\
\delta&\in (\frac{7}{10},1),\quad &&\gamma\in \big[\delta+\frac{1}{2},6\delta-3\big) \ \quad &&&\text{if $n=3$},    
\end{aligned}
\end{equation*}
then the initial condition \eqref{id1-high} in {\rm Theorem \ref{th1-high}} can be reduced to
\begin{equation*}
0<\rho_0\in L^1(\mathbb{R}^n),\quad \nabla \rho_0^{\delta-1}\in  D^{1,n}(\mathbb{R}^n)\cap D^{2}(\mathbb{R}^n),\quad \nabla\rho_0^\frac{\delta-1}{2}\in L^{2n}(\mathbb{R}^n),\quad\boldsymbol{u}_0 \in H^3(\mathbb{R}^n).   
\end{equation*}
\end{cor}

We now make some remarks on the results of this paper.
\begin{rk}
On one hand, the constraints on the viscosity coefficients  given in \eqref{kelaoxiusi} and \eqref{coeff} automatically imply
\begin{equation*}
\delta\geq 1-\frac{1}{n}.
\end{equation*}

On the other hand, for technical reasons, in deriving the global uniform upper bound of density $\rho$ in {\rm\S \ref{section-upper-density}}, a crucial step is to establish the $L^{p}(I)$-energy estimates with some $p>n$ for both the velocity and the effective velocity. This requires
\begin{equation}
\delta>\frac{1}{2}\quad \text{when $n=2$},\qquad \delta>7-2\sqrt{10}\,(\approx 0.675>\frac{2}{3})\quad \text{when $n=3$}.
\end{equation}
Therefore, since we focus on the global well-posedness for {\rm\textbf{CNS}} \eqref{eq:1.1} when $\delta\in (0,1)$, the conditions of $\delta$ given in \eqref{cd1} are optimal under our methodology.
\end{rk}

\begin{rk}
The initial conditions \eqref{id1-high}--\eqref{th78zx} required  in {\rm Theorem \ref{th1-high}} identify a general class of spherically symmetric initial data, which ensures the unique solvability of the Cauchy problem \eqref{eq:1.1}--\eqref{kelaoxiusi} with \eqref{eqs:CauchyInit}--\eqref{e1.3}. For example, one can  choose $(\rho_0,\boldsymbol{u}_0)$ satisfying the following constraints{\rm:} 
$\boldsymbol{u}_0(\boldsymbol{x})=u_0(r)\frac{\boldsymbol{x}}{r} \in C_{\rm c}^\infty(\mathbb{R}^n)$, 
and $0<\rho_0(\boldsymbol{x})\in C^\infty(\mathbb{R}^n)$ 
with $\rho_0(\boldsymbol{x})=\rho_0(r)$ such that
\begin{equation*}
0<\lim_{r\to \infty}r^{\sigma}\rho_0(r)< \infty \qquad    
\text{for $\ \max\{n,\,\frac{n-2}{2\gamma-2}\}<\sigma<\frac{1}{1-\delta}$}.
\end{equation*}
\end{rk}

\begin{rk}\label{rk-nabla}
We give some comments on the  constraints for $\rho_0$ given 
in {\rm Theorem \ref{th1-high}}.
On one hand, it follows from \eqref{id1-high}--\eqref{th78zx} and {\rm Lemma \ref{initial3}} that $\rho_0\in L^\infty(\mathbb{R}^n).$
On the other hand, the boundedness of the effective velocity 
$\boldsymbol{v}=\boldsymbol{u}+\frac{2a_1\delta}{\delta-1} \nabla\rho^{\delta-1}$ plays a key role in our analysis, 
which requires that $\boldsymbol{v}(0,\boldsymbol{x})\in L^\infty(\mathbb{R}^n)$. 
Since $\boldsymbol{u}_0\in L^\infty(\mathbb{R}^n)$ holds by classical Sobolev embedding theorems, we still need 
$\nabla\rho_0^{\delta-1} \in L^\infty(\mathbb{R}^n)$. This follows directly from $\nabla \rho_0^{\delta-1}\in D^{1,n}(\mathbb{R}^n)$ in \eqref{id1-high} due to the critical Sobolev embedding $D^{1,n}(\mathbb{R}^n) \hookrightarrow L^{\infty}(\mathbb{R}^n)$ for spherically symmetric vector functions {\rm(}see {\rm Lemma \ref{Hk-Ck-vector}} 
in {\rm Appendix \ref{improve-sobolev}}{\rm)}.

\end{rk}

\begin{rk}
The distinction of the range of $\gamma$ for the $2$-{\rm D} case and $3$-{\rm D} case stems from the factor $r^\frac{m}{2}$ $(m=n-1)$ in the BD entropy estimates under the spherical coordinates $($see {\rm Lemma \ref{energy-BD}} in {\rm \S\ref{section-upper-density}}$)${\rm :} 
\begin{equation}\label{bd135}
\sup_{t\in[0,T]}\big\|r^{\frac{m}{2}}(\rho^{\delta-\frac{1}{2}})_r\big\|_{L^2(I)} \leq C_0, 
\end{equation}
where $I=[0,\infty)$ and $C_0>0$ is a constant depending only on $(\rho_0,u_0,a_1,\delta,\gamma,A,n)$. Notably, \eqref{bd135} is some $r$-weighted $L^2(I)$-estimate of $(\rho^{\delta-\frac{1}{2}})_r$, 
and the power of $r$ in the $3$-{\rm D} case is larger than 
that in the $2$-{\rm D} case.  

In fact, in {\rm \S\ref{section-upper-density}}, a crucial step in establishing the global uniform upper bound for $\rho$ is to derive some new $r$-weighted estimates for $\rho$ near the origin via the Hardy inequality $($see {\rm Lemma \ref{hardy}} in {\rm Appendix \ref{appA}}$)$ and the BD entropy estimates \eqref{bd135}{\rm :} 
\begin{equation}\label{314}
\sup_{t\in[0,T]}\|r^K\rho\|_{L^p(0,1)} \leq C \qquad\, \text{for $p\in [1,\infty]$},
\end{equation}
where $C>0$ is a constant depending on $(C_0,p,K)$, and exponent $K$ satisfies
\begin{equation*}
K > -\frac{1}{p} \quad \text{if $n=2$},\qquad K\geq \frac{1}{2\delta-1}-\frac{1}{p}\quad \text{if $n=3$}.
\end{equation*}
Clearly, the range of $K$ is smaller when $n=3$, and this is main reason that narrows the admissible range of $\gamma$ in this case. More detailed calculations can be found in {\rm \S \ref{subsection4.2}}.
\end{rk}

\smallskip
\subsection{Global Spherically Symmetric Solutions of the Degenerate \textbf{CNS} with Strictly Positive Initial Density} 
As a direct application of our approach, we can establish the global well-posedness of regular solutions with general data when $\bar\rho>0$ in \eqref{e1.3}. We first define the corresponding regular solutions of the Cauchy problem \eqref{eq:1.1}--\eqref{kelaoxiusi} with \eqref{eqs:CauchyInit}--\eqref{e1.3}.

\begin{mydef}\label{cjk-po}
Let  $\bar\rho>0$ in \eqref{e1.3} and $T>0$. 
The vector function $(\rho, \boldsymbol{u})$ is called a regular solution of the Cauchy problem \eqref{eq:1.1}{\rm--}\eqref{kelaoxiusi} with \eqref{eqs:CauchyInit}{\rm--}\eqref{e1.3}  in $[0,T]\times \mathbb{R}^n$ $(n=2$ or $3)$ if $(\rho,\boldsymbol{u})$ satisfy the following properties{\rm:}
\begin{equation*}\begin{aligned}
\mathrm{(i)}& \ (\rho, \boldsymbol{u}) \ \text{satisfies this problem in the sense of distributions};\\
\mathrm{(ii)}& \ 
\inf_{(t,\boldsymbol{x})\in [0,T]\times\mathbb{R}^n} \rho(t,\boldsymbol{x})>0,\quad  
\rho-\bar\rho\in C([0,T];H^3(\mathbb{R}^n));\\[2pt]
\mathrm{(iii)}& \ \boldsymbol{u}\in C([0,T]; H^3(\mathbb{R}^n))\cap L^2([0,T]; D^{4}(\mathbb{R}^n)),\\[2pt]
& \ \boldsymbol{u}_t\in C([0,T]; H^{1}(\mathbb{R}^n))\cap L^2([0,T]; D^{2}(\mathbb{R}^n)).
\end{aligned}
\end{equation*}
\end{mydef}

When the initial density is strictly positive, our main result on the global well-posedness 
of the  regular solutions of the Cauchy problem \eqref{eq:1.1}--\eqref{kelaoxiusi} with \eqref{eqs:CauchyInit}--\eqref{e1.3} is  stated as follows:
\begin{thm}\label{thm-positive}
Let $n=2$ or $3$, $\bar\rho>0$ in \eqref{e1.3}, and
\eqref{coeff} and\eqref{cd1} hold. Assume that the initial data  $(\rho_0, \boldsymbol{u}_0)(\boldsymbol{x})$ are spherically symmetric and satisfy 
\begin{equation}\label{id1-high-positive}
\begin{aligned}
&\inf_{\boldsymbol{x}\in \mathbb{R}^n}\rho_0(\boldsymbol{x})>0,\qquad (\rho_0-\bar\rho,\boldsymbol{u}_0) \in H^3(\mathbb{R}^n). 
\end{aligned}
\end{equation}
Then, for any $T>0$,  the Cauchy problem \eqref{eq:1.1}{\rm--}\eqref{kelaoxiusi} with \eqref{eqs:CauchyInit}{\rm--}\eqref{e1.3}  admits a unique global regular solution $(\rho,\boldsymbol{u})(t,\boldsymbol{x})$
in $[0,T]\times\mathbb{R}^n$ that satisfies \eqref{er2-high}, \eqref{3dclassical2}, and \begin{equation}\label{decay-est-positive}
C(T)^{-1}\leq \rho(t,\boldsymbol{x})\leq C(T) \qquad \text{for all $(t,\boldsymbol{x})\in [0,T]\times \mathbb{R}^n$},
\end{equation}
where $C(T)\geq 1$ is a constant depending only on $(T,\rho_0,\boldsymbol{u}_0,n,a_1,\delta,\gamma,A)$. Moreover, $(\rho, \boldsymbol{u})$ is spherically symmetric with form \eqref{2.8}.
\end{thm}

\section{Notations, Reformulations, and Main Strategies}\label{Section2}
In this section, we first present some notations and conventions in \S\ref{section-notaions}, which are frequently used throughout this paper. Next, in \S\ref{see1}, we introduce an enlarged reformulation for the degenerate \textbf{CNS} \eqref{eq:1.1}   to deal with the degeneracy caused by the far-field vacuum. In \S\ref{subsection-strategy}, based on such a reformulation,  we show the main strategy and new ideas in our analysis.  

\subsection{Notations}\label{section-notaions}
Throughout the rest of this paper, 
unless otherwise specified, we adopt the notations in \eqref{eulerspace} and the following ones.
\subsubsection{Notations in M-D Eulerian coordinates} 
\begin{itemize}
\smallskip
\item We always let $n=2$ or $3$ be the dimension of the Euclidean  space $\mathbb{R}^n$, and denote  $m:=n-1$.
\smallskip
\item For a variable $\boldsymbol{y} \in \mathbb{R}^l$ ($l\geq 2$), its $i$-th component is denoted by $y_i$ ($1\leq i\leq l$), and $\boldsymbol{y} = (y_1, \cdots\!, y_l)^\top$.  We always let $\boldsymbol{x}=(x_1,\cdots\!,x_n)^\top$ be the spatial variable of $\mathbb{R}^n$.
\smallskip
\item For any vector function $\boldsymbol{f}: E\subset \mathbb{R}^l \to \mathbb{R}^q$ ($l,q\geq 2$, $E$ is a measurable set), its $i$-th component is denoted by $f_i$ ($1\leq i\leq q$), and $\boldsymbol{f} = (f_1, \cdots\!, f_q)^\top$.
\smallskip
\item For any function $f$ defined on a measurable subset of $\mathbb{R}^l$ ($l\geq 1$), if the independent variable of $f$ is $\boldsymbol{y}=(y_1,\cdots\!,y_l)^\top$, then 
\begin{align*}
&\qquad \partial_{\boldsymbol{y}}^{\boldsymbol{\varsigma}} f=\partial_{y_1}^{\varsigma_1}\cdots\partial_{y_l}^{\varsigma_l} f=f_{\underbrace{\text{\tiny$y_1\cdots y_1$}}_{\text{$\varsigma_1$-times}}\cdots\underbrace{\text{\tiny$y_l\cdots y_l$}}_{\text{$\varsigma_l$-times}}}= \frac{\partial^{\varsigma_1+\cdots+ \varsigma_l}}{\partial y_1^{\varsigma_1}\cdots \partial y_l^{\varsigma_l}}f \quad\ \text{for }\boldsymbol{\varsigma}=(\varsigma_1,\cdots\!,\varsigma_l)\in \mathbb{N}^l,\\[-7pt]
&\qquad \nabla_{\boldsymbol{y}} f=(\partial_{y_1} f,\cdots\!,\partial_{y_l} f)^\top,\qquad \Delta_{\boldsymbol{y}} f=\sum_{i=1}^l \partial_{y_i}^2 f,\\[-7pt]
&\qquad \nabla_{\boldsymbol{y}}^k f \text{ denotes one generic } \partial_{\boldsymbol{y}}^{\boldsymbol{\varsigma}} f \text{ with }|\boldsymbol{\varsigma}|=\sum_{i=1}^l \varsigma_i=k \text{ for integer }k\geq 2,\\[-7pt]
&\qquad |\nabla_{\boldsymbol{y}}^k f|=\Big(\sum_{|\varsigma|=k}|\partial_{y_1}^{\varsigma_1} \cdots\partial^{\varsigma_l}_{y_l}f|^2\Big)^\frac{1}{2} \quad\text{ for } k\in \mathbb{N}^*.
\end{align*}
In particular, for the derivatives with respect to the
variable $\boldsymbol{x}=(x_1,\cdots\!,x_n)^\top\in \mathbb{R}^n$, we use the notations 
$(\partial_i^{\varsigma_i},\partial^{\boldsymbol{\varsigma}},\nabla,\Delta,\nabla^k)=(\partial_{x_i}^{\varsigma_i}, \partial_{\boldsymbol{x}}^{\boldsymbol{\varsigma}},\nabla_{\boldsymbol{x}},\Delta_{\boldsymbol{x}},\nabla_{\boldsymbol{x}}^k)$.

\smallskip
\item If $\boldsymbol{f}: E\subset \mathbb{R}^l \to \mathbb{R}^q$ ($l,q\geq 2$, $E$ is a measurable set) is a vector function with the independent variable $\boldsymbol{y}=(y_1,\cdots\!,y_l)^\top$ and $X \in \{\partial_{y_i}^{\varsigma_i},\partial_{\boldsymbol{y}}^\varsigma,\Delta_{\boldsymbol{y}},\nabla_{\boldsymbol{y}}^k\}$, then 
\begin{align*}
&X\boldsymbol{f}=\big(Xf_1,\cdots\!,Xf_q\big)^\top, \ \quad \nabla_{\boldsymbol{y}}\boldsymbol{f}=\begin{pmatrix} 
\partial_{y_1} f_1 & \partial_{y_2} f_1 & \cdots & \partial_{y_l} f_1\\[2mm]
\partial_{y_1} f_2 & \partial_{y_2} f_2 & \cdots & \partial_{y_l} f_2 \\[2mm]
\vdots & \vdots & \ddots & \vdots \\[2mm]
\partial_{y_1} f_q & \partial_{y_2} f_q & \cdots & \partial_{y_l} f_q
\end{pmatrix}_{q\times l},\\
&|\nabla_{\boldsymbol{y}}^k \boldsymbol{f}|=\Big(\sum_{i=1}^q\sum_{|\varsigma|=k}\big|\partial_{y_1}^{\varsigma_1} \cdots\partial^{\varsigma_l}_{y_l}f_{i}\big|^2\Big)^\frac{1}{2} \quad \text{ for } k\in \mathbb{N}^*.
\end{align*}
Moreover, if $l=j+i$ with $j\geq 0$ and the independent variable $\boldsymbol{y}$ takes the form $\boldsymbol{y}=(\boldsymbol{s},\boldsymbol{\tilde{y}})^\top$ with $\boldsymbol{s}=(s_1,\cdots\!,s_j)^\top$ and $\boldsymbol{\tilde{y}}=(\tilde{y}_1,\cdots\!,\tilde{y}_i)^\top$, then
\begin{equation*}
\diver_{\boldsymbol{\tilde{y}}} \boldsymbol{f}=\sum_{k=1}^i \partial_{\tilde{y}_k}f_k.  
\end{equation*}
In particular, if $j=0,1$, $\,i=n$, and $\boldsymbol{\tilde{y}}=\boldsymbol{x}=(x_1,\cdots\!,x_n)^\top\in \mathbb{R}^n$, then $\diver=\diver_{\boldsymbol{x}}$.
\end{itemize}

\subsubsection{Notations in M-D spherical  coordinates}\label{Nor}
\begin{itemize}
\item We always let $r=|\boldsymbol{x}|$ be the radial distance in spherical coordinates, and let $r\in I:=[0,\infty)$. Moreover,  the following conventions are adapted:
\begin{align*}
|f|_p=\|f\|_{L^p(I)},\quad   \|f\|_k=\|f\|_{H^k(I)},\quad  \|f\|_{k,p}=\|f\|_{W^{k,p}(I)}
\qquad \,\,\text{for $k\in \mathbb{N}^*$}.
\end{align*}
In particular, for the Sobolev spaces defined on the open interval $(a,b)\subset I$, we use the abbreviation $X(a,b)=X((a,b))$ for $X=L^p$, $W^{k,p}$, and $H^k$.
\item To simplify the notations, in the rest of the paper, we define the operator $\mathrm{D}_r$ as 
\begin{equation*}
\qquad \begin{aligned}
\mathrm{D}_rf&:=(f_r,\frac{f}{r}),\qquad \mathrm{D}_r^2f:=(f_{rr},(\frac{f}{r})_r),\\
\mathrm{D}_r^3f&:=(f_{rrr},\frac{f_{rr}}{r},(\frac{f}{r})_{rr},\frac{1}{r}(\frac{f}{r})_r),\qquad
\mathrm{D}_r^4f :=(f_{rrrr},(\frac{f_{rr}}{r})_r,(\frac{f}{r})_{rrr},(\frac{1}{r}(\frac{f}{r})_r)_r).
\end{aligned}
\end{equation*}
Notice from Lemma \ref{lemma-initial}  in Appendix \ref{appb} that the operator $\mathrm{D}_r$ can be formally seen as a representative of the gradient $\nabla$ in the radial coordinate. Besides, the following useful properties will be used in later analysis: for $f=f(t,r)$,
\begin{equation*}
\begin{aligned}
&(\mathrm{D}_r f)_r=\mathrm{D}_r^2 f,\qquad (\mathrm{D}_r^3 f)_r=\mathrm{D}_r^4 f,\qquad \mathrm{D}_r^j f_t=(\mathrm{D}_r^j f)_t \quad \text{for $j=1,2,3,4$},\\
&|f_r|^2+\Big|\frac{f}{r}\Big|^2=|\mathrm{D}_r f|^2,\qquad |f_{rr}|^2+\Big|(\frac{f}{r})_r\Big|^2=|\mathrm{D}_r^2 f|^2,\\
&|(\mathrm{D}_r^2 f)_r|^2+\Big|\frac{\mathrm{D}_r^2 f}{r}\Big|^2=|\mathrm{D}_r^3 f|^2,\qquad |(\mathrm{D}_r^2 f)_{rr}|^2+\Big|\big(\frac{\mathrm{D}_r^2 f}{r}\big)_r\Big|^2=|\mathrm{D}_r^4 f|^2.
\end{aligned}
\end{equation*} 
\end{itemize}

\subsubsection{Other notations and conventions}\label{othernote}
\begin{itemize}\smallskip
\item $C^\ell(\overline\Omega)$ $(\ell\in \mathbb{N},\,C(\overline\Omega)=C^0(\overline\Omega))$ 
denotes the space of all functions $f$ for which $\nabla^j f$ $(0\leq j\leq \ell)$ is bounded and uniformly continuous in $\Omega\subset \mathbb{R}^n$, which is equipped with the norm:
\begin{equation*}
\|f\|_{C^\ell(\overline{\Omega})}:= \max_{0\leq j\leq \ell}\|\nabla^j f\|_{L^\infty(\Omega)}.
\end{equation*}
In particular, if $\Omega=\mathbb{R}^n$, $C^\ell(\overline{\mathbb{R}^n})$ denotes the space of all functions $f$ for which $\nabla^j f$ $(0\leq j\leq \ell)$ is bounded and uniformly continuous in $\mathbb{R}^n$.

\smallskip
\item $C^\infty_{\rm c}(\Omega)$ denotes the space of all functions $f$ for which $\nabla^j f$ $(j\in \mathbb{N})$ is continuous and compactly supported in $\Omega\subset \mathbb{R}^n$.

\smallskip
\item For any function spaces $(X,X_1,\cdots\!,X_k)$ and functions $(h,f,f_1,\cdots\!,f_k)$,
\begin{equation*}
\|f\|_{X_1\cap\cdots\cap X_k}:=\sum_{i=1}^k\|f\|_{X_i},\qquad \|h(f_1,\cdots\!,f_k)\|_{X}:=\sum_{i=1}^k\|hf_i\|_X.
\end{equation*}

\smallskip
\item For any $n\times n$ real matrix $\mathcal{S}$, $\mathcal{S}_{ij}$ denotes its $(i,j)$-th entry. $\mathcal{S}:\mathcal{N}:=\sum_{i,j=1}^n \mathcal{S}_{ij}\mathcal{N}_{ij}$ for any $n\times n$ matrices $(\mathcal{S},\mathcal{N})$.  Moreover, $\mathrm{SO}(n)$  denotes the set of all $n\times n$ real orthogonal matrices $\mathcal{O}$ such that $\det \mathcal{O}=1$, where $\det \mathcal{O}$ is the determinant of $\mathcal{O}$.

\smallskip
\item $\delta_{ij}$  is the Kronecker symbol satisfying $\delta_{ij} =1$ when $i = j$, 
and $\delta_{ij} =0$ otherwise.

\smallskip
\item $\langle\cdot,\cdot\rangle_{X^*\times X}$ denotes the pairing between the space $X$ and its dual space $X^*$. 
\smallskip

\item To integrate the estimates for both the $2$-D case and $3$-D case, we introduce the following parameter $n^*$:
\begin{equation*} 
n^*:=\frac{2n}{n-2}=\begin{cases}
\infty&\quad \text{if $n=2$},\\
6&\quad \text{if $n=3$}.
\end{cases}
\end{equation*}
\end{itemize}

\subsection{An Enlarged Reformulation}\label{see1}

We first make an enlarged  reformulation for the Cauchy problem \eqref{eq:1.1}--\eqref{kelaoxiusi} with \eqref{eqs:CauchyInit}--\eqref{e1.3}. In terms of variables
\begin{equation}\label{bianliang}
\phi=\frac{A\gamma}{\gamma-1} \rho^{\gamma-1},\qquad \boldsymbol{\psi}=\frac{\delta}{\delta-1}\nabla \rho^{\delta-1}=\frac{\delta}{\delta-1}\big(\frac{A\gamma}{\gamma-1}\big)^{\frac{1-\delta}{\gamma-1}}\nabla \phi^{\frac{\delta-1}{\gamma-1}}=(\psi_{1},\psi_{2},\psi_{3})^\top,
\end{equation}
and $\boldsymbol{u}$, system (\ref{eq:1.1})  can be rewritten as
\begin{equation}\label{eq:cccq}
\begin{cases}
\displaystyle
\phi_t+\boldsymbol{u}\cdot \nabla \phi+(\gamma-1)\phi \diver \boldsymbol{u}=0,\\[4pt]
\displaystyle
\boldsymbol{u}_t+\boldsymbol{u}\cdot\nabla \boldsymbol{u} +\nabla \phi+a\phi^{2\iota}L\boldsymbol{u}=\boldsymbol{\psi} \cdot Q(\boldsymbol{u}),\\[4pt]
\displaystyle
\boldsymbol{\psi}_t+\nabla (\boldsymbol{u}\cdot \boldsymbol{\psi})+(\delta-1)\boldsymbol{\psi}\diver \boldsymbol{u} +\delta a\phi^{2\iota}\nabla \diver \boldsymbol{u}=\boldsymbol{0},
 \end{cases}
\end{equation}
where $L \boldsymbol{u}$ and $Q(\boldsymbol{u})$ are  defined in \eqref{operatordefinition}, and 
\begin{equation} \label{xishu}
\begin{aligned}
& a=\big(\frac{A\gamma}{\gamma-1}\big)^{\frac{1-\delta}{\gamma-1}} \qquad \text{and} \quad \iota=\frac{\delta-1}{2(\gamma-1)}<0.
\end{aligned}
\end{equation}

We aim to establish global spherically symmetric solutions of (3.2) with the form:
\begin{equation}\label{2.}
(\phi, \boldsymbol{u}, \boldsymbol{\psi})(t, \boldsymbol{x})=(\phi(t,|\boldsymbol{x}|), u(t,|\boldsymbol{x}|) \frac{\boldsymbol{x}}{|\boldsymbol{x}|}, \psi(t,|\boldsymbol{x}|) \frac{\boldsymbol{x}}{|\boldsymbol{x}|}),
\end{equation}
with the initial data:
\begin{equation}\label{sfana1}
\begin{aligned}
(\phi, \boldsymbol{u}, \boldsymbol{\psi})(0, \boldsymbol{x}) =\left(\phi_0, \boldsymbol{u}_0, \boldsymbol{\psi}_0\right)(\boldsymbol{x}) :=(\phi_0(|\boldsymbol{x}|), u_0(|\boldsymbol{x}|) \frac{\boldsymbol{x}}{|\boldsymbol{x}|}, \frac{a\delta}{\delta-1}\nabla\phi_0^{2\iota}(|\boldsymbol{x}|) \frac{\boldsymbol{x}}{|\boldsymbol{x}|})  
\end{aligned}
\end{equation}
for $\boldsymbol{x} \in \mathbb{R}^n$ and the far-field asymptotic condition:
\begin{equation}\label{sfanb1}
(\phi, \boldsymbol{u})\to (0,\boldsymbol{0})\qquad \text{as $|\boldsymbol{x}|\to \infty$ \ for $t \geq 0$}.
\end{equation}

Then we define the corresponding regular solutions of the Cauchy problem \eqref{eq:cccq}--\eqref{sfanb1}.
\begin{mydef}\label{cjk-reform}
Let $T>0$. The vector function $(\phi, \boldsymbol{u},\boldsymbol{\psi})$ is called a regular solution of the Cauchy problem \eqref{eq:cccq}{\rm--}\eqref{sfanb1}  in $[0,T]\times \mathbb{R}^n$ $(n=2$ or $3)$ if $(\phi, \boldsymbol{u},\boldsymbol{\psi})$ satisfy the following properties{\rm:}
\begin{equation*}\begin{aligned}
\mathrm{(i)}& \ (\phi, \boldsymbol{u},\boldsymbol{\psi}) \ \text{satisfies this problem in the sense of distributions};\\
\mathrm{(ii)}& \ 0<\phi^\frac{1}{\gamma-1}\in C([0,T];L^1(\mathbb{R}^n))\cap L^\infty([0,T]\times \mathbb{R}^n), 
\quad   \phi \in C([0,T]; D^1(\mathbb{R}^n)\cap D^3(\mathbb{R}^n)),\\[2pt]
& \ \boldsymbol{\psi}\in C([0,T]; L^\infty(\mathbb{R}^n)\cap D^{1,n}(\mathbb{R}^n)\cap D^2(\mathbb{R}^n)),\ \ \nabla\phi^\iota\in C([0,T]; L^{2n}(\mathbb{R}^n));\\[2pt]
\mathrm{(iii)}& \ \boldsymbol{u}\in C([0,T]; H^3(\mathbb{R}^n))\cap L^2([0,T]; D^{4}(\mathbb{R}^n)),\ \  \phi^{\iota}\nabla \boldsymbol{u} \in C([0,T]; L^2(\mathbb{R}^n)),\\[2pt]
& \ \boldsymbol{u}_t\in C([0,T]; H^1(\mathbb{R}^n))\cap L^2([0,T]; D^2(\mathbb{R}^n)),\quad    \boldsymbol{u}_{tt} \in L^2([0,T]; L^2(\mathbb{R}^n)), \\[2pt]
& \ \phi^\iota\nabla \boldsymbol{u}_t \in  L^\infty([0,T];L^2(\mathbb{R}^n)),\quad   \phi^{2\iota}\nabla^2 \boldsymbol{u} \in C([0,T]; L^2(\mathbb{R}^n)),\\[2pt] 
& \ \phi^{2\iota}\nabla^3 \boldsymbol{u} \in C([0,T]; L^2(\mathbb{R}^n)),\quad \phi^{2\iota}\nabla^4 \boldsymbol{u} \in L^2([0,T]; L^2(\mathbb{R}^n)).
\end{aligned}
\end{equation*}
\end{mydef}

Since $\partial_{x_i} \psi_j=\partial_{x_j} \psi_i$ ($i,j=1,\cdots\!,n$), equations $\eqref{eq:cccq}_3$ can be rewritten as
\begin{equation}\label{kuzxc}
\boldsymbol{\psi}_t+\sum_{l=1}^n \mathcal{A}_l \partial_l\boldsymbol{\psi}+\mathcal{B}\boldsymbol{\psi}+\delta a\phi^{2\iota}\nabla \diver \boldsymbol{u}=\boldsymbol{0},
\end{equation}
where $\mathcal{A}_l=(a^l_{ij})_{n\times n}$ for  $i,j,l=1,\cdots\!,n$,
are symmetric matrices with $a^l_{ij}=u_l$ for $i=j$ and otherwise $a^l_{ij}=0$,  and $\mathcal{B}=(\nabla u)^\top+(\delta-1)\diver u\,\mathbb{I}_n$. This implies that the source term $\boldsymbol{\psi}$ can be controlled by the symmetric hyperbolic system \eqref{kuzxc}.

Note that  the enlarged system \eqref{eq:cccq}   consists of (up to leading order) 
\begin{itemize}
\item {\it scalar transport} equation $\eqref{eq:cccq}_1$ for $\phi$;
\smallskip
\item {\it singular parabolic}  system  $\eqref{eq:cccq}_2$ for the velocity $\boldsymbol{u}$;
\smallskip
\item {\it symmetric hyperbolic} system  $\eqref{eq:cccq}_4$ but with several singular source terms for $\boldsymbol{\psi}$,
\end{itemize}
such a structure enables us to  establish  the following local well-posedness  of the Cauchy problem \eqref{eq:cccq}{\rm--}\eqref{sfanb1}.

\begin{thm}\label{th1}
Let $n=2$ or $3$, and let $(\delta,\gamma)$ satisfy
\begin{equation}\label{canshu}
\gamma\in (1,\infty),\qquad\delta\in (1-\frac{1}{n},1).    
\end{equation}
Assume that  the initial data $(\phi_0, \boldsymbol{u}_0,\boldsymbol{\psi}_0)$ satisfy 
\begin{equation}\label{th78qq}
\begin{aligned}
&0<\phi_0^\frac{1}{\gamma-1}\in L^1(\mathbb{R}^n),\quad \nabla\phi_0\in H^2 (\mathbb{R}^n),\quad \boldsymbol{\psi}_0 \in  D^{1,n}(\mathbb{R}^n)\cap  D^2(\mathbb{R}^n),\\
&\nabla\phi_0^\iota\in L^{2n}(\mathbb{R}^n),\qquad \boldsymbol{u}_0\in H^3(\mathbb{R}^n),
\end{aligned}
\end{equation}
and the initial  compatibility conditions
\begin{equation}\label{th78zxq}
\nabla \boldsymbol{u}_0=\phi^{-\iota}_0 \mathcal{G}_1,\qquad
\nabla (\phi^{2\iota}_0L\boldsymbol{u}_0)=\phi^{-\iota}_0\mathcal{G}_2,\qquad  L\boldsymbol{u}_0= \phi^{-2\iota}_0\boldsymbol{g}_*,
\end{equation}
for some matrices $(\mathcal{G}_1,\mathcal{G}_2) \in L^2(\mathbb{R}^n)$ and a vector function  $\boldsymbol{g}_*\in L^2(\mathbb{R}^n)$. Then there exists $T_*>0$ such that the Cauchy problem \eqref{eq:cccq}{\rm--}\eqref{sfanb1} admits a unique regular solution $(\phi, \boldsymbol{u},\boldsymbol{\psi})(t,\boldsymbol{x})$  in $[0,T_*]\times \mathbb{R}^n$  satisfying \eqref{er2-high} with $T$ replaced by $T_*$, and
\begin{equation}
\boldsymbol{\psi}=\frac{a\delta}{\delta-1}\nabla \phi^{2\iota} \qquad\text{for \textit{a.e.} $(t,\boldsymbol{x})\in (0,T_*)\times \mathbb{R}^n$}.
\end{equation}
Moreover, $(\phi, \boldsymbol{u},\boldsymbol{\psi})$ is spherically symmetric with form \eqref{2.}.
\end{thm}

The proofs for Theorem \ref{th1} are given in \S \ref{section-local-regular}. Moreover, at the end of \S \ref{section-local-regular}, we show that Theorem \ref{th78qq} indeed implies the local well-posedness of regular solutions of the Cauchy problem \eqref{eq:1.1}--\eqref{kelaoxiusi} and \eqref{eqs:CauchyInit}--\eqref{e1.3}  with general smooth, spherically symmetric initidata and far-field vacuum (\textit{i.e.}, $\bar{\rho}=0$), which are stated in Theorem \ref{thm-loc} below.
\begin{thm}\label{thm-loc}
Let $n=2$ or $3$, and \eqref{canshu} hold. If the initial data $(\rho_0, \boldsymbol{u}_0)(\boldsymbol{x})$ are spherically symmetric, and satisfy \eqref{id1-high}--\eqref{th78zx}, then there exists $T_*>0$ such that the Cauchy problem \eqref{eq:1.1}{\rm--}\eqref{kelaoxiusi} with \eqref{eqs:CauchyInit}{\rm--}\eqref{e1.3}  admits a unique regular solution $(\rho, \boldsymbol{u})(t, \boldsymbol{x})$ in $[0, T_*] \times \mathbb{R}^n$ satisfying \eqref{er2-high} with $T$ replaced by $T_*$, and
\begin{equation*}
(\rho^{\gamma-1})_{t t} \in C([0, T_*] ; L^2(\mathbb{R}^n)) \cap L^2([0,T_*];D^1(\mathbb{R}^n)), \quad(\nabla \rho^{\delta-1})_{tt} \in L^2([0, T_*];L^2(\mathbb{R}^n)) .    
\end{equation*}
Moreover, $(\rho, \boldsymbol{u})$ is spherically symmetric with form \eqref{2.8}.    
\end{thm}
  
\subsection{Main Strategies}\label{subsection-strategy}
We outline the proof for the global well-posedness of the Cauchy problem \eqref{eq:1.1}--\eqref{kelaoxiusi} with 
\eqref{eqs:CauchyInit}--\eqref{coeff},  including the global upper bound of the density, the $L^\infty(I)$-estimate of the effective velocity, the non-formation of vacuum inside the fluids, and the uniform estimates for the regular solutions. For brevity, we only sketch the proof for the $3$-D case; the $2$-D case can be treated in a similar way.

\subsubsection{Global uniform upper bound of $\rho$}
The first major task is to obtain the global uniform upper bound of density $\rho$. Due to the coordinate singularity at the origin, we divide the proof into two parts: the interior region $B_1=\{\boldsymbol{x}:\,|\boldsymbol{x}|<1\}$ and the exterior region $\mathbb{R}^3\backslash B_1$. 

In the exterior domain, by the Sobolev embedding $H^1(1,\infty)\hookrightarrow L^\infty(1,\infty)$, the conservation of total mass, and the BD entropy estimates (see Lemma \ref{energy-BD}), we obtain
\begin{equation}\label{01}
\|\rho\|_{L^\infty([0,T]\times [1,\infty))}\leq C_0.
\end{equation}

In the interior region $B_1$, the analysis is more subtle due to the geometric singularity. To overcome this difficulty, we introduce the effective velocity $\boldsymbol{v}=\boldsymbol{u}+\frac{2a_1\delta}{\delta-1}\nabla\rho^{\delta-1}$ (its radial projection is $v=u+\frac{2a_1\delta}{\delta-1}(\rho^{\delta-1})_r$) and employ a combination of the Hardy inequality and weighted estimates for $\rho$. First, due to  $W^{1,p}(\mathbb{R}^3)\hookrightarrow L^\infty(\mathbb{R}^3)$ for $p>3$, we observe
\begin{equation}\label{000}
\begin{aligned}
\|\rho^{\frac{1}{p}+\delta-1}\|_{L^\infty(B_1)}&\leq C(p)\|\rho^{\frac{1}{p}+\delta-1}\|_{W^{1,p}(B_1)} \leq C(p)\|\rho\|_{L^{p\delta-p+1}(B_1)}^{\delta-1+\frac{1}{p}}+C(p)\big\|\rho^{\frac{1}{p}}(u,v)\big\|_{L^p}\\
&\leq C(p)\|r^\frac{2}{p\delta-p+1}\rho\|_{L^{p\delta-p+1}(0,1)}^{\delta-1+\frac{1}{p}}+C(p)\big |(r^2\rho)^{\frac{1}{p}}(u,v)\big |_{p}:=J_*+J_{**}.
\end{aligned}   
\end{equation}

To control the right-hand side of \eqref{000}, we derive $r$-weighted $L^q$-estimates for $\rho$ via the Hardy inequality and the BD entropy estimate $\nabla\rho^{\delta-\frac{1}{2}}\in L^\infty([0,T];L^2(\mathbb{R}^n))$:
\begin{equation}\label{02}
\big\|r^{\frac{1}{2\delta-1}-\frac{1}{q}}\rho\big\|_{L^\infty([0,T];L^q(0,1))}\leq C(q) \qquad \text{for $q\in [2\delta-1,\infty]$}.
\end{equation}
Then, if $p\in (3,\frac{1}{1-\delta})$ (since $\delta>\frac{2}{3}$), $J_*$ can be controlled by using the H\"older inequality and \eqref{02}. To estimate $J_{**}$, we establish the $L^p(I)$-energy estimates for both $u$ and $v$ via multiplying $\eqref{e1.5hh}_2$ and equation \eqref{eq:effective2} by $r^2|u|^{p-2}u$ and $r^2\rho|v|^{p-2}v$, respectively. By handling the exterior part via \eqref{01}, and treating the interior part via \eqref{02} and an iterative scheme involving the Hardy inequality, we arrive at the estimates of the form:
\begin{equation}\label{--3}
\big\|(r^2\rho)^{\frac{1}{p}}(u,v)\big\|_{L^\infty([0,T];L^p(I))}+ \big\|(r^2\rho^\delta)^\frac{1}{2}|u|^\frac{p-2}{2}\mathrm{D}_ru\big\|_{L^p([0,T];L^p(I))} \leq C(p,T) 
\end{equation}
for any $p\in [2,\tilde{p}(\delta))$. Here, $\tilde{p}(\delta)$ is a parameter that depends only  on $\delta$, which is strictly larger than $3$ whenever $\delta>7-2\sqrt{10}$ (for $2$-D case, $\tilde{p}(\delta)>2$ requires $\delta>\frac{1}{2}$). Hence,  \eqref{000}, combined with \eqref{02}--\eqref{--3}, yields  the uniform upper bound of $\rho$ in  $B_1$.

\subsubsection{Global $L^\infty(I)$-estimate for the effective velocity.}
With the uniform upper bound of $\rho$ at hand, we proceed to control the $L^\infty(I)$-bound for $v$. From the evolution equation \eqref{eq:effective2} and the method of characteristics, we obtain  
\begin{equation}\label{05}
\|v\|_{L^\infty([0,T]\times I)}\leq |v_0|_{\infty}+\frac{A\gamma}{2a_1\delta}\|\rho^{\gamma-\delta}u\|_{L^1([0,T];L^\infty(I))} \leq C_0+C(T)\|\rho^{1-\delta}u\|_{L^2([0,T];L^\infty(I))}^\frac{1}{2}.
\end{equation}

To estimate the integral term above, we first consider the $L^2(I)$-estimate for $\rho^\frac{1}{2}u$. Multiplying $\eqref{e1.5hh}_2$ by $u$ instead of $r^2u$ and integrating over $I$, we obtain an energy inequality with an extra integral term involving $\frac{\rho v u^2}{r}$. This term can be further controlled by $|v|_{\infty}$, yielding
\begin{equation*} 
\|\rho^{\frac{1}{2}}u\|_{L^\infty([0,T];L^2(I))} +\|\rho^{\frac{\delta}{2}} \mathrm{D}_r u\|_{L^2([0,T];L^2(I))} \leq C(T)\big(\|v\|_{L^\infty([0,T]\times I)}^2+1\big),
\end{equation*}
which, combined with the Sobolev embedding $W^{1,1}(I)\hookrightarrow L^\infty(I)$, yields that, all $\varepsilon\in(0,1)$,
\begin{equation}\label{rhogamma1}
\|\rho^{1-\delta}u\|_{L^2([0,T];L^\infty(I))}\leq C(\varepsilon,T) \big(1+\|v\|_{L^1([0,T];L^\infty(I))}^2\big)+\varepsilon\|v\|_{L^\infty([0,T]\times I)}^2.
\end{equation}
Substituting \eqref{rhogamma1} into \eqref{05} and choosing $\varepsilon$ sufficiently small, the Gr\"onwall inequality gives the desired bound $\|v\|_{L^\infty([0,T]\times I)}\leq C(T)$.

\subsubsection{Non-formation of cavitation inside the fluids.}
Using the $L^\infty$ bounds for $\rho$ and $v$, we can further show that no vacuum forms inside the fluid in finite time, provided that no vacuum state occurs initially. The crucial step is to derive
\begin{equation}\label{07}
\|u\|_{L^\infty([0,T];L^2(I))}+\|\rho^\frac{\delta-1}{2}\mathrm{D}_r u\|_{L^2([0,T];L^2(I))}+\|u\|_{L^4([0,T];L^\infty(I))} \leq C(T).
\end{equation}
In fact, once \eqref{07} is established, we can derive from $v=u+\frac{2a_1\delta}{\delta-1}(\rho^{\delta-1})_r$ and $L^\infty(I)$-estimate of $v$ that $\|(\rho^{\delta-1})_r\|_{L^\infty([0,T];L^2(0,1))}\leq C(T)$, and obtain from the continuity equation $\eqref{e1.5hh}_1$ that $\|\rho^{\delta-1}\|_{L^\infty([0,T];L^2(0,1))}\leq C(T)$. Hence, the embedding $H^1(0,1)\hookrightarrow L^\infty(0,1)$ yields $\|\rho^{\delta-1}\|_{L^\infty([0,T];L^\infty(0,1))}\leq C(T)$, which implies the strictly positive lower bound for $\rho$ in $[0,T]\times[0,1]$. Similar estimates can be extended to any bounded interval $[0,R]$ with constants depending on $R$, which give the pointwise estimate for $\rho$ in $[0,T]\times I$.

To establish \eqref{07}, we can first derive from \eqref{rhogamma1} and the $L^\infty(I)$-estimate of $v$ that $\|\rho^{1-\delta}u\|_{L^2([0,T];L^\infty(I))}\leq C(T)$. Then, based on this, multiplying $\eqref{e1.5hh}_2$ by $\rho^{-\alpha} u$ with $\alpha=\max\{0,\frac{5\delta-3}{2}\}$ and integrating over $I$, we obtain an intermediate $L^5(I)$-estimate for $\rho^\frac{1-\alpha}{5}u$:
\begin{equation*}
\big\|\rho^\frac{1-\alpha}{5}u\big\|_{L^\infty([0,T];L^5(I))}
+\big\|\rho^\frac{\delta-\alpha}{2}|u|^\frac{3}{2}\mathrm{D}_r u\big\|_{L^2([0,T];L^2(I))}\leq C(T),
\end{equation*}
which is crucial and further leads to $\|\rho^\frac{1-\delta}{2}u\|_{L^2([0,T];L^\infty(I))}\leq C(T)$. Multiplying $\eqref{cosingu}_2$ by $\rho^{-1}u$ and integrating over $I$, together with $L^2([0,T];L^\infty(I))$-bound of $\rho^\frac{1-\delta}{2}u$ and the $L^\infty(I)$-estimates for $(\rho,v)$, give the desired estimate \eqref{07}.

\subsubsection{Global \textit{a priori} estimates of the regular solutions}
With the $L^\infty(I)$-bounds for $(\rho,v)$ established, we turn to the enlarged system \eqref{e2.2} to obtain the global estimates for the regular solutions. To simplify the calculations, we provide two auxiliary lemmas (see Lemmas \ref{im-1}--\ref{im-2} in \S\ref{section-global2}), indicating the equivalence between the $W^{k,p}(\mathbb{R}^3)$-norms and weighted $H^k(\mathbb{R}^3)$-norms of the gradient and the divergence of a spherically symmetric vector function, which are useful in establishing the elliptic estimates for $u$.

Another key step is to obtain the $L^2([0,T];L^\infty(I))$-estimate for $\mathrm{D}_r u$. To this end, we establish an inequality of the form: $|r\mathrm{D}_r u|_{\infty}\leq C(T)(|ru_t|_{2}+1)$ via equation $\eqref{e2.2}_2$ and the Hardy inequality. This allows us to close the first-order temporal estimate for $u$, thereby obtaining $\mathrm{D}_r u\in L^2([0,T];L^\infty(I))$. With this estimate at hand, together with Lemmas \ref{im-1}--\ref{im-2}, we can systematically establish the remaining global uniform estimates for $(\phi,u,\psi)$.

\section{Global Uniform Upper Bound of the Density}\label{section-upper-density}

This section is devoted to establishing the global upper bound of $\rho$. In \S\ref{section-upper-density}--\S \ref{se46}, we denote $C_0 \in[1, \infty)$ a generic constant depending only on $\left(\rho_0, u_0, n, a_1, A, \gamma,\delta\right)$, and $C\left(\nu_1, \cdots\!, \nu_k\right) \in[1, \infty)$ a generic constant depending on $C_0$ and parameters $\left(\nu_1, \cdots\!, \nu_k\right)$, which may be different at each occurrence. Besides, we define the characteristic functions $(\chi^\flat_\sigma,\chi^\sharp_\sigma)$ ($\sigma>0$) as
\begin{equation}\label{chi-sigma}
\chi^\flat_\sigma(r):=
\begin{cases}
1&\text{for $r\in[0,\sigma)$},\\
0&\text{for $r\in[\sigma,\infty)$},
\end{cases}\qquad \ \ \chi^\sharp_\sigma:=1-\chi^\flat_\sigma.
\end{equation}
We always let 
\begin{equation}\label{del-gam}
\begin{aligned}
&\delta\in (\frac{1}{2},1),\ \ &&\gamma\in (1,\infty)\qquad\,\, &&&\text{if $n=2$};\\
&\delta\in (7-2\sqrt{10},1),\ \ &&\gamma\in (1,6\delta-3)\qquad\,\, &&&\text{if $n=3$}.
\end{aligned}
\end{equation}

\subsection{Reformulation in the Spherical Coordinates}

Let $T>0$ be any fixed time and let the vector $(\rho,\boldsymbol{u})(t,\boldsymbol{x})$ be the unique regular solution of the Cauchy problem \eqref{eq:1.1}--\eqref{kelaoxiusi} with  \eqref{eqs:CauchyInit}--\eqref{e1.3} in $[0,T_*]\times\mathbb{R}^n$, which has the following form:
\begin{equation}\label{e1.4}
\rho(t,\boldsymbol{x})=\rho(t,r), \quad \boldsymbol{u}(t,\boldsymbol{x})=u(t,r)\frac{\boldsymbol{x}}{r}\qquad \, \text{for $r=|\boldsymbol{x}|$}.
\end{equation}
Then the Cauchy problem \eqref{eq:1.1}--\eqref{kelaoxiusi} with  \eqref{eqs:CauchyInit}--\eqref{e1.3} can be transformed to the following \textbf{IBVP} in $[0,T]\times I$:
\begin{equation}\label{e1.5}
\begin{cases}
\displaystyle 
\rho_t+u\rho_r+\rho\big(u_r+\frac{m u}{r}\big)=0,\\[3pt]
\displaystyle
\rho u_t+\rho uu_r+A(\rho^\gamma)_r=2a_1\delta\big(\rho^\delta u_r+\frac{m\rho^\delta u}{r} \big)_r-\frac{2a_1 m(\rho^\delta)_r u}{r},\\[3pt]
\displaystyle
(\rho, u)|_{t=0}=(\rho_0, u_0 )\qquad\text{for  $r\in I$},\\[3pt]
\displaystyle
u|_{r=0}=0\qquad\qquad\qquad \ \text{for $t\in(0,T]$},\\[3pt]
\displaystyle
(\rho,u)\to \left(0,0\right) \qquad\qquad \, \text{as $r\to \infty$ \ for $t\in (0,T]$}.
\end{cases}
\end{equation}

Next, by Lemma  \ref{lemma-initial} in Appendix B, we can transform the statements of Theorems \ref{th1}--\ref{thm-loc} into spherical coordinates.
\begin{lem}\label{rth1} Let $n=2$ or $3$ and \eqref{cd1} hold. Assume that the initial data $(\rho_0, u_0)(r)$ satisfy 
\begin{equation}\label{etm}
\begin{aligned}
&\rho_0>0,\qquad r^m\rho_0\in L^1(I),\qquad r^\frac{m}{2n}(\rho_0^\frac{\delta-1}{2})_r\in L^{2n}(I),\\[2pt]
&r^{\frac{m}{n}} \mathrm{D}_r(\rho_0^{\delta-1})_{r} \in L^n(I),\qquad r^{\frac{m}{2}}  \mathrm{D}_r^2(\rho_0^{\delta-1})_{r} \in L^2(I),\\[4pt]
&r^{\frac{m}{2}} ((\rho_0^{\gamma-1})_r,\mathrm{D}_r(\rho_0^{\gamma-1})_{r},\mathrm{D}_r^2(\rho_0^{\gamma-1})_{r},u_0,\mathrm{D}_ru_0,\mathrm{D}_r^2u_0,\mathrm{D}_r^3u_0)\in L^2(I),
\end{aligned}
\end{equation}
and the initial  compatibility conditions{\rm :}
\begin{equation}\label{xiangrongxing}
\mathrm{D}_r u_0=r^{-\frac{m}{2}}\phi^{-\iota}_0(g_1 ,g_2),\qquad \mathrm{D}_r(\phi^{2\iota}_0 \tilde Lu_0)=r^{-\frac{m}{2}}\phi^{-\iota}_0(g_3,g_4),\qquad \tilde Lu_0=r^{-\frac{m}{2}} \phi^{-2\iota}_0g_5,
\end{equation}
from some functions $g_i$ {\rm(}$1\leq i\leq 5${\rm)} belonging to $L^2(I)$, where $\tilde Lu_0:=((u_0)_r+\frac{m}{r}u_0)_r$ and the operator $\mathrm{D}_r$ is defined in {\rm\S\ref{Nor}}. Then there exist both $T_*>0$ and a unique regular solution $(\rho, u)(t,r)$ in $[0,T_*]\times I$ of problem \eqref{e1.5}  satisfying 
\begin{equation}\label{spd}
\begin{aligned}
&\rho>0,\quad r^m\rho \in C([0,T_*];L^1(I)),\quad (\rho^{\delta-1})_{r}\in C([0,T_*];L^\infty(I)),\\
&r^{\frac{m}{2}}((\rho^{\gamma-1})_r,\mathrm{D}_r(\rho^{\gamma-1})_r, \mathrm{D}_r^2(\rho^{\gamma-1})_r)\in C([0,T_*];L^2(I)),\\
&r^{\frac{m}{n}} \mathrm{D}_r(\rho^{\delta-1})_{r}\in C([0,T_*];L^n(I)), \quad  r^{\frac{m}{2}} \mathrm{D}_r^2(\rho^{\delta-1})_{r}\in C([0,T_*];L^2(I)),\\
&r^{\frac{m}{2}}\big(u,\mathrm{D}_ru, \mathrm{D}_r^2u, \mathrm{D}_r^3u,u_t,\mathrm{D}_ru_t\big)\in C([0,T_*];L^2(I)),\\
&r^{\frac{m}{2}}\rho^\frac{\delta-1}{2}\mathrm{D}_ru\in C([0,T_*];L^2(I)),\quad  r^{\frac{m}{2}}\rho^\frac{\delta-1}{2} \mathrm{D}_ru_t\in L^\infty([0,T_*];L^2(I)),\\
&r^{\frac{m}{2}}\rho^{\delta-1}(\mathrm{D}_r^2u,\mathrm{D}_r^3u)\in C([0,T_*];L^2(I)),\quad r^{\frac{m}{2}} (\mathrm{D}_r^4 u,\mathrm{D}_r^2 u_t,\rho^{\delta-1}\mathrm{D}_r^4 u)\in L^2([0,T_*];L^2),\\
&t^{\frac{1}{2}}r^{\frac{m}{2}}(u_{tt},\mathrm{D}_r^2 u_t,\mathrm{D}_r^4 u)\in L^\infty([0,T_*];L^2),\quad t^{\frac{1}{2}}r^{\frac{m}{2}} \mathrm{D}_r u_{tt} \in L^2([0,T_*];L^2).
\end{aligned}
\end{equation}
In addition, the following regularity properties hold{\rm :} 
\begin{equation}\label{spd2}
(\rho,\rho_r,\rho_t,u,\mathrm{D}_r u)\in C([0,T_*]\times C(\bar I)),\qquad (\mathrm{D}_r^2 u,u_t)\in C((0,T_*]\times C(\bar I)).
\end{equation}
\end{lem}

\subsection{Some Basic Estimates and Radial Weighted Estimates for $\rho$}\label{subsection4.2}
First, the so-called effective velocity is defined as follows.
\begin{mydef}\label{def-effective}
Let $(\rho,u)$ be defined as in {\rm \S 1}. The effective velocity $v$ is defined as
\begin{equation}\label{V-expression}
v=u+2a_1\delta\rho^{\delta-2}\rho_r=u+\frac{2a_1\delta}{\delta-1}(\rho^{\delta-1})_r.
\end{equation}
Besides, define $v_0:=v|_{t=0}=u_0+2a_1\delta\rho_0^{\delta-2}(\rho_0)_r$.
\end{mydef}

Then we have the following energy estimates and the BD entropy estimates.
\begin{lem}\label{energy-BD}
For all $t\in [0,T]$,
\begin{equation*}\label{e-1.1}
\begin{aligned}
\int_0^\infty r^m(\rho u^2+ \rho^\gamma)(t,r)\,\mathrm{d}r+\int_0^t\int_0^\infty r^m \rho^\delta|\mathrm{D}_r u|^2\, \mathrm{d}r\mathrm{d}s&\leq C_0,\\
\int_0^\infty r^m\big(\rho v^2+|(\rho^{\delta-\frac{1}{2}})_r|^2+ \rho^\gamma\big)(t,r)\,\mathrm{d}r+ \int_0^t\int_0^\infty r^m\rho^{\gamma+\delta-3}\rho_r^2\,\mathrm{d}r\mathrm{d}s&\leq C_0,
\end{aligned}
\end{equation*}
where the operator $\mathrm{D}_r$ is defined in {\rm\S\ref{Nor}}.
\end{lem}
\begin{proof}
We divide the proof into two steps.

\smallskip
\textbf{1.} First, multiplying $\eqref{e1.5}_2$ by $r^m u$, along with $\eqref{e1.5}_1$, yields
\begin{equation}\label{408}
\begin{aligned}
\big(\frac{1}{2} r^m\rho u^2+\frac{A}{\gamma-1}\rho^\gamma\big)_t &=-2a_1 r^m\rho^\delta \underline{\Big(\delta u_r^2-2(1-\delta)m u_r\frac{u}{r}+(m-(1-\delta) m^2)\frac{u^2}{r^2}\Big)}_{:=I_1}\\
&\quad+\Big(r^m\big(2a_1 \rho^\delta u \big(\delta u_r+(\delta-1)m\frac{u}{r}\big) -\frac{1}{2}  \rho u^3-\frac{A\gamma}{\gamma-1} u\rho^\gamma\big)\Big)_r.
\end{aligned}
\end{equation}
Letting  $(X,Y)=(u_r,\frac{u}{r})$, then $I_1$ becomes a binary form:
\begin{equation*}
I_1=\delta X^2-2(1-\delta)m XY+(m-(1-\delta)m^2) Y^2    
\end{equation*}
and its discriminant $\mathscr{D}$ satisfies
\begin{equation*} 
\mathscr{D}=4(1-\delta)^2m^2-4\delta(m-(1-\delta)m^2)<0 \qquad \text{whenever} \ \ \delta>\frac{m}{m+1}=1-\frac{1}{n}.    
\end{equation*}
Hence, there exists a constant $c^*_\delta>0$, depending only on $\delta$, such that 
\begin{equation*}
I_1\geq c^*_\delta(X^2+Y^2)=c^*_\delta\big(u_r^2+\frac{u^2}{r^2}\big),
\end{equation*}
which, along with \eqref{408}, leads to
\begin{equation}\label{eap1}
\begin{aligned}
&\,\big( \frac{1}{2} r^m\rho u^2+\frac{A}{\gamma-1}r^m\rho^\gamma\big)_t+2a_1 c^*_\delta r^m\rho^\delta \big(u_r^2+\frac{u^2}{r^2}\big) \\
&\leq\Big(r^m\underline{\big(2a_1 \rho^\delta u \big(\delta u_r+(\delta-1)m\frac{u}{r}\big) -\frac{1}{2}  \rho u^3-\frac{A\gamma}{\gamma-1} u\rho^\gamma\big)}_{:=\mathcal{B}_1}\Big)_r.
\end{aligned}
\end{equation}

Next, we prove $r^m \mathrm{D}_r \mathcal{B}_1 \in L^1(I)$ for \textit{a.e.} $t\in (0,T)$, so that Lemma \ref{calculus} can be applied to obtain
\begin{equation}\label{int-B1}
\int_0^\infty (r^m\mathcal{B}_1)_r\,\mathrm{d}r=0.
\end{equation}
Indeed, it follows from \eqref{spd}--\eqref{spd2} that, for \textit{a.e.} $t\in (0,T)$,
\begin{equation*}
r^m\rho\in L^1(I),\quad (\rho,(\rho^{\delta-1})_r,u)\in L^\infty(I),\quad r^\frac{m}{2}(u,\mathrm{D}_r u,\mathrm{D}_r^2 u,(\rho^{\gamma-1})_r)\in L^2(I).
\end{equation*}
Then we obtain from the above that
\begin{align*}
&\begin{aligned}
|r^m \mathrm{D}_r \mathcal{B}_1|_1&\leq C_0\big|(r^{m}\rho^\gamma \mathrm{D}_r u,r^mu \rho(\rho^{\gamma-1})_r)\big|_1+C_0|r^{m-1}\rho^{\delta}u\mathrm{D}_r u|_1\\
&\quad +C_0\big|\big(r^m\rho(\rho^{\delta-1})_ru\mathrm{D}_r u,r^m\rho^{\delta} u_r\mathrm{D}_r u,r^m\rho^{\delta} u\mathrm{D}_r^2 u\big)\big|_1\\
&\quad+C_0\big|\big(r^{m-1}\rho u^3,r^m\rho^{2-\delta}(\rho^{\delta-1})_ru^3,r^m\rho u^2u_r\big)\big|_1
\end{aligned}\\ 
&\qquad\qquad \ \ \begin{aligned}
&\leq C_0\big(|\mathrm{D}_r u|_\infty|r^m\rho|_1|\rho|_\infty^{\gamma-1}+|\rho|_\infty|r^\frac{m}{2}(\rho^{\gamma-1})_r|_2|r^\frac{m}{2}u|_2\big)\\
&\quad +C_0|\rho|_\infty^{\delta} \big(|r^\frac{m}{2}\mathrm{D}_r u|_2^2+|r^\frac{m}{2}u|_2 |r^\frac{m}{2}\mathrm{D}_r^2 u|_2\big)\\
&\quad +C_0|\rho|_\infty|(\rho^{\delta-1})_r|_\infty|r^\frac{m}{2}u|_2 |r^\frac{m}{2}\mathrm{D}_r u|_2\\
&\quad +C_0|r^m\rho|_1|u|_\infty^2 |\mathrm{D}_r u|_\infty+C_0|\rho|_\infty^{2-\delta}|(\rho^{\delta-1})_r|_\infty|u|_\infty|r^\frac{m}{2}u|_2^2<\infty.
\end{aligned}
\end{align*}

Thus, integrating \eqref{eap1} over $[0,t]\times I$, along with \eqref{int-B1}, gives
\begin{equation*}\label{eap2}
\begin{aligned}
&\int_0^\infty \big(\frac{1}{2} r^m\rho u^2+\frac{A}{\gamma-1}r^m\rho^\gamma\big)(t,\cdot) \,\mathrm{d}r+ 2a_1 c^*_\delta\int_0^t\int_0^\infty r^m\rho^\delta|\mathrm{D}_r u|^2\,\mathrm{d}r\,\mathrm{d}s\\
&\leq C_0\big(|(r^m\rho_0)^\frac{1}{2}u_0|_2^2+|r^\frac{m}{\gamma}\rho_0|_\gamma^\gamma\big), 
\end{aligned}
\end{equation*}
where the initial data can be handled by Lemmas \ref{ale1}, \ref{initial3}, and \ref{lemma-initial}:
\begin{equation*}
\begin{aligned}
|(r^m\rho_0)^\frac{1}{2} u_0|_2&\leq C_0|r^m\rho_0|_1^\frac{1}{2}|u_0|_\infty \leq C_0 \|\rho_0\|_{L^1}^\frac{1}{2}\|\boldsymbol{u}_0\|_{L^\infty} \leq C_0,\\
|r^\frac{m}{\gamma}\rho_0|_\gamma&\leq C_0|r^m\rho_0|_1 |\rho_0|_\infty^{\gamma-1} \leq C_0 \|\rho_0\|_{L^1} \|\rho_0\|_{L^\infty}^{\gamma-1} \leq C_0.
\end{aligned}
\end{equation*}

\smallskip
\textbf{2.} Multiplying $\eqref{e1.5}_1$ by $2a_1\delta\rho^{\delta-1}$ and applying $\partial_r$ to the resulting equality give 
\begin{align*}
2a_1\delta\rho(\rho^{\delta-2}\rho_r)_{t}+2a_1\delta\rho u(\rho^{\delta-2}\rho_r)_{r}+2a_1\delta\big(\rho^{\delta}(u_r+\frac{mu}{r})\big)_r-\frac{2a_1 m (\rho^{\delta})_ru}{r}=0,
\end{align*}
which, together with $\eqref{e1.5}_2$, leads to
\begin{equation}\label{ess}
\rho(v_t+u v_r)+A(\rho^\gamma)_r=0.
\end{equation}

Then, multiplying \eqref{ess} by $r^m v$ and 
integrating the resulting equality over $[0,t]\times I$, we derive from $\eqref{e1.5}_1$ and the similar argument in Step 1 that
\begin{equation*}
\int_0^\infty r^m(\rho v^2+\rho^\gamma)(t,\cdot)\,\mathrm{d}r+ \int_0^t\int_0^\infty r^m\rho^{\gamma+\delta-3}\rho_r^2\,\mathrm{d}r\mathrm{d}s
\leq C_0\big(|(r^m\rho_0)^\frac{1}{2}v_0 |_2^2+1\big).
\end{equation*}
For the $L^2$-boundedness of $(r^m\rho_0)^{\frac{1}{2}}v_0$, it follows from Lemmas \ref{ale1}, \ref{initial3}, and \ref{lemma-initial} that
\begin{equation*}
|(r^m\rho_0)^\frac{1}{2} v_0|_2 \leq C_0|r^m\rho_0|_1^\frac{1}{2}|(u_0,(\rho_0^{\delta-1})_r)|_\infty \leq C_0 \|\rho_0\|_{L^1}^\frac{1}{2}\|(\boldsymbol{u}_0,\nabla(\rho_0^{\delta-1}))\|_{L^\infty} \leq C_0.
\end{equation*}

Finally, the $L^2$-estimate of $r^\frac{m}{2}(\rho^{\delta-\frac{1}{2}})_r$ follows directly from the $L^2$-estimates of $(r^m\rho)^\frac{1}{2}(u,v)$ and \eqref{V-expression}. This completes the proof.
\end{proof}

Clearly, by \eqref{V-expression} and \eqref{ess}, we have
\begin{cor}\label{cor-v}
The effective velocity $v$ satisfies the following equation{\rm:}
\begin{equation}\label{eq:effective2}
v_t+uv_r+ \frac{A\gamma}{2a_1\delta} \rho^{\gamma-\delta} (v-u)=0.
\end{equation}
\end{cor}

Next, we establish the $L^p$-estimates, $p\in [1,\infty]$, of $r^m\rho$ away from the origin.
\begin{lem}\label{far-p-infty}
For any $\omega>0$, there exists a constant $C(\omega)>0$ such that 
\begin{equation*}
\begin{aligned}
|r^m \rho(t)|_1&\leq C_0&&\quad \text{for all $t\in[0,T]$},\\
|\chi^\sharp_\omega r^m \rho(t)|_p&\leq C(\omega)&&\quad \text{for all $p\in (1,\infty]$ and $t\in[0,T]$}.
\end{aligned}
\end{equation*}
\end{lem}
\begin{proof}
First, integrating $\eqref{e1.5}_1$ over $I$, we obtain 
from  Lemma \ref{calculus} and $r^m \mathrm{D}_r(\rho u) \in L^1(I)$ for \textit{a.e.} $t\in (0,T)$ due to \eqref{spd}--\eqref{spd2} that 
\begin{equation*} 
\frac{\mathrm{d}}{\mathrm{d}t}\int_0^\infty r^m\rho \,\mathrm{d}r=-\int_0^\infty (r^m \rho u)_r \,\mathrm{d}r=r^m \rho u|_{r=0}=0.
\end{equation*}
Then integrating the above over $[0,t]$ yields 
\begin{equation}\label{L1-rho}
|r^m\rho(t)|_1=|r^m\rho_0|_1\leq C_0 \qquad \text{for any $t\in [0,T]$}.
\end{equation}

Next, let $\omega>0$. It follows from \eqref{L1-rho}, Lemmas \ref{energy-BD} and \ref{calculus}, and the H\"older inequality that, for any $t\in [0,T]$,
\begin{equation*}
\begin{aligned}
|\chi^\sharp_\omega r^m\rho|_\infty&\leq |\chi^\sharp_\omega (r^m\rho)_r|_1\leq C_0\big(|\chi_\omega^\sharp r^m\rho^{\frac{3}{2}-\delta} (\rho^{\delta-\frac{1}{2}})_r|_1 + |\chi_\omega^\sharp r^{m-1}\rho|_1\big)\\
& \leq C_0\big(|\chi^\sharp_\omega\rho|_\infty^{1-\delta}|r^{m}\rho|_{1}^{\frac{1}{2}}|r^\frac{m}{2}(\rho^{\delta-\frac{1}{2}})_r|_2 +  |\chi_\omega^\sharp r^{-1}|_\infty |r^{m}\rho|_1\big)\\
&\leq C_0 |\chi^\sharp_\omega r^{m(\delta-1)}|_\infty |\chi^\sharp_\omega r^m\rho|_\infty^{1-\delta}+C(\omega)
\leq C(\omega)\big(|\chi^\sharp_\omega r^m\rho|_\infty^{1-\delta}+1\big),
\end{aligned}
\end{equation*}
which, along with the Young inequality, gives $|\chi^\sharp_\omega r^m\rho|_\infty \leq C(\omega)$ for all $t\in [0,T]$. 

Finally, it follows from the above and \eqref{L1-rho} that, 
for any $p\in(1,\infty)$ and $t\in [0,T]$,
\begin{equation*}
|\chi^\sharp_\omega r^m\rho|_p\leq |\chi^\sharp_\omega r^m\rho|_\infty^{1-\frac{1}{p}}|r^m\rho|_1^\frac{1}{p} \leq C(\omega),
\end{equation*}
where $C(\omega)\in [1,\infty)$ is a generic  constant depending on $C_0$ and $\omega$, which may be different at each occurrence. This completes the proof.
\end{proof}

The next lemma concerns the $r$-weighted $L^p$-estimates ($p\in[1,\infty]$) for $\rho$ near the origin.
\begin{lem}\label{l4.3}
For any $t\in[0,T]$,
\begin{enumerate}
\item[$\mathrm{(i)}$] if $n=2$, there exist two positive constants $C(\nu,p)$ and $C(\nu)$ such that 
\begin{equation*}
\begin{aligned}
&|\chi_1^\flat r^{\nu}\rho(t)|_p\leq C(p,\nu) &&\quad\text{for all $\nu>-\frac{1}{p}$ and $p\in (0,\infty)$},\\
&|\chi_1^\flat r^{\nu}\rho(t)|_\infty\leq C(\nu)&&\quad\text{for all $\nu>0$};
\end{aligned}
\end{equation*}
\item[$\mathrm{(ii)}$] if $n=3$, there exists two positive constants $C(p)$ and $C_0$ such that
\begin{equation*}
\big|\chi_1^\flat r^{\frac{1}{2\delta-1}-\frac{1}{p}}\rho(t)\big|_p \leq C(p)\quad\text{for all $p\in[2\delta-1,\infty)$},\qquad\big|\chi_1^\flat r^{\frac{1}{2\delta-1}}\rho(t)\big|_\infty \leq C_0.
\end{equation*}
\end{enumerate}
\end{lem}
\begin{proof}
We divide the proof into two steps.

\smallskip
\textbf{1.} By Lemmas \ref{energy-BD}--\ref{far-p-infty} and \ref{hardy} 
and the H\"older inequality, we obtain that, for all $\nu>0$,
\begin{equation*}
\begin{aligned}
|\chi_1^\flat r^\nu \rho|_\infty^{2\delta-1}&=|\chi_1^\flat r^{\nu(\delta-\frac{1}{2})} \rho^{\delta-\frac{1}{2}}|_\infty^2 \leq C(\nu)\big|\chi_1^\flat r^{\nu(\delta-\frac{1}{2})+\frac{1}{2}}(\rho^{\delta-\frac{1}{2}},(\rho^{\delta-\frac{1}{2}})_r)\big|_2^2\\
&\leq C(\nu) |\chi_1^\flat r \rho^{2\delta-1}|_1+C(\nu)| r^{\frac{1}{2}} (\rho^{\delta-\frac{1}{2}})_r|_2^2 \\
&\leq C(\nu) |\chi_1^\flat r|_1^{2-2\delta}|r \rho|_1^{2\delta-1}+C(\nu)\leq C(\nu).
\end{aligned}
\end{equation*}

Next, for all $p\in(0,\infty)$ and $\nu>-\frac{1}{p}$, we can fix $\varepsilon$ such that $0<\varepsilon<\min\{p\nu +1,1\}$, and then obtain
\begin{equation*}
|\chi_1^\flat r^\nu \rho|_p^p=\int_0^\infty  r^{p\nu} \rho^p\,\mathrm{d}r \leq \Big(\int_0^\infty  r^{p\nu-\varepsilon} \,\mathrm{d}r\Big) |\chi_1^\flat r^\frac{\varepsilon}{p}\rho|_\infty^p\leq C(p,\nu).
\end{equation*}
This completes the proof of (i).

\smallskip
\textbf{2.} For any $p\in[2\delta-1,\infty)$, it follows from Lemmas \ref{energy-BD}--\ref{far-p-infty} and \ref{hardy} and the H\"older inequality that  
\begin{equation*}
\begin{aligned}
|\chi_1^\flat r^{\frac{1}{2\delta-1}-\frac{1}{p}}\rho|_p&=\big|\chi_1^\flat r^{\frac{p-2\delta+1}{2p}}\rho^{\delta-\frac{1}{2}}\big|_{\frac{2p}{2\delta-1}}^\frac{2}{2\delta-1}\leq C(p)\big|\chi_1^\flat r (\rho^{\delta-\frac{1}{2}},(\rho^{\delta-\frac{1}{2}})_r)\big|_2^\frac{2}{2\delta-1}\\
&\leq C(p)\big(|\chi_1^\flat r^2|_1^{\frac{2-2\delta}{2\delta-1}} |r^2 \rho|_1+\big|\chi_1^\flat r(\rho^{\delta-\frac{1}{2}})_r\big|_2^\frac{2}{2\delta-1}\big)\leq C(p).
\end{aligned}
\end{equation*}
Similarly, we can also obtain
\begin{equation*}
|\chi_1^\flat r^{\frac{1}{2\delta-1}}\rho|_\infty=\big|\chi_1^\flat r^\frac{1}{2}\rho^{\delta-\frac{1}{2}}\big|_{\infty}^\frac{2}{2\delta-1}\leq C_0\big|\chi_1^\flat r (\rho^{\delta-\frac{1}{2}},(\rho^{\delta-\frac{1}{2}})_r)\big|_2^\frac{2}{2\delta-1} \leq C_0.
\end{equation*}
This completes the proof of (ii). 
\end{proof}

\subsection{$L^p(I)$-Energy Estimates for $(u,v)$}

We aim to establish the following $L^{p}(I)$-estimates for $(r^m\rho)^{\frac{1}{p}}(u,v)$:
\begin{lem}\label{lemma-uv-lp}
Let $\tilde{p}_{m}(\delta)>2$ be a special parameter defined by
\begin{equation*}
\tilde{p}_m(\delta):=\frac{2\delta(m\delta-(m-1))+2\sqrt{\delta(m\delta-(m-1))((m+1)\delta-m)}}{m(1-\delta)^2}.
\end{equation*}
Then, for any $p\in [2,\tilde{p}_{m}(\delta))$, there exists a constant $C(p,T)$ such that, for all $t\in[0,T]$,
\begin{equation*}
\big|(r^m\rho)^\frac{1}{p}(u,v)(t)\big|_{p}^{p}+ \int_0^t\big(\big|(r^m\rho^\delta)^\frac{1}{2}|u|^\frac{p-2}{2}\mathrm{D}_r u\big|_{2}^{2}+\big|(r^m\rho^{\gamma-\delta+1})^\frac{1}{p} v\big|_{p}^{p}\big)\,\mathrm{d}s\leq C(p,T).
\end{equation*}
\end{lem}

The proof of Lemma \ref{lemma-uv-lp} is divided into the following several lemmas. First, we prove that $\tilde{p}_m(\delta)>n$ if $\delta$ satisfies conditions \eqref{del-gam}.
\begin{lem}\label{Lemma-pp}
Let $\tilde{p}_m(\delta)$ be defined as in {\rm Lemma \ref{lemma-uv-lp}}. Then $\tilde{p}_m(\delta)$ is the larger root of the following quadratic equation with respect to $p${\rm:}
\begin{equation*}
m^2(\delta-1)^2p^2-4\delta(\delta m^2-m^2+m)p+4\delta(\delta m^2-m^2+m)=0.
\end{equation*}
In particular, the following properties hold{\rm:}
\begin{equation*}
\begin{aligned}
2<\tilde{p}_1(\delta)&=\frac{2\delta^2+2\delta\sqrt{2\delta-1}}{(\delta-1)^2}\quad&&\text{if $m=1$ and $\delta\in (\frac{1}{2},1)$};\\
3<\tilde{p}_2(\delta)&=\frac{\delta(2\delta-1)+\sqrt{\delta(2\delta-1)(3\delta-2)}}{(\delta-1)^2}\quad&&\text{if $m=2$ and $\delta\in (7-2\sqrt{10},1)$}.
\end{aligned}
\end{equation*}
\end{lem}
\begin{proof}
The first assertion follows directly from the formula for the roots of quadratic equations.

Next, if $m=1$, notice that
\begin{equation*}
\tilde p_1(\delta)>2\iff   \delta\sqrt{2\delta-1}+(2\delta-1)>0,
\end{equation*}
which automatically holds for $\delta\in (\frac{1}{2},1)$. 

Finally, if $m=2$, then
\begin{equation}\label{ifff}
\tilde p_2(\delta)>3\iff  \sqrt{\delta(2\delta-1)(3\delta-2)}>\delta^2-5\delta+3.  
\end{equation}
If $\delta^2-5\delta+3\leq 0$, we have
\begin{equation*}
\frac{5-\sqrt{13}}{2}\leq \delta<1,
\end{equation*}
so that the right-hand side of \eqref{ifff} automatically holds; 
while, if $\delta^2-5\delta+3>0$, 
\begin{equation*}
\eqref{ifff}\iff \delta(2\delta-1)(3\delta-2)>(\delta^2-5\delta+3)^2\impliedby 7-2\sqrt{10}<\delta<\frac{5-\sqrt{13}}{2}.
\end{equation*}
Hence, based on the above, we see that 
$\tilde p_2(\delta)>3$ for any $\delta\in (7-2\sqrt{10},1)$.
\end{proof}

Next, we concern the $L^p(I)$-estimate for $(r^m\rho)^{\frac{1}{p}} u$.

\begin{lem} \label{lem-u-lp}
Let $\tilde{p}_{m}(\delta)$ be defined as in {\rm Lemma \ref{lemma-uv-lp}}. Then, for any $p\in [2,\tilde{p}_{m}(\delta))$ and $\varepsilon\in (0,1)$, there exist $c_p>0$, depending only on $(n,\delta,p)$, and $C(p,\varepsilon)>0$ such that 
\begin{equation}\label{dt-u-p}
\begin{aligned}
&\,\frac{1}{p}\frac{\mathrm{d}}{\mathrm{d}t}\big|(r^m\rho)^{\frac{1}{p}} u\big|_{p}^{p} + a_1 c_p\big|(r^m\rho^\delta)^\frac{1}{2} |u|^\frac{p-2}{2}\mathrm{D}_r u\big|_2^2\\
&\leq C(p, \varepsilon)\big(1+\big|(r^m\rho)^\frac{1}{p}u\big|_{p}^{p}\big) + \varepsilon \big|(r^m\rho^{\gamma-\delta+1})^\frac{1}{p}v\big|_{p}^{p}.
\end{aligned}
\end{equation}
\end{lem}
\begin{proof}
We divide the proof into four steps.

\smallskip
\textbf{1.} Let $\tilde{p}_{m}(\delta)$ be defined as in Lemma \ref{lemma-uv-lp}. Multiplying $\eqref{e1.5}_2$ by $r^m|u|^{p-2}u$ with $p\in [2,\tilde{p}_{m}(\delta))$ gives
\begin{equation}\label{4.17}
\begin{aligned}
&\,\frac{1}{p}(r^m\rho |u|^p)_t+2a_1 r^m \rho^\delta |u|^{p-2}\Big(\underline{\delta(p-1)u_r^2-mp(1-\delta)u_r\frac{u}{r}+(\delta m^2-m^2+m)\frac{u^2}{r^2}}_{:=I_2}\Big)\\
&=\Big(r^m\big(\underline{2a_1 \rho^\delta|u|^{p-2}u\big(\delta u_r+(\delta-1)m\frac{u}{r}\big)-A \rho^\gamma |u|^{p-2}u -\frac{1}{p} \rho u|u|^p\big)}_{\mathcal{B}_2}\Big)_r\\
&\quad +Ar^m\rho^\gamma |u|^{p-2}\big((p-1)u_r+\frac{m}{r}u\big).
\end{aligned}
\end{equation}

\smallskip
\textbf{2.} We show that, for any $p\in [2,\tilde{p}_{m}(\delta))$, there exists a constant $c_p>0$ depending only on $(p,\delta,n)$ such that
\begin{equation}\label{mianaim}
I_2\geq c_p \big(u_r^2+ \frac{u^2}{r^2}\big).
\end{equation}

First, note that $I_2$ can be written in the following simple form:
\begin{equation*} 
I_2=\delta(p-1)X^2-mp(1-\delta)XY+(\delta m^2-m^2+m)Y^2,
\end{equation*}
where $(X,Y)=(u_r,\frac{u}{r})$, and its discriminant $\mathscr{D}(p;\delta)$ is  
\begin{equation*}
\mathscr{D}(p;\delta)=m^2(1-\delta)^2p^2-4\delta(\delta m^2-m^2+m)p+4\delta(\delta m^2-m^2+m),
\end{equation*}
which takes the same form as the quadratic equation given in Lemma \ref{Lemma-pp}. 

Consequently, by Lemma \ref{Lemma-pp}, $\mathscr{D}(\tilde{p}_m(\delta);\delta)=0$, and $\tilde{p}_m(\delta)>n$ is the larger root of the equation $\mathscr{D}(p;\delta)=0$. Moreover, since $\mathscr{D}(2;\delta)<0$ (see the proof of Lemma \ref{energy-BD}),  
we obtain that $\mathscr{D}(p;\delta)<0$ for any $p\in [2,\tilde{p}_m(\delta))$, which implies claim \eqref{mianaim}.

\smallskip
\textbf{3.} We prove $r^m\mathrm{D}_r\mathcal{B}_2 \in L^1(I)$ for \textit{a.e.} $t\in (0,T)$. Then Lemma \ref{calculus} leads to 
\begin{equation}\label{int-B2}
\int_0^\infty (r^m\mathcal{B}_2)_r\,\mathrm{d}r =0.
\end{equation}
Indeed, since $|\mathrm{D}_r\mathcal{B}_2|\leq C(p)|u|^{p-2}|\mathrm{D}_r\mathcal{B}_1|$ and $u\in C(\bar I)$ for $t\in (0,T]$ due to \eqref{spd2}, we can obtain claim \eqref{int-B2} by following an argument similar to the proof of \eqref{int-B1}.

\smallskip
\textbf{4.} Integrating \eqref{4.17} over $I$, together with \eqref{mianaim}--\eqref{int-B2}, yields 
\begin{equation}\label{add}
\begin{aligned}
&\,\frac{1}{p} \frac{\mathrm{d}}{\mathrm{d}t}\big|(r^m\rho)^{\frac{1}{p}} u\big|_{p}^{p} +2a_1 c_p\int_0^\infty r^m \rho^\delta |u|^{p-2}|\mathrm{D}_r u|^2\,\mathrm{d}r\\
&\leq A\int_0^\infty r^m\rho^\gamma |u|^{p-2}\big((p-1)u_r+\frac{m}{r}u\big)\,\mathrm{d}r:=I_3.
\end{aligned}
\end{equation}

To estimate $I_3$, we see from Lemmas \ref{far-p-infty}--\ref{l4.3} and the H\"older and Young inequalities that
\begin{equation}\label{J1}
\begin{aligned}
I_3&\leq C(p)\big|(r^m\rho^\delta)^{\frac12}|u|^{\frac{p-2}{2}}\mathrm{D}_r u\big|_2\big|r^{\frac{m}{2}}\rho^{\gamma-\frac{\delta}{2}}|u|^{\frac{p-2}{2}}\big|_2\\
&\leq \frac{a_1 c_p}{32}\big|(r^m\rho^\delta)^{\frac12}|u|^{\frac{p-2}{2}}\mathrm{D}_r u\big|_2^2+C(p)\big(\big|\chi_1^\flat r^{\frac{m}{2}}\rho^{\gamma-\frac{\delta}{2}}|u|^{\frac{p-2}{2}}\big|_2^2+ \big|\chi_1^\sharp r^{\frac{m}{2}}\rho^{\gamma-\frac{\delta}{2}}|u|^{\frac{p-2}{2}}\big|_2^2\big)\\
&\leq  \frac{a_1 c_p}{32}\big|(r^m\rho^\delta)^{\frac12}|u|^{\frac{p-2}{2}}\mathrm{D}_r u\big|_2^2+C(p)\big|\chi_1^\flat r^\frac{p+m-2}{p\gamma-p\delta+\delta} \rho\big|_{p\gamma-p\delta+\delta}^\frac{2p\gamma-2p\delta+2\delta}{p}\big|(r^{m-2}\rho^\delta)^\frac{1}{p}u\big|_{p}^{p-2}\\
&\quad +C(p)|\chi_1^\sharp r^{m(1+\delta-2\gamma)}|_\infty |\chi_1^\sharp r^m\rho|_{\frac{2p\gamma-p(1+\delta)+2}{2}}^\frac{2p\gamma-p(1+\delta)+2}{p} \big|(r^m\rho)^\frac{1}{p}u\big|_{p}^{p-2}\\
&\leq \frac{a_1 c_p}{16}\big|(r^m\rho^\delta)^{\frac12}|u|^{\frac{p-2}{2}}\mathrm{D}_r u\big|_2^2+C(p)\Big(1+\big|(r^m\rho)^\frac{1}{p}u\big|_{p}^{p}+\underline{\big|\chi_1^\flat r^\frac{p+m-2}{p\gamma-p\delta+\delta} \rho\big|_{p\gamma-p\delta+\delta}^{p\gamma-p\delta+\delta}}_{:=I_{3,1}}\Big),
\end{aligned}
\end{equation}
where we have used Lemma \ref{far-p-infty} with the fact that $\frac{2p\gamma-p(1+\delta)+2}{2}>1$. 

\smallskip
\textbf{4.1. Estimate for $I_{3,1}$ when $n=2$.} Note that, if $n=2$ ($m=1$),
\begin{equation*}
\frac{p+m-2}{p\gamma-p\delta+\delta}=\frac{p-1}{p\gamma-p\delta+\delta}>0.
\end{equation*}
Thus, we obtain from Lemma \ref{l4.3} that
\begin{equation}\label{g1,1-1}
I_{3,1} \leq C(p).
\end{equation}

\smallskip
\textbf{4.2. Estimate for $I_{3,1}$ when $n=3$.} Set 
\begin{equation*}
\mathfrak{a}_1:=p,\qquad \mathfrak{b}_1:=p\gamma-p\delta+\delta,\qquad \ \mathcal{I}_{(l_1,l_2)}:=\int_0^1  r^{l_1}\rho^{l_2}\,\mathrm{d}r.
\end{equation*}
First, it follows from \eqref{V-expression}, Lemma \ref{far-p-infty} (with $\omega=\frac{1}{2}$), integration by parts, and the H\"older and Young inequalities that, for any $\varepsilon_1\in (0,1)$,
\begin{equation}\label{a1-b1}
\begin{aligned}
I_{3,1} &=\mathcal{I}_{(\mathfrak{a}_1,\mathfrak{b}_1)}=\frac{1}{\mathfrak{a}_1+1}\rho^{\mathfrak{b}_1}(1)-\frac{\mathfrak{b}_1}{2a_1\delta(\mathfrak{a}_1+1)}\int_0^1  r^{\mathfrak{a}_1+1} \rho^{\mathfrak{b}_1-\delta+1}(v-u)\,\mathrm{d}r\\
&\leq C(\mathfrak{a}_1,\mathfrak{b}_1)\Big(\big|\chi_\frac{1}{2}^\sharp \rho\big|_\infty^{\mathfrak{b}_1}+ \big(\int_0^1  r^\frac{p\mathfrak{a}_1+p-2}{p-1} \rho^{\frac{p\mathfrak{b}_1-\gamma}{p-1}+1-\delta} \,\mathrm{d}r\big)^\frac{p-1}{p}\big|\chi_1^\flat(r^2\rho^{\gamma-\delta+1})^\frac{1}{p}(u,v)\big|_{p}\Big)\\
&\leq C(p,\mathfrak{a}_1,\mathfrak{b}_1)\Big(1+\frac{1}{\varepsilon_1}\int_0^1  r^\frac{p\mathfrak{a}_1+p-2}{p-1} \rho^{\frac{p\mathfrak{b}_1-\gamma}{p-1}+1-\delta} \,\mathrm{d}r\Big)+\varepsilon_1 \big|\chi_1^\flat(r^2\rho^{\gamma-\delta+1})^\frac{1}{p}(u,v)\big|_{p}^{p}.
\end{aligned}
\end{equation}

Next, for $j\in \mathbb{N}^*$, define the following two sequences $\{\mathfrak{a}_j\}_{j=1}^\infty$ and $\{\mathfrak{b}_j\}_{j=1}^\infty$ 
as
\begin{equation*}
\mathfrak{a}_{j+1}=\frac{p}{p-1}\mathfrak{a}_j+\frac{p-2}{p-1},\quad\,\, \mathfrak{b}_{j+1}=\frac{p}{p-1}\mathfrak{b}_j-\frac{\gamma}{p-1}+1-\delta.
\end{equation*}
Certainly, we can solve for $(\mathfrak{a}_{j+1},\mathfrak{b}_{j+1})$ from the above that, for $j\in \mathbb{N}^*$,
\begin{equation}\label{tongxiang}
\begin{aligned}
\mathfrak{a}_{j+1}&=2(p-1)\big(\frac{p}{p-1}\big)^{j}-p+2,\\
\mathfrak{b}_{j+1}&=(\gamma-2\delta+1)(p-1)\big(\frac{p}{p-1}\big)^{j}+\gamma+(p-1)(\delta-1).    
\end{aligned}
\end{equation}
Then \eqref{a1-b1} becomes
\begin{equation*}
\mathcal{I}_{(\mathfrak{a}_1,\mathfrak{b}_1)}\leq C(p,\mathfrak{a}_1,\mathfrak{b}_1)\Big(1+\frac{1}{\varepsilon_1}\mathcal{I}_{(\mathfrak{a}_2,\mathfrak{b}_2)}\Big)+\varepsilon_1 \big|\chi_1^\flat(r^2\rho^{\gamma-\delta+1})^\frac{1}{p}(u,v)\big|_{p}^{p},
\end{equation*}
and we can follow a calculation similar to \eqref{a1-b1} to obtain that,
for all $j\in \mathbb{N}^*$ and $\varepsilon_j\in (0,1)$,
\begin{equation}\label{aj-bj}
\mathcal{I}_{(\mathfrak{a}_j,\mathfrak{b}_j)}\leq C(p,\mathfrak{a}_j,\mathfrak{b}_j)\Big(1+\frac{1}{\varepsilon_j}\mathcal{I}_{(\mathfrak{a}_{j+1},\mathfrak{b}_{j+1})}\Big)+\varepsilon_j \big|\chi_1^\flat(r^2\rho^{\gamma-\delta+1})^\frac{1}{p}(u,v)\big|_{p}^{p}.
\end{equation}
Define the constants
\begin{equation*}
\varepsilon_0:=1,\qquad C(p,\mathfrak{a}_0,\mathfrak{b}_0):=1.
\end{equation*}
\eqref{aj-bj} implies that, for $j\in \mathbb{N}^*$,
\begin{equation}\label{g11-po}
\begin{aligned}
I_{3,1}=\mathcal{I}_{(\mathfrak{a}_1,\mathfrak{b}_1)}
&\leq \sum_{k=1}^j\Big(\prod_{l=0}^{k-1}\frac{C(p,\mathfrak{a}_l,\mathfrak{b}_l)}{\varepsilon_{l}}\Big)C(p,\mathfrak{a}_k,\mathfrak{b}_k)
+\Big(\prod_{l=1}^j \frac{C(p,\mathfrak{a}_l,\mathfrak{b}_l)}{\varepsilon_l}\Big)\mathcal{I}_{(\mathfrak{a}_{j+1},\mathfrak{b}_{j+1})}\\
&\quad + \sum_{k=1}^j\Big(\prod_{l=0}^{k-1}\frac{C(p,\mathfrak{a}_l,\mathfrak{b}_l)}{\varepsilon_l}\Big) \varepsilon_{k} \big|\chi_1^\flat(r^2\rho^{\gamma-\delta+1})^\frac{1}{p}(u,v)\big|_{p}^{p}.    
\end{aligned}    
\end{equation}

To estimate $\mathcal{I}_{(\mathfrak{a}_{j+1},\mathfrak{b}_{j+1})}$, we see that, for each $p\in [2,\tilde{p}_m(\delta))$, $\delta\in (7-2\sqrt{10},1)$, and $\gamma\in (1,6\delta-3)$, 
both $\{\mathfrak{a}_j\}_{j=1}^\infty$ and $\{\mathfrak{b}_j\}_{j=1}^\infty$ given in \eqref{tongxiang} are strictly increasing  as $j\to\infty$. 
Moreover, for these $(p,\gamma,\delta)$, 
\begin{equation}\label{635}
\begin{aligned}
&\ \, \frac{\mathfrak{a}_{j+1}}{\mathfrak{b}_{j+1}}\geq \frac{1}{2\delta-1}-\frac{1}{\mathfrak{b}_{j+1}}\\
&\iff  \frac{2(p-1)(\frac{p}{p-1})^{j}-p+3}{(\gamma-2\delta+1)(p-1)(\frac{p}{p-1})^{j}+\gamma+(p-1)(\delta-1)}\geq \frac{1}{2\delta-1}\\
&\iff \gamma\leq 6\delta-3-\frac{(3\delta-2)p+1-\delta}{(p-1)(\frac{p}{p-1})^{j}+1}:=f_{p,\delta}(j).    
\end{aligned}
\end{equation}
Here, sequence $\{f_{p,\delta}(j)\}_{j=1}^\infty$ satisfies that, for each $p\in [2,\tilde{p}_m(\delta))$ and $\delta\in (7-2\sqrt{10},1)$,
\begin{equation*}
f_{p,\delta}(j) \text{ is strictly increasing as $j\to\infty$},\qquad\lim_{j\to\infty}f_{p,\delta}(j)=6\delta-3>1.   
\end{equation*}
Hence, for each $\gamma\in (1,6\delta-3)$, we can choose $j=j_0$ to be sufficiently large such that $\eqref{635}_3$ holds, 
which in turn yields $\eqref{635}_1$. 
As a consequence, fixing this constant $j_0$ (depending only on $(p,\gamma,\delta)$), 
we obtain from Lemma \ref{l4.3} that
\begin{align*}
\mathcal{I}_{(\mathfrak{a}_{j_0+1},\mathfrak{b}_{j_0+1})}=\int_0^\infty r^{\mathfrak{a}_{j_0+1}}\rho^{\mathfrak{b}_{j_0+1}}\,\mathrm{d}r \leq C(p).
\end{align*}
Substituting $\mathcal{I}_{(\mathfrak{a}_{j_0+1},\mathfrak{b}_{j_0+1})}$ into \eqref{g11-po} with $j=j_0$ yields
\begin{equation}\label{g11-po'}
\begin{aligned}
I_{3,1}&\leq \sum_{k=1}^{j_0}\Big(\prod_{l=0}^{k-1}\frac{C(p,\mathfrak{a}_l,\mathfrak{b}_l)}{\varepsilon_{l}}\Big)C(p,\mathfrak{a}_k,\mathfrak{b}_k)
+\Big(\prod_{l=1}^{j_0} \frac{C(p,\mathfrak{a}_l,\mathfrak{b}_l)}{\varepsilon_l}\Big)C(p)\\
&\quad + \sum_{k=1}^{j_0}\Big(\prod_{l=0}^{k-1}\frac{C(p,\mathfrak{a}_l,\mathfrak{b}_l)}{\varepsilon_l}\Big) \varepsilon_{k} \big|\chi_1^\flat(r^2\rho^{\gamma-\delta+1})^\frac{1}{p}(u,v)\big|_{p}^{p}.    
\end{aligned}    
\end{equation}

To further reduce inequality \eqref{g11-po'}, for $\tilde\varepsilon\in (0,1)$ sufficiently small, 
we set 
\begin{align*}
\varepsilon_{k}&=\frac{\tilde\varepsilon}{j_0}\,
\prod_{l=0}^{k-1}\frac{\varepsilon_l}{C(p,\mathfrak{a}_l,\mathfrak{b}_l)}=\big(\frac{\tilde\varepsilon}{j_0}\big)^{2^{k-1}}\prod_{l=0}^{k-1} \frac{1}{ C(p,\mathfrak{a}_l,\mathfrak{b}_l)^{2^{k-l-1}}} \qquad 
\text{for $1\leq k\leq j_0$}.
\end{align*}
Then it follows from the above, \eqref{g11-po'}, and Lemma \ref{l4.3} that
\begin{equation}\label{g1,1-2}
\begin{aligned}
I_{3,1}&\leq C(p,\tilde\varepsilon) + \Big(\sum_{k=1}^{j_0} \frac{\tilde\varepsilon}{j_0}\Big) \big|\chi_1^\flat(r^2\rho^{\gamma-\delta+1})^\frac{1}{p}(u,v)\big|_{p}^{p}\\[-2pt]
&\leq C(p,\tilde\varepsilon)+\tilde\varepsilon\big(\big|(r^2\rho^{\gamma-\delta+1})^\frac{1}{p}v\big|_{p}^{p}+\big|\chi_1^\flat r^\frac{2}{\gamma-2\delta+1}\rho\big|_\infty^{\gamma-2\delta+1}|\rho^\frac{\delta}{p} u|_{p}^{p}\big)\\[2pt]
&\leq C(p,\tilde\varepsilon)+ \tilde\varepsilon \big|(r^m\rho^{\gamma-\delta+1})^\frac{1}{p}v\big|_{p}^{p}+C_0\tilde\varepsilon\big|(r^{m-2}\rho^\delta)^\frac{1}{p} u\big|_{p}^{p},
\end{aligned}    
\end{equation}
where we have also used the fact that $\frac{2}{\gamma-2\delta+1}>\frac{1}{2\delta-1}$ whenever $\gamma\in (1,6\delta-3)$.

Thus, for both cases $n=2$ and $n=3$, collecting \eqref{J1}--\eqref{g1,1-1} and \eqref{g1,1-2}, then choosing $\tilde{\varepsilon}$ small enough such that
\begin{equation*}
0<\tilde{\varepsilon}< \min\big\{\frac{a_1 c_p}{16C(p)},\frac{\varepsilon}{C(p)},1\big\} \qquad\text{with $\varepsilon\in (0,1)$},
\end{equation*}
we conclude that, for all $\varepsilon\in (0,1)$,
\begin{equation}\label{g1''}
\begin{aligned}
I_3&\leq \frac{a_1 c_p}{8} \big|(r^m\rho^\delta)^{\frac12} |u|^{\frac{p-2}{2}} \mathrm{D}_r u\big|_2^2+C(p,\varepsilon)\big(1+\big|(r^m\rho)^\frac{1}{p}u\big|_{p}^{p}\big) + \varepsilon \big|(r^m\rho^{\gamma-\delta+1})^\frac{1}{p}v\big|_{p}^{p}.
\end{aligned}    
\end{equation}

Finally, substituting  \eqref{g1''} into \eqref{add} implies the desired estimate.
\end{proof}

In addition, we show the following $L^p(I)$-estimates for $(r^m\rho)^{\frac{1}{p}}v$.
\begin{lem}\label{lem-v-lp}
Let $\tilde{p}_m(\delta)$ be defined as in {\rm Lemma \ref{lemma-uv-lp}}. Then, for all $p\in [2,\tilde{p}_m(\delta))$, there exists a constant $C(p)>0$ such that
\begin{equation}\label{dt-v-p}
\frac{\mathrm{d}}{\mathrm{d}t}\big|(r^m\rho)^\frac{1}{p} v\big|_{p}^{p}+\frac{pA\gamma}{4a_1\delta}\big|(r^m\rho^{\gamma-\delta+1})^\frac{1}{p} v\big|_{p}^{p}\leq C(p)\big(\big|(r^{m-2}\rho^\delta)^\frac{1}{p} u\big|_{p}^{p}+\big|(r^m\rho)^\frac{1}{p} u\big|_{p}^{p}\big).
\end{equation}
\end{lem}
\begin{proof}
Let $\tilde{p}_m(\delta)$ be defined as in Lemma \ref{lemma-uv-lp}. First, multiplying \eqref{eq:effective2} by $r^m\rho|v|^{p-2}v$ with $p\in[2,\tilde{p}_m(\delta))$, along with $\eqref{e1.5}_1$, gives
\begin{equation}\label{430}
\frac{1}{p}(r^m\rho |v|^p)_t+ \frac{1}{p}\big(r^m\underline{\rho u|v|^p}_{:=\mathcal{B}_3}\big)_r+\frac{A\gamma}{2a_1\delta} r^m\rho^{\gamma-\delta+1} |v|^p=\frac{A\gamma}{2a_1\delta} r^m\rho^{\gamma-\delta+1} uv|v|^{p-2}.
\end{equation}

Next, we show that $r^m\mathrm{D}_r  \mathcal{B}_3 \in L^1(I)$ for \textit{a.e.} $t\in (0,T)$ so that we can apply Lemma \ref{calculus} 
to obtain
\begin{equation}\label{int-B3}
\int_0^\infty (r^m\mathcal{B}_3)_r\,\mathrm{d}r =0.
\end{equation}
Indeed, based on \eqref{spd}--\eqref{spd2}, we have 
\begin{equation*}
r^m\rho\in L^1(I),\quad r^\frac{m}{2} u \in L^2(I),\quad r^\frac{m}{n} (\rho^{\delta-1})_{rr} \in L^n(I),
\quad (\rho,u,\mathrm{D}_r u,(\rho^{\delta-1})_r)\in L^\infty(I)
\end{equation*} 
for {\it a.e.} $t\in (0,T)$. Thus, we obtain from the H\"older inequality that 
\begin{align*}
&\begin{aligned}
|r^m \mathrm{D}_r \mathcal{B}_3|_1&\leq C(p)\big(\big|r^{m-1}\rho |u|^{p+1}\big|_1+\big|r^{m-1}\rho u|(\rho^{\delta-1})_r|^p\big|_1+\big|r^{m}\rho^{2-\delta}(\rho^{\delta-1})_r |u|^{p+1}\big|_1\big)\\
&\quad+C(p)\big(\big|r^{m}\rho^{2-\delta} u|(\rho^{\delta-1})_r|^{p+1}\big|_1+\big|r^{m}\rho |u|^pu_r\big|_1+\big|r^{m}\rho u_r|(\rho^{\delta-1})_r|^p\big|_1\big)\\
&\quad +C(p)\big|r^{m}\rho u|(\rho^{\delta-1})_r|^{p-1}|(\rho^{\delta-1})_{rr}|\big|_1
\end{aligned}\\
&\qquad \quad \quad \hspace{2mm} \begin{aligned}
&\leq C(p)|r^m\rho|_1 |\mathrm{D}_r u|_\infty  |(u,(\rho^{\delta-1})_r)|_\infty^p +C(p)|r^m\rho|_1|\rho^{1-\delta}|_\infty|(\rho^{\delta-1})_r|_\infty|u|_\infty^{p+1}\\
&\quad+C(p)|r^m\rho|_1 |\rho|_\infty^{1-\delta} |(\rho^{\delta-1})_r|_\infty^{p+1}|u|_\infty \\
&\quad +C(p)|(\rho^{\delta-1})_r|_\infty^{p-1}|r^\frac{m}{2}u|_2|r^m\rho|_{1}^\frac{n-2}{2n}|\rho|_{\infty}^\frac{n+2}{2n}|r^\frac{m}{n}(\rho^{\delta-1})_{rr}|_n <\infty.
\end{aligned}
\end{align*}

Now, integrating \eqref{430} over $I$, then we see  from \eqref{int-B3}, Lemmas \ref{far-p-infty}--\ref{l4.3}, and the H\"older and Young inequalities that
\begin{equation*}
\begin{aligned}
&\,\frac{1}{p}\frac{\mathrm{d}}{\mathrm{d}t}\big|(r^m\rho)^\frac{1}{p} v\big|_{p}^{p}+\frac{A\gamma}{2a_1\delta}\big|(r^m\rho^{\gamma-\delta+1})^\frac{1}{p} v\big|_{p}^{p}\leq \frac{A\gamma}{2a_1\delta} \int_0^\infty r^m\rho^{\gamma-\delta+1} uv|v|^{p-2}\,\mathrm{d}r\\
&\leq \frac{A\gamma}{4a_1\delta}\big|(r^m\rho^{\gamma-\delta+1})^\frac{1}{p} v\big|_{p}^{p}+\frac{A\gamma}{4a_1\delta}\big(\big|\chi_1^\flat (r^m\rho^{\gamma-\delta+1})^\frac{1}{p}u\big|_p^p+\big|\chi_1^\sharp (r^m\rho^{\gamma-\delta+1})^\frac{1}{p} u\big|_p^p\big)\\
&\leq \frac{A\gamma}{4a_1\delta}\big|(r^m\rho^{\gamma-\delta+1})^\frac{1}{p} v\big|_{p}^{p}+ \frac{A\gamma}{4a_1\delta}\big|\chi_1^\flat r^\frac{2}{\gamma-2\delta+1}\rho\big|_\infty^{\gamma-2\delta+1} \big|(r^{m-2}\rho^\delta)^\frac{1}{p} u\big|_{p}^{p}\\
&\quad +\frac{A\gamma}{4a_1\delta}|\chi_1^\sharp \rho|_\infty^{\gamma-\delta} \big|(r^m\rho)^\frac{1}{p} u\big|_{p}^{p}\\
&\leq \frac{A\gamma}{4a_1\delta}\big|(r^m\rho^{\gamma-\delta+1})^\frac{1}{p} v\big|_{p}^{p}+C(p)\big(\big|(r^{m-2}\rho^\delta)^\frac{1}{p} u\big|_{p}^{p}+\big|(r^m\rho)^\frac{1}{p} u\big|_{p}^{p}\big),
\end{aligned}
\end{equation*}
where we have also used Lemma \ref{l4.3} with the facts that
\begin{equation*}
\frac{2}{\gamma-2\delta+1}>
\begin{cases}
0 \quad&\text{if $n=2$},\\
\frac{1}{2\delta-1}\quad &\text{if $n=3$ \ and \  $\gamma\in(1,6\delta-3)$}.   
\end{cases}
\end{equation*}

This completes the proof of Lemma \ref{lem-v-lp}.
\end{proof}

Now, combining Lemmas \ref{lem-u-lp}--\ref{lem-v-lp}, we give the proof for Lemma \ref{lemma-uv-lp}.

\begin{proof}[Proof for Lemma \ref{lemma-uv-lp}]
Let $p\in [2,\tilde{p}_{m}(\delta))$. First, multiplying both sides of \eqref{dt-v-p} by $\frac{8a_1\delta\varepsilon}{A\gamma p}$ 
with $\varepsilon\in (0,1)$, we arrive at
\begin{equation*} 
\frac{8a_1\delta\varepsilon}{A\gamma p}\frac{\mathrm{d}}{\mathrm{d}t}\big|(r^m\rho)^\frac{1}{p} v\big|_{p}^{p}+2\varepsilon\big|(r^m\rho^{\gamma-\delta+1})^\frac{1}{p} v\big|_{p}^{p}\leq C(p)\varepsilon\big(\big|(r^m\rho)^\frac{1}{p} u\big|_{p}^{p}+\big|(r^{m-2}\rho^\delta)^\frac{1}{p} u\big|_{p}^{p}\big).
\end{equation*}
Then summing the above inequality with \eqref{dt-u-p} yields 
\begin{equation*}
\begin{aligned}
&\,\frac{\mathrm{d}}{\mathrm{d}t}\Big(\frac{1}{p}\big|(r^m\rho)^{\frac{1}{p}} u\big|_{p}^{p}+\frac{8a_1\delta\varepsilon}{A\gamma p} \big|(r^m\rho)^\frac{1}{p} v\big|_{p}^{p}\Big) + a_1 c_p \big|(r^m\rho^\delta)^{\frac{1}{2}}|u|^{\frac{p-2}{2}}\mathrm{D}_r u\big|_2^2 +\varepsilon\big|(r^m\rho^{\gamma-\delta+1})^\frac{1}{p} v\big|_{p}^{p}\\
&\leq C(p)\big(1+\big|(r^m\rho)^\frac{1}{p}u\big|_{p}^{p}\big)+ C(p,\varepsilon)+C(p)\varepsilon\big|(r^{m-2}\rho^\delta)^\frac{1}{p} u\big|_{p}^{p}.
\end{aligned}
\end{equation*}

As a consequence, we can set
\begin{equation*}
\varepsilon=\min\big\{\frac{a_1 c_p}{100C(p)},\frac{1}{2}\big\},
\end{equation*}
and then apply the Gr\"onwall inequality to the resulting inequality to obtain
the desired estimate. For the $L^{p}(I)$-boundedness of $(r^m\rho_0)^{\frac{1}{p}}(u_0,v_0)$, we see from Lemmas \ref{ale1}, \ref{initial3}, and \ref{lemma-initial} that
\begin{align*}
\big|(r^m\rho_0)^\frac{1}{p}(u_0,v_0)\big|_p&\leq |r^m\rho_0|_1^\frac{1}{p}|(u_0,(\rho_0^{\delta-1})_r)|_{\infty}\leq C_0\|\rho_0\|_{L^1}^\frac{1}{p}\|(\boldsymbol{u}_0,\nabla \rho_0^{\delta-1})\|_{L^\infty} \leq C(p).
\end{align*}

This completes the proof.
\end{proof}

\subsection{Uniform Upper Bound of the Density}
\begin{lem}\label{important2}
There exists a constant $C(T)>0$ such that
\begin{equation*}
|\rho(t)|_\infty\leq C(T)  \qquad \text{for all $t\in[0,T]$}.
\end{equation*}
\end{lem}
\begin{proof}
By Lemma \ref{far-p-infty}, it remains to establish the uniform upper bound of $\rho$ near the origin:
\begin{equation}\label{near}
|\chi_1^\flat \rho(t)|_\infty \leq C(T) \qquad \text{for all $t\in [0,T]$}.
\end{equation}

First, it follows from Lemma \ref{l4.3} that, for both the 2-D case and 3-D case,
\begin{equation}\label{2d-1}
|\chi_1^\flat \rho(t)|_{2\delta-1}\leq C_0 \qquad\text{for all $t\in [0,T]$}.
\end{equation}

Next, note that $\tilde{p}_{m}(\delta)>n$ due to Lemma \ref{Lemma-pp}. Thus, for any fixed $p_0$ satisfying
\begin{equation*}
n<p_0<\min\big\{\frac{1}{1-\delta},\tilde{p}_{m}(\delta)\big\},
\end{equation*}
we obtain from \eqref{V-expression}, \eqref{2d-1}, Lemmas  \ref{lemma-uv-lp} and \ref{ale1}, and the H\"older and Young inequalities that, for all $t\in [0,T]$,
\begin{equation*}
\begin{aligned}
|\chi_1^\flat \rho|_\infty^{2\delta-1}&\leq  C_0|\chi_1^\flat (\rho^{2\delta-1},\rho^\delta(\rho^{\delta-1})_r)|_1 \leq  C_0+C_0|\chi_1^\flat\rho^{\delta} (u,v)|_1\\
&\leq  C_0+C_0|\chi_1^\flat r^{-\frac{m}{p_0-1}}|_1^\frac{p_0-1}{p_0}|\chi_1^\flat\rho|_\infty^\frac{\delta p_0-1}{p_0}|(r^m\rho)^\frac{1}{p_0}(u,v)|_{p_0}\leq C(T)+\frac{1}{20}|\chi_1^\flat \rho|_\infty^{2\delta-1},
\end{aligned}
\end{equation*}
where we have also used the fact that $0<\frac{\delta p_0-1}{p_0}<2\delta-1$.

Therefore, we obtain \eqref{near}, and hence the global uniform upper bound for $\rho$.
\end{proof}

\section{Global Uniform Bound of the Effective Velocity}\label{section-effective}

This section is devoted to establishing the global uniform $L^\infty(I)$-estimate of the effective velocity $v$ in spherical coordinates.  We first establish the following $L^2(I)$-estimate for $\rho^\frac{1}{2}u$.
\begin{lem}\label{rho u-L2}
There exists a constant $C(T)>0$ such that
\begin{equation*}
|\rho^\frac{1}{2}u(t)|_2^2+ \int_0^t|\rho^\frac{\delta}{2}\mathrm{D}_r u|_2^2\,\mathrm{d}s\leq C(T)\Big(\sup_{s\in [0,t]}|v|_\infty^2+1\Big)\qquad \text{for all $t\in [0,T]$}.
\end{equation*}
\end{lem}
\begin{proof}
Multiplying both sides of $\eqref{e1.5}_2$ by $u$, along with $\eqref{e1.5}_1$ and \eqref{V-expression}, leads to
\begin{equation}\label{5-1}
\begin{aligned}
&\,\frac{1}{2}(\rho u^2)_t + 2a_1 \rho^\delta \Big(\delta u_r^2+ (\delta-1) m\frac{u u_r}{r}+  \frac{m}{2} \frac{u^2}{r^2}\Big)\\
&=\Big(\underline{2a_1\delta\rho^\delta uu_r+ a_1 (2\delta-1) m \frac{\rho^\delta u^2}{r}-\frac{1}{2}\rho u^3-A\rho^\gamma u}_{:=\mathcal{B}_4}\Big)_r+A\rho^\gamma u_r -\frac{m}{2}\frac{\rho vu^2}{r}.
\end{aligned}
\end{equation}
Here we need to show  that $\mathcal{B}_4\in W^{1,1}(I)$ and $\mathcal{B}_4|_{r=0}=0$ for {\it a.e.} $t\in (0,T)$, which allows us to apply Lemma \ref{calculus} to obtain
\begin{equation}\label{eq:B4}
\int_0^\infty (\mathcal{B}_4)_r\,\mathrm{d}r=-\mathcal{B}_4|_{r=0}=0.
\end{equation}
This can be seen as follows: 
On one hand, $\mathcal{B}_4|_{r=0}=0$ follows directly 
from \eqref{V-expression}, $u|_{r=0}=0$, 
and the fact that $(\rho,u,\mathrm{D}_r u)\in C(\bar I)$ for each $t\in (0,T]$ due to \eqref{spd2}. 
On the other hand,  if $n=2$ $(m=1)$, \eqref{spd}--\eqref{spd2} 
imply  
\begin{equation*}
\begin{aligned}
r\rho\in L^1(I),\quad \sqrt{r} (u,\mathrm{D}_r u,\rho^\frac{\delta-1}{2}\mathrm{D}_r u,\mathrm{D}_r^2 u)\in L^2(I),\quad (\rho,\mathrm{D}_r u,u_r,(\rho^{\delta-1})_r)\in L^\infty(I) 
\end{aligned}
\end{equation*}
for {\it a.e.} $t\in (0,T)$, so that 
\begin{equation*}
\begin{aligned}
|\rho|_1&\leq |\chi_1^\flat\rho|_1+|\chi_1^\sharp\rho|_1\leq |\rho|_\infty+|\chi_1^\sharp r^{-1}|_\infty|r\rho|_1<\infty, \\
|\rho^\frac{\delta}{2}\mathrm{D}_r u|_2&\leq |\chi_1^\flat\rho^\frac{\delta}{2}\mathrm{D}_r u|_2+ |\chi_1^\sharp\rho^\frac{\delta}{2}\mathrm{D}_r u|_2 \leq |\rho|_\infty^\frac{\delta}{2} |\mathrm{D}_r u|_\infty+|\chi_1^\sharp r^{-\frac{1}{2}}|_\infty|\rho|_\infty^\frac{1}{2} |\chi_1^\sharp \sqrt{r}\rho^\frac{\delta-1}{2}\mathrm{D}_r u|_2<\infty,
\end{aligned}
\end{equation*}
which, along with the H\"older inequality, yields 
\begin{align*}
&\quad \,\begin{aligned}
|\mathcal{B}_4|_1 &\leq  C_0\big(|\rho^\delta u\mathrm{D}_r u|_1+\big|\rho |u|^{3}\big|_1+|\rho^\gamma u|_1\big)\\
&\leq C_0\big(|\rho|_\infty^\delta |\sqrt{r}\mathrm{D}_r u|_2^2+|\rho|_\infty|u|_\infty|\sqrt{r}u|_2 |\sqrt{r}\mathrm{D}_r u|_2+|r\rho|_1|\rho|_\infty^{\gamma-1} |\mathrm{D}_r u|_\infty\big)<\infty,
\end{aligned}\\[1mm]
&\begin{aligned}
|(\mathcal{B}_4)_r|_1&\leq C_0\big|\big(\rho(\rho^{\delta-1})_ru\mathrm{D}_r u,\rho^\delta u\mathrm{D}_r^2 u,\rho^\delta u_r\mathrm{D}_r u\big)\big|_1\\
&\quad +C_0\big|\big(\rho^{2-\delta}(\rho^{\delta-1})_r u^3,\rho u^2u_r\big)\big|_1 + C_0\big|\big(\rho^{\gamma-\delta+1}(\rho^{\delta-1})_r u,\rho^{\gamma} u_r\big)\big|_1
\end{aligned}\\
&\qquad\quad \ \begin{aligned}
&\leq C_0|\rho|_\infty|(\rho^{\delta-1})_r|_\infty|\sqrt{r}\mathrm{D}_r u|_2^2+C_0|\rho|_\infty^\delta  |\sqrt{r}\mathrm{D}_r u|_2|\sqrt{r}\mathrm{D}_r^2 u|_2\\
&\quad +C_0|\rho^\frac{\delta}{2}u_r|_2 |\rho^\frac{\delta}{2}\mathrm{D}_r u|_2+C_0|\rho|_1\big(|\rho|_\infty^{1-\delta}|(\rho^{\delta-1})_r|_\infty|u|_\infty^3+|u|_\infty^2|u_r|_\infty\big)\\
&\quad+C_0|\rho|_1\big(|\rho|_\infty^{\gamma-\delta}|(\rho^{\delta-1})_r|_\infty|u|_\infty+|\rho|_\infty^{\gamma-1}|u_r|_\infty\big)<\infty.
\end{aligned}
\end{align*}
Furthermore, if $n=3$ $(m=2)$, \eqref{spd}--\eqref{spd2}  imply
\begin{equation*}
(u,r\mathrm{D}_r u,r\mathrm{D}_r^2 u)\in L^2(I),\quad (\rho,u,\mathrm{D}_r u,(\rho^{\delta-1})_r )\in L^\infty(I),
\quad r^2\rho\in L^1(I) 
\end{equation*}
for {\it a.e.} $t\in (0,T)$, so that
\begin{equation*}
\begin{aligned}
|\rho|_1&\leq |\chi_1^\flat\rho|_1+|\chi_1^\sharp\rho|_1\leq |\rho|_\infty+ |\chi_1^\sharp r^{-2}|_\infty |r^2\rho|_1<\infty,\qquad |\rho^\delta|_{2}\leq |\rho|_1^\frac{1}{2}|\rho|_\infty^{\delta-\frac{1}{2}}<\infty,\\
|\rho^\frac{\delta}{2}\mathrm{D}_r u|_2&\leq  |\chi_1^\flat\rho^\frac{\delta}{2}\mathrm{D}_r u|_2+ |\chi_1^\sharp\rho^\frac{\delta}{2}\mathrm{D}_r u|_2 \leq |\rho|_\infty^\frac{\delta}{2} |\mathrm{D}_r u|_\infty+|\chi_1^\sharp r^{-1}|_\infty|\rho|_\infty^\frac{1}{2} |\chi_1^\sharp r\rho^\frac{\delta-1}{2}\mathrm{D}_r u|_2<\infty,
\end{aligned}
\end{equation*}
which, along with the H\"older inequality, yields
\begin{align*}
&\quad \,\begin{aligned}
|\mathcal{B}_4|_1 &\leq C_0\big( |\rho^\delta u\mathrm{D}_r u|_1+\big|\rho |u|^3\big|_1+|\rho^\gamma u|_1\big)\\
&\leq C_0|\rho|_\infty^\frac{\delta}{2}|u|_2 |\rho^\frac{\delta}{2}\mathrm{D}_r u|_2+C_0|\rho|_1\big(|u|_\infty^3+|\rho|_\infty^{\gamma-1}|u|_\infty\big) <\infty,
\end{aligned}\\[1mm]
&\begin{aligned}
|(\mathcal{B}_4)_r|_1&\leq C_0\big|\big(\rho(\rho^{\delta-1})_ru\mathrm{D}_r u,\rho^\delta u\mathrm{D}_r^2 u,\rho^\delta u_r\mathrm{D}_r u\big)\big|_1\\
&\quad +C_0\big|\big(\rho^{2-\delta}(\rho^{\delta-1})_r u^3,\rho u^2u_r\big)\big|_1 + C_0\big|\big(\rho^{\gamma-\delta+1}(\rho^{\delta-1})_r u,\rho^{\gamma} u_r\big)\big|_1
\end{aligned}\\
&\qquad\quad \ \begin{aligned}
&\leq C_0|\rho|_2|(\rho^{\delta-1})_r|_\infty|u|_2 |\mathrm{D}_r u|_\infty+C_0|\rho^\delta|_2 |\mathrm{D}_r u|_\infty |r\mathrm{D}_r^2 u|_2\\
&\quad +C_0|\rho^\frac{\delta}{2}u_r|_2 |\rho^\frac{\delta}{2}\mathrm{D}_r u|_2+C_0|\rho|_1\big(|\rho|_\infty^{1-\delta}|(\rho^{\delta-1})_r|_\infty|u|_\infty^3+|u|_\infty^2|u_r|_\infty\big)\\
&\quad+C_0|\rho|_1\big(|\rho|_\infty^{\gamma-\delta}|(\rho^{\delta-1})_r|_\infty|u|_\infty+|\rho|_\infty^{\gamma-1}|u_r|_\infty\big)<\infty.
\end{aligned}
\end{align*}

Integrating \eqref{5-1} over $I$, we obtain from \eqref{eq:B4}, Lemmas \ref{far-p-infty} and \ref{important2}, and the H\"older and Young inequalities that
\begin{equation}\label{503}
\begin{aligned}
&\,\frac{1}{2}\frac{\mathrm{d}}{\mathrm{d}t}|\rho^\frac{1}{2}u|_2^2+ \underline{2a_1 \int_0^\infty\rho^\delta \Big(\delta u_r^2+ (\delta-1) m\frac{u u_r}{r}+  \frac{m}{2} \frac{u^2}{r^2}\Big)\,\mathrm{d}r}_{:=I_4}\\
&=A\int_0^\infty \rho^\gamma u_r\,\mathrm{d}r \ \underline{-\frac{m}{2}\int_0^\infty \frac{\rho v u^2}{r}\,\mathrm{d}r}_{:=I_5}\leq C_0|\rho|_{2\gamma-\delta}^{\gamma-\frac{\delta}{2}}|\rho^\frac{\delta}{2}u_r|_2+I_5\\
&\leq C_0\big(|\chi_1^\flat\rho|_{\infty}^{\gamma-\frac{\delta}{2}}+|\chi_1^\sharp r^{-m}|_\infty^{\gamma-\frac{\delta}{2}}\big|\chi_1^\sharp r^m\rho|_{2\gamma-\delta}^{\gamma-\frac{\delta}{2}}\big) |\rho^\frac{\delta}{2}u_r|_2+I_5\\
&\leq C(\varepsilon,T) +\varepsilon|\rho^\frac{\delta}{2}u_r|_2^2+I_5.
\end{aligned}
\end{equation}

First, for $I_4$, according to the Young inequality
\begin{equation*}
\Big|(\delta-1)mu_r \frac{u }{r}\Big|\leq m(\delta-1)^2u_r^2+\frac{m}{4}\frac{u^2}{r^2},
\end{equation*}
we can find a constant $c_*>0$, depending only on $(a_1,n,\delta)$, such that  
\begin{equation}
I_4\geq 2a_1 \int_0^\infty\rho^\delta \Big((-m\delta^2+(2m+1)\delta-m)u_r^2+ \frac{m}{4} \frac{u^2}{r^2}\Big)\,\mathrm{d}r \geq  c_* |\rho^\frac{\delta}{2}\mathrm{D}_r u|_2^2,
\end{equation}
where we have also used the fact that $-m\delta^2+(2m+1)\delta-m>0$ for $\delta\in (\frac{m}{m+1},1)$.

For $I_5$, it follows from Lemma \ref{important2} and the H\"older and Young inequalities that
\begin{equation*}
\begin{aligned}
\text{if $n=2$:}\quad I_5&\leq C_0|v|_\infty |\rho|_\infty^{1-\delta}\Big|(r\rho^\delta)^\frac{1}{2}\frac{u}{r}\Big|_2^2 \leq C(T)|v|_\infty \big|(r^{m-2}\rho^\delta)^\frac{1}{2}u\big|_2^2;\\
\text{if $n=3$:}\quad I_5&\leq C_0|v|_\infty |\rho|_\infty^{1-\delta}|\rho^\frac{\delta}{2} u|_{2}\Big|\rho^\frac{\delta}{2}\frac{u}{r}\Big|_2 \leq C(T)|v|_\infty^2\big|(r^{m-2}\rho^\delta)^\frac{1}{2}u\big|_2^2+\frac{c_*}{8}\Big|\rho^\frac{\delta}{2}\frac{u}{r}\Big|_2^2,
\end{aligned}
\end{equation*}
which yields that, for both $n=2$ and $n=3$,
\begin{equation}\label{R2-dim2}
I_5\leq C(T)(1+|v|_\infty^2)\big|(r^{m-2}\rho^\delta)^\frac{1}{2}u\big|_2^2+\frac{c_*}{8}\Big|\rho^\frac{\delta}{2}\frac{u}{r}\Big|_2^2.
\end{equation}

Thus, collecting \eqref{503}--\eqref{R2-dim2} and choosing $\varepsilon$ sufficiently small, we have 
\begin{equation*} 
\frac{\mathrm{d}}{\mathrm{d}t}|\rho^\frac{1}{2}u|_2^2+c_* |\rho^\frac{\delta}{2}\mathrm{D}_r u|_2^2 \leq C(T)+C(T)\big(1+|v|_\infty^2\big)\big|(r^{m-2}\rho^\delta)^\frac{1}{2}u\big|_2^2.
\end{equation*}

Finally, integrating the above over $[0,t]$, together with Lemma \ref{lemma-uv-lp}, yields the desired result. To check the $L^2(I)$-boundedness of $\rho_0^{\frac{1}{2}}u_0$, we see from Lemmas \ref{ale1}, \ref{initial3}, and \ref{lemma-initial} that
\begin{equation*}
\begin{aligned}
|\rho_0^\frac{1}{2}u_0|_2&\leq |\chi_1^\flat \rho_0^\frac{1}{2}u_0|_2+|\chi_1^\sharp \rho_0^\frac{1}{2}u_0|_2\leq |\rho_0|_\infty^\frac{1}{2}|u_0|_\infty+ |\rho_0|_\infty^\frac{1}{2}|\chi_1^\sharp r^{-\frac{m}{2}}|_\infty |r^\frac{m}{2}u_0|_2\\
&\leq C_0\|\rho_0\|_{L^\infty}^\frac{1}{2}\big(\|\boldsymbol{u}_0\|_{L^\infty}+\|\boldsymbol{u}_0\|_{L^2}\big)\leq C_0\|\boldsymbol{u}_0\|_{H^2}\leq C_0.
\end{aligned}
\end{equation*}

This completes the proof of Lemma \ref{rho u-L2}.
\end{proof}

Next, we aim to derive the $L^2([0,T];L^\infty(I))$-estimate for $\rho^{1-\delta} u$. We first need an auxiliary lemma associated with 
parameter $\tilde{p}_m(\delta)$.
\begin{lem}\label{lemma-ppp}
Let $\tilde{p}_m(\delta)$ be defined as in {\rm Lemma \ref{lemma-uv-lp}}, and let $(\delta,\gamma)$ satisfy \eqref{del-gam}.
\begin{itemize}
\item[{\rm(i)}] If $n=2$ $(m=1)$, then
\begin{equation}\label{510'}
\delta<f^\mathsf{a}(\tilde{p}_{1}(\delta)), \qquad\delta<f^\mathsf{b}(\tilde{p}_{1}(\delta)),
\end{equation}
with $f^\mathsf{a}(\tau):=\frac{2\tau-1}{2\tau+1}$ and $f^\mathsf{b}(\tau):=\frac{2\tau+3}{2\tau+7};$  

\item[{\rm(ii)}] If $n=3$ $(m=2)$, then
\begin{equation}\label{510}
\delta<f^\mathsf{c}(\tilde{p}_{2}(\delta)),\qquad \text{with } f^\mathsf{c}(\tau):=\frac{\tau+2}{\tau+4}.
\end{equation}
\end{itemize}
\end{lem}
\begin{proof}
We divide the proof into two steps.

\smallskip
\textbf{1. Proof for \eqref{510'}.} Since $\tilde{p}_1(\delta)>2$ in view of Lemma \ref{Lemma-pp},  a direct calculation gives
\begin{equation*}
\begin{aligned}
\delta<f^\mathsf{a}(\tilde{p}_{1}(\delta))&\iff 5\delta^2-1+4\delta\sqrt{2\delta-1}>0,\\
\delta<f^\mathsf{b}(\tilde{p}_{1}(\delta))&\iff 11\delta^2-10\delta+3+4\delta\sqrt{2\delta-1}>0,
\end{aligned}
\end{equation*}
which, of course, leads to \eqref{510'}, due to $\delta\in(\frac{1}{2},1)$.

\smallskip
\textbf{2. Proof for \eqref{510}.} Similarly, via a direct calculation, we have
\begin{equation*}
\delta<f^\mathsf{c}(\tilde{p}_{2}(\delta))\iff (2\delta-1)(3\delta-2)+\sqrt{\delta(2\delta-1)(3\delta-2)}>0,
\end{equation*}
and the right-hand sides of the above hold automatically, due to $1>\delta>\frac{2}{3} >\frac{1}{2}$.
\end{proof}

As a direct consequence of Lemma \ref{lemma-ppp}, we have
\begin{lem}\label{fffaa}
Let $\tilde{p}_m(\delta)$ be defined as in {\rm Lemma \ref{lemma-uv-lp}}, and let $(\delta,\gamma)$ satisfy \eqref{del-gam}. Then there exists a parameter $q\in (n,\tilde{p}_{m}(\delta))$  such that
\begin{itemize}
\item[{\rm(i)}] if $n=2$ $(m=1)$, then
\begin{equation} 
\qquad(2q-1)-(2q+1)\delta\geq 0,\quad(2q+3)-(2q+7)\delta\geq 0,\quad
(2q+1)-(2q+3)\delta> 0;
\end{equation}
\item[{\rm(ii)}] if $n=3$ $(m=2)$, then
\begin{equation} 
 (q+2)-(q+4)\delta\geq 0,\qquad q-(q+1)\delta\geq 0,\qquad
(q+1)-(q+2)\delta>0.
\end{equation}
\end{itemize}
\end{lem}
\begin{proof}
Based on Lemma \ref{lemma-ppp}, since $\tilde{p}_{m}(\delta)>n$ in view of Lemma \ref{Lemma-pp} and $(f^\mathsf{a},f^\mathsf{b},f^\mathsf{c})(\tau)$ are all strictly increasing with respect of $\tau$ for $\tau\in (1,\infty)$, we can find $q\in (n,\tilde{p}_{m}(\delta))$, sufficiently closing to $\tilde{p}_m(\delta)$, such that
\begin{equation*} 
\delta\leq g(q)<g(\tilde{p}_{m}(\delta)) \qquad \text{for function $g=f^{\mathsf{a}}$, $f^{\mathsf{b}}$, or $f^{\mathsf{c}}$}.
\end{equation*}
Therefore, we can derive from the above that such a parameter $q$ satisfies, if $n=2$ ($m=1$), 
\begin{equation*}
\begin{aligned}
&(2q-1)-(2q+1)\delta\geq 0,\quad(2q+3)-(2q+7)\delta\geq 0,\\[4pt]
&(2q+1)-(2q+3)\delta=(2q-1)-(2q+1)\delta+2(1-\delta)> 0;
\end{aligned}
\end{equation*}
if $n=3$ $(m=2)$,
\begin{equation*}
\begin{aligned}
&(q+2)-(q+4)\delta\geq 0,\\[4pt]
&q-(q+1)\delta=(q+2)-(q+4)\delta+(3\delta-2)>0,\\[4pt]
&(q+1)-(q+2)\delta=(q+2)-(q+4)\delta+(2\delta-1)>0.
\end{aligned}
\end{equation*}

This completes the proof.
\end{proof}

Now, we establish the $L^2([0,T];L^\infty(I))$-estimate for $\rho^{1-\delta} u$.

\begin{lem}\label{cru3}
For any $\varepsilon\in (0,1)$, there exists a constant $C(\varepsilon,T)>0$ such that
\begin{equation*}
\int_0^t |\rho^{1-\delta}u|_{\infty}^2\, \mathrm{d}s\leq C(\varepsilon,T)\Big(1+ \int_0^t |v|_\infty^2\,\mathrm{d}s\Big)+\varepsilon \sup_{s\in[0,t]}|v|_\infty^2\qquad\text{for all $t\in [0,T]$}.
\end{equation*}
\end{lem}
\begin{proof}
We divide the proof into three steps.

\smallskip
\textbf{1.} Let the parameter $q\in (n,\tilde{p}_m(\delta))$ be chosen as in Lemma \ref{fffaa}. First, by \eqref{V-expression}, Lemma \ref{calculus}, and the H\"older inequality, we have
\begin{equation}\label{5122}
\begin{aligned}
|\rho^{1-\delta}u|_\infty^{q}&\leq |(\rho^{q(1-\delta)} |u|^{q})_r|_1\leq C_0\big(\big|\rho^{(q+1)(1-\delta)}|u|^{q}(v,u)\big|_1+\big|\rho^{q(1-\delta)} |u|^{q-1}u_r\big|_1\big)\\[6pt]
&\leq C_0|v|_\infty\big|\rho^{(q+1)(1-\delta)}u^{q}\big|_1+C_0|\rho^{1-\delta}u|_\infty \big|\rho^{q(1-\delta)}u^{q}\big|_1\\[-4pt]
&\quad +C_0|\rho^{1-\delta}u|_\infty^\frac{q-2}{2}\big|\rho^{(q+2)-(q+3)\delta}u^q \big|_1^\frac{1}{2}|\rho^\frac{\delta}{2}u_r|_2 :=\sum_{i=6}^{8} I_i.
\end{aligned}
\end{equation}

\smallskip
\textbf{2.} If $n=2$ $(m=1)$,  by Lemmas \ref{lemma-uv-lp}, \ref{important2}, and \ref{fffaa}, and the Young inequality, we have
\begin{align}
&\begin{aligned}
I_6&\leq C_0|v|_\infty|\rho|_\infty^\frac{(2q+1)-(2q+3)\delta}{2}  \big|(r^{-1}\rho^\delta)^\frac{1}{q}u\big|_q^\frac{q}{2} \big|(r\rho)^\frac{1}{q}u\big|_q^\frac{q}{2}\leq C_0|v|_\infty\big|(r^{m-2}\rho^\delta)^\frac{1}{q}u\big|_q^\frac{q}{2},\notag 
\end{aligned}\\
&\begin{aligned}[b]
I_7&\leq C_0|\rho^{1-\delta}u|_\infty|\rho|_\infty^\frac{(2q-1)-(2q+1)\delta}{2} \big|(r^{-1}\rho^\delta)^\frac{1}{q}u\big|_q^\frac{q}{2} \big|(r\rho)^\frac{1}{q}u\big|_q^\frac{q}{2} \\
&\leq C(T)\big|(r^{m-2}\rho^\delta)^\frac{1}{q}u\big|_q^\frac{q^2}{2q-2}+\frac{1}{20}|\rho^{1-\delta}u|_\infty^{q},
\end{aligned}\label{5122'}\\
&\begin{aligned}
I_{8}&\leq C_0|\rho^{1-\delta}u|_\infty^\frac{q-2}{2}|\rho^\frac{\delta}{2}u_r|_2|\rho|_\infty^{\frac{(2q+3)-(2q+7)\delta}{4}} \big|(r^{-1}\rho^\delta)^\frac{1}{q}u\big|_q^\frac{q}{4}\big|(r\rho)^\frac{1}{q}u\big|_q^\frac{q}{4}\notag\\
&\leq C(T)|\rho^\frac{\delta}{2}u_r|_2^\frac{2q}{q+2}\big|(r^{m-2}\rho^\delta)^\frac{1}{q}u\big|_q^\frac{q^2}{2q+4} +\frac{1}{20}|\rho^{1-\delta}u|_\infty^{q}.
\end{aligned}
\end{align}

Combining \eqref{5122}--\eqref{5122'}, we arrive at the estimate of form:
\begin{equation*}
|\rho^{1-\delta}u|_\infty^2\leq C(T)\Big(|v|_\infty^\frac{2}{q}\big|(r^{m-2}\rho^\delta)^\frac{1}{q}u\big|_q+\big|(r^{m-2}\rho^\delta)^\frac{1}{q}u\big|_q^\frac{q}{q-1}+|\rho^\frac{\delta}{2}u_r|_2^\frac{4}{q+2}\big|(r^{m-2}\rho^\delta)^\frac{1}{q}u\big|_q^\frac{q}{q+2}\Big).
\end{equation*}
Since $q>2$, integrating the above over $[0,t]$, along with Lemmas \ref{lemma-uv-lp} and \ref{rho u-L2}, and the H\"older and Young inequalities, yields that, for all $\varepsilon\in (0,1)$,
\begin{equation*}
\begin{aligned}
\int_0^t |\rho^{1-\delta}u|_\infty^2\,\mathrm{d}s&\leq C(T)\int_0^t\Big(|v|_\infty^\frac{2}{q}\big|(r^{m-2}\rho^\delta)^\frac{1}{q}u\big|_q+\big|(r^{m-2}\rho^\delta)^\frac{1}{q}u\big|_q^\frac{q}{q-1}\Big)\,\mathrm{d}s\\
&\quad +C(T)\Big(\int_0^t |\rho^\frac{\delta}{2}u_r|_2^2\,\mathrm{d}s\Big)^\frac{2}{q+2}\Big(\int_0^t \big|(r^{m-2}\rho^\delta)^\frac{1}{q}u\big|_q \,\mathrm{d}s\Big)^\frac{q}{q+2}\\
&\leq C(\varepsilon,T)\int_0^t \big(|v|_\infty^2 +\big|(r^{m-2}\rho^\delta)^\frac{1}{q}u\big|_q^q+1\big)\,\mathrm{d}s +\varepsilon\sup_{s\in[0,t]}|v|_\infty^2\\
&\leq C(\varepsilon,T)\Big(1+ \int_0^t |v|_\infty^2\,\mathrm{d}s\Big)+\varepsilon \sup_{s\in[0,t]}|v|_\infty^2.
\end{aligned}
\end{equation*}

\smallskip
\textbf{3.} If $n=3$ $(m=2)$, we can similarly obtain from Lemmas \ref{important2} and \ref{fffaa}, and the Young inequality that
\begin{equation}\label{5123}
\begin{aligned}
I_6&\leq C_0|v|_\infty|\rho|_\infty^{(q+1)-(q+2)\delta}|\rho^{\frac{\delta}{q}}u|_{q}^{q}\leq C(T) |v|_\infty\big|(r^{m-2}\rho^\delta)^\frac{1}{q}u\big|_q^{q}, \\
I_7&\leq C_0|\rho^{1-\delta}u|_\infty|\rho|_\infty^{q-(q+1)\delta}|\rho^{\frac{\delta}{q}}u|_{q}^{q}\leq C(T)\big|(r^{m-2}\rho^\delta)^\frac{1}{q}u\big|_q^\frac{q^2}{q-1}+\frac{1}{20}|\rho^{1-\delta}u|_\infty^{q},\\
I_{8}&\leq C_0|\rho^{1-\delta}u|_\infty^\frac{q-2}{2}|\rho|_\infty^{\frac{(q+2)-(q+4)\delta}{2}}|\rho^{\frac{\delta}{q}}u|_{q}^\frac{q}{2}|\rho^\frac{\delta}{2}u_r|_2\\
&\leq C(T)\big(|\rho^\frac{\delta}{2}u_r|_2^{\frac{2q}{3}}+\big|(r^{m-2}\rho^\delta)^\frac{1}{q}u\big|_q^\frac{q^2}{q-1}\big) +\frac{1}{20}|\rho^{1-\delta}u|_\infty^{q}.
\end{aligned}
\end{equation}
 
Combining \eqref{5122} and \eqref{5123}, we arrive at the estimate of form:
\begin{equation*}
|\rho^{1-\delta}u|_\infty^2\leq C(T)\Big(|v|_\infty^\frac{2}{q} \big|(r^{m-2}\rho^\delta)^\frac{1}{q}u\big|_q^2+\big|(r^{m-2}\rho^\delta)^\frac{1}{q}u\big|_q^\frac{2q}{q-1}+ |\rho^\frac{\delta}{2}u_r|_2^{\frac{4}{3}}\Big).
\end{equation*}
Since $q>3$, integrating the above over $[0,t]$, along with Lemmas \ref{lemma-uv-lp} and \ref{rho u-L2}, and the H\"older and Young inequalities, yields that, for all $\varepsilon\in (0,1)$,
\begin{equation*}
\begin{aligned}
\int_0^t |\rho^{1-\delta}u|_\infty^2\,\mathrm{d}s&\leq C(T)\int_0^t\Big(|v|_\infty^\frac{2}{q} \big|(r^{m-2}\rho^\delta)^\frac{1}{q}u\big|_q^2+\big|(r^{m-2}\rho^\delta)^\frac{1}{q}u\big|_q^\frac{2q}{q-1}+ |\rho^\frac{\delta}{2}u_r|_2^{\frac{4}{3}}\Big)\,\mathrm{d}s\\
&\leq C(T)\int_0^t \big(|v|_\infty^2 +\big|(r^{m-2}\rho^\delta)^\frac{1}{q}u\big|_q^q+1\big)\,\mathrm{d}s +C(T)\Big(\int_0^t |\rho^\frac{\delta}{2}u_r|_2^{2} \,\mathrm{d}s\Big)^\frac{2}{3}\\
&\leq C(T)\Big(1+ \int_0^t |v|_\infty^2\,\mathrm{d}s\Big)+C(T)\Big(\sup_{s\in[0,t]}|v|_\infty^2\Big)^\frac{4}{3}\\
&\leq C(\varepsilon,T)\Big(1+ \int_0^t |v|_\infty^2\,\mathrm{d}s\Big)+\varepsilon \sup_{s\in[0,t]}|v|_\infty^2.
\end{aligned}
\end{equation*}

This completes the proof of Lemma \ref{cru3}.
\end{proof}

Finally, the uniform $L^\infty(I)$-estimate of $v$ can be derived as follows.
\begin{lem}\label{l4.4}
There exists a constant $C(T)>0$ such that
\begin{equation*}
|v(t)|_\infty\leq C(T) \qquad \text{for all $t\in [0,T]$}.
\end{equation*}
\end{lem}
\begin{proof} 
First, define the flow map $X(t,r): [0,T]\times I \to I$ as 
\begin{equation*}
X_t(t,r)=u(t,X(t,r))\qquad\text{with} \quad X(0,r)=r.
\end{equation*}
Then \eqref{eq:effective2}, together with the above, implies the following ODE:
\begin{equation*}
\frac{\mathrm{d}}{\mathrm{d}t}v(t,X(t,r))+\frac{A\gamma}{2a_1\delta}(\rho^{\gamma-\delta}v)(t,X(t,r))=\frac{A\gamma}{2a_1\delta}(\rho^{\gamma-\delta}u)(t,X(t,r)),
\end{equation*}
which, along with the method of characteristic, yields
\begin{equation}\label{try}
\begin{aligned}
v(t,X(t,r))&=v_0(r)\exp\Big(-\int_0^t \frac{A\gamma}{2a_1\delta} \rho^{\gamma-\delta}(\tau,X(\tau,r))\,\mathrm{d}\tau\Big)\\
&\quad +\frac{A\gamma}{2a_1\delta}\int_0^t (\rho^{\gamma-\delta}u)(s,X(s,r))\ \exp\Big(\!-\!\int_s^t \frac{A\gamma}{2a_1\delta} \rho^{\gamma-\delta}(\tau,X(\tau,r))\,\mathrm{d}\tau\Big)\,\mathrm{d}s.
\end{aligned}
\end{equation}

Notice that $\rho\geq 0$, and  $v_0$ is uniformly bounded owing to Lemmas \ref{ale1}, \ref{initial3}, and \ref{lemma-initial}:
\begin{equation*}
|v_0|_\infty\leq C_0|(u_0,(\rho_0^{\delta-1})_r)|_\infty \leq C_0\|(\boldsymbol{u}_0,\nabla(\rho_0^{\delta-1}))\|_{L^\infty}\leq C_0(\|\boldsymbol{u}_0\|_{H^2}+1)\leq C_0.
\end{equation*}
Then \eqref{try}, together with the above, Lemmas \ref{important2} and \ref{cru3}, and the H\"older inequality, yields that, for any $\varepsilon\in (0,1)$,
\begin{equation*} 
\begin{aligned}
\sup_{s\in[0,t]}|v|_\infty&\leq  C_0\Big(|v_0|_\infty+\int_0^t|\rho|_\infty^{\gamma-1}|\rho^{1-\delta}u|_\infty\mathrm{d}s\Big) \leq  C_0+C(T)\Big(\int_0^t |\rho^{1-\delta}u|_\infty^2\mathrm{d}s\Big)^\frac{1}{2}\\
&\leq C_0+C(T)\Big(C(\varepsilon,T)\Big(1+\int_0^t |v|_\infty^2\,\mathrm{d}s\Big)+ \varepsilon\sup_{s\in[0,t]}|v|_\infty^2\Big)^\frac{1}{2},
\end{aligned}
\end{equation*}
which leads to
\begin{equation}\label{e-1.17} 
\sup_{s\in[0,t]}|v|_\infty^2 \leq  C(\varepsilon,T)\big(1+\int_0^t |v|_\infty^2\,\mathrm{d}s\big)+ C(T)\varepsilon\sup_{s\in[0,t]}|v|_\infty^2, 
\end{equation}
where we have used the fact that $\sup_{s\in[0,t]}|v|_\infty^2=\big(\sup_{s\in[0,t]}|v|_\infty\big)^2$. 

Finally, choosing $\varepsilon$ in \eqref{e-1.17} sufficiently small and then applying the Gr\"onwall inequality provide the desired result.
\end{proof}

\section{Non-Formation of Cavitation
inside the Fluids  in Finite Time}\label{section-nonformation}

This section is devoted to showing that the cavitation does not form inside the fluids in finite 
time and establishing the lower-bound estimate for the density.

\subsection{Refined $\rho$-Weighted Estimates for the Velocity}\label{sub-refine}

First, by Lemmas \ref{cru3}--\ref{l4.4}, we can directly obtain the $L^2([0,T];L^\infty(I))$-estimate of $\rho^{1-\delta}u$. 
\begin{lem}\label{cru4}
There exists a constant $C(T)>0$ such that
\begin{equation*}
\int_0^t |\rho^{1-\delta}u|_{\infty}^2\, \mathrm{d}s\leq C(T) \qquad\text{for all $t\in [0,T]$}.
\end{equation*}    
\end{lem}

Next, we show the refined $\rho$-weighted estimates for $u$.
\begin{lem}\label{ele-0}
Let 
\begin{equation}\label{alphamp}
\alpha=\max\big\{0,\frac{5\delta-3}{2}\big\}. 
\end{equation}
There  exists  a constant $C(T)>0$ such that 
\begin{equation*}
|\rho^\frac{1-\alpha}{5} u(t)|_5^5 +\int_0^t |\rho^\frac{\delta-\alpha}{2}|u|^\frac{3}{2}\mathrm{D}_r u|_2^2\,\mathrm{d}s\leq C(T) \qquad\text{for any $t\in [0,T]$}.
\end{equation*}
\end{lem}
\begin{proof}
Let $\alpha$ be given as in \eqref{alphamp}. We divide the proof into three steps.

\smallskip
\textbf{1. $L^2$-estimate for $\rho^\frac{1-\alpha}{2}u$.} We first see from $\eqref{e1.5}_1$ that
\begin{equation}\label{mass-alpha}
(\rho^{1-\alpha})_t+(u\rho^{1-\alpha})_r-\alpha\rho^{1-\alpha}u_r+m(1-\alpha)\frac{\rho^{1-\alpha}u}{r}=0.
\end{equation}

Next, multiplying $\eqref{e1.5}_2$ by $\rho^{-\alpha}u$, together with \eqref{V-expression} and \eqref{mass-alpha}, gives
\begin{equation}\label{6-3}
\!\!\!\!\begin{aligned}
&\,\big(\frac{1}{2}\rho^{1-\alpha}|u|^2\big)_t+2a_1\delta\rho^{\delta-\alpha}\Big(u_r^2-m\frac{1-\delta}{\delta-\alpha}u_r\frac{u}{r}+\frac{m}{2}\frac{1-\alpha}{\delta-\alpha}\frac{u^2}{r^2}\Big)\\
&=\Big(\underline{a_1\delta\rho^{\delta-\alpha}\big(2uu_r+m\frac{2\delta-\alpha-1}{\delta-\alpha}\frac{u^{2}}{r}\big)-\frac{\alpha+3}{6}\rho^{1-\alpha}u|u|^2}_{:=\mathcal{B}_5}\Big)_r\\
&\quad\!-\frac{A\gamma}{2a_1\delta}\rho^{\gamma-\delta+1-\alpha}(v-u)u+\rho^{1-\alpha}vu\Big(\alpha u_r-\frac{m(1-\alpha)}{2}\frac{u}{r}\Big)+\frac{\alpha(1-\alpha)}{12a_1\delta}\rho^{2-\delta-\alpha}(v-u)u^{3}\!.
\end{aligned}
\end{equation}
Here, for the detailed derivation of \eqref{6-3}, see Appendix \ref{appd}. We need to show that $\mathcal{B}_5\in W^{1,1}(I)$ and $\mathcal{B}_5|_{r=0}=0$ for \textit{a.e.} $t\in (0,T)$, so that we can apply Lemma \ref{calculus} to obtain
\begin{equation}\label{int-B5}
\int_0^\infty (\mathcal{B}_5)_r\,\mathrm{d}r=-\mathcal{B}_5|_{r=0}=0.
\end{equation}
Indeed, on one hand, $\mathcal{B}_5|_{r=0}=0$ follows from $u|_{r=0}=0$ and $(\rho,u,\mathrm{D}_r u)\in C(\bar I)$ for $t\in (0,T]$ due to \eqref{spd2}. On the other hand, we see from \eqref{spd}--\eqref{spd2} that, for \textit{a.e.} $t\in (0,T)$,
\begin{equation*}
(\rho,(\rho^{\delta-1})_r,u,\mathrm{D}_r u)\in L^\infty(I),\qquad r^\frac{m}{2}\big(u,\frac{\mathrm{D}_r^2 u}{r}\big)\in L^2(I),
\end{equation*}
which implies that, for \textit{a.e.} $t\in (0,T)$,
\begin{equation}\label{uur}
\begin{aligned}
|(u,\mathrm{D}_r u)|_2&\leq  |\chi_1^\flat (u,\mathrm{D}_r u)|_2+C_0|\chi_1^\sharp(u,u_r)|_2\\
&\leq |(u,\mathrm{D}_r u)|_\infty+C_0|\chi_1^\sharp r^{-\frac{m}{2}}|_\infty |r^\frac{m}{2}(u,u_r)|_2<\infty.
\end{aligned}
\end{equation}
Then it follows from \eqref{uur}, $\alpha<\delta<1$, and the H\"older inequality that
\begin{align*}
&\quad \, |\mathcal{B}_5|_1\leq C_0|\rho|_\infty^{\delta-\alpha}|u|_2 |\mathrm{D}_r u|_2+C_0|\rho|_\infty^{1-\alpha}|u|_\infty|u|_2^2<\infty,\\
&\begin{aligned}
|(\mathcal{B}_5)_r|_1&\leq C_0 |\rho^{1-\alpha}(\rho^{\delta-1})_ru\mathrm{D}_r u|_1+C_0 |\rho^{\delta-\alpha}u_r\mathrm{D}_r u|_1+C_0 |\rho^{\delta-\alpha}u\mathrm{D}_r^2 u|_1\\
&\quad+ C_0\big|\rho^{2-\delta-\alpha}(\rho^{\delta-1})_r u^3\big|_1+C_0|\rho^{1-\alpha}u^2u_r|_1
\end{aligned}\\
&\qquad\quad \ \begin{aligned}
&\leq C_0\big(|\rho|_\infty^{1-\alpha}|(\rho^{\delta-1})_r|_\infty |u|_2+ |\rho|_\infty^{\delta-\alpha}|u_r|_2\big) |\mathrm{D}_r u|_2+C_0|\rho|_\infty^{\delta-\alpha}|r^\frac{2-m}{2}u|_2 |r^\frac{m-2}{2}\mathrm{D}_r^2 u|_2\\
&\quad+ C_0|\rho|_\infty^{2-\delta-\alpha}|(\rho^{\delta-1})_r|_\infty|u|_\infty|u|_2^2+C_0|\rho|_\infty^{1-\alpha}|u|_2^2|u_r|_\infty<\infty.
\end{aligned}
\end{align*}

Based on \eqref{int-B5}, integrating \eqref{6-3} over $I$, we obtain
\begin{equation}\label{6.3'}
\begin{aligned}
&\,\frac{1}{2}\frac{\mathrm{d}}{\mathrm{d}t}|\rho^\frac{1-\alpha}{2}u|_2^2+ \underline{2a_1\delta\int_0^\infty\rho^{\delta-\alpha}\Big(u_r^2-m\frac{1-\delta}{\delta-\alpha}u_r\frac{u}{r}+\frac{m}{2}\frac{1-\alpha}{\delta-\alpha}\frac{u^2}{r^2}\Big)\,\mathrm{d}r}_{:=I_{9}}\\
&=-\frac{A\gamma}{2a_1\delta}\int_0^\infty\rho^{\gamma-\delta+1-\alpha}(v-u)u\,\mathrm{d}r+\int_0^\infty\rho^{1-\alpha}vu\Big(\alpha u_r+\frac{m(\alpha-1)}{2}\frac{u}{r}\Big)\,\mathrm{d}r\\
&\quad+\frac{\alpha(1-\alpha)}{12a_1\delta}\int_0^\infty\rho^{2-\delta-\alpha}(v-u)u^3\,\mathrm{d}r:=\sum_{i=10}^{12}I_i.
\end{aligned}
\end{equation}

To derive the lower bound for $I_{9}$, let $X=\sqrt{2a_1\delta} \rho^\frac{\delta-\alpha}{2}u_r$ and  $Y=\sqrt{2a_1\delta} \rho^\frac{\delta-\alpha}{2} \frac{u}{r}$. Then write $I_{9}$ in view of $(X,Y)$ as
\begin{equation*}
I_{9}=\int_0^\infty \Big(X^2-m\frac{1-\delta}{\delta-\alpha}XY+\frac{m}{2}\frac{1-\alpha}{\delta-\alpha}Y^2\Big)\,\mathrm{d}r.    
\end{equation*}
Clearly, the integrand in the above is a quadratic form with respect to $(X,Y)$ and its discriminant $\mathscr{D}$ is strictly negative:
\begin{equation*}
\mathscr{D}=m^2\frac{(1-\delta)^2}{(\delta-\alpha)^2}-2m\frac{1-\alpha}{\delta-\alpha}<0.
\end{equation*}
Hence, there exists a constant $c_{*1}>0$, which depends only on $(n,\delta)$, such that
\begin{equation}
I_{9}\geq c_{*1}\int_0^\infty (X^2+Y^2)\,\mathrm{d}r=2a_1\delta c_{*1} |\rho^\frac{\delta-\alpha}{2}\mathrm{D}_r u|_2^2.  
\end{equation}

We continue to estimate $I_{10}$--$I_{12}$. It follows from Lemmas \ref{far-p-infty}, \ref{important2}, and \ref{l4.4}, and the H\"older and Young inequalities that
\begin{equation}\label{6.8'}
\begin{aligned}
I_{10}&\leq C_0|\rho|_\infty^{\gamma-\delta}|\rho^\frac{1-\alpha}{2}u|_2^2+C_0\big|\rho^{\gamma-\delta+\frac{1-\alpha}{2}} v\big|_2|\rho^\frac{1-\alpha}{2}u|_{2}\leq C(T)\big(|\rho^\frac{1-\alpha}{2}u|_2^2+\big|\rho^{\gamma-\delta+\frac{1-\alpha}{2}} v\big|_2^2\big),\\
I_{11}&\leq C_0|\rho|_\infty^\frac{1-\delta}{2}|v|_\infty|\rho^\frac{1-\alpha}{2}u|_2  |\rho^\frac{\delta-\alpha}{2}\mathrm{D}_r u|_2\leq C(T)|\rho^\frac{1-\alpha}{2}u|_2^2+\frac{a_1\delta c_{*1}}{8} |\rho^\frac{\delta-\alpha}{2}\mathrm{D}_r u|_2^2,\\
I_{12}&\leq C_0|\rho^{2-\delta-\alpha}vu^3|_1\leq C_0|v|_\infty|\rho^{1-\delta}u|_\infty |\rho^\frac{1-\alpha}{2}u|_2^2\leq C(T)\big(1+|\rho^{1-\delta}u|_\infty^2\big) |\rho^\frac{1-\alpha}{2}u|_2^2,
\end{aligned}
\end{equation}
where, in $I_{12}$, we have also used the facts that $0\leq \alpha<1$ and 
\begin{equation*}
-\frac{\alpha(1-\alpha)}{12a_1\delta}\int_0^\infty \rho^{2-\delta-\alpha}u^4\,\mathrm{d}r\leq 0.
\end{equation*}

Therefore, collecting \eqref{6.3'}--\eqref{6.8'} yields
\begin{equation}\label{66}
\frac{\mathrm{d}}{\mathrm{d}t}|\rho^\frac{1-\alpha}{2}u|_2^2+a_1\delta c_{*1} |\rho^\frac{\delta-\alpha}{2}\mathrm{D}_r u|_2^2 \leq C(T)(1+|\rho^{1-\delta}u|_\infty^2) |\rho^\frac{1-\alpha}{2}u|_2^2+C(T)\big|\rho^{\gamma-\delta+\frac{1-\alpha}{2}}v\big|_{2}^{2}.
\end{equation}

\smallskip
\textbf{2. Estimate for $v$.} We multiply \eqref{eq:effective2} by $\rho^{2(\gamma-\delta)+(1-\alpha)}v$ and obtain from $\eqref{e1.5}_1$ that
\begin{equation}\label{duo-v}
\begin{aligned}
&\,\frac{1}{2}\big(\rho^{2(\gamma-\delta)+(1-\alpha)}v^{2}\big)_t+ \frac{1}{2} \big(\underline{u\rho^{2(\gamma-\delta)+(1-\alpha)}v^{2}}_{:=\mathcal{B}_6}\big)_r+\frac{A\gamma}{2a_1\delta}\rho^{3(\gamma-\delta)+(1-\alpha)} v^2\\
&=\big(\frac{\alpha}{2}-(\gamma-\delta)\big)\rho^{2(\gamma-\delta)+(1-\alpha)} v^2 u_r-\big((\gamma-\delta)+\frac{1-\alpha}{2}\big)m\rho^{2(\gamma-\delta)+(1-\alpha)} v^2\frac{u}{r}\\
&\quad +\frac{A\gamma}{2a_1\delta}\rho^{3(\gamma-\delta)+(1-\alpha)}  vu.
\end{aligned}
\end{equation}
We need to show that $\mathcal{B}_6\in W^{1,1}(I)$ and $\mathcal{B}_6|_{r=0}=0$ 
for {\it a.e.} $t\in (0,T)$, which allows us to apply Lemma \ref{calculus} to obtain
\begin{equation}\label{int-B6}
\int_0^\infty (\mathcal{B}_6)_r\,\mathrm{d}r=-\mathcal{B}_6|_{r=0}=0.   
\end{equation}
Indeed, on one hand,  to obtain $\mathcal{B}_6|_{r=0}=0$, we first note that $v=u+\frac{2a_1\delta}{\delta-1}(\rho^{\delta-1})_r$ and
\begin{equation*}
(u,(\rho^{\delta-1})_r)\in L^\infty(I)
\qquad\text{for {\it a.e.} $t\in(0,T)$}
\end{equation*}
due to \eqref{spd}--\eqref{spd2}, which implies  
\begin{equation}\label{psi,wuqiong0}
v\in L^\infty(I)  \qquad\text{for {\it a.e.} $t\in (0,T)$}.
\end{equation}
Then it follows from $u|_{r=0}$ and $(\rho,u) \in C(\bar I)$ for each $t\in (0,T]$ (due to \eqref{spd2}) that $\mathcal{B}_6|_{r=0}=0$. 
On the other hand, it follows from \eqref{spd}--\eqref{spd2} and \eqref{uur} that 
\begin{align*}
&(u,\mathrm{D}_r u)\in L^2(I),\quad  (\rho,u,u_r,(\rho^{\delta-1})_r)\in L^\infty(I),\\
&r^\frac{m}{2} (u,u_r,\mathrm{D}_r(\rho^{\gamma-1})_r)\in L^2(I),\quad r^\frac{m}{n} (\rho^{\delta-1})_{rr} \in L^n(I)
\end{align*}
for {\it a.e.} $t\in (0,T)$. Thus, we obtain from Lemma \ref{hardy} that
\begin{equation*}
\begin{aligned}
|(\rho^{\gamma-1})_r|_2&\leq  \big|\chi_1^\flat(\rho^{\gamma-1})_r\big|_2+\big|\chi_1^\sharp(\rho^{\gamma-1})_r\big|_2\\
&\leq |\chi_1^\flat r^\frac{2-m}{2}|_\infty\big|r^\frac{m-2}{2}(\rho^{\gamma-1})_r\big|_2 + |\chi_1^\sharp r^{-\frac{m}{2}}|_\infty \big|r^\frac{m}{2}(\rho^{\gamma-1})_r\big|_2<\infty,
\end{aligned}
\end{equation*}
which, along with \eqref{V-expression}, \eqref{psi,wuqiong0}, and the H\"older and Young inequalities, yields
\begin{align*}
&\quad\, |\mathcal{B}_6|_1 \leq C_0|u|_\infty\big(|\rho|_\infty^{2(\gamma-\delta)+(1-\alpha)}|u|_2^2+|\rho|_\infty^{1-\alpha}|(\rho^{\gamma-1})_r|_2^2\big)<\infty,\\[1mm]
&\begin{aligned}
|(\mathcal{B}_6)_r|_1&\leq C_0\big|\big(u_r\rho^{2(\gamma-\delta)+(1-\alpha)} v^2, u\rho^{\gamma-2\delta+2-\alpha}(\rho^{\gamma-1})_r v^2,u\rho^{2(\gamma-\delta)+(1-\alpha)} vv_r\big)\big|_1\\
&\leq C_0|u_r|_\infty\big(|\rho|_\infty^{2(\gamma-\delta)+(1-\alpha)}|u|_2^2+|\rho|_\infty^{1-\alpha}|(\rho^{\gamma-1})_r|_2^2\big)\\
&\quad +C_0|\rho|_\infty^{\gamma-2\delta+2-\alpha}|(\rho^{\gamma-1})_r|_2|u|_2|v|_\infty^2\\
&\quad+ C_0|\rho|_\infty^{2(\gamma-\delta)+(1-\alpha)}|v|_\infty\big( |\mathrm{D}_r u|_2|\chi_1^\flat r(u_r,(\rho^{\delta-1})_{rr})|_2+|u|_2|\chi_1^\sharp (u_r,(\rho^{\delta-1})_{rr})|_2\big)\\
&\leq C_0|u_r|_\infty\big(|\rho|_\infty^{2(\gamma-\delta)+(1-\alpha)}|u|_2^2+|\rho|_\infty^{1-\alpha}|(\rho^{\gamma-1})_r|_2^2\big)\\
&\quad +C_0|\rho|_\infty^{\gamma-2\delta+2-\alpha}|(\rho^{\gamma-1})_r|_2|u|_2|v|_\infty^2\\
&\quad+ C_0|\rho|_\infty^{2(\gamma-\delta)+(1-\alpha)} |\mathrm{D}_r u|_2|v|_\infty\big(|\chi_1^\flat r^\frac{2-m}{2}|_\infty |r^\frac{m}{2} u_r|_2+|\chi_1^\flat r^\frac{1}{n}|_{n^*}|r^\frac{m}{n} (\rho^{\delta-1})_{rr}|_n\big)\\
&\quad +C_0|\rho|_\infty^{2(\gamma-\delta)+(1-\alpha)}|u|_2|v|_\infty\big(|\chi_1^\sharp r^{-\frac{m}{2}}|_\infty |r^\frac{m}{2} u_r|_2+|\chi_1^\sharp r^{-\frac{m}{n}}|_{n^*} |r^\frac{m}{n} (\rho^{\delta-1})_{rr}|_n\big)<\infty,    
\end{aligned}
\end{align*}
where $n^*$ is the parameter defined in \S \ref{othernote}.

Now, integrating \eqref{duo-v} over $I$, we obtain from \eqref{int-B6}, Lemmas \ref{important2} and \ref{l4.4}, and the H\"older and Young inequalities that
\begin{equation}\label{new-v-0}
\begin{aligned}
&\,\frac{1}{2}\frac{\mathrm{d}}{\mathrm{d}t}\big|\rho^{\gamma-\delta+\frac{1}{2}(1-\alpha)}v\big|_2^2+ \frac{A\gamma}{2a_1\delta} \big|\rho^{\frac{3}{2}(\gamma-\delta)+\frac{1}{2}(1-\alpha)}v\big|_2^2\\
&\leq C_0 \int_0^\infty \rho^{2(\gamma-\delta)+(1-\alpha)} v^2 |\mathrm{D}_r u| \,\mathrm{d}r +C_0\int_0^\infty\rho^{3(\gamma-\delta)+(1-\alpha)} |v||u|\,\mathrm{d}r\\
&\leq  C_0\big|\rho^{\gamma-\delta+\frac{1}{2}(1-\alpha)}v\big|_2\Big(|\rho|_\infty^\frac{2\gamma+1-3\delta}{2}|v|_\infty |\rho^\frac{\delta-\alpha}{2}\mathrm{D}_r u|_2+ |\rho|_\infty^{2\gamma-2\delta} |\rho^\frac{1-\alpha}{2}u|_2\Big)\\
&\leq  C(T)\big(\big|\rho^{\gamma-\delta+\frac{1}{2}(1-\alpha)}v\big|_2^2+|\rho^\frac{1-\alpha}{2}u|_2^2\big)+\frac{a_1\delta c_{*1}}{8} |\rho^\frac{\delta-\alpha}{2}\mathrm{D}_r u|_2^2. 
\end{aligned}        
\end{equation}

Combining \eqref{new-v-0} with \eqref{66} gives
\begin{equation*} 
\begin{aligned}
&\,\frac{\mathrm{d}}{\mathrm{d}t}\big(|\rho^\frac{1-\alpha}{2}u|_2^2+\big|\rho^{\gamma-\delta+\frac{1}{2}(1-\alpha)}v\big|_2^2\big)+ \frac{a_1\delta c_{*1}}{2} |\rho^\frac{\delta-\alpha}{2}\mathrm{D}_r u|_2^2\\
&\leq C(T)\big(1+|\rho^{1-\delta}u|_\infty^2\big) |\rho^\frac{1-\alpha}{2}u|_2^2+C(T)\big|\rho^{\gamma-\delta+\frac{1-\alpha}{2}}v\big|_{2}^{2},
\end{aligned}
\end{equation*}
which, along with the Gr\"onwall inequality and Lemma \ref{cru4}, yields that, for all $t\in [0,T]$,
\begin{equation}\label{47-0}
|\rho^\frac{1-\alpha}{2}u(t)|_2^2+\big|\rho^{\gamma-\delta+\frac{1}{2}(1-\alpha)}v(t)\big|_2^2+ \int_0^t |\rho^\frac{\delta-\alpha}{2}\mathrm{D}_r u|_2^2\,\mathrm{d}s\leq  C(T),
\end{equation}
where the $L^2(I)$-bound of $(\rho_0^\frac{1-\alpha}{2}u_0,\rho_0^{\gamma-\delta+\frac{1}{2}(1-\alpha)}v_0)$ follows from Lemmas \ref{ale1}, \ref{initial3}, and \ref{lemma-initial}:
\begin{equation*}
\begin{aligned}
|\rho_0^\frac{1-\alpha}{2}u_0|_2&\leq |\rho_0|_\infty^\frac{1-\alpha}{2}\big(|\chi_1^\flat u_0|_2+|\chi_1^\sharp u_0|_2\big)\leq C_0\big(|u_0|_\infty+|\chi_1^\sharp r^{-\frac{m}{2}}|_\infty |r^\frac{m}{2}u_0|_2\big) \\
&\leq C_0\big(\|\boldsymbol{u}_0\|_{L^\infty}+\|\boldsymbol{u}_0\|_{L^2}\big)\leq C_0\|\boldsymbol{u}_0\|_{H^2}\leq C_0,\\
\big|\rho_0^{\gamma-\delta+\frac{1}{2}(1-\alpha)} v_0\big|_2&\leq |\rho_0|_\infty^{\gamma-\delta}|\rho_0^{\frac{1-\alpha}{2}}u_0|_2+C_0 |\rho_0|_\infty^\frac{1-\alpha}{2}|(\rho_0^{\gamma-1})_r|_2\\
&\leq C_0\big(1+|\chi_1^\flat(\rho_0^{\gamma-1})_r|_\infty+|\chi_1^\sharp r^{-\frac{m}{2}}|_\infty|r^\frac{m}{2}(\rho_0^{\gamma-1})_r|_2\big)\\
&\leq C_0 \big(1+\|\nabla\rho_0^{\gamma-1}\|_{L^\infty}+\|\nabla\rho_0^{\gamma-1}\|_{L^2}\big)\leq C_0\big(1+\|\nabla\rho_0^{\gamma-1}\|_{H^2}\big)\leq C_0.
\end{aligned}
\end{equation*}

\smallskip
\textbf{3. $L^5$-estimates for $\rho^\frac{1-\alpha}{5}u$.} Multiplying $\eqref{e1.5}_2$ by $\rho^{-\alpha}|u|^{\ell-2}u$ with $\ell=3,4,5$, together with \eqref{V-expression} and \eqref{mass-alpha}, gives
\begin{align}
&\,\big(\frac{1}{\ell}\rho^{1-\alpha}|u|^\ell\big)_t+2(\ell-1)a_1\delta\rho^{\delta-\alpha}|u|^{\ell-2}\Big(u_r^2-m\frac{1-\delta}{\delta-\alpha}u_r\frac{u}{r}+\frac{m}{\ell}\frac{1-\alpha}{\delta-\alpha}\frac{u^2}{r^2}\Big)\notag\\
&=\Big(\underline{a_1\delta\rho^{\delta-\alpha}|u|^{\ell-2}\big(2uu_r+2m\frac{\ell\delta-\ell\alpha- \ell+1}{\ell(\delta-\alpha)}\frac{u^{2}}{r}\big) -\frac{(\ell-1)\alpha+\ell+1}{\ell(\ell+1)}\rho^{1-\alpha}u|u|^\ell}_{:=\mathcal{B}_7}\Big)_r\label{uell}\\
&\quad -\frac{A\gamma}{2a_1\delta}\rho^{\gamma-\delta+1-\alpha}(v-u)|u|^{\ell-2}u+\rho^{1-\alpha}vu|u|^{\ell-2}\Big(\alpha u_r-\frac{m(1-\alpha)}{\ell}\frac{u}{r}\Big)\notag\\
&\quad +\frac{(\ell-1)\alpha(1-\alpha)}{2\ell(\ell+1)a_1\delta}\rho^{2-\delta-\alpha}(v-u)u|u|^{\ell}.\notag
\end{align}
The detailed derivation of \eqref{uell} is given in Appendix \ref{appd}. Then we need to show that $\mathcal{B}_7\in W^{1,1}(I)$ and $\mathcal{B}_7|_{r=0}=0$ for \textit{a.e.} $t\in(0,T)$, so that we can apply Lemma \ref{calculus} to obtain
\begin{equation}\label{int-B7}
\int_0^\infty (\mathcal{B}_7)_r\,\mathrm{d}r=-\mathcal{B}_7|_{r=0}=0.
\end{equation}
Note that, this proof can be derived similarly from that of \eqref{int-B5} in Step 1, since $u\in C(\bar I)$ for $t\in (0,T]$ due to \eqref{spd2}, and  $|\mathcal{B}_7|+|(\mathcal{B}_7)_r|\leq C(\ell)|u|^{\ell-2}(|\mathcal{B}_5|+|(\mathcal{B}_5)_r|)$.

Hence, based on \eqref{int-B7}, integrating \eqref{uell} over $I$, we have
\begin{equation}\label{6.3}
\begin{aligned}
&\,\frac{1}{\ell}\frac{\mathrm{d}}{\mathrm{d}t}|\rho^\frac{1-\alpha}{\ell}u|_\ell^\ell+ \underline{2(\ell-1)a_1\delta\int_0^\infty\rho^{\delta-\alpha}|u|^{\ell-2}\Big(u_r^2-m\frac{1-\delta}{\delta-\alpha}u_r\frac{u}{r}+\frac{m}{\ell}\frac{1-\alpha}{\delta-\alpha}\frac{u^2}{r^2}\Big)\,\mathrm{d}r}_{:=I_{13}}\\
&=-\frac{A\gamma}{2a_1\delta}\int_0^\infty\rho^{\gamma-\delta+1-\alpha}(v-u)|u|^{\ell-2}u\,\mathrm{d}r+\int_0^\infty\rho^{1-\alpha}vu|u|^{\ell-2}\Big(\alpha u_r+\frac{m(\alpha-1)}{\ell}\frac{u}{r}\Big)\,\mathrm{d}r\\
&\quad+\frac{(\ell-1)\alpha(1-\alpha)}{2\ell(\ell+1)a_1\delta}\int_0^\infty\rho^{2-\delta-\alpha}(v-u)u|u|^{\ell}\,\mathrm{d}r:=\sum_{i=14}^{16}I_i.
\end{aligned}
\end{equation}

To estimate $I_{13}$, we claim that there exists a constant $c_{*\ell}>0$, depending only on $(n,\delta,\ell)$, such that
\begin{equation}\label{panbie}
\begin{aligned}
I_{13}\geq c_{*\ell}\int_0^\infty (X^2+Y^2)\,\mathrm{d}r=2a_1\delta(\ell-1)c_{*\ell}\Big(\big|\rho^\frac{\delta-\alpha}{2}|u|^\frac{\ell-2}{2}u_r\big|_2^2+\Big|\rho^\frac{\delta-\alpha}{2}|u|^\frac{\ell-2}{2}\frac{u}{r}\Big|_2^2\Big),
\end{aligned} 
\end{equation}
where $(X,Y)$ are given by
\begin{equation*}
X=\sqrt{2(\ell-1)a_1\delta} \rho^\frac{\delta-\alpha}{2}|u|^\frac{\ell-2}{2}u_r,\qquad Y=\sqrt{2(\ell-1)a_1\delta} \rho^\frac{\delta-\alpha}{2}|u|^\frac{\ell-2}{2}\frac{u}{r}.
\end{equation*}
To this end, write $I_{13}$ in view of $(X,Y)$ as
\begin{equation*}
I_{13}=\int_0^\infty \Big(X^2-m\frac{1-\delta}{\delta-\alpha}XY+\frac{m}{\ell}\frac{1-\alpha}{\delta-\alpha}Y^2\Big)\,\mathrm{d}r.    
\end{equation*}
Since the integrand in the above is a quadratic form with respect to $(X,Y)$ and its discriminant $\mathscr{D}$ is strictly negative:
\begin{equation*}
\mathscr{D}=m^2\frac{(1-\delta)^2}{(\delta-\alpha)^2}-\frac{4m}{\ell}\frac{1-\alpha}{\delta-\alpha}<0,
\end{equation*}
we obtain claim \eqref{panbie}.

To estimate $I_{14}$--$I_{16}$, it follows from the fact that $0\leq \alpha<1$, Lemmas \ref{far-p-infty}, \ref{important2}, and \ref{l4.4}, and the H\"older and Young inequalities that
\begin{align}
&\begin{aligned}[b]
I_{14}&\leq C_0|\rho|_\infty^{\gamma-\delta}|\rho^\frac{1-\alpha}{\ell}u|_\ell^\ell+C_0|\rho|_\infty^{\gamma-\delta}|v|_\infty\big|\rho^\frac{1-\alpha}{\ell-1}u\big|_{\ell-1}^{\ell-1}\leq C(T)\big(|\rho^\frac{1-\alpha}{\ell}u|_\ell^\ell+|\rho^\frac{1-\alpha}{\ell-1}u|_{\ell-1}^{\ell-1}\big),\\
I_{15}&\leq C(\ell)|\rho|_\infty^\frac{1-\delta}{2}|v|_\infty|\rho^\frac{1-\alpha}{\ell}u|_\ell^\frac{\ell}{2}\big|\rho^\frac{\delta-\alpha}{2}|u|^\frac{\ell-2}{2}\mathrm{D}_r u\big|_2\\
&\leq C(\ell,T)|\rho^\frac{1-\alpha}{\ell}u|_\ell^\ell+\frac{a_1 \delta c_{*\ell}}{8}\big|\rho^\frac{\delta-\alpha}{2}|u|^\frac{\ell-2}{2}\mathrm{D}_r u\big|_2^2,
\end{aligned}\label{6.8}\\
&\begin{aligned}
I_{16}&\leq C(\ell)|\rho^{2-\delta-\alpha}v|u|^{\ell+1}|_1\leq C(\ell)|v|_\infty|\rho^{1-\delta}u|_\infty |\rho^\frac{1-\alpha}{\ell}u|_\ell^\ell \\
&\leq C(\ell,T)\big(1+|\rho^{1-\delta}u|_\infty^2\big) |\rho^\frac{1-\alpha}{\ell}u|_\ell^\ell,\notag
\end{aligned}
\end{align}
where we have used the following fact in $I_{16}$:
\begin{equation*}
-\frac{(\ell-1)\alpha(1-\alpha)}{2\ell(\ell+1)a_1\delta}\int_0^\infty\rho^{2-\delta-\alpha}|u|^{\ell+2}\,\mathrm{d}r\leq 0.
\end{equation*}

Therefore, collecting \eqref{6.3}--\eqref{6.8} yields
\begin{equation*}
\begin{aligned}
&\,\frac{\mathrm{d}}{\mathrm{d}t}|\rho^\frac{1-\alpha}{\ell}u|_\ell^\ell+a_1\delta c_{*\ell}\big|\rho^\frac{\delta-\alpha}{2}|u|^\frac{\ell-2}{2}\mathrm{D}_r u\big|_2^2\\
&\leq C(\ell,T)\big(1+|\rho^{1-\delta}u|_\infty^2\big) |\rho^\frac{1-\alpha}{\ell}u|_\ell^\ell+C(T)|\rho^\frac{1-\alpha}{\ell-1}u|_{\ell-1}^{\ell-1}.
\end{aligned}
\end{equation*}
Applying the Gr\"onwall inequality, we can iteratively obtain from \eqref{47-0} and Lemma \ref{cru4} that, for $\ell=3,4,5$ and any $t\in [0,T]$,
\begin{equation*}
|\rho^\frac{1-\alpha}{\ell}u(t)|_\ell^\ell+ \int_0^t\big|\rho^\frac{\delta-\alpha}{2}|u|^\frac{\ell-2}{2}\mathrm{D}_r u\big|_2^2\,\mathrm{d}s\leq  C(\ell,T).
\end{equation*}
Here, for the $L^\ell(I)$-boundedness of $\rho_0^\frac{1-\alpha}{\ell}u_0$, we see from Lemmas \ref{ale1}, \ref{initial3}, and \ref{lemma-initial} that
\begin{equation*}
\begin{aligned}
|\rho_0^\frac{1-\alpha}{\ell}u_0|_\ell&\leq |\rho_0|_\infty^\frac{1-\alpha}{\ell}\big(|\chi_1^\flat u_0|_\ell+|\chi_1^\sharp u_0|_\ell\big)\leq C(\ell)\big(|u_0|_\infty+|\chi_1^\sharp r^{-\frac{m}{\ell}}|_\infty |r^\frac{m}{\ell}u_0|_\ell\big) \\
&\leq C(\ell)\big(\|\boldsymbol{u}_0\|_{L^\infty}+\|\boldsymbol{u}_0\|_{L^\ell}\big)\leq C(\ell)\|\boldsymbol{u}_0\|_{H^2}\leq C(\ell).
\end{aligned}
\end{equation*}

This completes the proof.
\end{proof}

Based on Lemma \ref{ele-0}, we have the following $L^2([0,T];L^\infty(I))$-estimate for $\rho^\frac{1-\delta}{2}u$.
\begin{lem}\label{cru5}
There exists a constant $C(T)>0$ such that
\begin{equation*}
\int_0^t |\rho^\frac{1-\delta}{2}u|_{\infty}^2\, \mathrm{d}s\leq C(T) \qquad\text{for all $t\in [0,T]$}.
\end{equation*}    
\end{lem}

\begin{proof}
First, let $\alpha$ be given as in \eqref{alphamp}. Then we have
\begin{equation*}
2\alpha-5\delta+3\geq 0.
\end{equation*}

Next, we obtain from the above, \eqref{V-expression}, Lemmas \ref{important2}, \ref{l4.4}, \ref{ele-0}, and \ref{calculus}, and the H\"older inequality that
\begin{equation*}
\begin{aligned}
|\rho^\frac{1-\delta}{2}u|_\infty^{5}&\leq \big|(\rho^\frac{5(1-\delta)}{2} |u|^{5})_r\big|_1\leq C_0\big(\big|\rho^\frac{5(1-\delta)}{2}|u|^{3}uu_r\big|_1+\big|\rho^{\frac{7}{2}(1-\delta)} (v-u)u^{5}\big|_1\big)\\
&\leq C_0|\rho|_\infty^{\frac{2\alpha-5\delta+3}{2}}\big(|\rho|_\infty^\frac{1-\delta}{2}|\rho^\frac{1-\alpha}{5}u|_{5}^\frac{5}{2}\big|\rho^\frac{\delta-\alpha}{2}|u|^\frac{3}{2}u_r\big|_2+|v|_\infty |\rho|_\infty^{1-\delta}|\rho^\frac{1-\alpha}{5}u|_{5}^{5}\big)\\
&\quad +C_0|\rho|_\infty^\frac{2\alpha-5\delta+3}{2} |\rho^{1-\delta} u|_\infty|\rho^{\frac{1-\alpha}{5}}u|_{5}^{5}\leq C(T)\big(\big|\rho^\frac{\delta-\alpha}{2}|u|^\frac{3}{2}u_r\big|_2+|\rho^{1-\delta} u|_\infty+1\big).
\end{aligned}
\end{equation*}
Consequently, integrating the above over $[0,t]$, together with Lemmas \ref{cru4}--\ref{ele-0}, leads to the desired estimate.
\end{proof}

Finally, we can obtain the $L^\infty([0,T];L^2(I))$-estimate and $L^4([0,T];L^\infty(I))$-estimate for $u$.
\begin{lem}\label{ele}
There  exists  a constant $C(T)>0$ such that 
\begin{equation*}
|(\rho^{\gamma-\delta} v,u)(t)|_2^2 +\int_0^t ( |\rho^\frac{\delta-1}{2}\mathrm{D}_r u|_2^2+|u|_\infty^4 )\,\mathrm{d}s\leq C(T) \qquad\text{for any $t\in [0,T]$}.
\end{equation*}
\end{lem}

\begin{proof}
We divide the proof into four steps.

\smallskip
\textbf{1.}
First, multiplying $\eqref{e1.5}_2$ by $\rho^{-1}u$, along with \eqref{V-expression}, gives
\begin{equation}\label{eq:6.14}
\begin{aligned}
&\,\frac{1}{2}(u^2)_t +a_1\delta\rho^{\delta-1}\big(2 u_r^2+ m\frac{u^2}{r^2}\big)\\
&=\Big(\underline{2a_1\delta\rho^{\delta-1} u_ru+a_1 \delta m\rho^{\delta-1}\frac{u^2}{r}}_{:=\mathcal{B}_8}\ \underline{-\frac{2}{3}u^3}_{:=\mathcal{B}_9}\Big)_r-\frac{A\gamma}{2a_1\delta}\rho^{\gamma-\delta}(v-u)u\\
&\quad + vu_ru-\frac{(1-\delta)m}{2}(v-u)\frac{u^2}{r}.
\end{aligned} 
\end{equation}

\smallskip
\textbf{2.} We show that $\mathcal{B}_8\in L^\frac{n}{n-1}(I)\cap D^{1,1}(I)$, $\mathcal{B}_9\in W^{1,1}(I)$, and $(\mathcal{B}_8+\mathcal{B}_9)|_{r=0}=0$ for {\it a.e.} 
$t\in (0,T)$, which allows us to apply Lemma \ref{calculus} to obtain
\begin{equation}\label{int-B8B9}
\int_0^\infty (\mathcal{B}_8)_r+(\mathcal{B}_9)_r\,\mathrm{d}r=-(\mathcal{B}_8+\mathcal{B}_9)|_{r=0}=0.
\end{equation}
$\mathcal{B}_9\in W^{1,1}(I)$ and $(\mathcal{B}_8+\mathcal{B}_9)|_{r=0}=0$ follow  from $u|_{r=0}=0$ 
and the facts that $\rho>0$, $u\in L^2(I)$, and $(\rho,u,u_r)\in C(\bar I)$ for \textit{a.e.} $t\in (0,T)$ due to \eqref{spd}--\eqref{spd2} and \eqref{uur}. Hence, it remains to prove $\mathcal{B}_8\in L^\frac{n}{n-1}(I)\cap D^{1,1}(I)$. Let 
\begin{equation*}
X_1:=\rho^{\delta-1}\nabla\boldsymbol{u}\cdot \boldsymbol{u},\qquad X_2:=\rho^{\delta-1}(\diver\boldsymbol{u})\boldsymbol{u}
\end{equation*}
On one hand, we see from Theorem \ref{thm-loc}, Lemma \ref{lemma-L6} (since $X_1$ and $X_2$ are spherically symmetric vector functions), and the H\"older inequality that, for $i=1,2$ and \textit{a.e.} $t\in (0,T)$,
\begin{equation*}
\begin{aligned}
\|X_i\|_{L^\frac{n}{n-1}}\leq C_0\|\nabla X_i\|_{L^1}&\leq C_0\big(\|\nabla\rho^{\delta-1}\|_{L^\infty}\|\nabla\boldsymbol{u}\|_{L^2}+\|\rho^{\delta-1}\nabla^2\boldsymbol{u}\|_{L^2}\big)\|\boldsymbol{u}\|_{L^2}\\
&\quad +C_0\|\rho^\frac{\delta-1}{2}\nabla\boldsymbol{u}\|_{L^2}^2<\infty.
\end{aligned}
\end{equation*}
This, along with Lemma \ref{lemma-initial}, gives
\begin{equation}\label{622}
\begin{aligned}
&\ r^\frac{m(n-1)}{n}\rho^{\delta-1}u u_r\in L^\frac{n}{n-1}(I),\quad r^\frac{m(n-1)}{n}\rho^{\delta-1}u \big(u_r+\frac{m}{r}u\big)\in L^\frac{n}{n-1}(I)\\
&\implies r^\frac{m(n-1)}{n}\rho^{\delta-1}u\mathrm{D}_r u\in L^\frac{n}{n-1}(I)\qquad \text{for \textit{a.e.} $t\in (0,T)$}.
\end{aligned}
\end{equation}
On the other hand, it follows from \eqref{spd}--\eqref{spd2} and \eqref{uur} that, for {\it a.e.} $t\in (0,T)$, 
\begin{equation*}
\begin{aligned}
&\rho>0\ \ \text{on $I$},\qquad (\rho^{\delta-1})_r\in L^\infty(I),\\
&(\rho,u,\mathrm{D}_r u)\in L^2(I)\cap L^\infty(I),\quad 
r^\frac{m}{2}\big(u ,\rho^{\delta-1}\frac{\mathrm{D}_r^2 u}{r}\big)\in L^2(I),
\end{aligned}
\end{equation*}
which, along with \eqref{622} and the H\"older inequality, leads to
\begin{align*}
& \begin{aligned}
|\mathcal{B}_8|_\frac{n}{n-1}&\leq C_0 |\chi_1^\flat \rho^{\delta-1}u\mathrm{D}_r u|_\frac{n}{n-1}+C_0 |\chi_1^\sharp \rho^{\delta-1}u\mathrm{D}_r u|_\frac{n}{n-1}\\ 
&\leq C_0|\chi_1^\flat \rho^{\delta-1}|_\infty |u|_\infty |\mathrm{D}_r u|_\infty +C_0|\chi_1^\sharp r^\frac{m(1-n)}{n}|_\infty \big|r^\frac{m(n-1)}{n}\rho^{\delta-1}u\mathrm{D}_r u\big|_{\frac{n}{n-1}} <\infty,
\end{aligned}\\
&\begin{aligned}
|(\mathcal{B}_8)_r|_1&\leq C_0 |(\rho^{\delta-1})_ru\mathrm{D}_r u|_1+ C_0\Big|\rho^{\delta-1}\big((u_r)^2,u_{rr}u,u_r\frac{u}{r},u(\frac{u}{r})_r\big)\Big|_1\\
&\leq C_0\big(|(\rho^{\delta-1})_r|_\infty|u|_2 |\mathrm{D}_r u|_2+ |\rho^\frac{\delta-1}{2}\mathrm{D}_r u|_2^2+|r^\frac{2-m}{2}u|_2 |r^\frac{m-2}{2}\rho^{\delta-1}\mathrm{D}_r^2 u|_2\big)<\infty.
\end{aligned}
\end{align*}

\smallskip
\textbf{3.} Integrating \eqref{eq:6.14} over $I$, together with \eqref{int-B8B9}, yields
\begin{equation}\label{ell}
\begin{aligned}
&\,\frac{1}{2}\frac{\mathrm{d}}{\mathrm{d}t}|u|_2^2+ 2a_1\delta|\rho^\frac{\delta-1}{2}u_r|_2^2+ma_1\delta\Big|\rho^\frac{\delta-1}{2}\frac{u}{r}\Big|_2^2\\
&=-\frac{A\gamma}{2a_1\delta}\int_0^\infty \rho^{\gamma-\delta}(v-u) u\,\mathrm{d}r+ \int_0^\infty v u_ru \,\mathrm{d}r-\frac{(1-\delta)m}{2}\int_0^\infty\! (v-u)\frac{u^2}{r}\,\mathrm{d}r :=\sum_{i=17}^{19} I_i.
\end{aligned}
\end{equation}
It follows from Lemmas \ref{important2}, \ref{l4.4}, and \ref{cru5}, and the H\"older and Young inequalities that 
\begin{equation}\label{j3-j5}
\begin{aligned}
I_{17}&\leq C_0\big(|\rho^{\gamma-\delta} v|_2|u|_2 +|\rho|_\infty^{\gamma-\delta} |u|_2^2\big) \leq  C(T)|(\rho^{\gamma-\delta} v,u)|_2^2,\\
I_{18}&\leq C_0|\rho|_\infty^\frac{1-\delta}{2}|v|_\infty|u|_2 |\rho^\frac{\delta-1}{2}u_r|_2 \leq \frac{a_1\delta}{8}|\rho^\frac{\delta-1}{2}u_r|_2^2+C(T)|u|_2^2,\\
I_{19}&\leq C_0\big(|\rho|_\infty^\frac{1-\delta}{2}|v|_\infty+|\rho^\frac{1-\delta}{2}u|_\infty\big)|u|_2 \Big|\rho^\frac{\delta-1}{2}\frac{u}{r}\Big|_2\leq \frac{a_1\delta}{8}\Big|\rho^\frac{\delta-1}{2}\frac{u}{r}\Big|_2^2+C(T)|u|_2^2, 
\end{aligned}
\end{equation}

Combining \eqref{ell} with \eqref{j3-j5} gives
\begin{equation}\label{pol}
\frac{\mathrm{d}}{\mathrm{d}t}|u|_2^2+ a_1\delta |\rho^\frac{\delta-1}{2}\mathrm{D}_r u|_2^2\leq  C(T)|(\rho^{\gamma-\delta} v,u)|_2^2. 
\end{equation}

\smallskip
\textbf{4.} For the $L^2(I)$-estimate of $\rho^{\gamma-\delta} v$, 
we first multiply \eqref{eq:effective2} by $\rho^{2\gamma-2\delta}v$ and then obtain from $\eqref{e1.5}_1$ that
\begin{equation}\label{eq:B6-pre}
\begin{aligned}
&\,\frac{1}{2}(\rho^{2\gamma-2\delta}v^2)_t+ \frac{1}{2} \big(\underline{u\rho^{2\gamma-2\delta}v^2}_{:=\mathcal{B}_{10}}\big)_r+\frac{A\gamma}{2a_1\delta}\rho^{3\gamma-3\delta} v^2\\
&=\big(\frac{1}{2}+\delta-\gamma\big)\rho^{2\gamma-2\delta} v^2 u_r-(\gamma-\delta)m\rho^{2\gamma-2\delta} v^2\frac{u}{r}+\frac{A\gamma}{2a_1\delta}\rho^{3\gamma-3\delta} vu.
\end{aligned}
\end{equation}
We need to show that $\mathcal{B}_{10}\in W^{1,1}(I)$ and $\mathcal{B}_{10}|_{r=0}=0$ 
for {\it a.e.} $t\in (0,T)$, which allows us to apply Lemma \ref{calculus} to obtain
\begin{equation}\label{eq:B6}
\int_0^\infty (\mathcal{B}_{10})_r\,\mathrm{d}r=-\mathcal{B}_{10}|_{r=0}=0.   
\end{equation}
This proof can be achieved by basically following that of \eqref{int-B6} (with $\alpha=1$) in the proof of Lemma \ref{ele-0}. We omit the details here for brevity.

Thus, integrating \eqref{eq:B6-pre} over $I$, we obtain from \eqref{eq:B6}, Lemmas \ref{important2} and \ref{l4.4}, and the H\"older and Young inequalities that
\begin{equation}\label{new-v}
\begin{aligned}
&\,\frac{1}{2}\frac{\mathrm{d}}{\mathrm{d}t}|\rho^{\gamma-\delta} v|_2^2+ \frac{A\gamma}{2a_1\delta} \big|\rho^\frac{3\gamma-3\delta}{2} v\big|_2^2\\
&=\int_0^\infty \rho^{2\gamma-2\delta} v^2 \Big(\big(\frac{1}{2}+\delta-\gamma\big)u_r -(\gamma-\delta)m \frac{u}{r}\Big)\,\mathrm{d}r+\frac{A\gamma}{2a_1\delta}\int_0^\infty \rho^{3\gamma-3\delta} vu\,\mathrm{d}r\\
&\leq  C_0|\rho|_\infty^\frac{2\gamma+1-3\delta}{2}|v|_\infty |\rho^{\gamma-\delta} v|_2  |\rho^\frac{\delta-1}{2}\mathrm{D}_r u|_2+C_0|\rho|_\infty^{2\gamma-2\delta} |\rho^{\gamma-\delta} v|_2 |u|_2\\
&\leq  C(T)|(\rho^{\gamma-\delta} v,u)|_2^2+\frac{a_1\delta}{8} |\rho^\frac{\delta-1}{2}\mathrm{D}_r u|_2^2. 
\end{aligned}        
\end{equation}

Combining \eqref{pol} with \eqref{new-v} gives
\begin{equation*} 
\frac{\mathrm{d}}{\mathrm{d}t}|(\rho^{\gamma-\delta} v,u)|_2^2+ \frac{a_1\delta}{2} |\rho^\frac{\delta-1}{2}\mathrm{D}_r u|_2^2\leq  C(T)|(\rho^{\gamma-\delta} v,u)|_2^2. 
\end{equation*}
which, along with the Gr\"onwall inequality, yields
\begin{equation}\label{47}
|(\rho^{\gamma-\delta} v,u)(t)|_2^2 + \int_0^t |\rho^\frac{\delta-1}{2}\mathrm{D}_r u|_2^2\,\mathrm{d}s\leq  C(T)\qquad\text{for all $t\in [0,T]$}.
\end{equation}
We still needs to check the $L^2(I)$-boundedness of $(\rho_0^{\gamma-\delta} v_0,u_0)$. 
Indeed, it follows from Lemmas \ref{ale1}, \ref{initial3}, and \ref{lemma-initial} that
\begin{align*}
|u_0|_2&\leq |\chi_1^\flat u_0|_2+|\chi_1^\sharp u_0|_2\leq |u_0|_\infty+|\chi_1^\sharp r^{-\frac{m}{2}}|_\infty |r^\frac{m}{2}u_0|_2 \\
&\leq C_0\big(\|\boldsymbol{u}_0\|_{L^\infty}+\|\boldsymbol{u}_0\|_{L^2}\big)\leq C_0\|\boldsymbol{u}_0\|_{H^2}\leq C_0,\\[1mm]
|\rho_0^{\gamma-\delta} v_0|_2&\leq C_0|\rho_0^{\gamma-\delta}(u_0,(\rho_0^{\delta-1})_r)|_2\leq C_0\big(|\rho_0|_\infty^{\gamma-\delta}|u_0|_2+ |(\rho_0^{\gamma-1})_r|_2\big) \\
&\leq C_0|\rho|_{\infty}^{\gamma-\delta}+C_0|\chi_1^\flat r^\frac{2-m}{2} |_\infty |r^\frac{m-2}{2}(\rho_0^{\gamma-1})_r|_2+C_0|\chi_1^\sharp r^{-\frac{m}{2}}|_\infty |r^\frac{m}{2} (\rho_0^{\gamma-1})_r|_2\\
&\leq C_0\big(\|\rho\|_{L^\infty}^{\gamma-\delta}+\|\nabla (\rho_0^{\gamma-1})\|_{H^1}\big) \leq C_0.
\end{align*}
 
Finally, it follows from \eqref{47}, Lemmas \ref{important2} and \ref{calculus}, and the H\"older inequality that 
\begin{equation*}
\begin{aligned}
\int_0^t |u|_\infty^4\,\mathrm{d}s&\leq \int_0^t |(u^2)_r|_1^2 \,\mathrm{d}s\leq 4\int_0^t |u|_2^2|u_r|_2^2 \,\mathrm{d}s\\
&\leq 4\Big(\sup_{s\in [0,t]}|u|_2^2\Big)\Big(\sup_{s\in [0,t]}|\rho|_\infty^{1-\delta}\Big)\int_0^t |\rho^\frac{\delta-1}{2}u_r|_2^2 \,\mathrm{d}s\leq C(T).
\end{aligned}
\end{equation*}

This completes the proof.
\end{proof}

\subsection{Non-Formation of Cavitation Inside the Fluids}
Now, with the help of Lemmas \ref{l4.4} and \ref{ele}, we can show the pointwise estimates of $\rho$  
in the domain containing the origin.

\begin{lem}\label{lemma-inf-rho}
 Suppose that 
\begin{equation}\label{inf-rho0}
\inf_{z\in[0,r]}\rho_0(z)=\underline{\rho}(r)>0 \qquad\text{for $r>0$},
\end{equation}
with $\underline{\rho}(r)$, defined on $I$, satisfying $\underline{\rho}(r)\to 0$ as $r\to\infty$. 
Then,  for any  $R>0$, there exists a constant $C(T)>0$ such that 
\begin{equation*}
\begin{gathered}
\inf_{(t,r)\in [0,T]\times [0,R]} \rho(t,r)\geq \frac{C(T)^{-1}\underline{\rho}(R)}{\big(R^{\frac{1}{2-2\delta}}+1\big)(\underline{\rho}(R)+1)}.
\end{gathered}   
\end{equation*}
In particular, the cavitation does not form in $[0,T]\times \{\boldsymbol{x}\in\mathbb{R}^n:\,|\boldsymbol{x}|\leq R\}$ 
for any  $R>0$.
\end{lem}

\begin{proof}
First, it follows from Lemma \ref{ale1} that, for all $R>0$ and $t\in [0,T]$,
\begin{equation}\label{wuqiong-2-2}
|\chi_R^\flat\rho^{\delta-1}(t)|_\infty\leq C_0\Big(\sqrt{\frac{1+R}{R}} |\chi_R^\flat\rho^{\delta-1}(t)|_2+|\chi_R^\flat(\rho^{\delta-1})_r(t)|_2\Big).
\end{equation}
Then, by  \eqref{V-expression} and Lemmas \ref{l4.4} and \ref{ele}, we obtain that, for all $t\in [0,T]$,
\begin{equation}\label{323}
|\chi_R^\flat (\rho^{\delta-1})_r(t)|_2\leq C_0|\chi_R^\flat (v,u)(t)|_2 \leq C_0\big(\sqrt{R}|v(t)|_\infty+ |u(t)|_2\big)\leq C(T)(\sqrt{R}+1).
\end{equation}

Next, multiplying $\eqref{e1.5}_1$ by $\chi_R^\flat (2\delta-2)\rho^{2\delta-3}$ and integrating over $I$,  we have
\begin{equation}\label{627}
\frac{\mathrm{d}}{\mathrm{d}t}|\chi_R^\flat \rho^{\delta-1}|_2^2=-2\int_0^R  u(\rho^{\delta-1})_r\rho^{\delta-1}\,\mathrm{d}r-(2\delta-2) \int_0^R \big(u_r+\frac{m}{r}u\big)\rho^{2\delta-2}\,\mathrm{d}r:=\sum_{i=20}^{21}I_{i}.
\end{equation}
Then, for $I_{20}$--$I_{21}$, we obtain from \eqref{wuqiong-2-2}--\eqref{323} and the H\"older and Young inequalities that
\begin{equation}\label{628}
\begin{aligned}
I_{20}&\leq 2|u|_\infty|\chi_R^\flat(\rho^{\delta-1})_r|_2|\chi_R^\flat\rho^{\delta-1}|_2\leq C(T)(R+1)|u|_\infty^2+|\chi_R^\flat\rho^{\delta-1}|_2^2,\\
I_{21}&\leq C_0 |\rho^\frac{\delta-1}{2}\mathrm{D}_r u|_2|\chi_R^\flat \rho^{\delta-1}|_2|\chi_R^\flat \rho^{\delta-1}|_\infty^\frac{1}{2} \\
&\leq C(T) \Big(\frac{R+1}{R}\Big)^\frac{1}{4} |\rho^\frac{\delta-1}{2}\mathrm{D}_r u|_2|\chi_R^\flat \rho^{\delta-1}|_2^\frac{3}{2} + C(T) (\sqrt{R}+1)^\frac{1}{2} |\rho^\frac{\delta-1}{2}\mathrm{D}_r u|_2|\chi_R^\flat \rho^{\delta-1}|_2\\
&\leq C(T)\Big(\frac{R+1}{R}+\sqrt{R}+1\Big) + \big(1+ |\rho^\frac{\delta-1}{2}\mathrm{D}_r u|_2^2 \big)|\chi_R^\flat \rho^{\delta-1}|_2^2.
\end{aligned}
\end{equation}

Combining \eqref{627}--\eqref{628}, along with Lemmas \ref{ele} and the Gr\"onwall inequality, yields
\begin{equation}\label{324}
\begin{aligned}
|\chi_R^\flat\rho^{\delta-1}(t)|_2^2 &\leq C(T)\Big(|\chi_R^\flat \rho_0^{\delta-1}|_2^2+\frac{R+1}{R}+\sqrt{R}+1\Big)\\
&\leq C(T)\Big(R|\chi_R^\flat \rho_0^{\delta-1}|_\infty^2+\frac{R+1}{R}+\sqrt{R}+1\Big).
\end{aligned}
\end{equation}

Collecting \eqref{wuqiong-2-2}--\eqref{323} and \eqref{324}, and letting $R_0\geq 1$ be a fixed sufficiently large constant, we obtain that, for all $t\in [0,T]$ and $R\geq R_0$,
\begin{equation*}
|\chi_R^\flat\rho^{\delta-1}(t)|_\infty\leq C(T)(\sqrt{R}+1) \big(|\chi_R^\flat\rho_0^{\delta-1}|_\infty+1 \big)\leq C(T)(\sqrt{R}+1) \big(\underline{\rho}^{\delta-1}(R)+1 \big), 
\end{equation*}
which, implies that, for all $(t,r)\in[0,T]\times [0,R]$ and $R\geq R_0$, 
\begin{equation*}
 \rho(t,r)\geq \frac{C(T)^{-1}\underline{\rho}(R)}{\big(R^{\frac{1}{2-2\delta}}+1\big)(\underline{\rho}(R)+1)}.  
\end{equation*}

Finally, for $R\leq R_0$, it follows from Lemma \ref{ale1}, \eqref{wuqiong-2-2}--\eqref{323}, and \eqref{324} that 
\begin{equation*}
\begin{aligned}
|\chi_{R}^\flat\rho^{\delta-1}(t)|_\infty&\leq |\chi_{R_0}^\flat\rho^{\delta-1}(t)|_\infty \leq C_0\Big(\sqrt{\frac{1+R_0}{R_0}} |\chi_{R_0}^\flat\rho^{\delta-1}(t)|_2+|\chi_{R_0}^\flat(\rho^{\delta-1})_r(t)|_2\Big)\notag \\
&\leq  C(T) \big(|\chi_{R_0}^\flat\rho_0^{\delta-1}|_\infty+1 \big)\leq C(T),
\end{aligned}
\end{equation*}
which implies that $\rho(t,r)\geq C(T)^{-1}$ for all $(t,r)\in[0,T]\times [0,R]$ and $R\leq R_0$.
\end{proof}

\section{Global  Estimates for  
Regular Solutions with Far-Field Vacuum}\label{section-global2}
The goal of this section is to establish the global uniform estimates for the regular solutions when  $\bar\rho=0$. Let $T>0$ be any fixed time, and let $(\rho, u)(t,r)$ be the regular solution of problem \eqref{e1.5} in $[0,T]\times I$ 
obtained in Theorem \ref{rth1}. Moreover, throughout this section, we always assume that 
\eqref{cd1} holds.

Next, we consider the enlarged system \eqref{eq:cccq} in spherical coordinates. 
Specifically, we introduce the following two important quantities: 
\begin{equation}\label{tr}
\phi=\frac{A\gamma}{\gamma-1}\rho^{\gamma-1}, \qquad \psi=\frac{a\delta}{\delta-1}(\phi^{2\iota})_r=\frac{\delta}{\delta-1}(\rho^{\delta-1})_r.
 \end{equation}
Then \eqref{e1.5} can be rewritten as the problem for $(\phi,u,\psi)$ in $[0,T]\times I$:
\begin{equation}\label{e2.2}
\begin{cases}
\displaystyle \phi_t+u\phi_r+(\gamma-1)\phi\big(u_r+ \frac{m}{r}u\big)=0,\\[4pt]
\displaystyle u_t+u u_r+\phi_r=2a_1\delta a\phi^{2\iota}\big(u_r+ \frac{m}{r}u\big)_r+2a_1 \psi\big(\delta u_r+m(\delta-1)\frac{u}{r}\big),\\[4pt]
\displaystyle \psi_t+u\psi_r+\big(\delta u_r+(\delta-1)m\frac{u}{r}\big)\psi+\delta a\phi^{2\iota}\big(u_r+\frac{m}{r}u\big)_r=0,\\[4pt]
\displaystyle  (\phi,u,\psi)|_{t=0}=(\phi_0,u_0,\psi_0)
=(\frac{A\gamma}{\gamma-1}\rho_0^{\gamma-1},u_0,\frac{\delta}{\delta-1}(\rho_0^{\delta-1})_r) \qquad\text{for $r\in I$},\\[4pt]
\displaystyle u|_{r=0}=0 \qquad\qquad\qquad\qquad\qquad\quad \,\text{for $t\in [0,T]$},\\[4pt]
\displaystyle  (\phi,u)\to (0,0) \qquad \,\ \ \text{as $r\to\infty$}  \qquad \text{for $t\in [0,T]$}.
\end{cases}
\end{equation}
Clearly, the effective velocity $v$ and its equation \eqref{eq:effective2} take the following forms, respectively:
\begin{align}
&v=u+2a_1 \psi=u+\frac{2a_1 \delta  a}{\delta-1}(\phi^{2\iota})_r=u+\frac{2a_1\delta a}{\gamma-1}\phi^{2\iota-1}\phi_r,\label{v-express-2}\\
&v_t+uv_r+\frac{\gamma-1}{2a_1 \delta  a}\phi^{1-2\iota}(v-u)=0.\label{eq:effective1}
\end{align}

\subsection{Some Auxiliary Lemmas}
The following two lemmas concern div-curl estimates and weighted div-curl estimates for spherically symmetric functions in spherical coordinates, which are frequently used in our analysis. 

\begin{lem}\label{im-1}
Let $p\in(1,\infty)$ and $\boldsymbol{f}(\boldsymbol{x})=f(r)\frac{\boldsymbol{x}}{r} \in C_{\rm c}^\infty(\mathbb{R}^n)$. 
Then 
\begin{equation*}
\begin{aligned}
&\|\nabla^j\boldsymbol{f}\|_{L^p} \sim \hspace{0.3mm}|r^\frac{m}{p}\mathrm{D}_r^jf|_{p}\sim \Big|r^\frac{m}{p}\mathrm{D}_r^{j-1}\big(f_r+\frac{m}{r}f\big)\Big|_{p}\quad&&\text{for $j=1,2$},\\
&\|\nabla^k\boldsymbol{f}\|_{L^p}\sim |r^\frac{m}{p}\mathrm{D}_r^kf|_{p} \sim \Big|r^\frac{m}{p}\mathrm{D}_r^{k-2}\big(f_r+\frac{m}{r}f\big)_r \Big|_p\quad&&\text{for $k=3,4$},
\end{aligned}
\end{equation*}
where $F_1\sim F_2$ denotes $C^{-1}F_1\leq F_2\leq CF_1$ for some constant $C>0$ depending only on $(n,p)$. 
\end{lem}
\begin{proof}
Let $\boldsymbol{f}(\boldsymbol{x})=f(r)\frac{\boldsymbol{x}}{r}$. First, since $\boldsymbol{f}$ is curl-free, we have the following identity 
\begin{equation*}
\Delta \boldsymbol{f} = \nabla\diver\boldsymbol{f}\implies \nabla \boldsymbol{f}= -\nabla(-\Delta)^{-1}\nabla\diver \boldsymbol{f},
\end{equation*}
where $(-\Delta)^{-1}$ is defined via the Fourier transform: 
\begin{equation*}
\big((-\Delta)^{-1}g\big)(\boldsymbol{x})
:=\mathcal{F}^{-1}\big[\frac{1}{4\pi|\boldsymbol{\omega}|^2}\mathcal{F}[g](\boldsymbol{\omega})\big] (\boldsymbol{x}),
\end{equation*}
with $\mathcal{F}[g](\boldsymbol{\omega})$ as
the Fourier transform of $g\in C^\infty_{\rm c}(\mathbb{R}^n)$: $\mathcal{F}[g](\boldsymbol{\omega})=\int_{\mathbb{R}^n} g(\boldsymbol{x})e^{-2\pi\mathrm{i}\boldsymbol{x}\cdot\boldsymbol{\omega}}\,\mathrm{d}\boldsymbol{x}$.

Thus, by the Mihlin--H\"ormander multiplier theorem (see \cite[Theorem 6.2.7]{lg}), we have 
\begin{equation*} 
\|\nabla^j\diver\boldsymbol{f}\|_{L^p}\sim\|\nabla^{j+1}\boldsymbol{f}\|_{L^p}\qquad \text{for integer $0\leq j\leq 3$},
\end{equation*}
which, together with $\diver\boldsymbol{f}=f_r+\frac{m}{r}f$ and  Lemma \ref{lemma-initial}, leads to the desired results. 
\end{proof}

\begin{lem}\label{im-2}
Let $\boldsymbol{f}(t,\boldsymbol{x})=f(t,r)\frac{\boldsymbol{x}}{r}\in C_\mathrm{c}^\infty([0,T]\times\mathbb{R}^n)$. Then, for any $q\geq 2\iota$, there exists a constant $C(q,T)>0$ such that, for all $t\in [0,T]$,
\begin{equation*}
\begin{aligned}
|r^\frac{m}{2}\phi^{q}\mathrm{D}_r^j f|_2&\leq 2\Big|r^\frac{m}{2}\phi^{q}\partial_r^{j-1}\big(f_r+\frac{m}{r}f\big)\Big|_2+C(q,T) (1+|u|_\infty)|r^\frac{m}{2}\mathrm{D}_r^{j-1}f|_2\quad&&\text{for $j=1,2$},\\
|r^\frac{m}{2}\phi^{q}\mathrm{D}_r^k f|_2&\leq 2\Big|r^\frac{m}{2}\phi^q \mathrm{D}_r^{k-2}\big(f_r+ \frac{m}{r}f\big)_{r}\Big|_2+C(q,T)(1+|u|_\infty)|r^\frac{m}{2} \mathrm{D}_r^{k-1} f|_2\quad&&\text{for $k=3,4$}.
\end{aligned}
\end{equation*}
\end{lem}
\begin{proof}
Let $q\geq 2\iota$.  For brevity, we only give the proof for the highest-order estimates, and the others can be derived similarly.

First, since the following identity holds: $f_{rrrr}=r(\frac{f}{r})_{rrrr}+ 4 (\frac{f}{r})_{rrr}$, we deduce from a direct calculation that
\begin{equation}\label{L6}
\begin{aligned}
&\,\Big|r^\frac{m}{2}\phi^q \big(f_r+ \frac{m}{r}f\big)_{rrr}\Big|_2^2\\
&=\int_0^\infty r^m\phi^{2q} \Big(|f_{rrrr}|^2 +m^2 \Big|\big(\frac{f}{r}\big)_{rrr}\Big|^2\Big) \,\mathrm{d}r+ 2m\int_0^\infty r^m\phi^{2q} f_{rrrr} \big(\frac{f}{r}\big)_{rrr}\,\mathrm{d}r\\
&=\int_0^\infty  r^m\phi^{2q} \Big(|f_{rrrr}|^2 +(m^2+8m) \Big|\big(\frac{f}{r}\big)_{rrr}\Big|^2\Big) \,\mathrm{d}r +2m\int_0^\infty r^{m+1}\phi^{2q} \big(\frac{f}{r}\big)_{rrrr} \big(\frac{f}{r}\big)_{rrr}\,\mathrm{d}r\\
&=\int_0^\infty r^m\phi^{2q} \Big(|f_{rrrr}|^2\! +\!7m\Big|\big(\frac{f}{r}\big)_{rrr}\Big|^2\Big) \mathrm{d}r  -\underline{\frac{2qm}{\iota}\int_0^\infty\! r^{m+1}\phi^{2q-2\iota}(\phi^{2\iota})_r \Big|\big(\frac{f}{r}\big)_{rrr}\Big|^2 \mathrm{d}r}_{:=G_1}\!\!\!.    
\end{aligned}
\end{equation}

Then, for $G_1$, we obtain from the identity: $f_{rrr}=r(\frac{f}{r})_{rrr}+ 3 (\frac{f}{r})_{rr}$, \eqref{v-express-2}, Lemmas \ref{important2} and \ref{l4.4}, and the H\"older and Young inequalities that
\begin{equation}\label{G1}
\begin{aligned}
G_1&=\frac{qm(\delta-1)}{a_1\delta a\iota}\int_0^\infty r^{m+1}\phi^{2q-2\iota}(v-u)\Big|\big(\frac{f}{r}\big)_{rrr}\Big|^2 \,\mathrm{d}r\\
&\leq C(q)|\phi|_\infty^{q-2\iota}|(v,u)|_\infty \Big|r^\frac{m+2}{2} \big(\frac{f}{r}\big)_{rrr}\Big|_2\Big|r^\frac{m}{2}\phi^q \big(\frac{f}{r}\big)_{rrr}\Big|_2\\
&\leq C(q,T)(1+|u|_\infty) |r^\frac{m}{2} (\mathrm{D}_r^2 f)_r|_2\Big|r^\frac{m}{2}\phi^q \big(\frac{f}{r}\big)_{rrr}\Big|_2\\
&\leq C(q,T)\big(1+|u|_\infty^2\big)|r^\frac{m}{2} (\mathrm{D}_r^2 f)_r|_2^2+\frac{1}{100}\Big|r^\frac{m}{2}\phi^q \big(\frac{f}{r}\big)_{rrr}\Big|_2^2,
\end{aligned}
\end{equation}
which, along with \eqref{L6}, gives
\begin{equation}\label{L6'}
|r^\frac{m}{2}\phi^q(\mathrm{D}_r^2 f)_{rr}|_2^2 \leq 2\Big|r^\frac{m}{2}\phi^q \big(f_r+ \frac{m}{r}f\big)_{rrr}\Big|_2^2 +C(q,T)\big(1+|u|_\infty^2\big) |r^\frac{m}{2}(\mathrm{D}_r^2 f)_{r}|_2^2.  
\end{equation}

Next, thanks to the identity
\begin{equation}\label{L6''}
\big(\frac{f_{rr}}{r}\big)_r=\big(\frac{f}{r}\big)_{rrr}+2\Big(\frac{1}{r}\big(\frac{f}{r}\big)_r\Big)_r= r \Big(\frac{1}{r}\big(\frac{f}{r}\big)_r\Big)_{rr}+4 \Big(\frac{1}{r}\big(\frac{f}{r}\big)_r\Big)_r,
\end{equation}
we deduce from a direct calculation and integration by parts that
\begin{equation}\label{9.6}
\begin{aligned}
&\,\Big|r^\frac{m}{2}\phi^q \Big(\frac{1}{r}\big(f_r+ \frac{m}{r}f\big)_r\Big)_r\Big|_2^2=\Big|r^\frac{m}{2}\phi^{q}\Big(r\big(\frac{1}{r}(\frac{f}{r})_r\big)_{rr}+(m+4) \big(\frac{1}{r}(\frac{f}{r})_r\big)_r\Big)\Big|_2^2\\
&=\int_0^\infty  r^m\phi^{2q} \Big(\Big|r\Big(\frac{1}{r}\big(\frac{f}{r}\big)_r\Big)_{rr}\Big|^2 +(m+4)^2 \Big|\Big(\frac{1}{r}\big(\frac{f}{r}\big)_r\Big)_r\Big|^2\Big)\,\mathrm{d}r\\
&\quad +2(m+4)\int_0^\infty  r^{m+1}\phi^{2q} \Big(\frac{1}{r}\big(\frac{f}{r}\big)_r\Big)_{rr} \Big(\frac{1}{r}\big(\frac{f}{r}\big)_r\Big)_r\,\mathrm{d}r\\
&=\int_0^\infty  r^m\phi^{2q} \Big(\Big|r\Big(\frac{1}{r}\big(\frac{f}{r}\big)_r\Big)_{rr}\Big|^2 +(3m+12)\Big|\Big(\frac{1}{r}\big(\frac{f}{r}\big)_r\Big)_r\Big|^2\Big)\,\mathrm{d}r\\
&\quad -\underline{\frac{2q(m+4)}{\iota}\int_0^\infty r^{m+1}\phi^{2q-2\iota}(\phi^{2\iota})_r\Big|\Big(\frac{1}{r}\big(\frac{f}{r}\big)_r\Big)_r\Big|^2 \,\mathrm{d}r}_{:=\widetilde G_1}.
\end{aligned}
\end{equation}
Then, for $\widetilde G_1$, using the identity
\begin{equation*}
\frac{f_{rr}}{r}=r\Big(\frac{1}{r}\big(\frac{f}{r}\big)_{r}\Big)_r+\frac{3}{r}\big(\frac{f}{r}\big)_r,
\end{equation*}
repeating a calculation similar to \eqref{G1}, we obtain
\begin{equation*} 
\begin{aligned}
\widetilde G_1&=\frac{q(m+4)(\delta-1)}{a_1\delta a\iota}\int_0^\infty r^{m+1}\phi^{2q-2\iota}(v-u)\Big|\Big(\frac{1}{r}\big(\frac{f}{r}\big)_r\Big)_r\Big|^2 \,\mathrm{d}r\\
&\leq C(q)|\phi|_\infty^{q-2\iota}|(v,u)|_\infty \Big|r^\frac{m+2}{2}\Big(\frac{1}{r}\big(\frac{f}{r}\big)_r\Big)_r\Big|_2\Big|r^\frac{m}{2}\phi^q \Big(\frac{1}{r}\big(\frac{f}{r}\big)_r\Big)_r\Big|_2\\
&\leq C(q,T)(1+|u|_\infty) |r^\frac{m-2}{2}\mathrm{D}_r^2 f|_2\Big|r^\frac{m}{2}\phi^q \Big(\frac{1}{r}\big(\frac{f}{r}\big)_r\Big)_r\Big|_2\\
&\leq C(q,T)\big(1+|u|_\infty^2\big)|r^\frac{m-2}{2}\mathrm{D}_r^2 f|_2^2+\frac{1}{100}\Big|r^\frac{m}{2}\phi^q \Big(\frac{1}{r}\big(\frac{f}{r}\big)_r\Big)_r\Big|_2^2.
\end{aligned}
\end{equation*}

Therefore, \eqref{9.6}, combined with the above and \eqref{L6''}, implies 
\begin{equation*}
\Big|r^\frac{m}{2}\phi^q\big(\frac{\mathrm{D}_r^2 f}{r}\big)_r\Big|_2^2 \leq 2\Big|r^\frac{m}{2}\phi^q \Big(\frac{1}{r}\big(f_r+ \frac{m}{r}f\big)_r\Big)_r\Big|_2^2 +C(q,T)\big(1+|u|_\infty^2\big)|r^\frac{m-2}{2}\mathrm{D}_r^2 f|_2^2.    
\end{equation*}
which, together with \eqref{L6'}, leads to the desired estimate.
\end{proof}

\subsection{Zeroth- and First-Order Estimates of the Velocity}
The first lemma concerns the zeroth-order energy estimate for $u$.
\begin{lem}\label{l4.5}
There exists a constant $C(T)>0$ such that 
\begin{equation*}
|r^{\frac{m}{2}}(u,\phi^{1-2\iota} v)(t)|_2^2 +\int_0^t |r^{\frac{m}{2}}\phi^{\iota}\mathrm{D}_r u|_2^2 \, \mathrm{d}s\leq C(T)\qquad\text{for any $t\in [0,T]$}.
\end{equation*}
\end{lem}

\begin{proof} We divide the proof into two steps.

\smallskip
\textbf{1.} Multiplying $\eqref{e2.2}_2$ by $r^m u$, together with \eqref{v-express-2}, gives
\begin{equation}\label{eq:B7-pre}
\begin{aligned}
&\,\frac{1}{2}(r^m u^2)_t+2a_1\delta a r^m\phi^{2\iota} \big(u_r+\frac{m}{r}u\big)^2-2a_1\delta a\Big(r^m \underline{\phi^{2\iota}u\big(u_r+\frac{m}{r}u\big)}_{:=\mathcal{B}_{11}}\Big)_r \\
&= -\frac{\gamma-1}{2a_1\delta a }r^m\phi^{1-2\iota}(v-u)u+ r^m(v-2u)uu_r.
\end{aligned}
\end{equation}
Then we show that $r^m \mathrm{D}_r \mathcal{B}_{11} \in L^1(I)$ for {\it a.e.} $t\in (0,T)$, which allows us to apply Lemma \ref{calculus} to obtain
\begin{equation}\label{eq:B7}
\int_0^\infty (r^m\mathcal{B}_{11})_r\,\mathrm{d}r=0.
\end{equation}
Indeed, based on \eqref{spd}, we see that, for \textit{a.e.} $t\in (0,T)$, 
\begin{equation*}
\begin{aligned}
(\phi^{2\iota})_r \in L^\infty(I),\quad r^\frac{m}{2}(\phi^\iota \mathrm{D}_r u,\phi^{2\iota} \mathrm{D}_r^2 u)\in L^2(I),\quad r^\frac{m}{2} (u,\mathrm{D}_r u)\in L^2(I).
\end{aligned}
\end{equation*}
Then it follows from the H\"older inequality that 
\begin{equation*}
\begin{aligned}
|r^m \mathrm{D}_r \mathcal{B}_{11}|_1&\leq C_0\big(\big|r^{m-1}\phi^{2\iota}u\mathrm{D}_r u\big|_1+ |r^m(\phi^{2\iota})_ru\mathrm{D}_r u|_1 + |r^{m}\phi^{2\iota}u_r\mathrm{D}_r u|_1+ |r^{m}\phi^{2\iota}u\mathrm{D}_r^2 u|_1\big)\\
&\leq C_0\big(|r^\frac{m}{2}\phi^\iota\mathrm{D}_r u |_2^2+ |r^\frac{m}{2}u|_2|(\phi^{2\iota})_r|_\infty |r^\frac{m}{2}\mathrm{D}_r u|_2+ |r^\frac{m}{2}u|_2 |r^\frac{m}{2}\phi^{2\iota}\mathrm{D}_r^2 u|_2\big)<\infty.    
\end{aligned}
\end{equation*}    

Thus, integrating \eqref{eq:B7-pre} over $I$, together with \eqref{eq:B7}, yields
\begin{equation}\label{e4.15}
\begin{aligned}
&\,\frac{1}{2}\frac{\mathrm{d}}{\mathrm{d}t} |r^{\frac{m}{2}}u|_2^2 + 2a_1\delta a\Big|r^{\frac{m}{2}}\phi^\iota \big(u_r+\frac{m}{r}u\big)\Big|_2^2\\
&= -\frac{1}{4a_1\delta a \iota}\int_0^\infty r^m \phi^{1-2\iota}(v-u)u\,\mathrm{d}r + \int_0^\infty r^m (v-2u)uu_r\,\mathrm{d}r:=\sum_{i=2}^{3} G_i.
\end{aligned}
\end{equation}
Then, for $G_2$--$G_3$, it follows from \eqref{v-express-2}, Lemmas \ref{important2}, \ref{l4.4}, and \ref{im-2}, 
and the H\"older and Young inequalities that
\begin{equation}\label{G6-G8}   
\begin{aligned}
G_2&\leq  C_0 |r^\frac{m}{2}\phi^{1-2\iota} v|_{2} |r^{\frac{m}{2}}u|_{2} +C_0|\phi|_{\infty}^{1-2\iota} |r^{\frac{m}{2}}u|_{2}^2 \leq  C(T)|r^\frac{m}{2}(\phi^{1-2\iota} v,u)|_2^2,\\
G_3&\leq C_0|\phi|_\infty^{-\iota}|(v,u)|_{\infty}|r^{\frac{m}{2}}u|_{2} |r^{\frac{m}{2}}\phi^\iota u_r|_{2} \\
&\leq  C(T)\big(1+|u|_\infty^2\big)|r^{\frac{m}{2}}u|_{2}^2+\frac{a_1\delta a}{8}\Big|r^{\frac{m}{2}}\phi^\iota\big(u_r+\frac{m}{r}u\big)\Big|_2^2.
\end{aligned}
\end{equation}
Substituting \eqref{G6-G8} into \eqref{e4.15}, along with Lemma \ref{im-2}, leads to
\begin{equation}\label{eq4.16}
\frac{\mathrm{d}}{\mathrm{d}t} |r^{\frac{m}{2}}u|_2^2 + a_1\delta a |r^{\frac{m}{2}}\phi^\iota \mathrm{D}_r u|_2^2 \leq C(T)\big(1+|u|_\infty^2\big) |r^\frac{m}{2}(\phi^{1-2\iota} v,u)|_2^2.
\end{equation}

\smallskip
\textbf{2. $L^2(I)$-estimate for $r^\frac{m}{2}\phi^{1-2\iota} v$.}  We multiply \eqref{eq:effective1} by $r^m\phi^{2-4\iota}v$ and use $\eqref{e2.2}_1$ to obtain
\begin{equation}\label{b8pre}
\begin{aligned}
&\,\frac{1}{2} (r^m \phi^{2-4\iota}v^2)_t+\frac{\gamma-1}{2a_1\delta a}r^m \phi^{3-6\iota} v^2\\
&=-\frac{1}{2}(r^m \underline{u\phi^{2-4\iota}v^2}_{:=\mathcal{B}_{12}})_r-\big(\frac{2\delta+1}{2}-\gamma\big)r^m\phi^{2-4\iota}v^2\big(u_r+\frac{m}{r}u\big)+\frac{\gamma-1}{2a_1\delta a}r^m \phi^{3-2\iota} vu.
\end{aligned}
\end{equation}
Next, we need to  show that $r^m \mathrm{D}_r\mathcal{B}_{12} \in L^1(I)$  for {\it a.e.} $t\in (0,T)$, which allows us to apply Lemma \ref{calculus} to obtain
\begin{equation}\label{eq:B8}
\int_0^\infty (r^m\mathcal{B}_{12})_r\,\mathrm{d}r=0.
\end{equation}
Indeed,  \eqref{spd}--\eqref{spd2}, \eqref{psi,wuqiong0}, and Lemmas \ref{ale1} and \ref{lemma-initial} imply 
\begin{equation*}
(\phi, u, \mathrm{D}_ru,v)\in L^\infty(I),\quad r^\frac{m}{n^*}\phi\in L^{n^*}(I),\quad r^\frac{m}{n}\psi_{r}\in L^n(I),\quad  
r^\frac{m}{2}(\phi_r,u,u_r)\in L^2(I)
\end{equation*}
for {\it a.e.} $t\in (0,T)$, where $n^*$ is defined in \S\ref{othernote}. Then it follows from \eqref{v-express-2} and the H\"older inequality that
\begin{equation*}
\begin{aligned}
|r^m \mathrm{D}_r\mathcal{B}_{12}|_1&\leq C_0\big|\big(r^{m-1}u\phi^{2-4\iota} v^2,r^m u_r\phi^{2-4\iota} v^2,r^m u\phi^{1-4\iota} \phi_r v^2,r^{m}u\phi^{2-4\iota} vv_r\big)\big|_1\\
&\leq C_0|\mathrm{D}_r u|_\infty\big(|\phi|_\infty^{2-4\iota}|r^\frac{m}{2}u|_2^2+ |r^\frac{m}{2}\phi_r|_2^2\big)\\
&\quad +C_0|\phi|_\infty^{1-4\iota}|v|_\infty |r^\frac{m}{2}u|_2\big(|v|_\infty|r^\frac{m}{2}\phi_r|_2+|\phi|_\infty |r^\frac{m}{2}u_r|_2+|r^\frac{m}{n^*}\phi|_{n^*} |r^\frac{m}{n}\psi_r|_n\big)<\infty.  
\end{aligned}
\end{equation*} 

Thus, integrating \eqref{b8pre} over $I$, we obtain from  \eqref{eq:B8}, Lemmas \ref{important2}  and \ref{l4.4}, 
and the H\"older and Young inequalities that
\begin{equation}\label{555}
\begin{aligned}
&\,\frac{1}{2}\frac{\mathrm{d}}{\mathrm{d}t}|r^\frac{m}{2}\phi^{1-2\iota} v|_2^2+ \frac{\gamma-1}{2a_1\delta a} \big|r^\frac{m}{2}\phi^{\frac{3}{2}-3\iota} v\big|_2^2\\
&=\big(\frac{2\delta+1}{2}-\gamma\big)\int_0^\infty  r^m\phi^{2-4\iota} v^{2} \big(u_r+\frac{m}{r}u\big)\,\mathrm{d}r +\frac{\gamma-1}{2a_1\delta a}\int_0^\infty r^m\phi^{6-4\iota} vu\,\mathrm{d}r\\
&\leq  C_0|\phi|_\infty^{1-\iota}|v|_\infty |r^\frac{m}{2}\phi^{1-2\iota} v|_2|r^\frac{m}{2}\phi^\iota\mathrm{D}_r u|_2+C_0|\phi|_\infty^{2-2\iota} |r^\frac{m}{2}\phi^{1-2\iota} v|_2 |r^\frac{m}{2}u|_2\\
&\leq  C(T)|r^\frac{m}{2}(\phi^{1-2\iota} v,u)|_2^2+\frac{a_1\delta a}{8} |r^\frac{m}{2}\phi^\iota\mathrm{D}_r u|_2^2,
\end{aligned}    
\end{equation}
which, together with \eqref{eq4.16}, gives
\begin{equation}\label{eq4.17}
\frac{\mathrm{d}}{\mathrm{d}t}|r^\frac{m}{2}(\phi^{1-2\iota} v,u)|_2^2 + \frac{a_1\delta a}{2}|r^\frac{m}{2}\phi^\iota\mathrm{D}_r u|_2^2\leq C(T)\big(1+|u|_\infty^2\big)|r^\frac{m}{2}(\phi^{1-2\iota} v,u)|_2^2.
\end{equation}

Applying  the Gr\"onwall  inequality to \eqref{eq4.17}, together with Lemma \ref{ele}, yields
\begin{equation*} 
|r^\frac{m}{2}(\phi^{1-2\iota} v,u)|_2^2+\int_0^t |r^\frac{m}{2}\phi^\iota\mathrm{D}_r u|_2^2 \,\mathrm{d}s\leq C(T)\qquad \text{for all $t\in [0,T]$}
\end{equation*}
where, the $L^2(I)$-boundedness of $r^\frac{m}{2}(\phi_0^{1-2\iota} v_0,u_0)$ follows from Lemmas \ref{ale1}, \ref{initial3}, and  \ref{lemma-initial}:
\begin{equation}\label{chuzhi0}
\begin{aligned}
|r^\frac{m}{2}(\phi_0^{1-2\iota}v_0,u_0)|_2 &\leq C_0|r^\frac{m}{2}(\phi_0^{1-2\iota}u_0,\phi_0^{1-2\iota}\psi_0,u_0)|_2 \\
&\leq C_0\big(\|\phi_0\|_{L^\infty}^{1-2\iota}\|\boldsymbol{u}_0\|_{L^2}+\|\nabla\phi_0\|_{L^2} +\|\boldsymbol{u}_0\|_{L^2}\big)\leq C_0. 
\end{aligned}
\end{equation}

This completes the proof.
\end{proof}

Next, we establish the following intermediate estimate for the first-order derivative of $u$.
\begin{lem}\label{l4.6-0}
There exists a constant $C(T)>0$ such that 
\begin{equation*} 
|r^\frac{m}{2} \mathrm{D}_r u|_2^2+\int_0^t |r^{\frac{m}{2}}\phi^{-\iota}u_t|_2^2\,\mathrm{d}s\leq C(T)\qquad\text{for any $t\in [0,T]$}.
\end{equation*}
\end{lem}
\begin{proof}
First, multiplying $\eqref{e2.2}_2$ by $r^m \phi^{-2\iota}u_t$, together with $\eqref{e2.2}_1$ and \eqref{v-express-2}, gives 
\begin{equation}\label{7.5-1}
\begin{aligned}
&\,a_1\delta a\frac{\mathrm{d}}{\mathrm{d}t}\Big|\big(u_r+\frac{m}{r}u\big)\Big|_2^2+r^m\phi^{-2\iota}u_t^2-2a_1\delta a \Big(r^m\underline{\big(u_r+\frac{m}{r}u\big)u_t}_{:=\mathcal{B}_{13}}\Big)_r\\
&= -\frac{\gamma-1}{2a_1\delta a}r^m\phi^{1-4\iota}(v-u)u_t+r^m\phi^{-2\iota} u_t\Big( (v-u)\big(\delta u_r+m(\delta-1)\frac{u}{r}\big) -  uu_r\Big).
\end{aligned}
\end{equation}
Then we show that $r^m \mathrm{D}_r\mathcal{B}_{13} \in L^1(I)$ for \textit{a.e.} $t\in (0,T)$, which allows us to apply Lemma \ref{calculus} to obtain 
\begin{equation}\label{int-B13}
\int_0^\infty (r^m\mathcal{B}_{13})_r\,\mathrm{d}r=0.
\end{equation}
To obtain this, we see from \eqref{spd} that
\begin{equation*}
r^\frac{m}{2} (\mathrm{D}_r u,\mathrm{D}_r^2 u,u_t,\mathrm{D}_r u_t)\in L^2(I)\qquad\text{for \textit{a.e.} $t\in (0,T)$}.
\end{equation*}
Then it follows from the H\"older inequality that
\begin{equation*}
\begin{aligned}
|r^m \mathrm{D}_r\mathcal{B}_{13}|_1&\leq C_0|r^{m-1}(\mathrm{D}_r u) u_t|_1+C_0 |r^{m}(\mathrm{D}_r^2 u)u_t|_1+C_0|r^{m}(\mathrm{D}_r u)u_{tr}|_1\\
&\leq C_0 |r^\frac{m}{2} \mathrm{D}_r u|_2|r^\frac{m}{2}\mathrm{D}_r u_t|_2+C_0 |r^\frac{m}{2}\mathrm{D}_r^2 u|_2|r^\frac{m}{2}u_t|_2<\infty.
\end{aligned}
\end{equation*}

Next, integrating \eqref{7.5-1} over $I$, together with \eqref{int-B13}, we have
\begin{equation}\label{e-1.200}
\begin{aligned}
&\,a_1\delta a\frac{\mathrm{d}}{\mathrm{d}t}\Big|r^\frac{m}{2}\big(u_r+\frac{m}{r}u\big)\Big|_2^2+|r^{\frac{m}{2}}\phi^{-\iota}u_t|_2^2\\
&= -\frac{\gamma-1}{2a_1\delta a}\int_0^\infty r^m\phi^{1-4\iota}(v-u)u_t\,\mathrm{d}r \\
&\quad + \int_0^\infty r^m\phi^{-2\iota} u_t\Big( (v-u)\big(\delta u_r+m(\delta-1)\frac{u}{r}\big)-uu_r\Big)\,\mathrm{d}r:=\sum_{i=4}^{5} G_i.
\end{aligned}
\end{equation}
For $G_4$--$G_{5}$, by Lemmas \ref{important2}, \ref{l4.4}, 
and \ref{l4.5}, and the H\"older and Young inequalities, we have
\begin{equation}\label{G9-G110}
\begin{aligned}
G_4&\leq  C_0\big(|\phi|_\infty^{-\iota}|r^{\frac{m}{2}} \phi^{1-2\iota} v|_{2} + |\phi|_\infty^{1-3\iota} |r^{\frac{m}{2}}u|_2\big) |r^{\frac{m}{2}}\phi^{-\iota}u_t|_2\leq C(T)+\frac{1}{8} |r^{\frac{m}{2}}\phi^{-\iota}u_t|^2_2,\\
G_5&\leq C_0|\phi|_\infty^{-\iota}|(v,u)|_{\infty} |r^{\frac{m}{2}}\mathrm{D}_r u|_{2}|r^{\frac{m}{2}}\phi^{-\iota}u_t|_{2}\\
&\leq C(T)\big(1+|u|_{\infty}^2\big) |r^{\frac{m}{2}}\mathrm{D}_r u|_{2}^2+\frac{1}{8}|r^{\frac{m}{2}}\phi^{-\iota}u_t|_{2}^2.
\end{aligned}    
\end{equation}

Combining \eqref{e-1.200}--\eqref{G9-G110} gives
\begin{equation*}
a_1\delta a\frac{\mathrm{d}}{\mathrm{d}t}\Big|r^\frac{m}{2}\big(u_r+\frac{m}{r}u\big)\Big|_2^2+\frac{1}{2}|r^{\frac{m}{2}}\phi^{-\iota}u_t|_2^2\leq C(T)\big(1+|u|_{\infty}^2\big) |r^{\frac{m}{2}}\mathrm{D}_r u|_{2}^2+C(T),
\end{equation*}
which, along with Lemmas \ref{ele} and \ref{im-1}, the Gr\"onwall inequality, and the following estimate for the initial data in view of Lemma \ref{lemma-initial}:
\begin{equation*}
|r^\frac{m}{2}\mathrm{D}_r u_0|_2\leq C_0\|\nabla \boldsymbol{u}_0\|_{L^2}\leq C_0,
\end{equation*}
yields the desired estimate.
\end{proof}

The following lemma concerns the first-order energy estimate of $u$.
\begin{lem}\label{l4.6}
There exists a constant $C(T)>0$ such that 
\begin{equation*} 
|r^{\frac{m}{2}}\phi^\iota\mathrm{D}_r u(t)|_2^2+\int_0^t |r^{\frac{m}{2}}u_t|_2^2\,\mathrm{d}s\leq C(T)\qquad \text{for any $t\in [0,T]$}.
\end{equation*}
\end{lem}
\begin{proof}
We divide the proof into two steps.

\smallskip
\textbf{1.}  Multiplying $\eqref{e2.2}_2$ by $r^\frac{m}{2}$ and taking the $L^2(I)$-norm of the resulting equality, we obtain from \eqref{v-express-2}, Lemmas \ref{important2}, \ref{l4.4}, and \ref{l4.5}--\ref{l4.6-0} that   
\begin{equation*}
\begin{aligned}
\Big|r^{\frac{m}{2}}\phi^{2\iota}\big(u_r+\frac{m}{r}u\big)_r\Big|_2&\leq C_0 |r^{\frac{m}{2}}(u_t,uu_r,\phi_r,\psi\mathrm{D}_r u)|_2 \\
&\leq  C_0\big(|r^{\frac{m}{2}}(u_t, uu_r)|_2+ |\phi|_\infty^{-2\iota}|r^{\frac{m}{2}} \phi(v-u)|_2+ |r^{\frac{m}{2}} (v-u)\mathrm{D}_r u|_2\big) \\
&\leq  C(T)\big(|r^{\frac{m}{2}}u_t|_2+ |r^\frac{m}{2}\mathrm{D}_r u|_2|(v,u)|_\infty+|r^\frac{m}{2}\phi v|_2 +|\phi|_\infty|r^{\frac{m}{2}}u|_2\big)\\
&\leq C(T)\big(|r^{\frac{m}{2}}u_t|_2+|u|_\infty+1\big),
\end{aligned}
\end{equation*}
which, along with Lemmas \ref{important2}, \ref{im-2}, and \ref{l4.6-0}, yields that, for any $q\in [2\iota,0]$,
\begin{equation}\label{618}
|r^{\frac{m}{2}}\phi^{q}\mathrm{D}_r^2 u|_2\leq |\phi|_\infty^{q-2\iota} |r^{\frac{m}{2}}\phi^{2\iota}\mathrm{D}_r^2 u|_2\leq C(T)\big(|r^{\frac{m}{2}}u_t|_2+|u|_\infty+1\big).
\end{equation}

Next, it follows from \eqref{618}, Lemmas \ref{ele}, \ref{l4.5}--\ref{l4.6-0}, \ref{ale1}, and \ref{hardy}, and the H\"older and Young inequalities that, for all $t\in[0,T]$,
\begin{equation}\label{6.4-d}
\begin{aligned}
|r^\frac{m}{2}\mathrm{D}_r u|_\infty&\leq |\chi_1^\flat r^\frac{m}{2}\mathrm{D}_r u|_\infty+ |\chi_1^\sharp r^\frac{m}{2}\mathrm{D}_r u|_\infty\\
&\leq C_0\big|\chi_1^\flat r^\frac{m+1}{2} (\mathrm{D}_r u,\mathrm{D}_r^2 u)\big|_2 +C_0 |\chi_1^\sharp r^{\frac{m}{2}}\mathrm{D}_r u|_2+C_0|\chi_1^\sharp (r^\frac{m}{2}\mathrm{D}_r u)_r|_2 \\
&\leq C_0 |r^{\frac{m}{2}}(\mathrm{D}_r u,\mathrm{D}_r^2 u) |_2\leq C(T)\big(|r^{\frac{m}{2}}u_t|_2+|u|_\infty+1\big).
\end{aligned}
\end{equation}

Finally, according to \eqref{v-express-2},  Lemmas \ref{important2}, \ref{l4.4}, and \ref{l4.5}, we have 
\begin{equation}\label{6.6-1}
|r^{\frac{m}{2}}\phi_r|_2\leq C_0|r^{\frac{m}{2}}\phi^{1-2\iota}(v-u)|_2 \leq C_0 \big(|r^\frac{m}{2}\phi^{1-2\iota} v|_2+|\phi|_\infty^{1-2\iota}|r^{\frac{m}{2}}u|_2\big)\leq C(T).
\end{equation}

\smallskip
\textbf{2.} It follows from $\eqref{e2.2}_1$ that
\begin{equation}\label{phi2l}
(\phi^{2\iota})_t+u(\phi^{2\iota})_r+(\delta-1)\phi^{2\iota}\big(u_r+\frac{m}{r}u\big)=0.
\end{equation}

Multiplying $\eqref{e2.2}_2$ by $r^m u_t$, together with $\eqref{e2.2}_1$, \eqref{v-express-2}, and \eqref{phi2l}, gives
\begin{equation}\label{eq:B9pre}
\begin{aligned}
&\,a_1\delta a \Big(r^m\phi^{2\iota}\big(u_r+ \frac{m}{r}u\big)^2\Big)_t+r^m u_t^2-2a_1 \delta a\Big(r^m \underline{\phi^{2\iota}\big(u_r+ \frac{m}{r}u\big)u_t}_{:=\mathcal{B}_{14}}\Big)_r\\
&= r^m (v-2u) u_ru_t -r^m\phi_ru_t +\frac{1-\delta}{2} r^m (v-u)u\big(u_r+ \frac{m}{r}u\big)^2\\
&\quad-a_1\delta(\delta-1) a r^m\phi^{2\iota}\big(u_r+ \frac{m}{r}u\big)^3.    
\end{aligned}    
\end{equation}
Next, we need to  show that $r^m \mathrm{D}_r\mathcal{B}_{14} \in L^1(I)$ for {\it a.e.} $t\in (0,T)$, which allows us to apply Lemma \ref{calculus} to obtain
\begin{equation}\label{eq:B9}
\int_0^\infty (r^m\mathcal{B}_{14})_r\,\mathrm{d}r =0.
\end{equation}
To prove this, by \eqref{spd}, we have 
\begin{equation*}
r^\frac{m}{2}(\phi^\iota\mathrm{D}_r u,u_t,\phi^{2\iota}\mathrm{D}_r^2 u,\phi^\iota \mathrm{D}_r u_t)\in L^2(I)  \qquad \text{for {\it a.e.} $t\in (0,T)$}.
\end{equation*}
This implies from the H\"older inequality that 
\begin{equation*}
\begin{aligned}
|r^m \mathrm{D}_r\mathcal{B}_{14}|_1&\leq C_0 |r^{m-1}\phi^{2\iota} u_t \mathrm{D}_r u|_1+C_0 |r^{m}\phi^{2\iota}u_{tr}\mathrm{D}_r u|_1+C_0 |r^{m}\phi^{2\iota}u_{t}\mathrm{D}_r^2 u|_1\\
&\leq C_0 |r^\frac{m}{2}\phi^{\iota}\mathrm{D}_r u_t|_2 |r^\frac{m}{2}\phi^{\iota}\mathrm{D}_r u|_2+C_0|r^\frac{m}{2}u_t|_2 |r^\frac{m}{2}\phi^{2\iota}\mathrm{D}_r^2 u|_2<\infty.
\end{aligned}
\end{equation*}

Therefore, integrating \eqref{eq:B9pre} over $I$, together with \eqref{eq:B9}, yields
\begin{equation}\label{e-1.20}
\begin{aligned}
&\,a_1\delta a\frac{\mathrm{d}}{\mathrm{d}t}\Big|r^\frac{m}{2}\phi^\iota\big(u_r+\frac{m}{r}u\big)\Big|_2^2+| r^{\frac{m}{2}}u_t|_2^2\\
&= \int_0^\infty r^m ((v-2u)u_r- \phi_r )u_t\,\mathrm{d}r +\frac{1-\delta}{2} \int_0^\infty r^m (v-u)u\big(u_r+ \frac{m}{r}u\big)^2\,\mathrm{d}r\\[-2pt]
&\quad -a_1\delta(\delta-1) a \int_0^\infty r^m\phi^{2\iota}\big(u_r+ \frac{m}{r}u\big)^3\,\mathrm{d}r:=\sum_{i=6}^{8} G_i.
\end{aligned}
\end{equation}

Then, for $G_6$--$G_{7}$, it follows from \eqref{6.6-1}, Lemmas \ref{important2} and \ref{l4.4}, and the H\"older and Young inequalities that
\begin{equation}
\begin{aligned}
G_6& \leq  C_0\big(|\phi|_\infty^{-\iota}|(u,v)|_\infty|r^{\frac{m}{2}}\phi^\iota u_r|_{2}+|r^\frac{m}{2}\phi_r|_2\big)|r^{\frac{m}{2}}u_t|_{2}\\
&\leq  C(T)\big(1+|u|_\infty^2\big)|r^{\frac{m}{2}}\phi^\iota u_r|_{2}^2+C(T)+\frac{1}{8} |r^{\frac{m}{2}}u_t|_{2}^2,\\
G_7&\leq C_0|\phi|_\infty^{-2\iota}|(v,u)|_\infty|u|_\infty\Big|r^\frac{m}{2}\phi^\iota\big(u_r+\frac{m}{r}u\big)\Big|_2^2\leq C(T)\big(1+|u|_\infty^2\big)\Big|r^\frac{m}{2}\phi^\iota\big(u_r+\frac{m}{r}u\big)\Big|_2^2.    
\end{aligned}    
\end{equation}
To estimate $G_{8}$, we see from \eqref{6.4-d} that 
\begin{equation}\label{G9-G11}
\begin{aligned}
G_{8}&\leq C_0\Big|\phi^\iota\big(u_r+\frac{m}{r}u\big)\Big|_2\Big|r^\frac{m}{2}\phi^\iota\big(u_r+\frac{m}{r}u\big)\Big|_2 |r^\frac{m}{2}\mathrm{D}_r u|_\infty\\
&\leq C(T)\Big|\phi^\iota\big(u_r+\frac{m}{r}u\big)\Big|_2\Big|r^\frac{m}{2}\phi^\iota\big(u_r+\frac{m}{r}u\big)\Big|_2\big(|r^\frac{m}{2}u_t|_2+|u|_\infty+1\big)\\
&\leq C(T)\Big|\phi^\iota\big(u_r+\frac{m}{r}u\big)\Big|_2^2\Big|r^\frac{m}{2}\phi^\iota\big(u_r+\frac{m}{r}u\big)\Big|_2^2+C(T)\big(1+|u|_\infty^2\big)+\frac{1}{8}|r^\frac{m}{2}u_t|_2^2.
\end{aligned}
\end{equation}

Collecting \eqref{e-1.20}--\eqref{G9-G11}, we obtain 
from Lemmas \ref{im-2}--\ref{l4.5} that
\begin{equation}\label{e-1.23}
\begin{aligned}
&\,a_1\frac{\mathrm{d}}{\mathrm{d}t}\Big|\phi^{\iota}r^\frac{m}{2}\big(u_r+\frac{m}{r}u\big)\Big|_2^2+|r^{\frac{m}{2}}u_t|_2^2\\
&\leq C(T)+C(T)\Big(1+|u|_\infty^2+\Big|\phi^\iota\big(u_r+\frac{m}{r}u\big)\Big|_2^2\Big)\Big|r^\frac{m}{2}\phi^\iota\big(u_r+\frac{m}{r}u\big)\Big|_2^2\\
&\quad +C(T)\big(1+|u|_\infty^2\big)|r^\frac{m}{2}\phi^{\iota}u_r|_2^2\\
&\leq  C(T)\Big(1+|u|_\infty^2+\Big|\phi^\iota\big(u_r+\frac{m}{r}u\big)\Big|_2^2\Big)\Big|r^\frac{m}{2}\phi^\iota\big(u_r+\frac{m}{r}u\big)\Big|_2^2 + C(T)\big(1+|u|_\infty^4\big) ,
\end{aligned}
\end{equation}
which, along with Lemmas \ref{ele} and \ref{im-2} and the Gr\"onwall inequality, yields
\begin{equation}\label{eg}
|r^{\frac{m}{2}}\phi^\iota\mathrm{D}_r u(t)|_2^2+\int_0^t  |r^{\frac{m}{2}}u_t|_2^2\,\mathrm{d}s\leq C(T) |r^{\frac{m}{2}}\phi_0^\iota\mathrm{D}_r u_0|_2^2+C(T)\leq C(T).  
\end{equation}
Here, for the estimate of the initial data, it follows from the initial compatibility condition $\eqref{th78zxq}_1$ and Lemma \ref{lemma-initial} that
\begin{equation}
\begin{aligned}
|r^{\frac{m}{2}}\phi_0^\iota\mathrm{D}_r u_0|_2\leq C_0\|\phi_0^\iota\nabla\boldsymbol{u}_0\|_{L^2}\leq C_0\|\mathcal{G}_1\|_{L^2}\leq C_0.
\end{aligned}
\end{equation}

This completes the proof. 
\end{proof}

\subsection{Second-Order Estimates of the Velocity}
\begin{lem}\label{l4.8}
There exists a constant $C(T)>0$ such that
\begin{equation*}
|r^{\frac{m}{2}}u_t(t)|^2_2+\int_0^t |r^{\frac{m}{2}}\phi^\iota\mathrm{D}_r u_t|_2^2 \,\mathrm{d}s \leq C(T)
\qquad \text{for any $t\in[0,T]$}.
\end{equation*}
\end{lem}
 
\begin{proof} We divide the proof into three steps.

\smallskip 
\textbf{1.} Multiplying $\eqref{e2.2}_1$ by $r^\frac{m}{2}$ and taking the $L^2(I)$-norm of the resulting equation, we obtain from \eqref{618}, \eqref{6.6-1}, and Lemmas \ref{important2} and \ref{l4.6-0} that 
\begin{equation}\label{6.6-2}
\begin{aligned}
|r^{\frac{m}{2}}\phi_t|_2&\leq  |r^{\frac{m}{2}}u\phi_r|_2+(\gamma-1)\Big|r^\frac{m}{2}\phi(u_r+\frac{m}{r}u\big)\Big|_2\\
&\leq |u|_\infty|r^{\frac{m}{2}}\phi_r|_2+C_0|\phi|_\infty |r^{\frac{m}{2}}\mathrm{D}_r u|_2 \leq C(T)(|u|_\infty+1).
\end{aligned}    
\end{equation} 

\smallskip
\textbf{2.} Applying $\partial_t$ to $\eqref{e2.2}_2$, together with $\eqref{e2.2}_3$, \eqref{v-express-2}, and \eqref{phi2l}, gives
\begin{equation}\label{743utt}
\begin{aligned}
u_{tt}
&=-(u_tu_r+uu_{tr})-\phi_{tr}+2a_1\delta a \phi^{2\iota}\big(u_{tr}+\frac{m}{r}u_t\big)_{r}+(1-\delta)u(v-u)\big(u_r+\frac{m}{r}u\big)_r\\
&\quad -2a_1 \delta a \phi^{2\iota}\big(u_r+\frac{m}{r}u\big)_r\big((2\delta-1) u_r+2m(\delta-1)\frac{u}{r}\big)-2a_1  \psi_ru\big(\delta u_r+m(\delta-1)\frac{u}{r}\big) \\
&\quad - (v-u)\big(\delta u_r+ m(\delta-1)\frac{u}{r}\big)^2+(v-u)\big(\delta u_{tr}+ m(\delta-1)\frac{u_t}{r}\big).
\end{aligned}
\end{equation}
Then, multiplying the \eqref{743utt} by $r^mu_t$, we obtain from \eqref{v-express-2} that
\begin{equation}\label{eq:B10pre}
\begin{aligned}
&\,\frac{1}{2}(r^m u_t^2)_t+2a_1\delta a r^m\phi^{2\iota}\big(u_{tr}+ \frac{m}{r}u_t\big)^2\\
&=\Big(r^m \big(\underline{2a_1 \delta a \phi^{2\iota} u_t \big(u_{tr}+ \frac{m}{r}u_t\big)- u_t \phi_{t}-2a_1  \psi u\big(\delta u_r+m(\delta-1)\frac{u}{r}\big)u_t\big)}_{:=\mathcal{B}_{15}}\Big)_{r}\\
&\quad - r^m(u_tu_r+uu_{tr})u_t +r^m\phi_{t}\big(u_{tr}+\frac{m}{r}u_t\big)\\
&\quad -2a_1 \delta a r^m\phi^{2\iota}\big(u_r+\frac{m}{r}u\big)_r\big((2\delta-1) u_r +2m(\delta-1)\frac{u}{r}\big)u_t\\
&\quad +r^m(v-u)\big(\delta u_r+m(\delta-1)\frac{u}{r}\big)\Big(u\big(u_{tr}+\frac{m}{r}u_t\big)-(\delta-1)\big(u_r+\frac{m}{r}u\big) u_t\Big)\\
&\quad +r^m(v-u)(uu_{rr}+u_{tr})u_t.
\end{aligned}    
\end{equation}
Next, we need to  show that $r^m \mathrm{D}_r  \mathcal{B}_{15} \in L^1(I)$  for {\it a.e.} $t\in (0,T)$, which allows us to apply Lemma \ref{calculus} to obtain
\begin{equation}\label{eq:B10}
\int_0^\infty (r^m\mathcal{B}_{15})_r\,\mathrm{d}r =0.
\end{equation}
By \eqref{spd}--\eqref{spd2} and Lemmas \ref{ale1} and \ref{lemma-initial}, we obtain that, for {\it a.e.} $t\in (0,T)$,
\begin{equation*}
\begin{aligned}
&(\psi,u,\mathrm{D}_r u)\in L^\infty(I),\qquad r^\frac{m}{n^*}u\in L^{n^*}(I) ,\qquad r^\frac{m}{n}\psi_r\in L^{n}(I),\\
&r^\frac{m}{2}\big(\phi_t,\mathrm{D}_r u,u_t,\mathrm{D}_r^2 u,\mathrm{D}_r u_t,\phi^\iota \mathrm{D}_r u_t,\phi^{2\iota}\mathrm{D}_r^2 u_t\big)\in L^2(I).    
\end{aligned}
\end{equation*}
Then it follows from the H\"older inequality that
\begin{equation*} 
\begin{aligned}
|r^m \mathrm{D}_r \mathcal{B}_{15}|_1&\leq C_0\big|r^{m-1} u_t (\phi^{2\iota}\mathrm{D}_r u_t,\phi_t,\psi u\mathrm{D}_r u)\big|_1 +C_0\big|r^{m}u_{tr} (\phi^{2\iota}\mathrm{D}_r u_t,\phi_t,\psi u\mathrm{D}_r u)\big|_1\\
&\quad +C_0 \big|r^{m}u_{t} (\psi\mathrm{D}_r u_t,\phi^{2\iota}\mathrm{D}_r^2 u_t,\phi_{tr},\psi_r u \mathrm{D}_r u,\psi u_r \mathrm{D}_r u,\psi u \mathrm{D}_r^2 u)\big|_1\\
&\leq C_0 |r^\frac{m}{2}\phi^\iota\mathrm{D}_r u_t|_2^2+C_0 |r^\frac{m}{2} \mathrm{D}_r u_t|_2\big( |r^\frac{m}{2} \phi_t|_2+|\psi|_\infty |u|_\infty|r^\frac{m}{2}\mathrm{D}_r u|_2\big)\\
&\quad +C_0|r^\frac{m}{2}u_t|_2\big(|\psi|_\infty|r^\frac{m}{2}\mathrm{D}_r u_t|_2+|r^\frac{m}{2}(\phi^{2\iota}\mathrm{D}_r^2 u_t,\phi_{tr})|_2+|r^\frac{m}{n}\psi_r|_n|r^\frac{m}{n^*}u|_{n^*} |\mathrm{D}_r u|_\infty\big)\\
&\quad +C_0|r^\frac{m}{2}u_t|_2|\psi|_\infty\big(|u_r|_\infty |r^\frac{m}{2}\mathrm{D}_r u|_2+ |u|_\infty |r^\frac{m}{2}\mathrm{D}_r^2 u|_2\big)<\infty.
\end{aligned}    
\end{equation*}

Now, integrating \eqref{eq:B10pre} over $I$, together with \eqref{eq:B10}, yields
\begin{equation}\label{util2}
\begin{aligned}
&\,\frac{1}{2} \frac{\mathrm{d}}{\mathrm{d}t}|r^{\frac{m}{2}}u_t|^2_2+2a_1\delta a \Big|r^{\frac{m}{2}}\phi^\iota\big(u_{tr}+\frac{m}{r}u_t\big)\Big|_2^2\\
&=-\int_0^\infty r^m(u_tu_{r}+uu_{tr}) u_t\,\mathrm{d}r+\int_0^\infty r^m \phi_t\big(u_{tr}+\frac{m}{r}u_t\big)\,\mathrm{d}r\\
&\quad -2a_1 \delta a\int_0^\infty  r^m\phi^{2\iota}\big(u_r+\frac{m}{r}u\big)_r\big((2\delta-1) u_r +2m(\delta-1)\frac{u}{r}\big)u_t\,\mathrm{d}r\\
&\quad +\int_0^\infty r^m(v-u)\big(\delta u_r+m(\delta-1)\frac{u}{r}\big)\Big(u\big(u_{tr}+\frac{m}{r}u_t\big)-(\delta-1)\big(u_r+\frac{m}{r}u\big) u_t\Big)\,\mathrm{d}r\\
&\quad + \int_0^\infty r^m(v-u)(uu_{rr}+u_{tr})u_t\,\mathrm{d}r:=\sum_{i=9}^{13} G_i.
\end{aligned}    
\end{equation}

\smallskip
\textbf{3.} For $G_{9}$--$G_{13}$, it follows from \eqref{618}--\eqref{6.4-d}, \eqref{6.6-2}, Lemmas \ref{important2}, \ref{ele}, \ref{im-2}, \ref{l4.6}, and the H\"older and Young inequalities that
\begin{align*}
&\begin{aligned}
G_{9}&\leq |\phi|_\infty^{-\iota}|u|_{\infty}|r^{\frac{m}{2}}\phi^\iota u_{tr}|_2|r^{\frac{m}{2}}u_t|_2+ \big(|\chi_1^\flat r^\frac{m}{2} u_ru_t|_{2} +  |\chi_1^\sharp r^\frac{m}{2} u_ru_t|_{2}\big) |r^{\frac{m}{2}}u_t|_2\\
&\leq C(T)|u|_{\infty}|\phi^\iota r^{\frac{m}{2}}u_{tr}|_2|r^{\frac{m}{2}}u_t|_2+|\phi|_\infty^{-\iota} |\chi_1^\flat r^\frac{2-m}{2}|_\infty|r^\frac{m}{2} u_r|_{\infty} |r^{\frac{m-2}{2}}\phi^\iota u_t|_2 |r^{\frac{m}{2}}u_t|_2\\
&\quad +  |\chi_1^\sharp r^{-\frac{m}{2}}|_\infty |r^\frac{m}{2}u_r|_{\infty}|r^{\frac{m}{2}}u_t|_2^2\\
&\leq C(T)\big(|(u,r^\frac{m}{2}u_r)|_\infty^2+1\big)|r^{\frac{m}{2}}u_t|_2^2 +\frac{a_1\delta a}{80}\Big|r^{\frac{m}{2}}\phi^\iota\big(u_{tr}+\frac{m}{r}u_t\big)\Big|_2^2\\
&\leq C(T)\big(|r^{\frac{m}{2}}u_t|_2^2+|u|_\infty^2+1\big)|r^{\frac{m}{2}}u_t|_2^2 +\frac{a_1\delta a}{80}\Big|r^{\frac{m}{2}}\phi^\iota\big(u_{tr}+\frac{m}{r}u_t\big)\Big|_2^2,
\end{aligned}\\
&\begin{aligned}
G_{10}&\leq |\phi|_\infty^{-\iota}|r^{\frac{m}{2}}\phi_t|_2 \Big|r^{\frac{m}{2}}\phi^\iota\big(u_{tr}+\frac{m}{r}u_t\big)\Big|_2\leq C(T)(|u|_\infty+1)\Big|r^{\frac{m}{2}}\phi^\iota\big(u_{tr}+\frac{m}{r}u_t\big)\Big|_2\\
&\leq C(T)\big(|u|_\infty^2+1\big) +\frac{a_1\delta a}{80}\Big|r^{\frac{m}{2}}\phi^\iota\big(u_{tr}+\frac{m}{r}u_t\big)\Big|_2^2,
\end{aligned}\\
&\begin{aligned}
G_{11}&\leq C_0 |\phi^\iota\mathrm{D}_r u|_2\Big|r^\frac{m}{2}\phi^\iota\big(u_r+\frac{m}{r}u\big)_r\Big|_2|r^\frac{m}{2}u_t|_2\\
&\leq C(T)\big(|r^\frac{m}{2}u_t|_2^2+|u|_\infty^2+1\big)|r^\frac{m}{2}u_t|_2^2+C(T) |\phi^\iota \mathrm{D}_r u|_2^2,
\end{aligned}\\
&\begin{aligned}
G_{12}&\leq |(v,u)|_\infty |\phi|_\infty^{-2\iota}|u|_\infty\Big|r^\frac{m}{2}\phi^\iota\mathrm{D}_r u\Big|_2\Big|r^\frac{m}{2}\phi^\iota\big(u_{tr}+\frac{m}{r}u_t\big)\Big|_2\\
&\quad +|(v,u)|_\infty|\phi|_\infty^{-\iota} |\phi^\iota\mathrm{D}_r u|_2 |r^\frac{m}{2}\mathrm{D}_r u|_\infty|r^\frac{m}{2}u_t|_2\\
&\leq C(T)\big(1+|u|_\infty^2\big)\Big|r^\frac{m}{2}\phi^\iota\big(u_{tr}+\frac{m}{r}u_t\big)\Big|_2\\
&\quad +C(T)(1+|u|_\infty) |\phi^\iota\mathrm{D}_r u|_2 |r^\frac{m}{2}u_t|_2\big(1+|u|_\infty+|r^\frac{m}{2}u_t|_2\big)\\
&\leq C(T) \big(|u|_\infty^4+ |\phi^\iota\mathrm{D}_r u|_2^2+1\big)\big(|r^\frac{m}{2}u_t|_2^2+1\big)  +\frac{a_1\delta a}{80}\Big|r^\frac{m}{2}\phi^\iota\big(u_{tr}+\frac{m}{r}u_t\big)\Big|_2^2 ,
\end{aligned}\\
&\begin{aligned}
G_{13}&\leq C_0|(v,u)|_\infty\big(|u|_\infty|r^\frac{m}{2}u_{rr}|_2+|\phi|_\infty^{-\iota}|r^\frac{m}{2}\phi^\iota u_{tr}|_2\big)|r^\frac{m}{2}u_t|_2\\
&\leq C(T)\big((|r^\frac{m}{2}u_t|_2^2+|u|_\infty^2+1)|r^\frac{m}{2}u_t|_2^2+|u|_\infty^4+1\big)+\frac{a_1\delta a}{80}\Big|r^\frac{m}{2}\phi^\iota\big(u_{tr}+\frac{m}{r}u_t\big)\Big|_2^2.
\end{aligned}
\end{align*}

Substituting the above into \eqref{util2}, along with Lemma \ref{im-2}, gives
\begin{equation*} 
\frac{\mathrm{d}}{\mathrm{d}t}|r^{\frac{m}{2}}u_t|^2_2+a_1\delta a |r^{\frac{m}{2}}\phi^\iota\mathrm{D}_r u_t|_2^2 \leq C(T)\big(|r^{\frac{m}{2}}u_{t}|_2^2+|u|_\infty^4+ |\phi^\iota\mathrm{D}_r u|_2^2+1\big)\big(|r^{\frac{m}{2}}u_{t}|_2^2+1\big).
\end{equation*}
Integrating over $(\tau, t)(\tau\in (0,t))$ yields
\begin{equation}\label{lg1}
\begin{aligned}
&\,|r^{\frac{m}{2}}u_t(t)|_2^2+a_1\delta a\int_\tau^t |r^{\frac{m}{2}}\phi^\iota\mathrm{D}_r u_t|_2^2\,\mathrm{d}s\\
&\leq  |r^{\frac{m}{2}}u_t(\tau)|_2^2+C(T)\int_\tau^t\big(|r^{\frac{m}{2}}u_{t}|_2^2+|u|_\infty^4+ |\phi^\iota\mathrm{D}_r u|_2^2+1\big)\big(|r^\frac{m}{2}u_t|_2^2+1\big)\,\mathrm{d}s.
\end{aligned}
\end{equation}
For the $L^2(I)$-boundedness of $r^{\frac{m}{2}}u_t(\tau)$ 
on the right-hand side of \eqref{lg1}, we multiply $\eqref{e2.2}_2$ by $r^\frac{m}{2}$ and then take the $L^2(I)$-norm of the resulting equality to obtain 
\begin{equation*}
|r^{\frac{m}{2}}u_t(\tau)|_2\leq C_0\big(|u|_\infty|r^{\frac{m}{2}}u_r|_2+|r^{\frac{m}{2}}\phi_r|_2+ \Big|r^{\frac{m}{2}}\phi^{2\iota}\big(u_r+\frac{m}{r}u\big)_r\Big|_2+|\psi|_\infty |r^{\frac{m}{2}}\mathrm{D}_r u|_2\big)(\tau).
\end{equation*}
This, along with the time-continuity of $(\phi,u,\psi)$, the initial compatibility condition $\eqref{th78zxq}_3$, and 
Lemmas \ref{ale1}, \ref{initial3}, and \ref{lemma-initial}, yields
\begin{equation*}
\begin{aligned}
\limsup_{\tau\to 0} |r^{\frac{m}{2}}u_t(\tau)|_2&\leq C_0 \big(\|\boldsymbol{u}_0\|_{L^\infty}\|\nabla\boldsymbol{u}_0\|_{L^2}+ \|\nabla\phi_0\|_{L^2}+\|\phi_0^{2\iota}L\boldsymbol{u}_0\|_{L^2} + \|\boldsymbol{\psi}_0\|_{L^\infty} \|\nabla\boldsymbol{u}_0\|_{L^2}\big)\\[1pt]
&\leq C_0\big(\|\boldsymbol{u}_0\|_{H^2}+\|\boldsymbol{\psi}_0\|_{L^\infty}\big)\|\boldsymbol{u}_0\|_{H^2} +C_0\|\boldsymbol{g}_*\|_{L^2}+C_0 \|\nabla\phi_0\|_{L^2} \leq C_0.
\end{aligned}
\end{equation*}
where we have also used the following fact, since $\boldsymbol{u}_0$ is spherically symmetric:
\begin{equation}\label{Lu=}
L\boldsymbol{u}_0=-a_1 \Delta\boldsymbol{u}_0-(2\delta-1)a_1\nabla\diver \boldsymbol{u}_0=-2\delta a_1\nabla\diver \boldsymbol{u}_0.
\end{equation}

Hence, based on the above discussions, letting $\tau\to 0$ in \eqref{lg1}, combined with Lemmas \ref{ele} and \ref{l4.6} and the Gr\"onwall inequality, leads to the desired estimate.
\end{proof}

With the help of Lemma \ref{l4.8}, we can also obtain the following estimates:

\begin{lem}\label{lemma66}
There exists a constant $C(T)>0$ such that, for any $t\in [0,T]$ and $q\in [2\iota,0]$,
\begin{equation*}
|(u,r^\frac{m}{2}u_r)(t)|_\infty+\delta_{3n}|r^\frac{m}{6}(\phi,u,\mathrm{D}_r u)(t)|_{6} + |r^{\frac{m}{2}} (\phi_r,\phi_t,\phi^{q}\mathrm{D}_r^2 u)(t)|_2+ \int_0^t |\mathrm{D}_r u|_\infty^2\,\mathrm{d}s \leq C(T).
\end{equation*}
\end{lem}
\begin{proof}

First, it follows from \eqref{618}, Lemmas \ref{important2}, \ref{im-1}--\ref{l4.5}, \ref{l4.8}, \ref{GN-ineq}, and \ref{lemma-initial} that
\begin{equation*} 
\begin{aligned}
|u|_{\infty}&=\|\boldsymbol{u}\|_{L^\infty}\leq C_0\|\boldsymbol{u}\|_{L^2}^\frac{4-n}{4}\|\nabla^2\boldsymbol{u}\|_{L^2}^\frac{n}{4}\leq C_0|r^\frac{m}{2}u|_2^\frac{4-n}{4} \Big|r^\frac{m}{2} \big(u_r+\frac{m}{r}u\big)_r\Big|_2^\frac{n}{4}\\
&\leq C(T)\big(|r^\frac{m}{2}u_t|_2+|u|_\infty+1\big)^\frac{n}{4}\leq C(T)\big(|u|_\infty^\frac{n}{4}+1\big),
\end{aligned}
\end{equation*}
which, along with the Young inequality, leads to
\begin{equation*} 
|u(t)|_{\infty} \leq C(T).
\end{equation*}
Clearly, the above, together with \eqref{618}--\eqref{6.6-1} and \eqref{6.6-2} and Lemma \ref{l4.8}, also gives that, for all $t\in [0,T]$ and $q\in [2\iota,0]$,
\begin{equation}\label{6026}
|r^\frac{m}{2}u_r(t)|_\infty + |r^{\frac{m}{2}} (\phi_r,\phi_t)(t)|_2 +|r^{\frac{m}{2}}\phi^{q}\mathrm{D}_r^2 u(t)|_2 \leq C(T).
\end{equation}

Next, multiplying $\eqref{e2.2}_2$ by $r^\frac{m}{4}$ and then taking the $L^4(I)$-norm of the resulting equality, we obtain from \eqref{v-express-2}, \eqref{6026}, and Lemmas \ref{important2}, \ref{im-1}, \ref{l4.5}, \ref{l4.8}, \ref{GN-ineq}, and \ref{lemma-initial} that
\begin{equation*}
\begin{aligned}
\|\nabla^2 \boldsymbol{u}\|_{L^4}&\leq  C_0 \Big|r^\frac{m}{4} \big(u_r+\frac{m}{r}u\big)_r\Big|_4\leq  C_0|\phi|_\infty^{-2\iota}\Big|r^\frac{m}{4}\phi^{2\iota}\big(u_r+\frac{m}{r}u\big)_r\Big|_4\\
&\leq C(T)\big(|r^\frac{m}{4} u_t|_4+|(u,v)|_\infty |r^\frac{m}{4}\mathrm{D}_r u|_4+|r^\frac{m}{4}\phi^{1-2\iota}(v-u)|_4\big)\\
&\leq C(T)\|(\boldsymbol{u}_t,\nabla\boldsymbol{u})\|_{L^4}+C_0|\phi|_\infty^\frac{1-2\iota}{2}|(u,v)|_\infty^\frac{1}{2}\big(|r^\frac{m}{2}\phi^{1-2\iota} v|_2^\frac{1}{2}+|\phi|_\infty^\frac{1-2\iota}{2} |r^\frac{m}{2}u|_2^\frac{1}{2}\big)\\
&\leq C(T)\big(\|(\boldsymbol{u}_t,\nabla\boldsymbol{u})\|_{L^2}^\frac{4-n}{4}\|(\nabla\boldsymbol{u}_t,\nabla^2\boldsymbol{u})\|_{L^2}^\frac{n}{4} +1\big)\\
&\leq C(T)\big(|r^\frac{m}{2}(u_t,\mathrm{D}_r u)|_2^\frac{4-n}{4} |r^\frac{m}{2}(\mathrm{D}_r u_t,\mathrm{D}_r^2 u)|_2^\frac{n}{4} +1\big) \leq C(T)\big(|r^\frac{m}{2}\mathrm{D}_r u_t|_2^\frac{n}{4}+1\big),
\end{aligned}
\end{equation*}
which, together with Lemmas \ref{l4.6-0}, \ref{GN-ineq}, and \ref{lemma-initial}, and the Young inequality, yields 
\begin{equation}\label{ur-infty}
\begin{aligned}
|\mathrm{D}_r u|_\infty&\leq C_0\|\nabla \boldsymbol{u}\|_{L^\infty}\leq C_0\|\nabla \boldsymbol{u}\|_{L^2}^\frac{4-n}{n+4}\|\nabla^2\boldsymbol{u}\|_{L^4}^\frac{2n}{n+4}\\
&\leq C(T) |r^\frac{m}{2}\mathrm{D}_r u|_{2}^\frac{4-n}{n+4}\big( |r^\frac{m}{2}\mathrm{D}_r u_t|_2^\frac{n^2}{2(n+4)}+1\big) \leq C(T) \big(|r^\frac{m}{2}\mathrm{D}_r u_t|_2+1\big).
\end{aligned}   
\end{equation}
This, along with Lemma \ref{l4.8}, yields the $L^2([0,T];L^\infty)$ estimate for $\mathrm{D}_r u$.

Finally, when $n=3$, by \eqref{6026} and Lemmas \ref{ale1} and \ref{lemma-initial}, we have
\begin{equation}
\begin{aligned}
|r^\frac{m}{6}(\phi,u,\mathrm{D}_r u) |_{6} &\leq C_0\|(\phi,\boldsymbol{u},\nabla \boldsymbol{u})\|_{L^6} \leq C_0\|(\nabla \phi,\nabla \boldsymbol{u},\nabla^2 \boldsymbol{u})\|_{L^{2}}\\
&\leq C_0|r^\frac{m}{2}(\phi_r,\mathrm{D}_ru,\mathrm{D}_r^2 u) |_{2}\le C(T).
\end{aligned} 
\end{equation}

This completes the proof.
\end{proof}

We now show the first-order estimates for $(\psi,v)$.

\begin{lem}\label{l4.7}
There exists a constant $C(T)>0$ such that
\begin{equation*}\label{e4.34}
|\psi(t)|_\infty+ |r^{\frac{m}{n}}(\mathrm{D}_r v,\mathrm{D}_r \psi)(t)|_{n}+ |r^{\frac{m}{2}} \psi_t(t)|_{2}\leq C(T)\qquad \text{for any $t\in [0,T]$}.
\end{equation*}
\end{lem}
\begin{proof}
First, it follows from \eqref{v-express-2} and Lemmas \ref{l4.4} and \ref{lemma66} that
\begin{equation*}
|\psi(t)|_\infty\leq C_0|(v,u)(t)|_\infty\leq C(T)\qquad\text{for all $t\in [0,T]$}.    
\end{equation*}

Next, define the effective velocity $\boldsymbol{v}$ in coordinates $(t,\boldsymbol{x})$ as
\begin{equation}\label{MDv}
\boldsymbol{v}=\boldsymbol{u}+\frac{2a_1\delta a}{\delta-1}\nabla \phi^{2\iota}.    
\end{equation}
Then, from $\eqref{eq:cccq}_2$ and \eqref{kuzxc}, we can obtain its equation:
\begin{equation}\label{845}
\boldsymbol{v}_t+\boldsymbol{u}\cdot\nabla\boldsymbol{v}+\frac{\gamma-1}{2a_1\delta a}\phi^{1-2\iota}(\boldsymbol{v}-\boldsymbol{u})=\boldsymbol{0}.
\end{equation}
Applying $\mathrm{div}$ to \eqref{845} leads to the following equation in the sense of distributions:
\begin{equation}\label{nabla-v}
\begin{aligned}
(\diver\boldsymbol{v})_t+\diver\big(\boldsymbol{u} (\diver\boldsymbol{v})\big)&=(\diver\boldsymbol{u})(\diver\boldsymbol{v})-\nabla\boldsymbol{u}^\top:\nabla\boldsymbol{v}\\
&\quad -\frac{\gamma-\delta}{2a_1\delta a} \phi^{-2\iota}\nabla\phi\cdot(\boldsymbol{v}-\boldsymbol{u})-\frac{\gamma-1}{2a_1\delta a}\phi^{1-2\iota}(\diver\boldsymbol{v}-\diver\boldsymbol{u}).
\end{aligned}
\end{equation}
Using \eqref{MDv} and Lemmas \ref{th1} and \ref{ale1}, we obtain
\begin{equation}
\diver\boldsymbol{v}\in L^\infty([0,T];L^n(\mathbb{R}^n)),\qquad \boldsymbol{u}\in L^1([0,T];W^{1,\infty}(\mathbb{R}^n)),    
\end{equation}
and the right-hand side of \eqref{nabla-v} belongs to $L^1([0,T];L^n(\mathbb{R}^n))$. 
Thus, it follows from Lemma \ref{lemma-lions} that, for {\it a.e.} $t\in (0,T)$,
\begin{equation}
\begin{aligned}
&\frac{\mathrm{d}}{\mathrm{d}t}\|\diver\boldsymbol{v}\|_{L^n}^{n}+\frac{\gamma-1}{a_1\delta a}\big\|\phi^\frac{1-2\iota}{n}\diver\boldsymbol{v}\big\|_{L^n}^{n}\notag\\
&=\int_{\mathbb{R}^n}  (\diver\boldsymbol{u}) |\diver\boldsymbol{v}|^n\,\mathrm{d}\boldsymbol{x}-n\int_{\mathbb{R}^n} (\nabla\boldsymbol{u}^\top:\nabla\boldsymbol{v})|\diver\boldsymbol{v}|^{n-2}\diver\boldsymbol{v}\,\mathrm{d}\boldsymbol{x} \\
&\quad +\int_{\mathbb{R}^n} \Big(-\frac{n(\gamma-\delta)}{2a_1\delta a}\phi^{-2\iota}\nabla\phi\cdot(\boldsymbol{v}-\boldsymbol{u})+\frac{n(\gamma-1)}{2a_1\delta a}\phi^{1-2\iota}(\diver\boldsymbol{u})\Big)|\diver\boldsymbol{v}|^{n-2}\diver\boldsymbol{v}\,\mathrm{d}\boldsymbol{x}.\notag
\end{aligned}
\end{equation}
This, along with the spherical coordinate transformation, yields that, 
for {\it a.e.} $t\in (0,T)$,
\begin{equation}\label{derive-vr}
\begin{aligned}
&\frac{\mathrm{d}}{\mathrm{d}t}\Big|r^{\frac{m}{n}}\big(v_r+\frac{m}{r}v\big)\Big|_{n}^{n}+\frac{\gamma-1}{a_1\delta a}\Big|r^\frac{m}{n}\phi^\frac{1-2\iota}{n}\big(v_r+\frac{m}{r}v\big)\Big|_{n}^{n}\\
&=\int_0^\infty r^m \big(u_r+\frac{m}{r}u\big)\Big|v_r+\frac{m}{r}v\Big|^n\,\mathrm{d}r\\
&\quad -n\int_0^\infty r^m \Big(u_rv_r+\frac{m}{r^2}uv\Big)\Big|v_r+\frac{m}{r}v\Big|^{n-2}\big(v_r+\frac{m}{r}v\big)\,\mathrm{d}r\\
&\quad +  \int_0^\infty r^m\Big(-\frac{n(\gamma-\delta)}{2a_1\delta a}\phi^{-2\iota}\phi_r(v-u)+\frac{n(\gamma-1)}{2a_1\delta a}\phi^{1-2\iota}\big(u_r+\frac{m}{r}u\big)\Big)\\
&\qquad\qquad \times \Big|v_r+\frac{m}{r}v\Big|^{n-2}\big(v_r+\frac{m}{r}v\big)\,\mathrm{d}r.
\end{aligned}    
\end{equation}

We now continue to estimate the right-hand side of \eqref{derive-vr}. To achieve this, we first present some higher-order estimates for $\phi$. It follows from \eqref{tr}, \eqref{v-express-2},  Lemmas \ref{important2} and \ref{lemma66}--\ref{l4.7} that
\begin{equation}\label{dier}
\begin{aligned}
|r^\frac{m}{2}\mathrm{D}_r \phi_r|_2&\leq C_0 |r^\frac{m}{2}\mathrm{D}_r (\phi^{1-2\iota}\psi)|_2\leq C_0|\phi|_\infty^{-2\iota} |r^\frac{m}{2} (\phi_r\psi,\phi\mathrm{D}_r \psi)|_2\\
&\leq C_0|\phi|_\infty^{-2\iota} |\psi|_\infty|r^\frac{m}{2}\phi_r|_2+  C_0|\phi|_\infty^{-2\iota}|r^\frac{m}{n^*}\phi|_{n^*} |r^\frac{m}{n} \mathrm{D}_rv|_n\\
&\quad + C_0|\phi|_\infty^{-2\iota}|r^\frac{m}{n^*}\phi|_{n^*}|r^\frac{m}{n^*} \mathrm{D}_ru|_{n^*}^\frac{2-n}{2}|r^\frac{m}{2} \mathrm{D}_ru|_2^\frac{4-n}{2} \leq C(T)\big(|r^\frac{m}{n}\mathrm{D}_r v|_n+1\big), 
\end{aligned}
\end{equation}
which, along with Lemmas \ref{ale1}, \ref{lemma-initial}, and \ref{Hk-Ck-vector}, leads to
\begin{equation}\label{dier2}
|r^\frac{m}{n^*} \phi_r|_{n^*}\leq C_0 \|\nabla\phi\|_{L^{n^*}} \leq C_0 \|\nabla^2\phi\|_{L^{2}}\leq C_0|r^\frac{m}{2}\mathrm{D}_r \phi_r|_2 \leq C(T)\big(|r^\frac{m}{n}\mathrm{D}_r v|_n+1\big). 
\end{equation}

Hence, by \eqref{dier}--\eqref{dier2}, Lemmas \ref{important2}, \ref{l4.6-0}, and \ref{lemma66}, and the H\"older and Young inequalities, we have
\begin{equation*} 
\begin{aligned}
&\frac{\mathrm{d}}{\mathrm{d}t}\Big|r^{\frac{m}{n}}\big(v_r+\frac{m}{r}v\big)\Big|_{n}^{n}+\frac{\gamma-1}{a_1\delta a}\Big|(r^m\phi^{1-2\iota})^\frac{1}{n}\big(v_r+\frac{m}{r}v\big)\Big|_{n}^{n}\\
&\leq C_0 |\mathrm{D}_r u|_\infty |r^\frac{m}{n}\mathrm{D}_rv|_{n}^n +C_0|\phi|_\infty^{-2\iota} |r^\frac{m}{n^*}\phi_r|_{n^*}^\frac{n-2}{2} |r^\frac{m}{2}\phi_r|_{2}^\frac{4-n}{2}|(v,u)|_\infty |r^\frac{m}{2}\mathrm{D}_r v|_{n}^{n-1}\\
&\quad +C_0|\phi|_\infty^{1-2\iota} |r^\frac{m}{n^*}\mathrm{D}_r u|_{n^*}^\frac{n-2}{2} |r^\frac{m}{2}\mathrm{D}_r u|_{2}^\frac{4-n}{2} |r^\frac{m}{2}\mathrm{D}_r v|_{n}^{n-1}\\
&\leq C(T) \big(|\mathrm{D}_r u|_\infty^2+1\big)|r^\frac{m}{n}\mathrm{D}_r v|_{n}^{n}+C(T),
\end{aligned}    
\end{equation*}
which, along with Lemmas \ref{im-1}, \ref{lemma66}, \ref{ale1}, and \ref{lemma-initial}, and the Gr\"onwall inequality, gives
\begin{equation*}
\begin{aligned}
|r^\frac{m}{n}\mathrm{D}_r v(t)|_{n}&\leq C(T)\big(|r^\frac{m}{n}\mathrm{D}_r v_0|_{n}+1\big) \leq C(T)\big(\big|r^\frac{m}{n} (\mathrm{D}_r u_0,\mathrm{D}_r(\rho_0^{\delta-1})_r)\big|_{n}+1\big)\\
&\leq C(T)\big(\|(\nabla \boldsymbol{u}_0,\nabla^2 \rho_0^{\delta-1})\|_{L^{n}}+1\big)\leq C(T).
\end{aligned}  
\end{equation*}
This, together with \eqref{v-express-2} and Lemma \ref{l4.6-0}, yields 
\begin{equation*}
|r^\frac{m}{n}\mathrm{D}_r \psi(t)|_{n}\leq C_0 |r^\frac{m}{n}(\mathrm{D}_r v,\mathrm{D}_r u)(t)|_{n}\leq C(T)\qquad\text{for all $t\in [0,T]$}.
\end{equation*}

Finally, multiplying $\eqref{e2.2}_3$ by $r^\frac{m}{2}$ and then taking the $L^2(I)$-norm of the resulting equality, we obtain from the above and Lemmas \ref{l4.6-0} and \ref{lemma66} that
\begin{equation*} 
\begin{aligned}
|r^{\frac{m}{2}}\psi_t|_2&\leq  |r^{\frac{m}{2}}u\psi_r|_2+C_0 |r^\frac{m}{2} \psi\mathrm{D}_r u|_2+C_0 |r^{\frac{m}{2}}\phi^{2\iota}\mathrm{D}_r^2 u|_2 \\
&\leq  |r^{\frac{m}{n^*}}u|_{n^*}|r^{\frac{m}{n}}\psi_r|_n+C_0 |\psi|_\infty|r^{\frac{m}{2}}\mathrm{D}_r u|_2+C(T) \leq C(T).
\end{aligned}
\end{equation*}

This completes the proof. 
\end{proof}

The following lemma provides the higher-order estimates of $(\phi,u)$.
\begin{lem}\label{l4.9}  
There exists a constant $C(T)>0$ such that, for any $q\in [2\iota,0]$,
\begin{equation*}
|r^{\frac{m}{n^*}} \phi_{r}(t)|_{n^*}+|r^{\frac{m}{2}} (\mathrm{D}_r \phi_{r},\phi_{tr})(t)|_2+\int_{0}^{t} |r^{\frac{m}{2}}\phi^{q}\mathrm{D}_r^3 u|_2^2\,\mathrm{d}s\leq C(T)\qquad\text{for any $t\in [0,T]$},
\end{equation*}
where $n^*$ is defined as in {\rm\S\ref{othernote}}.
\end{lem}
\begin{proof} We divide the proof into two steps.

\smallskip
\textbf{1. Estimates on $\phi$.} First, by \eqref{dier}--\eqref{dier2} and Lemma \ref{l4.7}, we have
\begin{equation*}
|r^{\frac{m}{n^*}} \phi_{r}(t)|_{n^*}+|r^{\frac{m}{2}} \mathrm{D}_r \phi_{r} (t)|_2\leq C(T).
\end{equation*}
Next, for the $L^2(I)$-estimate of $r^\frac{m}{2}\phi_{tr}$, it follows from the above, \eqref{tr}, and Lemmas \ref{important2} and \ref{lemma66}--\ref{l4.7} that
\begin{equation*} 
\begin{aligned}
|r^{\frac{m}{2}}\phi_{tr}|_2&\leq C_0|r^\frac{m}{2}(\phi^{1-2\iota}\psi)_t|_2\leq C_0|\phi|_\infty^{-2\iota}|r^\frac{m}{2}(\phi_t\psi,\phi\psi_t)|_2 \\
&\leq C_0|\phi|_\infty^{-2\iota}(|\psi|_\infty |r^\frac{m}{2}\phi_t|_2+ |\phi|_\infty |r^\frac{m}{2}\psi_t|_2)\leq C(T).
\end{aligned}
\end{equation*}

\smallskip
\textbf{2. Estimates on $u$.} Applying $r^\frac{m}{2}\mathrm{D}_r$ on both sides of $\eqref{e2.2}_2$ and taking $L^2(I)$-norm of the resulting equality,
we obtain from Lemmas \ref{important2}, \ref{l4.6-0}--\ref{l4.6}, 
and \ref{lemma66}--\ref{l4.7} that
\begin{equation*} 
\begin{aligned}
\Big|r^{\frac{m}{2}}\mathrm{D}_r\big(\phi^{2\iota}\big(u_r+ \frac{m}{r}u\big)_{r}\big)\Big|_2&\leq C_0\big(\big|r^{\frac{m}{2}}\big(\mathrm{D}_ru_{t},|\mathrm{D}_ru|^2,u\mathrm{D}_r^2u\big)\big|_2+ |r^{\frac{m}{2}}\phi^{-2\iota}\phi_r\psi|_2\big)\\
&\quad +C_0\big(|r^{\frac{m}{2}}\phi^{1-2\iota}\mathrm{D}_r\psi|_2+\big|r^\frac{m}{2}\big(|\mathrm{D}_r\psi||\mathrm{D}_ru|,\psi \mathrm{D}_r^2u \big)\big|_2\big)\\
&\leq C_0\big(|\phi|_\infty^{-\iota}|r^{\frac{m}{2}}\phi^\iota \mathrm{D}_r u_t|_{2}+|u_r|_\infty|r^{\frac{m}{2}}\mathrm{D}_ru|_2+ |(u,\psi)|_\infty|r^{\frac{m}{2}}\mathrm{D}_r^2 u|_2\big)\\
&\quad +C_0 |\phi|_\infty^{-2\iota}|\psi|_\infty|r^{\frac{m}{2}}\phi_r|_2
+C_0|\phi|_\infty^{-2\iota}|r^\frac{m}{n^*}\phi|_{n^*}|r^{\frac{m}{n}}\mathrm{D}_r\psi|_n\\
&\quad +C_0|r^\frac{m}{n^*}\mathrm{D}_r u|_{n^*} |r^{\frac{m}{n}}\mathrm{D}_r\psi|_n \leq C(T)\big(|r^{\frac{m}{2}}\phi^\iota \mathrm{D}_ru_t|_{2}+|\mathrm{D}_r u|_\infty+1\big),
\end{aligned}    
\end{equation*}
which, along with \eqref{tr} and Lemmas \ref{lemma66}--\ref{l4.7}, implies
\begin{equation}\label{uxxxl1}
\begin{aligned}
\Big|r^{\frac{m}{2}} \phi^{2\iota}\mathrm{D}_r\big(u_r+ \frac{m}{r}u\big)_{r}\Big|_2&\leq C_0\Big|r^{\frac{m}{2}}\mathrm{D}_r\Big(\phi^{2\iota}\big(u_r+ \frac{m}{r}u\big)_{r}\Big)\Big|_2+C_0|\psi|_\infty  |r^{\frac{m}{2}}\mathrm{D}_r^2 u|_2\\
&\leq C(T)\big(|\phi^\iota r^{\frac{m}{2}}\mathrm{D}_r u_t|_{2}+|\mathrm{D}_r u|_\infty+1\big).
\end{aligned}
\end{equation}

Thus, we obtain from \eqref{uxxxl1}  and Lemmas \ref{im-2} and \ref{l4.8}--\ref{lemma66} that
\begin{equation}
\begin{aligned}
\int_0^t|r^\frac{m}{2}\phi^{2\iota}\mathrm{D}_r^3 u|_2^2\,\mathrm{d}s \leq C(T)\int_0^t\big(|\phi^\iota r^{\frac{m}{2}}\mathrm{D}_r u_t|_{2}^2+|\mathrm{D}_r u|_\infty^2+1\big)\,\mathrm{d}s \leq C(T).
\end{aligned}
\end{equation}

Finally, using Lemma \ref{important2} again, we arrive at the desired estimates. 
\end{proof}

\subsection{Third-Order Energy Estimates of the Velocity}
\begin{lem}\label{l4.10}  
There exists a constant $C(T)>0$ such that
\begin{equation*}
\delta_{3n}|r^\frac{m}{6}u_t(t)|_6+|r^{\frac{m}{2}}\phi^\iota\mathrm{D}_r u_t(t)|_2+\int_{0}^{t}  |r^{\frac{m}{2}}u_{tt}|^2_2 \,\mathrm{d}s\leq C(T)
\qquad\text{for any $t\in [0,T]$}.
\end{equation*}
\end{lem}

\begin{proof}
We divide the proof into two steps.

\smallskip 
\textbf{1.} Applying $r^mu_{tt}$ to $\eqref{e2.2}_2$, we obtain 
\begin{equation}\label{dtg15}
a_1\delta a \Big(r^m\phi^{2\iota}\big(u_{tr}+\frac{m}{r}u_t\big)^2\Big)_t+r^mu_{tt}^2=2a_1\delta a\Big(r^m\underline{\phi^{2\iota}\big(u_{tr}+\frac{m}{r}u_t\big)u_{tt}}_{:=\mathcal{B}_{16}}\Big)_r+G_{14},
\end{equation}
where
\begin{equation*}
\begin{aligned}
G_{14}&=-r^m(u_tu_r+uu_{tr}+\phi_{tr})  +2a_1 \delta a \iota r^m(\log\phi)_t \phi^{2\iota}\Big(2\big(u_r+\frac{m}{r}u\big)_r u_{tt}+ \big(u_{tr}+\frac{m}{r}u_t\big)^2\Big)\\
&\quad +2a_1  r^m\psi_t\big(\delta u_r+m(\delta-1)\frac{u}{r}\big)u_{tt} +2a_1  r^m\psi \big(\delta u_{tr}+m(\delta-1)\frac{u_t}{r}\big)u_{tt}.
\end{aligned}
\end{equation*}

\smallskip
\textbf{2.} We show that $r^m \mathrm{D}_r  \mathcal{B}_{16} \in L^1(I)$  for {\it a.e.} $t\in (0,T)$, so that Lemma \ref{calculus} implies
\begin{equation}\label{int-B16}
\int_0^\infty (r^m\mathcal{B}_{16})_r\,\mathrm{d}r =0.
\end{equation}
Indeed, by \eqref{spd}--\eqref{spd2}, we see that, for {\it a.e.} $t\in (0,T)$,
\begin{equation*}
(\psi,u,\mathrm{D}_r u)\in L^\infty(I), \qquad 
r^\frac{m}{2}(\phi_t,\mathrm{D}_r u,u_t,\mathrm{D}_r^2 u,\mathrm{D}_r u_t,\phi^\iota \mathrm{D}_r u_t,\phi^{2\iota}\mathrm{D}_r^2 u_t)\in L^2(I).
\end{equation*}
Then it follows from the H\"older inequality that
\begin{equation*} 
\begin{aligned}
|r^m \mathrm{D}_r \mathcal{B}_{16}|_1&\leq C_0\big(\big|r^{m-1}\phi^{2\iota}(\mathrm{D}_r u_t) u_{tt}\big|_1 + \big|r^{m} (\psi (\mathrm{D}_r u_t)u_{tt},\phi^{2\iota}(\mathrm{D}_r^2 u_t)u_{tt},,\phi^{2\iota}(\mathrm{D}_r u_t)u_{ttr})\big|_1\big)\\
&\leq C_0 |r^\frac{m}{2}\phi^\iota\mathrm{D}_r u_t|_2|r^\frac{m}{2}\phi^\iota\mathrm{D}_r u_{tt}|_2 \\
&\quad +C_0|r^\frac{m}{2}u_{tt}|_2\big(|\psi|_\infty|r^\frac{m}{2}\mathrm{D}_r u_t|_2+|r^\frac{m}{2} \phi^{2\iota}\mathrm{D}_r^2 u_t|_2\big) <\infty,
\end{aligned}    
\end{equation*}
which thus leads to \eqref{int-B16}.

\smallskip
\textbf{3.} To estimate $G_{14}$, we first see from from \eqref{tr} and $\eqref{e2.2}_1$ that
\begin{equation}\label{eq:logphi}
(\log\phi)_t+\frac{\gamma-1}{\delta a}u \phi^{-2\iota} \psi+(\gamma-1)\big(u_r+\frac{m}{r}u\big)=0,
\end{equation}
which, along with Lemmas \ref{important2} and \ref{lemma66}--\ref{l4.7}, yields
\begin{equation}\label{logphi}
|(\log\phi)_t|_\infty\leq C_0 \big(|u|_\infty|\phi|_\infty^{-2\iota}|\psi|_\infty+|\mathrm{D}_ru|_\infty\big)\leq C(T)\big(|\mathrm{D}_ru|_\infty+1\big).
\end{equation}

Next, by \eqref{logphi}, Lemmas \ref{important2} and \ref{l4.8}--\ref{l4.9}, 
and the H\"older and Young inequalities, we have
\begin{equation}\label{etrq}
\begin{aligned}
|G_{14}|_1&\leq C_0\big(|r^{\frac{m}{2}}(u_t,\psi_t)|_2|\mathrm{D}_ru|_\infty + |\phi|_\infty^{-\iota}|r^{\frac{m}{2}}\phi^{\iota}\mathrm{D}_ru_t|_2|(u,\psi)|_\infty+|r^{\frac{m}{2}}\phi_{tr}|_2\big)|r^{\frac{m}{2}}u_{tt}|_2\\
&\quad +C_0|(\log\phi)_t|_\infty\big(|r^\frac{m}{2}\phi^{2\iota}\mathrm{D}_r^2 u|_2|r^\frac{m}{2}u_{tt}|_2+|r^\frac{m}{2}\phi^{\iota}\mathrm{D}_r u_t|_2^2\big)\\
&\leq C(T) \big(|\mathrm{D}_ru|_\infty^2 +1\big)\big(|r^\frac{m}{2}\phi^{\iota}\mathrm{D}_r u_t|_2^2+1\big) +\frac{1}{8}|r^{\frac{m}{2}}u_{tt}|_2^2.
\end{aligned}    
\end{equation}

Hence, integrating \eqref{dtg15} over $[\tau,t]\times I$ with $\tau\in (0,t)$, along with \eqref{int-B16}--\eqref{etrq}, and Lemmas \ref{im-2} and \ref{l4.8}--\ref{lemma66}, implies 
\begin{equation}\label{etrq2}
\begin{aligned}
&\,|r^{\frac{m}{2}}\phi^\iota\mathrm{D}_r u_t(t)|^2_2+\int_{\tau}^{t}  |r^{\frac{m}{2}}u_{tt}|^2_2 \,\mathrm{d}s\\
&\leq C(T)\big(|r^{\frac{m}{2}}\phi^\iota\mathrm{D}_r u_t(\tau)|^2_2+1\big)+C(T) \int_\tau^t \big(|\mathrm{D}_ru|_\infty^2 +1\big)\big(|r^\frac{m}{2}\phi^{\iota}\mathrm{D}_r u_t|_2^2+1\big)\,\mathrm{d}s.
\end{aligned}
\end{equation}
For the $L^2(I)$-boundedness of $r^{\frac{m}{2}}\phi^\iota\mathrm{D}_r u_t(\tau)$ 
on the right-hand side of \eqref{etrq2}, we apply $r^\frac{m}{2}\phi^\iota \mathrm{D}_r$ to $\eqref{e2.2}_2$  and take the $L^2(I)$-norm of the resulting equality. Then it follows from \eqref{tr}, Lemmas \ref{important2} and \ref{l4.8}--\ref{l4.7}, and the H\"older inequality that, for $\tau \in (0,t)$,
\begin{equation*}
\begin{aligned}
|r^{\frac{m}{2}}\phi^\iota\mathrm{D}_r u_t |_2&\leq C_0\big(|\mathrm{D}_ru |_\infty|r^{\frac{m}{2}}\phi^\iota \mathrm{D}_r u |_2+|u |_\infty|r^{\frac{m}{2}}\phi^\iota \mathrm{D}_r^2 u |_2 +|\phi|_\infty^{1-\iota}|r^{\frac{m}{2}} \mathrm{D}_r\psi|_2\big)\\
&\quad + C_0 |r^\frac{m}{2}\phi|_2|\phi|_\infty^{-3\iota}| \psi|_\infty^2+ C_0\Big|r^{\frac{m}{2}}\phi^\iota \mathrm{D}_r\Big(\phi^{2\iota}\big(u_r+\frac{m}{r}u\big)_r\Big) \Big|_2\\
&\quad +C_0 |r^\frac{m}{n}\mathrm{D}_r\psi|_n \Big|r^\frac{m}{n^*}\phi^\iota \big(u_r+\frac{m}{r}u\big) \Big|_{n^*}+C_0|\psi |_\infty|r^{\frac{m}{2}}\phi^{\iota} \mathrm{D}_r^2 u|_2\\
&\leq C(T)\big(|\mathrm{D}_ru |_\infty  +1\big)  + C_0\Big|r^{\frac{m}{2}}\phi^\iota \mathrm{D}_r\Big(\phi^{2\iota}\big(u_r+\frac{m}{r}u\big)_r\Big) \Big|_2 \\
&\quad +C_0 |r^\frac{m}{n}\mathrm{D}_r\psi|_n \Big|r^\frac{m}{n^*}\phi^\iota \big(u_r+\frac{m}{r}u\big)\Big|_{n^*},
\end{aligned}
\end{equation*}
which, along with the time-continuity of $(\phi,u,\psi)$, \eqref{Lu=}, the initial compatibility condition \eqref{th78zxq}, and 
Lemmas \ref{ale1}, \ref{initial3}, and \ref{lemma-initial}, yields
\begin{equation*}
\begin{aligned}
\limsup_{\tau\to 0} |r^{\frac{m}{2}}\phi^\iota\mathrm{D}_ru_t(\tau)|_2&\leq C(T) \big(\|\nabla\boldsymbol{u}_0\|_{L^\infty}\!+\! \|\phi_0^\iota\nabla(\phi_0^{2\iota}L\boldsymbol{u}_0)\|_{L^2}\! +\! \|\nabla\boldsymbol{\psi}_0\|_{L^n} \|\phi_0^{\iota}\diver\boldsymbol{u}_0\|_{L^{n^*}}\!+1\big)\\[-4pt]
&\leq C(T)\big(\|\mathcal{G}_2\|_{L^2}+   \|\phi_0^{\iota}\diver\boldsymbol{u}_0\|_{H^2}+1\big)\\[1pt]
&\leq C(T)\big(\|\phi_0^{\iota}\nabla\boldsymbol{u}_0\|_{L^2}+\|\nabla\phi_0^{\iota}\|_{L^4}\|\nabla\boldsymbol{u}_0\|_{L^4}+\|\phi_0^{\iota}L\boldsymbol{u}_0\|_{L^2}\big)\\
&\quad +C(T)\|\phi_0\|_{L^\infty}^{-\iota}\big(\|\nabla^2(\phi_0^{\iota}\diver\boldsymbol{u}_0)\|_{L^2}+\|\nabla\boldsymbol{\psi}_0\|_{L^n}\|\diver\boldsymbol{u}_0 \|_{L^{n^*}}\big)\\
&\quad +C(T)\|\phi_0\|_{L^\infty}^{-\iota}\big(\|\phi_0\|_{L^\infty}^{-2\iota}\|\boldsymbol{\psi}_0\|_{L^\infty}^2\|\diver\boldsymbol{u}_0 \|_{L^{2}}+\|\phi_0^{2\iota}L\boldsymbol{u}_0\|_{L^2}+1\big)\\[1pt]
&\leq C(T)\big(\|(\mathcal{G}_1,\mathcal{G}_2,\boldsymbol{g}_*)\|_{L^2}+1\big)\leq C(T).
\end{aligned}
\end{equation*}
Consequently, based on the above discussions, letting $\tau\to 0$ in \eqref{etrq2}, together with Lemma \ref{lemma66} and the Gr\"onwall inequality, yields 
\begin{equation*}
|r^{\frac{m}{2}}\phi^\iota\mathrm{D}_r u_t(t)|^2_2+\int_{0}^{t}  |r^{\frac{m}{2}}u_{tt}|^2_2 \,\mathrm{d}s\leq C(T)\qquad\text{for all $t\in [0,T]$}.
\end{equation*}

Finally, if $n=3$, it follows from the above and Lemmas \ref{important2}, \ref{l4.8}, \ref{ale1}, and \ref{lemma-initial} that
\begin{equation*}
|r^\frac{m}{6} u_t|_{6}\leq C_0\|\boldsymbol{u}_t\|_{L^{6}}\leq C_0\|\boldsymbol{u}_t\|_{H^1}\leq C_0\big(|r^\frac{m}{2} u_t |_{2}+|\phi|_\infty^{-\iota}|r^\frac{m}{2}\phi^\iota\mathrm{D}_r u_t |_{2}\big)\leq C(T).
\end{equation*}

This completes the proof.
\end{proof}

\begin{lem}\label{l4.10-ell}  
There exists a constant $C(T)>0$ such that, for any $q\in [2\iota,0]$ and $t\in [0,T]$,
\begin{equation*}
|\mathrm{D}_r u(t)|_\infty\!+\!\delta_{3n}|r^\frac{m}{6}\mathrm{D}_r^2 u(t)|_{6}\!+\!|r^\frac{m}{2n}(\phi^\iota)_r(t)|_{2n}\!+\!|r^\frac{m}{2}\phi^q \mathrm{D}_r^3u(t)|_2\!+\!\int_{0}^{t} |r^\frac{m}{2}\phi^q\mathrm{D}_r^2u_t|_{2}^2 \mathrm{d}s\leq C(T).
\end{equation*}
\end{lem}
\begin{proof}
First, it follows from \eqref{ur-infty} and Lemmas \ref{important2} and \ref{l4.10},
\begin{equation}\label{ur-infty-t}
|\mathrm{D}_r u|_\infty\leq C(T) \big(|\phi|_\infty^{-\iota}|r^\frac{m}{2} \phi^\iota\mathrm{D}_r u_t|_2+ 1\big)\leq C(T)\qquad \text{for all $t\in [0,T]$},
\end{equation}
which, along with \eqref{uxxxl1}, and Lemmas \ref{im-2}, \ref{lemma66}, and \ref{l4.10}, yields that, for any $q\in [2\iota,0]$,
\begin{equation}\label{3333}
\begin{aligned}
|r^\frac{m}{2}\phi^q \mathrm{D}_r^3u|_{2} &\leq |\phi|_\infty^{q-2\iota}|r^\frac{m}{2}\phi^{2\iota}\mathrm{D}_r^3u|_{2}\leq C(T)\big(|r^{\frac{m}{2}}\phi^\iota\mathrm{D}_r u_t|_{2}+|\mathrm{D}_r u|_\infty+1\big)\leq C(T).
\end{aligned}    
\end{equation}

Next, applying $r^\frac{m}{2}\partial_t$ to $(\ref{e2.2})_2$ and taking $L^2(I)$-norm of the resulting equality, along with \eqref{logphi}, \eqref{ur-infty-t}, and Lemmas \ref{im-2} and \ref{l4.8}--\ref{l4.10}, we have 
\begin{equation}\label{utt}
\begin{aligned}
|r^\frac{m}{2}\phi^{2\iota}\mathrm{D}_r^2u_t|_{2}&\leq C(T)\Big(\Big|r^\frac{m}{2}\phi^{2\iota}\big(u_r+ \frac{m}{r}u\big)_{tr}\Big|_2+1\Big)\\
&\leq C(T)\big(\big|r^{\frac{m}{2}}\big(u_{tt},(u u_r)_t,\phi_{tr},(\psi \mathrm{D}_r u)_t,(\phi^{2\iota})_t \mathrm{D}_r^2 u\big)\big|_2+1\big)\\
&\leq C(T)\big(|r^{\frac{m}{2}}(u_{tt},\phi_{tr})|_2 +|r^{\frac{m}{2}}(u_t,\psi_t)|_2|\mathrm{D}_r u|_\infty+|r^{\frac{m}{2}}\mathrm{D}_r u_t|_2|(u,\psi)|_\infty \big) \\
&\quad +C(T)\big(|(\log\phi)_t|_\infty |r^\frac{m}{2}\phi^{2\iota}\mathrm{D}_r^2 u|_2+1\big) \leq C(T)\big(|r^{\frac{m}{2}}u_{tt}|_2 +1\big).
 \end{aligned}
\end{equation}
This, combined with Lemmas \ref{important2} and \ref{l4.10}, yields the $L^2([0,T];L^2)$-estimates of $r^\frac{m}{2}\phi^{q}\mathrm{D}_r^2u_t$.

To obtain $L^{2n}(I)$-estimate for $r^\frac{m}{2n}(\phi^\iota)_r$, we first see from \eqref{3333} and Lemmas \ref{important2}, \ref{lemma66}--\ref{l4.7}, \ref{ale1}, and \ref{lemma-initial} that
\begin{equation*}
\begin{aligned}
|r^\frac{m}{2n}\phi^\iota \mathrm{D}_r u|_{2n}&\leq C_0\|\phi^\iota\nabla^2\boldsymbol{u}\|_{L^{2n}}\leq C_0\|(\phi^\iota\nabla^2\boldsymbol{u},\phi^{-\iota}|\boldsymbol{\psi}||\nabla^2\boldsymbol{u}|,\phi^\iota\nabla^3\boldsymbol{u})\|_{L^2}\\
&\leq C_0|r^\frac{m}{2}\phi^\iota(\mathrm{D}_r^2u,\mathrm{D}_r^3u)|_{2}+C_0|\phi|_\infty^{-\iota}|\psi|_\infty|r^\frac{m}{2}\mathrm{D}_r^2u|_{2}\leq C(T).  
\end{aligned}
\end{equation*}
Then applying $r^m|(\phi^\iota)_r|^{2n-2}(\phi^\iota)_r\partial_r$ to $\eqref{e2.2}_1\times \phi^{\iota-1}$ and then integrating over $I$, together with the above, \eqref{ur-infty-t}, and the H\"older and Young inequalities, yield
\begin{equation*}
\begin{aligned}
\frac{\mathrm{d}}{\mathrm{d}t}|r^\frac{m}{2n}(\phi^\iota)_r|_{2n}^{2n}&\leq C_0|\mathrm{D}_r u|_\infty|r^\frac{m}{2n}(\phi^\iota)_r|_{2n}^{2n}+C_0|r^\frac{m}{2n}\phi^\iota \mathrm{D}_r u|_{2n}|r^\frac{m}{2n}(\phi^\iota)_r|_{2n}^{2n-1}\\
&\leq C(T)\big(|r^\frac{m}{2n}(\phi^\iota)_r|_{2n}^{2n}+1\big),
\end{aligned}
\end{equation*}
which, along with the Gr\"onwall inequality, implies the desired estimate.

Finally, if $n=3$, it follows from \eqref{3333} and Lemmas \ref{lemma66}, \ref{ale1}, and \ref{lemma-initial} that
\begin{equation*}
|r^\frac{m}{6}\mathrm{D}_r^2 u|_{6}\leq C_0\|\nabla^2\boldsymbol{u}\|_{L^{6}}\leq C_0\|\nabla^2\boldsymbol{u}\|_{H^1}\leq C_0|r^\frac{m}{2}(\mathrm{D}_r^2 u,\mathrm{D}_r^3 u)|_{2}\leq C(T).
\end{equation*}
\end{proof}

\subsection{The Highest-Order Estimates of $(\phi,\boldsymbol{\psi})$.}
Now, we focus on the $D^2(\mathbb{R}^n)$-estimates for $\boldsymbol{\psi}$ and  $D^3(\mathbb{R}^n)$-estimates for $\phi$.

\begin{lem}\label{lemma-psi-high}
There exists a constant $C(T)>0$ such that
\begin{equation*}
|r^\frac{m}{2}(\mathrm{D}_r^2\psi,\mathrm{D}_r \psi_t)(t)|_2\leq C(T)\qquad \text{for any $t\in [0,T]$}.
\end{equation*}
\end{lem}
\begin{proof} We divide the proof into three steps.

\smallskip
\textbf{1.} It follows from Lemmas \ref{l4.7}, \ref{ale1}, and \ref{hardy}, and the H\"older inequality that
\begin{equation}\label{vr-infty}
\begin{aligned}
|\chi_1^\flat r\mathrm{D}_r v|_\infty+|\chi_1^\sharp \mathrm{D}_r v|_\infty&\leq C_0\big(|\chi_1^\flat r^\frac{3}{2}(\mathrm{D}_r v,\mathrm{D}_r^2 v)|_2 + |\chi_1^\sharp(\mathrm{D}_r v,\mathrm{D}_r^2 v)|_2\big)\\
&\leq C_0\big(|\chi_1^\flat r^{\frac{3}{2}-\frac{m}{n}}|_{n^*}+|\chi_1^\flat r^{-\frac{m}{n}}|_{n^*}\big)|r^\frac{m}{n}\mathrm{D}_r v|_n+C_0|r^\frac{m}{2}\mathrm{D}_r^2 v|_2^2\\
&\leq C(T)\big(|r^\frac{m}{2}\mathrm{D}_r^2 v|_2^2+1\big).
\end{aligned}
\end{equation}

\smallskip
\textbf{2.}
Applying $\partial_l$ ($l=1,\cdots\!,n)$ to both sides of \eqref{nabla-v} leads to the following equation in the sense of distributions:
\begin{equation}\label{partdivev}
\begin{aligned}
&\, (\partial_l\diver\boldsymbol{v})_t+\diver\big(\boldsymbol{u} (\partial_l\diver\boldsymbol{v})\big)\\
&=(\diver\boldsymbol{u})(\partial_l\diver\boldsymbol{v})-\partial_l\boldsymbol{u}\cdot \nabla(\diver\boldsymbol{v})-\partial_l(\nabla\boldsymbol{u}^\top:\nabla\boldsymbol{v})\\
&\quad -\frac{\gamma-\delta}{2a_1\delta a}\big(\frac{1-\delta}{\delta a}\phi^{-4\iota}\psi_l\nabla\phi+\phi^{-2\iota}\nabla(\partial_l\phi)\big)\cdot(\boldsymbol{v}-\boldsymbol{u})-\frac{\gamma-\delta}{2a_1\delta a}\phi^{-2\iota}\nabla\phi\cdot(\partial_l\boldsymbol{v}-\partial_l\boldsymbol{u})\\
&\quad  -\frac{\gamma-\delta}{2a_1\delta a} \phi^{-2\iota}\partial_l\phi(\diver\boldsymbol{v}-\diver\boldsymbol{u})-\frac{\gamma-1}{2a_1\delta a}\phi^{1-2\iota}(\partial_l\diver\boldsymbol{v}-\partial_l\diver\boldsymbol{u}).
\end{aligned}
\end{equation}
It can be checked due to $\boldsymbol{v}=\boldsymbol{u}+2a_1\nabla\log\rho$, and Lemmas \ref{th1} and \ref{ale1} that 
\begin{equation*}
\nabla\diver\boldsymbol{v}\in L^\infty([0,T];L^2(\mathbb{R}^n)),\quad \boldsymbol{u}\in L^1([0,T];W^{1,\infty}(\mathbb{R}^n)),
\end{equation*}
and the right-hand side of \eqref{partdivev} belongs to $L^1([0,T];L^2(\mathbb{R}^n))$.
Then it follows from Lemma \ref{lemma-lions} that, for {\it a.e.} $t\in (0,T)$, 
\begin{equation}
\begin{aligned}
&\frac{\mathrm{d}}{\mathrm{d}t}\|\partial_l\diver\boldsymbol{v}\|_{L^2}^{2}+\frac{\gamma-1}{a_1\delta a}\big\|\phi^{\frac{1}{2}-\iota}\partial_l\diver\boldsymbol{v}\big\|_{L^2}^{2}\\
&=\int_{\mathbb{R}^n} \big( (\diver\boldsymbol{u}) (\partial_l\diver\boldsymbol{v}) -2 (\partial_l\boldsymbol{u}\cdot\nabla\diver\boldsymbol{v}) -2 \partial_l(\nabla\boldsymbol{u}^\top:\nabla\boldsymbol{v})\big)(\partial_l\diver\boldsymbol{v})\,\mathrm{d}\boldsymbol{x}\\
&\quad- \frac{\gamma-\delta}{a_1\delta a}\int_{\mathbb{R}^n}\big(\frac{1-\delta}{\delta a}\phi^{-4\iota}\psi_l\nabla\phi+\phi^{-2\iota}\nabla(\partial_l\phi)\big)\cdot(\boldsymbol{v}-\boldsymbol{u})(\partial_l\diver\boldsymbol{v})\,\mathrm{d}\boldsymbol{x}\\
&\quad - \frac{\gamma-\delta}{a_1\delta a}\int_{\mathbb{R}^n} \phi^{-2\iota}\big(\nabla\phi\cdot(\partial_l\boldsymbol{v}-\partial_l\boldsymbol{u})+\partial_l\phi(\diver\boldsymbol{v}-\diver\boldsymbol{u})\big)(\partial_l\diver\boldsymbol{v})\,\mathrm{d}\boldsymbol{x}\\
&\quad + \frac{\gamma-1}{a_1\delta a}\int_{\mathbb{R}^n}  \phi(\partial_l\diver\boldsymbol{u}) (\partial_l\diver\boldsymbol{v})\,\mathrm{d}\boldsymbol{x}.
\end{aligned}    
\end{equation}
Summing  above over $l=1,\cdots\!,n$, we obtain from the spherical coordinate transformation that, for {\it a.e.} $t\in (0,T)$,
\begin{equation}\label{nabla2-v}
\frac{\mathrm{d}}{\mathrm{d}t}\Big|r^\frac{m}{2}\big(v_r+\frac{m}{r}v\big)_{r}\Big|_2^2+\frac{\gamma-1}{a_1\delta a}\Big|(r^m\phi)^{\frac{1}{2}-\iota}\big(v_r+\frac{m}{r}v\big)_{r}\Big|_2^2:=\sum_{i=15}^{19}G_i,
\end{equation}
where
\begin{align*}
G_{15}&=\int_0^\infty  r^m\big(\frac{m}{r}u- u_r\big)\big(v_r+\frac{m}{r}v\big)_{r}^2\,\mathrm{d}r,\\
G_{16}&=-2\int_0^\infty r^m\Big(u_rv_r+\frac{m}{r^2}uv\Big)_r\big(v_r+\frac{m}{r}v\big)_{r}\,\mathrm{d}r,\\
G_{17}&= -\frac{\gamma-\delta}{a_1\delta a}\int_0^\infty r^m\big(\frac{1-\delta}{\delta a}\phi^{-4\iota}\psi \phi_r+\phi^{-2\iota}\phi_{rr}\big)(v-u)\big(v_r+\frac{m}{r}v\big)_{r}\,\mathrm{d}r,\\
G_{18}&=-\frac{\gamma-\delta}{a_1\delta a}\int_0^\infty r^m \phi^{-2\iota}\phi_{r}\Big(\big(2v_r+\frac{m}{r}v\big)-\big(2u_r+\frac{m}{r}u\big)\Big)\big(v_r+\frac{m}{r}v\big)_{r}\,\mathrm{d}r,\\
G_{19}&=\frac{\gamma-1}{a_1\delta a}\int_0^\infty  r^m\phi^{1-2\iota} \big(u_r+\frac{m}{r}u\big)_{r}\big(v_r+\frac{m}{r}v\big)_{r}\,\mathrm{d}r.
\end{align*}

For $G_{15}$--$G_{19}$, it follows from \eqref{vr-infty}, Lemmas \ref{important2}, \ref{l4.4}, \ref{lemma66}--\ref{l4.9}, \ref{l4.10-ell}, \ref{ale1}, and \ref{hardy}, and the H\"older and Young inequalities that
\begin{align*}\label{G16-20}
&\begin{aligned}
G_{15}&\leq C_0|\mathrm{D}_r u|_\infty|r^\frac{m}{2}\mathrm{D}_r^2 v|_2^2 \leq C(T)|r^\frac{m}{2}\mathrm{D}_r^2 v|_2^2,     
\end{aligned}\\
&\begin{aligned}
G_{16}&\leq C_0\big(\big|\chi_1^\flat r^\frac{m}{2}|\mathrm{D}_r^2 u| |\mathrm{D}_r v|\big|_2+\big|\chi_1^\sharp r^\frac{m}{2}|\mathrm{D}_r^2 u| |\mathrm{D}_r v|\big|_2+|\mathrm{D}_r u|_\infty|r^\frac{m}{2}\mathrm{D}_r^2 v|_2\big)|r^\frac{m}{2}\mathrm{D}_r^2 v|_2\\
&\leq C_0\big( |r^\frac{m}{2} \mathrm{D}_r^3 u|_2 |\chi_1^\flat r\mathrm{D}_r v|_\infty+ |r^\frac{m}{2}\mathrm{D}_r^2 u|_2|\chi_1^\sharp \mathrm{D}_r v|_\infty+|\mathrm{D}_r u|_\infty|r^\frac{m}{2}\mathrm{D}_r^2 v|_2\big)|r^\frac{m}{2}\mathrm{D}_r^2 v|_2\\
&\leq C(T)\big(|r^\frac{m}{2}\mathrm{D}_r^2 v|_2^2 +1\big),
\end{aligned}\\
&\begin{aligned}
G_{17}&\leq C_0\big(|\phi|_\infty^{-4\iota}|\psi|_\infty|r^\frac{m}{2}\phi_{r}|_2 +|\phi|_\infty^{-2\iota}|r^\frac{m}{2}\phi_{rr}|_2\big)|(u,v)|_\infty|r^\frac{m}{2}\mathrm{D}_r^2 v|_2\leq C(T)\big(|r^\frac{m}{2}\mathrm{D}_r^2 v|_2^2 +1\big),
\end{aligned}\\
&\begin{aligned}
G_{18}&\leq C_0|\phi|_\infty^{-2\iota}\big(|r^{\frac{m-2}{2}}\phi_r|_2 |(\chi_1^\flat rv_r,v)|_\infty+|r^{\frac{m}{2}}\phi_r|_2 |\chi_1^\sharp v_r|_\infty+|r^{\frac{m}{2}}\phi_r|_2|\mathrm{D}_r u|_\infty\big) |r^\frac{m}{2}\mathrm{D}_r^2 v|_2 \\
&\leq C(T) \big(|r^\frac{m}{2}\mathrm{D}_r^2 v|_2^2+1\big),
\end{aligned}\\
&\begin{aligned}
G_{19}&\leq C|\phi|_\infty^{1-2\iota}|r^\frac{m}{2}\mathrm{D}_r^2 u|_2|r^\frac{m}{2}\mathrm{D}_r^2 v|_2\leq C(T) \big(|r^\frac{m}{2}\mathrm{D}_r^2 v|_2^2+1\big).
\end{aligned}
\end{align*}

Thus, substituting the above into \eqref{nabla2-v} leads to
\begin{equation*}
\frac{\mathrm{d}}{\mathrm{d}t}\Big|r^\frac{m}{2}\big(v_r+\frac{m}{r}v\big)_r\Big|_2^2\leq C(T) \big(|r^\frac{m}{2}\mathrm{D}_r^2 v|_2^2+1\big),
\end{equation*}
which, along with Lemmas \ref{im-1} and \ref{lemma-initial}, and the Gr\"onwall inequality, yields
\begin{equation*}
\begin{aligned}
|r^\frac{m}{2}\mathrm{D}_r^2 v(t)|_2&\leq C(T)\big(|r^\frac{m}{2}\mathrm{D}_r^2 v_0|_2+1 \big)\leq C(T)\big(|r^\frac{m}{2}(\mathrm{D}_r^2u_0,\mathrm{D}_r^2\psi_0)|_2+1\big)\\
&\leq C(T)\big(\|(\nabla^2\boldsymbol{u}_0,\nabla^2\boldsymbol{\psi}_0)\|_{L^2}^2+1\big)\leq C(T).
\end{aligned}
\end{equation*}
Certainly, by \eqref{v-express-2} and Lemma \ref{lemma66}, we also have
\begin{equation}\label{high-psi}
|r^\frac{m}{2}\mathrm{D}_r^2 \psi(t)|_2\leq C_0\big(|r^\frac{m}{2}\mathrm{D}_r^2 v(t)|_2+|r^\frac{m}{2}\mathrm{D}_r^2 u(t)|_2\big)\leq C(T).
\end{equation}

\smallskip
\textbf{3.} Applying $r^\frac{m}{2}\mathrm{D}_r$ to $\eqref{e2.2}_3$ and taking the $L^2(I)$-norm of the resulting equality, we obtain from \eqref{tr}, \eqref{high-psi}, and Lemmas \ref{lemma66}--\ref{l4.7} and \ref{l4.10-ell} that
\begin{equation*}
|r^\frac{m}{2}\mathrm{D}_r \psi_t|_2\leq C_0\big(|r^\frac{m}{n^*}u|_{n^*}|r^\frac{m}{n}\mathrm{D}_r \psi|_n+|(u,\psi)|_\infty|r^\frac{m}{2}(\mathrm{D}_r^2u,\mathrm{D}_r^2 \psi)|_2 +|r^\frac{m}{2}\phi^{2\iota}\mathrm{D}_r^3u|_2 \big) \leq C(T).   
\end{equation*}

This completes the proof.
\end{proof}

\begin{lem}\label{Lemma6.12}
There exists a constant $C(T)>0$ such that, for all $q\in [2\iota,0]$,
\begin{equation*}
|r^\frac{m}{2}(\mathrm{D}_r^2 \phi_r,\mathrm{D}_r\phi_{tr})(t)|_2+\int_0^t\big(|\mathrm{D}_r^2 u|_\infty^2+|r^\frac{m}{2}\phi^q\mathrm{D}_r^4 u|_{2}^2\big) \,\mathrm{d}s\leq C(T)\qquad \text{for all $t\in [0,T]$}.
\end{equation*}
\end{lem}

\begin{proof} We divide the proof into two steps.

\smallskip
\textbf{1.} First, for the $L^2(I)$-estimate of $r^\frac{m}{2}\mathrm{D}_r^2 \phi_r$, by Lemmas \ref{important2}, \ref{lemma66}--\ref{l4.9}, and \ref{lemma-psi-high}, we have
\begin{equation}\label{lem612-2}
\begin{aligned}
|r^\frac{m}{2}\mathrm{D}_r^2 \phi_r|_2&\leq C_0|r^\frac{m}{2}\mathrm{D}_r^2(\phi^{1-2\iota}\psi)|_2\leq  C_0\big|r^\frac{m}{2}(\phi^{-4\iota}\phi_r\psi^2,\phi^{1-4\iota} \psi\mathrm{D}_r\psi,\phi^{1-2\iota}\mathrm{D}_r^2\psi)\big|_2,\\
&\leq C_0|\phi|_\infty^{-4\iota}\big(|r^\frac{m}{2}\phi_{r}|_2|\psi|_\infty +|r^\frac{m}{n^*} \phi_r|_{n^*} |\psi|_\infty|r^\frac{m}{n}\mathrm{D}_r\psi|_n\big)+C_0|\phi|_\infty^{1-2\iota}  |r^\frac{m}{2}\mathrm{D}_r^2\psi|_2 \\
&\leq C(T).
\end{aligned}
\end{equation}

Next, notice from \eqref{logphi} and Lemmas \ref{important2}, \ref{lemma66}--\ref{l4.9}, and \ref{lemma-psi-high} that
\begin{equation}\label{lem612-8}
\begin{aligned}
|r^\frac{m}{2}\mathrm{D}_r\phi_{tr} |_2&\leq C_0|r^\frac{m}{2}\mathrm{D}_r(\phi^{-2\iota}\phi_t\psi)|_2+C_0|r^\frac{m}{2}\mathrm{D}_r(\phi^{1-2\iota}\psi_t)|_2\\
&\leq C_0|r^\frac{m}{2}(\phi^{-4\iota}\phi_t \psi^2,\phi^{-2\iota}\phi_{tr}\psi,\phi^{1-2\iota}(\log\phi)_t\mathrm{D}_r\psi, \phi^{-2\iota}\phi_r\psi_t,\phi^{1-2\iota} \mathrm{D}_r\psi_t)|_2\\
&\leq C_0|\phi|_\infty^{-4\iota}|\psi|_\infty^2 |r^\frac{m}{2}\phi_t|_2+C_0|\phi|_\infty^{-2\iota}|(\phi,\psi)|_\infty |r^\frac{m}{2}(\phi_{tr},\mathrm{D}_r\psi_{t})|_2 \\
&\quad +C_0|\phi|_\infty^{-2\iota} |(\log\phi)_t|_\infty|r^\frac{m}{n^*}\phi|_{n^*} |r^\frac{m}{n}\mathrm{D}_r\psi|_n+C_0|\phi|_\infty^{-2\iota}|r^\frac{m}{2}\phi_r\psi_t|_2\\
&\leq C(T)\big(|r^\frac{m}{2}\phi_r\psi_t|_2+1\big).
\end{aligned}
\end{equation}
For the estimate of $|r^\frac{m}{2}\phi_r\psi_t|_2$ on the right-hand side of \eqref{lem612-8}, it follows from Lemmas \ref{lemma66}--\ref{l4.9}, \ref{ale1}, and \ref{hardy} that
\begin{equation}\label{lem612-9}
\begin{aligned}
|r^\frac{m}{2}\phi_r\psi_t|_2&\leq |\chi_1^\flat r^\frac{m}{2}\phi_r\psi_t|_2+|\chi_1^\sharp r^\frac{m}{2}\phi_r\psi_t|_2 \leq  |\chi_1^\flat r\phi_r|_\infty|r^\frac{m-2}{2}\psi_t|_2+|\chi_1^\sharp \phi_r|_\infty|r^\frac{m}{2}\psi_t|_2\\
&\leq  C(T)\big(|\chi_1^\flat r^\frac{3}{2}(\phi_r,\phi_{rr})|_2+|\chi_1^\sharp (\phi_r,\phi_{rr})|_2\big) \\
&\leq  C(T)\big(|\chi_1^\flat r^\frac{3-m}{2}|_\infty+ |\chi_1^\sharp r^{-\frac{m}{2}}|_\infty\big) |r^\frac{m}{2}(\phi_r,\phi_{rr})|_2\leq C(T),    
\end{aligned}
\end{equation}
Combining with \eqref{lem612-8}--\eqref{lem612-9} yields
\begin{equation}\label{lem612-10}
|r^\frac{m}{2}\mathrm{D}_r\phi_{tr}(t)|_2\leq C(T)\qquad \text{for all $t\in [0,T]$}.
\end{equation}

\smallskip
\textbf{2.} Applying $r^\frac{m}{2}\mathrm{D}_r^2$ to $\eqref{e2.2}_2$ and taking the $L^2(I)$-norm of the resulting equality, we first obtain 
from \eqref{lem612-2}, and Lemmas \ref{lemma66}--\ref{l4.7} and \ref{l4.10-ell}--\ref{lemma-psi-high} that
\begin{equation*}
\begin{aligned}
\Big|r^\frac{m}{2}\mathrm{D}_r^2\big(\phi^{2\iota}(u_r+\frac{m}{r}u)_{r}\big)\Big|_2&\leq C_0\big|r^\frac{m}{2}\big(\mathrm{D}_r^2u_{t},|\mathrm{D}_r u|| \mathrm{D}_r^2u|,u\mathrm{D}_r^3u,\mathrm{D}_r^2\phi_{r}\big)\big|_2\\
&\quad +C_0\big|r^\frac{m}{2}\big(|\mathrm{D}_r^2\psi||\mathrm{D}_r u|,|\mathrm{D}_r\psi||\mathrm{D}_r^2 u|, \psi \mathrm{D}_r^3 u\big)\big|_2\\
&\leq C_0|r^\frac{m}{2}(\mathrm{D}_r^2u_{t},\mathrm{D}_r^2\phi_r)|_2+C_0|(u,\psi)|_\infty |r^\frac{m}{2}\mathrm{D}_r^3u|_2\\
&\quad +C_0 |r^\frac{m}{n} \mathrm{D}_r\psi|_n|r^\frac{m}{n^*}\mathrm{D}_r^2 u|_{n^*} + C_0|r^\frac{m}{2}(\mathrm{D}_r^2 u,\mathrm{D}_r^2\psi)|_2|\mathrm{D}_r u|_\infty \\
&\leq C(T)\big(|r^\frac{m}{2}\mathrm{D}_r^2 u_t|_2+|\mathrm{D}_r^2 u|_\infty+1\big),
\end{aligned}
\end{equation*}
which also implies
\begin{equation}\label{lem612-4}
\begin{aligned}
\Big|r^\frac{m}{2} \phi^{2\iota}\mathrm{D}_r^2\big(u_r+\frac{m}{r}u\big)_{r} \Big|_2
&\leq C_0\Big|r^\frac{m}{2}\mathrm{D}_r^2\big(\phi^{2\iota}(u_r+\frac{m}{r}u)_{r}\big)\Big|_2
+C_0|\psi|_\infty|r^\frac{m}{2}\mathrm{D}_r^3 u|_2\\
&\quad +C_0|r^\frac{m}{n}\mathrm{D}_r\psi|_n|r^\frac{m}{n^*}\mathrm{D}_r^2 u|_{n^*}\\
&\leq C(T)\big(|r^\frac{m}{2}\mathrm{D}_r^2 u_t|_2+|\mathrm{D}_r^2 u|_\infty+1\big).
\end{aligned}
\end{equation}

For the $L^\infty(I)$-estimate of $\mathrm{D}_r^2 u$, we see Lemmas \ref{important2}, \ref{im-1}, \ref{lemma66}, \ref{GN-ineq} and \ref{lemma-initial} that
\begin{equation}\label{776'}
\begin{aligned}
|\mathrm{D}_r^2 u|_\infty&\leq C_0\|\nabla^2 \boldsymbol{u}\|_{L^\infty}\leq C_0\|\nabla^2 \boldsymbol{u}\|_{L^2}^\frac{4-n}{4}\|\nabla^4 \boldsymbol{u}\|_{L^2}^\frac{n}{4}\\
&\leq C_0 |r^\frac{m}{2}\mathrm{D}_r^2 u|_{2}^\frac{4-n}{4}|r^\frac{m}{2}\mathrm{D}_r^4 u|_{2}^\frac{n}{4}\leq C(T) \Big|r^\frac{m}{2} \phi^{2\iota}\mathrm{D}_r^2\big(u_r+\frac{m}{r}u\big)_{r} \Big|_2^\frac{n}{4}.
\end{aligned}
\end{equation}
Then \eqref{lem612-4}, along with \eqref{776'}, Lemmas \ref{im-2}, \ref{lemma66}, and \ref{l4.10-ell}, and the Young inequality, gives
\begin{equation}\label{4jir}
\begin{aligned}
|r^\frac{m}{2}\phi^{2\iota}\mathrm{D}_r^4 u|_{2}\leq C(T)\big(\big|r^\frac{m}{2} \phi^{2\iota}\mathrm{D}_r^2(u_r+\frac{m}{r}u)_{r} \big|_2+1\big)
\leq C(T)\big(|r^\frac{m}{2}\mathrm{D}_r^2 u_t|_2 +1\big).
\end{aligned}
\end{equation}

Finally, from \eqref{776'} and Lemmas \ref{important2} and \ref{l4.10-ell}, we obtain  the desired estimates for $u$.
\end{proof}

\subsection{Time-Weighted Estimates of the Velocity}
We now establish the time-weighted estimates for $u$.
\begin{lem}\label{Lemma6.13}  
There exists a constant $C(T)>0$ such that
\begin{equation*}
\begin{aligned}
&t|r^{\frac{m}{2}}u_{tt}(t)|^2_2+\int_{0}^{t} s |r^{\frac{m}{2}}\phi^\iota \mathrm{D}_r u_{tt}|^2_2 \,\mathrm{d}s\leq C(T) \qquad \text{for any $t\in [0,T]$}.
\end{aligned}
\end{equation*}
\end{lem}

\begin{proof} We divide the proof into three steps.

\smallskip
\textbf{1.} 
First, it follows from Lemmas \ref{important2}, \ref{l4.8}, \ref{l4.10-ell}, \ref{ale1}, and \ref{hardy} that
\begin{equation}\label{lem613-1}
\begin{aligned}
|u_t|_\infty&\leq C_0|(u_t,u_{tr})|_2\leq C_0|\chi_1^\flat (u_t,u_{tr})|_2+C_0|\chi_1^\sharp (u_t,u_{tr})|_2\\
&\leq C_0|\chi_1^\flat r(u_t,u_{tr},u_{trr})|_2+C_0|\chi_1^\sharp (u_t,u_{tr})|_2\\
&\leq C_0  |r^\frac{m}{2}(u_t,u_{tr},u_{trr})|_2 \leq C(T) \big(|r^\frac{m}{2}\phi^{2\iota}u_{trr}|_2+1\big),\\
|(\chi_1^\flat r\mathrm{D}_r u_t,\chi_1^\sharp \mathrm{D}_r u_t)|_\infty &\leq C_0|\chi_1^\flat r^\frac{3}{2}(\mathrm{D}_r u_t,\mathrm{D}_r^2 u_t)|_2+ C_0|\chi_1^\sharp (\mathrm{D}_r u_t,\mathrm{D}_r^2 u_t)|_2\\
&\leq C_0 |r^\frac{m}{2}(\mathrm{D}_r u_t,\mathrm{D}_r^2 u_t)|_2 \leq C(T) \big(|r^\frac{m}{2}\phi^{2\iota}\mathrm{D}_r^2 u_t|_2+1\big).
\end{aligned}    
\end{equation}

Next, applying $r^\frac{m}{2}\partial_t$ to $\eqref{e2.2}_3$ and taking the $L^2(I)$-norm of both sides of the resulting equality, we obtain from \eqref{logphi}, \eqref{lem613-1}, and
Lemmas \ref{lemma66}--\ref{l4.7} and \ref{l4.10}--\ref{lemma-psi-high} that
\begin{equation}\label{lem613-1'}
\begin{aligned}
|r^\frac{m}{2}\psi_{tt}|_2&\leq C_0\big|r^\frac{m}{2}\big(u_t\psi_r,u\psi_{tr},\psi_t \mathrm{D}_r u,\psi \mathrm{D}_r u_t,\phi^{2\iota}(\log\phi)_t\mathrm{D}_r^2 u,\phi^{2\iota}\mathrm{D}_r^2 u_t\big)\big|_2\\
&\leq C_0\big(|(u,\psi)|_\infty |r^\frac{m}{2}(\mathrm{D}_r u,\psi_{tr})|_2+|\mathrm{D}_r u|_\infty|r^\frac{m}{2}\psi_t|_2+|r^\frac{m}{n^*}u_t|_{n^*}|r^\frac{m}{n}\psi_r|_n\big)\\
&\quad +C_0|(\log\phi)_t|_\infty|r^\frac{m}{2}\phi^{2\iota}\mathrm{D}_r^2 u|_2+C_0|r^\frac{m}{2}\phi^{2\iota}\mathrm{D}_r^2 u_t|_2\\
&\leq C(T) \big(|r^\frac{m}{2}\phi^{2\iota}\mathrm{D}_r^2 u_t|_2+1\big).
\end{aligned}
\end{equation}

Finally, based on $\eqref{e2.2}_1$, it follows from Lemmas \ref{important2}, \ref{l4.8}--\ref{l4.9}, 
and \ref{l4.10-ell} that
\begin{equation}\label{cal-phitt}
\begin{aligned}
|r^\frac{m}{2}\phi_{tt}|_2&\leq |r^\frac{m}{2}(u\phi_r)_t|_2+C_0|r^\frac{m}{2}(\phi\mathrm{D}_r u)_t|_2\\
&\leq |u|_\infty |r^\frac{m}{2}\phi_{tr}|_2+|\phi|_\infty|\psi|_\infty|r^\frac{m}{2}u_t|_2\\
&\quad +C_0|r^\frac{m}{2}\phi_t|_2 |\mathrm{D}_r u|_\infty+C_0|\phi|_\infty^{1-\iota} |r^\frac{m}{2}\phi^\iota\mathrm{D}_r u_t|_2\leq C(T),
\end{aligned}    
\end{equation}
which, along with the chain rule, \eqref{tr}, \eqref{logphi}, \eqref{lem613-1'}, and Lemma \ref{lemma66}, leads to
\begin{equation}\label{phittr}
\begin{aligned}
|r^\frac{m}{2}\phi_{ttr}|_2&\leq C_0\big|r^\frac{m}{2}(\phi^{-2\iota}\phi_{tt}\psi,\phi^{-2\iota}(\log\phi)_{t}\phi_t\psi,\phi^{1-2\iota}(\log\phi)_t\psi_t,\phi^{1-2\iota}\psi_{tt})\big|_2\\
&\leq C_0\big(|\phi|_\infty^{-2\iota}|(\phi,\psi)|_\infty \big|r^\frac{m}{2}(\phi_{tt},\psi_{tt})\big|_2+|\phi|_\infty^{-2\iota}|(\log\phi)_t|_\infty|\psi|_\infty |r^\frac{m}{2}\phi_{t}|_2\big)\\
&\quad + C_0|\phi|_\infty^{1-2\iota}|(\log\phi)_t|_\infty |r^\frac{m}{2}\psi_t|_2\leq C(T) \big(|r^\frac{m}{2}\phi^{2\iota}\mathrm{D}_r^2 u_t|_2+1\big).
\end{aligned}    
\end{equation}

\vspace{1pt}
\textbf{2.} We now proceed to prove Lemma \ref{Lemma6.13}. 
Formally, applying $r^mu_{tt}\partial_{t}^2$ to both sides of $\eqref{e2.2}_2$ and integrating the resulting equality over $I$ yield
\begin{equation}\label{lem613-2}
\begin{aligned}
&\,\frac{1}{2}\frac{\mathrm{d}}{\mathrm{d}t}|r^\frac{m}{2}u_{tt}|_2^2+2a_1\delta a\Big|r^\frac{m}{2}\phi^\iota\big(u_{ttr}+\frac{m}{r}u_{tt}\big)\Big|_2^2 \\
&=4a_1\delta a\int_0^\infty r^m (\phi^{2\iota})_t\big(u_{tr}+\frac{m}{r}u_t\big)_r u_{tt}\,\mathrm{d}r +2a_1\delta a\int_0^\infty r^m(\phi^{2\iota})_{tt}\big(u_{r}+\frac{m}{r}u\big)_r u_{tt}\,\mathrm{d}r\\
&\quad +2a_1\int_0^\infty r^m \Big(\psi_{tt}\big(\delta u_r+m(\delta-1)\frac{u}{r}\big) +2\psi_t\big(\delta u_{tr}+m(\delta-1)\frac{u_t}{r}\big)+\psi u_{ttr}\Big) u_{tt}\,\mathrm{d}r\\
&\quad -\int_0^\infty r^m\big((uu_r)_{tt}+\phi_{ttr}\big)u_{tt}\,\mathrm{d}r:=\sum_{i=20}^{23}G_i
\end{aligned}
\end{equation}
for {\it a.e.} $t\in (\tau,T)$ and $\tau\in (0,T)$. Here, we temporarily assume that energy equality \eqref{lem613-2} holds, and the specific proof will be given in Step 3 below.

Then, for $G_{20}$, it follows from \eqref{logphi}, \eqref{lem613-1}, 
Lemmas \ref{important2}, \ref{im-1}, \ref{l4.7}, and \ref{lemma-psi-high}, and the H\"older and Young inequalities that
\begin{equation} 
\begin{aligned}
G_{20}&\leq C_0|(\log\phi)_t|_\infty|r^\frac{m}{2}\phi^{2\iota}\mathrm{D}_r^2 u_t|_2|r^\frac{m}{2}u_{tt}|_2+C_0|\chi_1^\flat r\mathrm{D}_r u_t|_\infty |r^\frac{m-2}{2}\psi_t|_2|r^\frac{m}{2}u_{tt}|_2\\
&\quad +C_0|\chi_1^\sharp \mathrm{D}_r u_t|_\infty |r^\frac{m}{2}\psi_t|_2|r^\frac{m}{2}u_{tt}|_2+C_0|\psi|_\infty |r^\frac{m}{2}u_{ttr}|_2|r^\frac{m}{2}u_{tt}|_2\\
&\leq C(T) \big(|r^\frac{m}{2}\phi^{2\iota}\mathrm{D}_r^2 u_t|_2+|r^\frac{m}{2}u_{ttr}|_2+1\big)|r^\frac{m}{2}u_{tt}|_2\\
&\leq C(T)\big(|r^\frac{m}{2}u_{tt}|_2^2+|r^\frac{m}{2}\phi^{2\iota}\mathrm{D}_r^2 u_t|_2^2+1\big)+\frac{a_1}{8} \Big|r^\frac{m}{2}\phi^\iota\big(u_{ttr}+\frac{m}{r}u_{tt}\big)\Big|_2^2.
\end{aligned}
\end{equation}

To estimate $G_{21}$, we first see from \eqref{tr} and \eqref{phi2l} that
\begin{equation*} 
(\phi^{2\iota})_{tt}=(1-\delta)\Big(\frac{1}{\delta a} (u_t \psi+u\psi_t)+ 2\iota \phi^{2\iota}(\log\phi)_t\big(u_{r}+\frac{m}{r}u\big)+ \phi^{2\iota}\big(u_{tr}+\frac{m}{r}u_t\big)\Big).
\end{equation*}
Then it follows from the above, \eqref{logphi}, \eqref{lem613-1}, 
Lemmas \ref{lemma66}--\ref{l4.7} and \ref{l4.10-ell}--\ref{lemma-psi-high}, and the H\"older and Young inequalities that
\begin{equation} 
\begin{aligned}
G_{21} &\leq C_0|(u,\psi)|_\infty|\mathrm{D}_r^2 u|_\infty|r^\frac{m}{2}(u_t,\psi_t)|_2|r^\frac{m}{2}u_{tt}|_2\\
&\quad +C_0|(\log\phi)_t|_\infty|\mathrm{D}_r u|_\infty |r^\frac{m}{2}\phi^{2\iota}\mathrm{D}_r^2 u|_2|r^\frac{m}{2}u_{tt}|_2\\
&\quad +C_0|\chi_1^\flat r\mathrm{D}_r u_t|_\infty |r^\frac{m-2}{2}\phi^{2\iota}\mathrm{D}_r^2 u|_2|r^\frac{m}{2}u_{tt}|_2 +C_0|\chi_1^\sharp \mathrm{D}_r u_t|_\infty |r^\frac{m}{2}\phi^{2\iota}\mathrm{D}_r^2 u|_2|r^\frac{m}{2}u_{tt}|_2 \\
&\leq C(T) \big(|\mathrm{D}_r^2 u|_\infty+|r^\frac{m}{2}\phi^{2\iota}\mathrm{D}_r^2 u_t|_2 +1\big)|r^\frac{m}{2}u_{tt}|_2\\
&\leq C(T)\big(|r^\frac{m}{2}u_{tt}|_2^2+|r^\frac{m}{2}\phi^{2\iota}\mathrm{D}_r^2 u_t|_2^2+|\mathrm{D}_r^2 u|_\infty^2+1\big).
\end{aligned}
\end{equation}

For $G_{22}$--$G_{23}$, it follows from $\eqref{e2.2}_3$, \eqref{lem613-1}--\eqref{lem613-1'}, \eqref{phittr}, Lemmas \ref{important2}, \ref{im-1}, \ref{lemma66}--\ref{l4.7}, and \ref{l4.10-ell}--\ref{lemma-psi-high}, and the H\"older and Young inequalities that
\begin{align}
&\begin{aligned}
G_{22}
&\leq C_0|\mathrm{D}_r u|_\infty|r^\frac{m}{2}\psi_{tt}|_2|r^\frac{m}{2}u_{tt}|_2+C_0|\chi_1^\flat r\mathrm{D}_r u_t|_\infty |r^\frac{m-2}{2}\psi_t|_2|r^\frac{m}{2}u_{tt}|_2\\
&\quad +C_0|\chi_1^\sharp \mathrm{D}_r u_t|_\infty |r^\frac{m}{2}\psi_t|_2|r^\frac{m}{2}u_{tt}|_2+C_0|\psi|_\infty |r^\frac{m}{2}u_{ttr}|_2|r^\frac{m}{2}u_{tt}|_2\\
&\leq C(T) \big(|r^\frac{m}{2}\phi^{2\iota}\mathrm{D}_r^2 u_t|_2+|r^\frac{m}{2}u_{ttr}|_2+1\big)|r^\frac{m}{2}u_{tt}|_2\\
&\leq C(T)\big(|r^\frac{m}{2}u_{tt}|_2^2+|r^\frac{m}{2}\phi^{2\iota}\mathrm{D}_r^2 u_t|_2^2+1\big)+\frac{a_1}{8} \Big|r^\frac{m}{2}\phi^\iota\big(u_{ttr}+\frac{m}{r}u_{tt}\big)\Big|_2^2,
\end{aligned}\\
&\begin{aligned}
G_{23}&\leq \big(|u_r|_\infty |r^\frac{m}{2}u_{tt}|_2+2|u_t|_\infty |r^\frac{m}{2}u_{tr}|_2 + |u|_\infty |r^\frac{m}{2} u_{ttr}|_2 +|r^\frac{m}{2}\phi_{ttr}|_2\big)|r^\frac{m}{2}u_{tt}|_2\\
&\leq C(T)\big(|r^\frac{m}{2}u_{tt}|_2^2+|r^\frac{m}{2}\phi^{2\iota}\mathrm{D}_r^2 u_t|_2^2+1\big)+\frac{a_1}{8}\Big|r^\frac{m}{2}\phi^\iota\big(u_{ttr}+\frac{m}{r}u_{tt}\big)\Big|_2^2.
\end{aligned}\label{G23}
\end{align}

Thus, collecting \eqref{lem613-2}--\eqref{G23}, along with Lemma \ref{im-2}, gives
\begin{equation*} 
\begin{aligned}
\frac{\mathrm{d}}{\mathrm{d}t}|r^\frac{m}{2}u_{tt}|_2^2+a_1\delta a |r^\frac{m}{2}\phi^\iota\mathrm{D}_r u_{tt}|_2^2&\leq C(T)\big(|r^\frac{m}{2}u_{tt}|_2^2+|r^\frac{m}{2}\phi^{2\iota}\mathrm{D}_r^2 u_t|_2^2+|\mathrm{D}_r^2 u|_\infty^2+1\big).
\end{aligned}
\end{equation*}
Multiplying the above by $t$ and integrating the resulting inequality over $[\tau, t]$ for $\tau\in(0,t)$, along with Lemmas \ref{im-2}, \ref{l4.10-ell}, and \ref{Lemma6.12}, imply 
\begin{equation}\label{etrq2'}
t|r^{\frac{m}{2}}u_{tt}(t)|_2^2+ \int_\tau^t s |r^\frac{m}{2}\phi^\iota \mathrm{D}_r u_{tt}|_2^2\,\mathrm{d}s\leq C(T)\big(\tau|r^{\frac{m}{2}}u_{tt}(\tau)|_2^2+1\big).
\end{equation}

Next, thanks to Lemma \ref{l4.10} and $r^{\frac{m}{2}}u_{tt}\in L^2([0,T];L^2(I))$, it follows from  Lemma \ref{bjr} that there exists a sequence $\{\tau_k\}_{k=1}^\infty$ such that $\tau_k\to 0$ and $\tau_k|r^\frac{m}{2}u_{tt}(\tau_k)|_2^2\to 0$ as $k\to \infty$. Choosing $\tau=\tau_k\to 0$  in  \eqref{etrq2'} yields 
\begin{equation}\label{qunjin2}
t|r^{\frac{m}{2}}u_{tt}(t)|_2^2+ \int_0^t s |r^\frac{m}{2}\phi^\iota \mathrm{D}_r u_{tt}|_2^2\,\mathrm{d}s\leq C(T)\qquad\text{for all $t\in [0,T]$}.
\end{equation}

\textbf{3.} In the final step, we give a rigorous proof of the energy equality \eqref{lem613-2}. Suppose that $\boldsymbol{\varphi}(\boldsymbol{x})=\varphi(r)\frac{\boldsymbol{x}}{r}$ is any given spherically symmetric vector function satisfying $r^\frac{m}{2}(\varphi,\mathrm{D}_r\varphi)\in L^2(I)$. By Lemma \ref{lemma-initial}, this is equivalent to $\boldsymbol{\varphi}\in H^1(\mathbb{R}^n)$. 
Then applying $r^m\varphi\partial_t$ to both sides of $\eqref{e2.2}_2$ and integrating over $I$ yield from \eqref{tr} and integration by parts that 
\begin{equation*}
\begin{aligned}
\int_0^\infty r^mu_{tt}\varphi\,\mathrm{d}r&=-2a_1\delta a\int_0^\infty r^m\phi^{2\iota}\big(u_{tr}+\frac{m}{r}u_t\big)\big(\varphi_r+\frac{m}{r}\varphi\big)\,\mathrm{d}r\\
&\quad  +2a_1\delta a\int_0^\infty r^m(\phi^{2\iota})_t\big(u_{r}+\frac{m}{r}u\big)_r\varphi\,\mathrm{d}r \\
&\quad + 2a_1\int_0^\infty r^m \Big(\psi_t \big(\delta u_r+m(\delta-1)\frac{u}{r}\big) +\psi  u_{tr}\Big) \varphi\,\mathrm{d}r\\
&\quad -\int_0^\infty r^m\big((uu_r)_t+\phi_{tr}\big) \varphi\,\mathrm{d}r  .
\end{aligned}    
\end{equation*}
Here, the process of integration by parts can be justified by following the similar discussions in \eqref{eq:B10pre}--\eqref{util2}. 

Next, differentiating the above with respect to $t$ leads to
\begin{equation*} 
\begin{aligned}
\frac{\mathrm{d}}{\mathrm{d}t}\int_0^\infty r^mu_{tt}\varphi\,\mathrm{d}r&=-2a_1\delta a\int_0^\infty r^m\phi^{2\iota}\big(u_{ttr}+\frac{m}{r}u_{tt}\big)\big(\varphi_r+\frac{m}{r}\varphi\big)\,\mathrm{d}r\\
&\quad -2a_1\delta a\int_0^\infty r^m(\phi^{2\iota})_t\big(u_{tr}+\frac{m}{r}u_t\big)\big(\varphi_r+\frac{m}{r}\varphi\big)\,\mathrm{d}r\\
&\quad +2a_1\delta a\int_0^\infty r^m\Big((\phi^{2\iota})_{tt}\big(u_{r}+\frac{m}{r}u\big)_r+(\phi^{2\iota})_t\big(u_{tr}+\frac{m}{r}u_t\big)_r\Big)\varphi\,\mathrm{d}r\\
&\quad + \int_0^\infty r^m \Big(2a_1\Big(\psi_t \big(\delta u_r+m(\delta-1)\frac{u}{r}\big) +\psi  u_{tr}\Big)_t-(uu_r)_{tt}-\phi_{ttr}\big)\Big) \varphi\,\mathrm{d}r.
\end{aligned}    
\end{equation*}
Based on the {\it a priori} assumption $r^\frac{m}{2}(u_{ttr},\frac{u_{tt}}{r},\phi_{ttr},\psi_{tt})\in L^2([\tau,T];L^2(I))$ for $\tau\in (0,T)$, 
and the similar calculations of $G_{20}$--$G_{23}$ in Step 2, we obtain 
\begin{equation*}
\frac{1}{\omega_n}\frac{\mathrm{d}}{\mathrm{d}t}\int_{\mathbb{R}^n} \boldsymbol{u}_{tt}\cdot\boldsymbol{\varphi}\,\mathrm{d}\boldsymbol{x}=\frac{\mathrm{d}}{\mathrm{d}t}\int_0^\infty r^mu_{tt}\varphi\,\mathrm{d}r\leq F(t)|r^\frac{m}{2}(\varphi,\mathrm{D}_r\varphi)|_2
\end{equation*}
for some positive function $F(t)\in L^2(\tau,T)$, where $\omega_n$ denotes the surface area of the $n$-sphere. Thus, it follows from Lemma 1.1 on \cite[page 250]{temam} and  Lemma \ref{lemma-initial}  that $\boldsymbol{u}_{ttt}\in L^2([\tau,T];H^{-1}(\mathbb{R}^n))$. Consequently, the energy equality follows  from the spherical coordinates transformation and the following identity, due to Lemma \ref{triple}:
\begin{equation*}
\frac{\mathrm{d}}{\mathrm{d}t}\int_0^\infty r^m|u_{tt}|^2\,\mathrm{d}r=\frac{1}{\omega_n}\frac{\mathrm{d}}{\mathrm{d}t}\int_{\mathbb{R}^n} |\boldsymbol{u}_{tt}|^2\,\mathrm{d}\boldsymbol{x}=\frac{2}{\omega_n}\langle \boldsymbol{u}_{ttt},\boldsymbol{u}_{tt}\rangle_{H^{-1}(\mathbb{R}^n)\times H^1(\mathbb{R}^n)}.
\end{equation*}
The proof of Lemma \ref{Lemma6.13} is completed.
\end{proof}

\begin{lem}\label{Lemma6.14}  
There exists a constant $C(T)>0$ such that, for any $q\in [2\iota,0]$,
\begin{equation*}
\sqrt{t}|r^{\frac{m}{2}}\phi^q(\mathrm{D}_r^2 u_t,\mathrm{D}_r^4 u)(t)|_2\leq C(T)\qquad\text{for all $t\in [0,T]$}.
\end{equation*}
\end{lem}
\begin{proof}
It follows from \eqref{utt} and Lemmas \ref{important2} and \ref{Lemma6.13} that 
\begin{equation*} 
\sqrt{t}|r^{\frac{m}{2}}\phi^q\mathrm{D}_r^2u_t(t)|_2\leq C(T)\sqrt{t}|r^{\frac{m}{2}}\phi^{2\iota}\mathrm{D}_r^2 u_t(t)|_2\leq C(T)\big(\sqrt{t}|r^\frac{m}{2}u_{tt}(t)|_2+1\big)\leq C(T).
\end{equation*}
Then \eqref{4jir}, together with the above, yields
\begin{equation*} 
\sqrt{t} |r^{\frac{m}{2}}\phi^q\mathrm{D}_r^4u(t)|_2\leq C(T)\sqrt{t} |r^{\frac{m}{2}}\phi^{2\iota}\mathrm{D}_r^4u(t)|_2\leq C(T)\big(\sqrt{t} |r^{\frac{m}{2}}\phi^q\mathrm{D}_r^2u_t(t)|_2+1\big)\leq C(T).
\end{equation*}
\end{proof}

\section{Global Well-Posedness of  Regular Solutions with Far-Field Vacuum}\label{se46}
Based on the local well-posedness stated in Theorem \ref{thm-loc} and the corresponding global uniform estimates obtained in \S \ref{section-upper-density}--\S \ref{section-global2}, we are now ready to give the proof for Theorem \ref{th1-high}. We divide the proof into three steps.

\smallskip
\textbf{1. Global well-posedness of regular solutions.} 
First, according to Theorem \ref{thm-loc}, there exists a 
regular solution $(\rho, \boldsymbol{u})(t,\boldsymbol{x})$ of the Cauchy problem \eqref{eq:1.1}--\eqref{kelaoxiusi} with \eqref{eqs:CauchyInit}--\eqref{e1.3} in $[0,T_*]\times \mathbb{R}^n$ for some $T_*>0$, which takes form \eqref{2.8}.

Second, let $\overline{T}_*>0$ be the life span of $(\rho, \boldsymbol{u})(t,\boldsymbol{x})$, 
and let $T$ be any fixed time satisfying $0<T<\overline{T}_*$.
Then collecting the uniform {\it a priori} bounds obtained in Lemmas \ref{far-p-infty}, \ref{important2}, and \ref{l4.5}--\ref{Lemma6.14}, together with \eqref{tr} and Lemma \ref{lemma-initial}, yields that, for any  $t\in [0,T]$,
\begin{equation}\label{global-unifrom33}
\begin{aligned}
\|\rho(t)\|_{L^1\cap L^\infty}+\|\nabla\rho^{\gamma-1}(t)\|_{H^2}+\|\nabla\rho^{\delta-1}(t)\|_{L^\infty\cap D^{1,n}\cap D^2}+\|\nabla\rho^\frac{\delta-1}{2}(t)\|_{L^{2n}} \leq C(T),\\[3pt]
\|\boldsymbol{u}(t)\|_{H^3}+\|\boldsymbol{u}_t(t)\|_{H^1}+\int_0^t \big\|(\nabla^4\boldsymbol{u},\nabla^2\boldsymbol{u}_t,\boldsymbol{u}_{tt})\big\|_{L^2}^2 \,\mathrm{d}s\leq C(T),\\ 
\big\|\rho^\frac{\delta-1}{2} (\nabla\boldsymbol{u},\nabla\boldsymbol{u}_t)(t)\big\|_{L^2}+\big\|\rho^{\delta-1} (\nabla^2\boldsymbol{u}, \nabla^3\boldsymbol{u})(t)\big\|_{L^2}+\int_0^t \big\|\rho^{\delta-1}\nabla^4\boldsymbol{u}\big\|_{L^2}^2 \,\mathrm{d}s\leq C(T),\\ 
\sqrt{t}\big\|(\nabla^4\boldsymbol{u},\nabla^2\boldsymbol{u}_t,\boldsymbol{u}_{tt})(t)\big\|_{L^2}+\int_0^t s\|\nabla\boldsymbol{u}_{tt}\|_{L^2}^2\,\mathrm{d}s\leq C(T).
\end{aligned}
\end{equation}

Clearly, $\overline{T}_*\ge T_*$.  Next, we show  that $\overline{T}_*=\infty$. Otherwise, if $\overline{T}_*<\infty$, 
according to  the uniform {\it a priori} estimates \eqref{global-unifrom33} and the weak convergence arguments, for any sequence $\{t_k\}_{k=1}^\infty$ satisfying $0<t_k<\overline{T}_*$ and $t_k\to \overline{T}_*$ as $k\to \infty$, there exist a subsequence (still denoted by) $\{t_{k}\}_{k=1}^\infty$ and limits $(\rho,\boldsymbol{u}, \bar{\boldsymbol{f}}, \bar{\boldsymbol{\psi}})(\overline{T}_*,\boldsymbol{x})$ satisfying 
\begin{align*}
&\rho(\overline{T}_*,\boldsymbol{x})\in L^p(\mathbb{R}^n) \ \ \text{for any $p\in (1,\infty]$},\qquad
\boldsymbol{u}(\overline{T}_*,\boldsymbol{x})\in H^3(\mathbb{R}^n),\\
&\bar{\boldsymbol{f}}(\overline{T}_*,\boldsymbol{x})\in H^2(\mathbb{R}^n),\qquad \bar{\boldsymbol{\psi}}(\overline{T}_*,\boldsymbol{x})\in L^\infty(\mathbb{R}^n)\cap D^{1,n}(\mathbb{R}^n)\cap D^{2}(\mathbb{R}^n), 
\end{align*}
and, as $k\to\infty$ and for $p\in (1,\infty)$,
\begin{equation}\label{f2}
\begin{aligned}
\rho(t_k,\boldsymbol{x})\to \rho(\overline{T}_*,\boldsymbol{x}) \qquad &\text{weakly\  \, in } L^p(\mathbb{R}^n),\\
\boldsymbol{u}(t_{k},\boldsymbol{x})\to  \boldsymbol{u}(\overline{T}_*,\boldsymbol{x}) \qquad &\text{weakly\  \,  in } H^3(\mathbb{R}^n),\\
(\rho,\nabla\rho^{\delta-1})(t_{k},\boldsymbol{x})\to  (\rho,\bar{\boldsymbol{\psi}})(\overline{T}_*,\boldsymbol{x}) \qquad &\text{weakly* in } L^\infty(\mathbb{R}^n),\\
\nabla\rho^{\gamma-1}(t_k,\boldsymbol{x})\to \bar{\boldsymbol{f}}(\overline{T}_*,\boldsymbol{x}) \qquad &\text{weakly  \, in } H^1(\mathbb{R}^n),\\ 
\nabla^2\rho^{\delta-1}(t_{k},\boldsymbol{x})\to  \nabla\bar{\boldsymbol{\psi}}(\overline{T}_*,\boldsymbol{x}) \qquad &\text{weakly  \, in } L^{n}(\mathbb{R}^n),\\
\nabla^3\rho^{\delta-1}(t_{k},\boldsymbol{x})\to  \nabla^2\bar{\boldsymbol{\psi}}(\overline{T}_*,\boldsymbol{x}) \qquad &\text{weakly  \, in } L^{2}(\mathbb{R}^n).
\end{aligned}
\end{equation}
Then we claim 
\begin{equation}\label{claim103}
(\bar{\boldsymbol{f}},\bar{\boldsymbol{\psi}})(\overline{T}_*,\boldsymbol{x})=(\nabla\rho^{\gamma-1},\nabla \rho^{\delta-1})(\overline{T}_*,\boldsymbol{x})
\qquad\text{for {\it a.e.} $\boldsymbol{x}\in \mathbb{R}^n$}.
\end{equation}
 For simplicity, we prove that $\bar{\boldsymbol{f}}(\overline{T}_*,\boldsymbol{x})=\nabla\rho^{\gamma-1}(\overline{T}_*,\boldsymbol{x})$ for {\it a.e.} $\boldsymbol{x}\in \mathbb{R}^n$, since the rest of \eqref{claim103} can be derived analogously. 

First, due to $\eqref{global-unifrom33}_1$, there exists a subsequence such that 
\begin{equation}\label{phibar}
\rho^{\gamma-1}(t_k,\boldsymbol{x})\to  \bar\phi(\overline{T}_*,\boldsymbol{x}) \qquad \text{weakly* in $L^\infty(\mathbb{R}^n)$}
\end{equation}
for some limit  $\phi(\overline{T}_*,\boldsymbol{x}) \in L^\infty(\mathbb{R}^n)$.
On the other hand, it follows from $\eqref{global-unifrom33}_1$ and Lemma \ref{lemma-inf-rho} that,  
for any fixed $R>0$, 
\begin{equation}\label{lll-222}
\sup_{t\in [0,T]}\|\rho^{\gamma-1}(t)\|_{H^2(B_R)}\leq C(R,T),\quad \quad \inf_{(t,\boldsymbol{x})\in [0,T]\times B_R} \rho(t,\boldsymbol{x})\geq C^{-1}(R,T)
\end{equation}
for some constant $C(R,T)>0$ depending only on $(C_0,R,T)$,
where $B_R=\{\boldsymbol{x}\in \mathbb{R}^n:\,|\boldsymbol{x}|< R\}$.

Since $H^2(B_R)$ is compactly embedded in $C(\overline{B_R})$, by extracting a subsequence, 
there exists a limit $0<\bar{\bar\phi}(\overline{T}_*,\boldsymbol{x})\in C(\mathbb{R}^n)$ such that, 
for each $R>0$,
\begin{equation}\label{BR}
\rho^{\gamma-1}(t_k,\boldsymbol{x})\to \bar{\bar\phi}(\overline{T}_*,\boldsymbol{x}) \qquad
\text{uniformly on $B_R$} \ \  \text{as $k\to\infty$}.
\end{equation}
Clearly,  $0<\bar\phi(\overline{T}_*,\boldsymbol{x})=\bar{\bar\phi}(\overline{T}_*,\boldsymbol{x})\in L^\infty(\mathbb{R}^n)\cap C(\mathbb{R}^n)$ due to the uniqueness of limits in \eqref{phibar} and \eqref{BR}, which, together with $\eqref{global-unifrom33}_1$ and  \eqref{lll-222}--\eqref{BR},   yields
\begin{equation} \label{denstylimit}
\rho(t_k,\boldsymbol{x})\to \bar\phi^\frac{1}{\gamma-1}(\overline{T}_*,\boldsymbol{x}) \qquad\text{for any $\boldsymbol{x}\in\mathbb{R}^n$} \,\,\, \text{as $k\to\infty$}. 
\end{equation}
Then it follows from  the uniqueness of the limits in $\eqref{f2}_1$ and \eqref{denstylimit} that  $\bar\phi^\frac{1}{\gamma-1}(\overline{T}_*,\boldsymbol{x})=\rho (\overline{T}_*,\boldsymbol{x})$, {\it i.e.}, $\bar\phi(\overline{T}_*,\boldsymbol{x})=\rho^{\gamma-1} (\overline{T}_*,\boldsymbol{x})$. 

Next, it follows from $\eqref{f2}_4$, \eqref{phibar},  $\bar\phi(\overline{T}_*,\boldsymbol{x})=\rho^{\gamma-1} (\overline{T}_*,\boldsymbol{x})$, and the Lebesgue dominated convergence theorem that, for any $\zeta(\boldsymbol{x})\in C^\infty_{\rm c}(\mathbb{R}^n)$ and $i=1,\cdots\!,n$, 
\begin{align*}
&\int_{\mathbb{R}^n}\rho^{\gamma-1}(\overline{T}_*,\boldsymbol{x})\zeta_{x_i}(\boldsymbol{x})\,\mathrm{d}\boldsymbol{x}=\lim_{k\to\infty}\int_{\mathbb{R}^n} \rho^{\gamma-1}(t_k,\boldsymbol{x})\zeta_{x_i}(\boldsymbol{x})\,\mathrm{d}\boldsymbol{x}\\
&=-\lim_{k\to\infty}\int_{\mathbb{R}^n}(\rho^{\gamma-1})_{x_i}(t_k,\boldsymbol{x})\zeta(\boldsymbol{x})\,\mathrm{d}\boldsymbol{x}=-\int_{\mathbb{R}^n}\bar{f}_i(\overline{T}_*,\boldsymbol{x})\zeta(\boldsymbol{x})\,\mathrm{d}\boldsymbol{x}.
\end{align*}
This implies that $\rho^{\gamma-1}(\overline{T}_*,\boldsymbol{x})$ admits the weak derivatives $(\rho^{\gamma-1})_{x_i}(\overline{T}_*,\boldsymbol{x})=\bar{f}_i(\overline{T}_*,\boldsymbol{x})\in L^2(\mathbb{R}^n)$ for $i=1,\cdots\!, n$, so that $\nabla\rho^{\gamma-1}(\overline{T}_*,\boldsymbol{x})=\bar{\boldsymbol{f}}(\overline{T}_*,\boldsymbol{x})$ for {\it a.e.} $\boldsymbol{x}\in \mathbb{R}^n$. 

We now continue to prove that $\overline{T}_*=\infty$. We aim to show that functions 
$(\rho, \boldsymbol{u})(\overline{T}_*,\boldsymbol{x})$ satisfy all the initial assumptions 
given in Theorem \ref{thm-loc}, which consist of showing that 
$(\rho, \boldsymbol{u})(\overline{T}_*,\boldsymbol{x})$ are spherically symmetric 
and satisfy \eqref{id1-high}--\eqref{th78zx}. 

\smallskip
\textbf{1.1. $(\rho,\boldsymbol{u})(\overline{T}_*,\boldsymbol{x})$ are spherically symmetric.} 
It suffices to show that $\boldsymbol{u}(\overline{T}_*,\boldsymbol{x})$ is spherically symmetric, 
since the proof for $\rho(\overline{T}_*,\boldsymbol{x})$ can be derived analogously. 
To achieve this, it suffices to show that 
$\boldsymbol{u}(\overline{T}_*,\boldsymbol{x})=(\mathcal{O}^\top\boldsymbol{u})(\overline{T}_*,\mathcal{O}\boldsymbol{x})$ for any $\mathcal{O}\in \mathrm{SO}(n)$. 
Indeed, since $\boldsymbol{u}(t_k,\boldsymbol{x})$ is spherically symmetric for each $t_k$ and, by \eqref{f2}, $\boldsymbol{u}(t_k,\boldsymbol{x})$ converges to $\boldsymbol{u}(\overline{T}_*,\boldsymbol{x})$ weakly in $L^2(\mathbb{R}^n)$ as $k\to \infty$, it follows from the coordinate transformation that, 
for any vector function $\boldsymbol{\zeta}(\boldsymbol{x})\in L^2(\mathbb{R}^n)$,
\begin{align*}
&\int_{\mathbb{R}^n} \boldsymbol{u}(\overline{T}_*,\boldsymbol{x})\cdot \boldsymbol{\zeta}(\boldsymbol{x})\,\mathrm{d}\boldsymbol{x}=\lim_{k\to\infty}\int_{\mathbb{R}^n}\boldsymbol{u}(t_k,\boldsymbol{x})\cdot \boldsymbol{\zeta}(\boldsymbol{x})\,\mathrm{d}\boldsymbol{x}\\
&=\lim_{k\to\infty}\int_{\mathbb{R}^n}(\mathcal{O}^\top\boldsymbol{u})(t_k,\mathcal{O}\boldsymbol{x})\cdot \boldsymbol{\zeta}(\boldsymbol{x})\,\mathrm{d}\boldsymbol{x}=\lim_{k\to\infty}\int_{\mathbb{R}^n}\boldsymbol{u}(t_k,\boldsymbol{x})\cdot (\mathcal{O}\boldsymbol{\zeta})(\mathcal{O}^\top\boldsymbol{x})\,\mathrm{d}\boldsymbol{x}\\
&=\int_{\mathbb{R}^n}\boldsymbol{u}(\overline{T}_*,\boldsymbol{x})\cdot (\mathcal{O}\boldsymbol{\zeta})(\mathcal{O}^\top\boldsymbol{x})\,\mathrm{d}\boldsymbol{x}=\int_{\mathbb{R}^n}(\mathcal{O}^\top\boldsymbol{u})(\overline{T}_*,\mathcal{O}\boldsymbol{x})\cdot \boldsymbol{\zeta}(\boldsymbol{x})\,\mathrm{d}\boldsymbol{x},
\end{align*}
which implies that 
$\boldsymbol{u}(\overline{T}_*,\boldsymbol{x})=(\mathcal{O}^\top\boldsymbol{u})(\overline{T}_*,\mathcal{O}\boldsymbol{x})$ for any $\mathcal{O}\in \mathrm{SO}(n)$.

\smallskip
\textbf{1.2. $(\rho,\boldsymbol{u})(\overline{T}_*,\boldsymbol{x})$ satisfies \eqref{id1-high}--\eqref{th78zx}.} It can be directly checked that \eqref{id1-high} holds. Next, $\rho|_{t=\overline{T}_*}\in L^1(\mathbb{R}^n)$ follows from \eqref{global-unifrom33}, $\bar\phi^\frac{1}{\gamma-1}|_{t=\overline{T}_*}= \rho|_{t=\overline{T}_*}$ in \eqref{denstylimit}, and Lemma \ref{Fatou}:
\begin{equation*}
\int_{\mathbb{R}^n} \rho(\overline{T}_*,\boldsymbol{x})\,\mathrm{d}\boldsymbol{x}=\int_{\mathbb{R}^n} \liminf_{k\to\infty}\rho(t_{k},\boldsymbol{x})\,\mathrm{d}\boldsymbol{x}\le \liminf_{k\to\infty}\int_{\mathbb{R}^n} \rho(t_{k},\boldsymbol{x})\,\mathrm{d}\boldsymbol{x}\leq C(T).
\end{equation*} 
To show \eqref{th78zx}, from \eqref{f2}--\eqref{claim103}, it suffices to prove 
\begin{equation}
(\rho^\frac{\delta-1}{2} \nabla\boldsymbol{u},\rho^{\delta-1} \nabla^2\boldsymbol{u}, \rho^{\delta-1}\nabla^3\boldsymbol{u})(\overline{T}_*,\boldsymbol{x})\in L^{2}(\mathbb{R}^n).
\end{equation}
For brevity, we only sketch the proof of $\rho^{\delta-1}\nabla^3\boldsymbol{u} (\overline{T}_*,\boldsymbol{x})\in L^{2}(\mathbb{R}^n)$. By $\eqref{global-unifrom33}_3$, denoting by $\{t_k\}_{k=1}^\infty$ the same sequence as in \eqref{f2}, we see that $\rho^{\delta-1}\nabla^3\boldsymbol{u}(t_{k},\boldsymbol{x})\to  \boldsymbol{F}(\overline{T}_*,\boldsymbol{x})$ weakly in  $L^2(\mathbb{R}^n)$. On the other hand, based on
\begin{equation}\label{rrrrr}
\begin{aligned}
&\,(\rho^{\delta-1}\nabla^3\boldsymbol{u})(t_{k},\boldsymbol{x})-(\rho^{\delta-1}\nabla^3\boldsymbol{u})(\overline{T}_*,\boldsymbol{x})\\
&=\big(\rho^{\delta-1}(t_{k},\boldsymbol{x})-\rho^{\delta-1}(\overline{T}_*,\boldsymbol{x})\big)\nabla^3\boldsymbol{u}(t_k,\boldsymbol{x})+  \rho^{\delta-1}(\overline{T}_*,\boldsymbol{x})\big(\nabla^3\boldsymbol{u}(t_{k},\boldsymbol{x})-\nabla^3\boldsymbol{u}(\overline{T}_*,\boldsymbol{x})\big),    
\end{aligned}
\end{equation}
using $\eqref{f2}_2$ and  $\rho (t_k,\boldsymbol{x})\to \rho(\overline{T}_*,\boldsymbol{x})$ uniformly on $B_R$ as $k\to\infty$ due to \eqref{BR}, we can show that the right-hand side of \eqref{rrrrr} converges weakly to $\boldsymbol{0}$  in $L^2(B_R)$ for each $R>0$. Thus, the uniqueness of the limits implies that $\rho^{\delta-1}\nabla^3\boldsymbol{u}|_{t=\overline{T}_*}=\boldsymbol{F}|_{t=\overline{T}_*}\in L^2(\mathbb{R}^n)$.

\smallskip
\textbf{1.3.}
To sum up, we have shown that $(\rho,\boldsymbol{u})(\overline{T}_*,\boldsymbol{x})$ satisfies all the initial assumptions on the initial data of Theorem \ref{thm-loc}. 
Thus, by Theorem \ref{thm-loc}, $(\rho, \boldsymbol{u})(t,\boldsymbol{x})$ can be uniquely extended to the  regular solution of the Cauchy problem \eqref{eq:1.1}--\eqref{kelaoxiusi} with \eqref{eqs:CauchyInit}--\eqref{e1.3} in  $[0,\overline{T}_*+T_{**}]\times \mathbb{R}^n$ for some $T_{**}>0$. This  contradicts to the definition  of $\overline{T}_*<\infty$. Therefore $\overline{T}_*=\infty$.

\smallskip
\textbf{2. Proof of Theorem \ref{th1-high} (i).}  First, it follows  from \eqref{global-unifrom33} and the relation:  
$(\delta-1)\nabla \rho=\rho^{2-\delta} \nabla \rho^{\delta-1}$ that, for any finite  $T>0$,
\begin{equation*}
\|\nabla^2\rho\|_{L^2(\mathbb{R}^n)} \leq C_0\|\rho\|_{L^{6-4\delta}(\mathbb{R}^n)}^{3-2\delta}\|\nabla\rho^{\delta-1}\|_{L^\infty(\mathbb{R}^n)}^2 +C_0\|\rho\|_{L^{n^*(2-\delta)}(\mathbb{R}^n)}^{2-\delta}\!\|\nabla^2\rho^{\delta-1}\|_{L^n(\mathbb{R}^n)}\leq C(T),
\end{equation*}
which, along with the fact that $\rho \in C([0,T];L^1(\mathbb{R}^n))$, gives that, for any $t_0\in [0,T]$,
\begin{equation*} 
\begin{aligned}
\lim_{t\to t_0}\|\rho(t)-\rho(t_0)\|_{L^\infty(\mathbb{R}^n)}&\leq C_0\lim_{t\to t_0}\|\rho(t)-\rho(t_0)\|_{L^1(\mathbb{R}^n)}^\frac{4-n}{n+4}\|\nabla^2\rho(t)-\nabla^2\rho(t_0)\|_{L^\infty(\mathbb{R}^n)}^\frac{2n}{n+4}\\
&\leq C(T)\lim_{t\to t_0}\|\rho(t)-\rho(t_0)\|_{L^1(\mathbb{R}^n)}^\frac{4-n}{n+4}=0. 
\end{aligned}
\end{equation*}
This implies that $\rho\in C([0,T];C(\overline{\mathbb{R}^n}))$.

Next, since $\nabla\rho^{\delta-1}$ is a spherically symmetric vector function, it follows directly from Lemma \ref{Hk-Ck-vector} in Appendix \ref{improve-sobolev} and $\nabla\rho^{\delta-1}\in C([0,T];D^{1,n}(\mathbb{R}^n))$ that $\nabla\rho^{\delta-1}\in C([0,T];C(\overline{\mathbb{R}^n}))$, which thus yields $\nabla\rho\in C([0,T];C(\overline{\mathbb{R}^n}))$ due to  $\rho\in C([0,T];C(\overline{\mathbb{R}^n}))$ and $(\delta-1)\nabla \rho=\rho^{2-\delta} \nabla \rho^{\delta-1}$.

Moreover, $\boldsymbol{u}\in C([0,T];C^1(\overline{\mathbb{R}^n}))$ follows directly from $\boldsymbol{u}\in C([0,T];H^3(\mathbb{R}^n))$ and  Lemma \ref{ale1}, and $\rho_t\in C([0,T];C(\overline{\mathbb{R}^n}))$ owing to $\rho_t=-\boldsymbol{u}\cdot\nabla \rho -\rho \diver\boldsymbol{u}$ and $(\rho,\boldsymbol{u})\in C([0,T];C^1(\overline{\mathbb{R}^n}))$.

Finally, thanks to \eqref{global-unifrom33}, we have 
\begin{equation*}
(t\boldsymbol{u}_t,t\nabla^2\boldsymbol{u})\in L^\infty([0,T];H^2(\mathbb{R}^n)), \qquad
((t\boldsymbol{u}_t)_t,(t\nabla^2\boldsymbol{u})_t)\in L^2([0,T];L^2(\mathbb{R}^n)),
\end{equation*}
which, along with Lemma \ref{triple}, yields $(t\boldsymbol{u}_t,t\nabla^2\boldsymbol{u})\in C([0,T];W^{1,4}(\mathbb{R}^n))$. Therefore, it follows from the above  and Lemma \ref{ale1} that $(\boldsymbol{u}_t,\nabla^2\boldsymbol{u})\in C((0,T]\times \mathbb{R}^n)$.

\medskip
\textbf{3. Proof of Theorem \ref{th1-high}(ii)--(iii).}  We only give the proof for the 3-D case, and the 2-D case follows analogously. First, since $(\sqrt{\rho},\sqrt{\rho}\boldsymbol{u})\in C([0,T];L^2(\mathbb{R}^3))$, then $\rho\boldsymbol{u}\in C([0,T];L^1(\mathbb{R}^3))$ and hence $r^2\rho u\in C([0,T];L^1(I))$. 
Next, via the spherical coordinate transformations: $r\in I$, $\theta_1\in [0,2\pi]$, $\theta_2\in [0,\pi]$, and
\begin{equation*}
\boldsymbol{x}=(x_1,x_2,x_3)^\top=(r\cos\theta_1\sin\theta_2,r\sin\theta_1\sin\theta_2,r\cos\theta_2)^\top,
\end{equation*}
$\mathcal{P}(t)\equiv \boldsymbol{0}$ follows directly  from the spherical symmetry of $\rho\boldsymbol{u}$:
\begin{equation*}
\begin{aligned}
\mathcal{P}(t)&=\int_{\mathbb{R}^3} \rho\boldsymbol{u}\,\mathrm{d}\boldsymbol{x}=\int_0^\pi\int_0^{2\pi}\int_0^\infty \frac{(x_1,x_2,x_3)^\top}{r}(r^2\rho u)\sin\theta_2\,\mathrm{d}r\mathrm{d}\theta_1\mathrm{d}\theta_2\\
&=\Big(\int_0^\infty  r^2 \rho u \,\mathrm{d}r\Big)\int_0^\pi\int_0^{2\pi} (\cos\theta_1\sin^2\theta_2,\sin\theta_1\sin^2\theta_2,\cos\theta_2\sin\theta_2)^\top \, \mathrm{d}\theta_1\mathrm{d}\theta_2 =\boldsymbol{0}.     
\end{aligned}
\end{equation*}

The conservation of total mass can be simply derived by integrating the mass equation $\eqref{eq:1.1}_1$ over $[0,T]\times \mathbb{R}^n$ and using the fact that $\rho\boldsymbol{u} \in C([0,T];W^{1,1}(\mathbb{R}^n))$.  Finally, Theorem \ref{th1-high} (iii) is a directly consequence of Lemmas \ref{far-p-infty}, \ref{important2}, and \ref{lemma-inf-rho}.

\section{Global Well-Posedness of   Regular  Solutions  with Strictly Positive Initial Density}\label{nonvacuumfarfield}

This section is devoted to establishing the global well-posedness of spherically symmetric classical solutions of the Cauchy problem of system  \eqref{eq:1.1} with general smooth initial data and
strictly positive initial density, {\it i.e.}, $\inf_{\boldsymbol{x}\in \mathbb{R}^n} \rho_0(\boldsymbol{x})>0$ so that $\bar{\rho}>0$ in 
\eqref{e1.3}. Certainly, in this case, under the spherical coordinates, the Cauchy problem \eqref{eq:1.1}--\eqref{kelaoxiusi} with \eqref{eqs:CauchyInit}--\eqref{e1.3} in $[0,T]\times \mathbb{R}^n$ for some $T>0$ can be written as the following \textbf{IBVP} in $[0,T]\times I$:
\begin{equation}\label{e1.5hh-}
\begin{cases}
\displaystyle 
\rho_t+(\rho u)_r+\frac{m\rho u}{r}=0,\\[1pt]
\displaystyle
\rho u_t+\rho uu_r+A(\rho^\gamma)_r=2a_1\delta\big(\rho^\delta u_r+\frac{m\rho^\delta u}{r} \big)_r-\frac{2a_1 m(\rho^\delta)_r u}{r},\\[4pt]
\displaystyle
(\rho,u)|_{t=0}=(\rho_0,u_0) \qquad\quad\,\,\,\, \text{for $r\in I$},\\[3pt]
\displaystyle
u|_{r=0}=0 \qquad\qquad\qquad\qquad\text{for $t\in [0,T]$},\\[3pt]
\displaystyle
(\rho,u)\to \left(\bar\rho,0\right)  \qquad\qquad\text{as $r\to \infty\,$ \, for $t\in [0,T]$}.
\end{cases}
\end{equation}

\subsection{Local Well-Posedness in M-D Coordinates} 
To establish the local well-posedness when $\inf_{\boldsymbol{x}\in \mathbb{R}^n} \rho_0(\boldsymbol{x})>0$, we consider the following reformulated system in M-D coordinates: 
\begin{equation}\label{re}
\begin{cases}
\rho_t+\diver(\rho\boldsymbol{u})=0,\\[3pt]
\displaystyle\boldsymbol{u}_t+\boldsymbol{u}\cdot\nabla\boldsymbol{u}+\frac{A\gamma}{\gamma-1} \nabla\rho^{\gamma-1}+L\boldsymbol{u}=\frac{\delta}{\delta-1}\nabla \rho^{\delta-1} \cdot Q(\boldsymbol{u}),
\end{cases}
\end{equation}
where the operators $(L,Q)$ are defined in \eqref{operatordefinition}, \textit{i.e.},
\begin{equation*}
L\boldsymbol{u}=-a_1\Delta\boldsymbol{u}-(a_1+a_2)\nabla\diver \boldsymbol{u},\qquad Q(\boldsymbol{u})=2a_1 D(\boldsymbol{u})+a_2\diver\boldsymbol{u}\,\mathbb{I}_n.
\end{equation*}
Notice that the positive lower bound of $\rho$ can be obtained via the transport 
equation $\eqref{re}_1$ in some short time. Then, employing the methods used by Matsumura--Nishida \cite{MN} and Sundbye \cite{Sundbye2}, we can establish the corresponding local well-posedness results of the regular solutions (as defined in Definition \ref{cjk-po}) when $\inf_{\boldsymbol{x}\in \mathbb{R}^n} \rho_0(\boldsymbol{x})>0$. 
We omit the proof here for brevity and give the statement of this theorem:

\begin{thm}\label{zth2-po}
Let $n=2$ or $3$, $\bar\rho>0$ in \eqref{e1.3}, and let \eqref{canshu} hold. If the initial data $(\rho_0, \boldsymbol{u}_0)(\boldsymbol{x})$ are spherically symmetric and satisfy \eqref{id1-high-positive}, then there exists $T_*>0$ such that the Cauchy problem \eqref{eq:1.1}--\eqref{kelaoxiusi} with \eqref{eqs:CauchyInit}--\eqref{e1.3}  admits a unique regular solution $(\rho, \boldsymbol{u})(t, \boldsymbol{x})$ in $[0, T_*] \times \mathbb{R}^n$ satisfying \eqref{er2-high}, \eqref{3dclassical2}, and \eqref{decay-est-positive} with $T$ replaced by $T_*$, and
\begin{equation*}
\rho_{tt} \in C([0, T_*] ; L^2(\mathbb{R}^n)) \cap L^2([0,T_*];D^1(\mathbb{R}^n)).
\end{equation*}
Moreover, $(\rho, \boldsymbol{u})$ is spherically symmetric with form \eqref{2.8}.
\end{thm}

\smallskip
\subsection{Global Well-Posedness in  Spherically Symmetric  Coordinates}

The proof for global well-posedness is organized as follows: First, we establish the  uniform upper bound of $\rho$. Next, we establish the uniform $L^\infty(\mathbb{R}^n)$-estimate for the effective velocity, and then obtain the  uniform lower bound of $\rho$. Finally, based on these, we obtain global uniform estimates for regular solutions and then show the desired  global well-posedness of regular solutions.

This part of the proof can be completed by combining the arguments in \S\ref{section-upper-density}--\S\ref{se46} with the method presented in \cite[Section 10]{CZZ1}.  For simplicity, we only list the points that require noticeable modifications in the following three most important parts:
\begin{itemize}
\item[\S\ref{subsub1}] Global uniform upper bound of the density;
\item[\S\ref{subsub2}] Global uniform bound of the effective velocity;
\item[\S\ref{subsub3}] Global uniform lower bound of the density.
\end{itemize}

In what follows, unless otherwise specified, we adopt all the notations and defined quantities in \S\ref{section-upper-density}--\S\ref{section-nonformation}, and denote $C_0 \in[1, \infty)$ a generic constant depending only on $\left(\bar\rho,\rho_0, u_0, n, a_1, A, \gamma,\delta\right)$, and $C\left(\nu_1, \cdots\!, \nu_k\right) \in[1, \infty)$ a generic constant depending on $C_0$ and parameters $\left(\nu_1, \cdots\!, \nu_k\right)$, which may be different at each occurrence.

\subsubsection{Global uniform upper bound of $\rho$}\label{subsub1}
We can obtain the global uniform upper bound of $\rho$ by combining the arguments used in \S\ref{section-upper-density} and  \cite[Section 10]{CZZ1}.
The major modification lies in the establishment of $L^p(I)$-estimates for $(u,v)$, namely, Lemma \ref{lemma-uv-lp}. It should be mentioned that Lemma \ref{l4.3} also holds in this case and, for convenience, we present here a modified version of Lemmas \ref{energy-BD}--\ref{far-p-infty}, which is adapted to the case where $\bar\rho>0$ in \eqref{e1.3}.

\begin{lem}\label{lemma991}
Assume that $\bar{\rho}>0$ in \eqref{e1.3}. Then {\rm Lemma \ref{energy-BD}} in {\rm\S\ref{section-upper-density}} holds with $\rho^\gamma$ replaced by $j_\gamma(\rho):=(\rho^\gamma-\bar\rho^\gamma)-\gamma\bar\rho^{\gamma-1}(\rho-\bar\rho)$ and, for any $\omega>0$, there exists $C(\omega)>0$ such that
\begin{equation*}
|\chi^\sharp_\omega \rho(t)|_\infty\leq C(\omega)\qquad\text{for any $t\in [0,T]$}.
\end{equation*}
\end{lem}

Now, we give the proof for Lemma \ref{lemma-uv-lp} when $\bar\rho>0$ in \eqref{e1.3}.
\begin{proof}[Proof of Lemma \ref{lemma-uv-lp} when $\bar{\rho}>0$]

We divide the rest of the proof into two steps.

\smallskip
\textbf{1.} Following a similar argument in Steps 1--3 of the proof of Lemma \ref{lem-u-lp} when $\bar{\rho}=0$, we multiply $\eqref{e1.5hh-}_2$ by $r^m|u|^{p-2}u$, with $p\in [2,\tilde{p}_m(\delta))$, and integrate the resulting equality. Then we arrive at
\begin{equation}\label{4.172}
\frac{1}{p}\frac{\mathrm{d}}{\mathrm{d}t}\big|(r^m\rho)^\frac{1}{p}u\big|_p^p+2a_1 c_p\int_0^\infty r^m \rho^\delta |u|^{p-2}|\mathrm{D}_r u|^2\,\mathrm{d}r\leq -A\int_0^\infty r^m(\rho^\gamma)_r|u|^{p-2}u\,\mathrm{d}r:=G^*_1.
\end{equation}

To estimate $G^*_1$, we introduce the following smooth cut-off functions:
\begin{equation}\label{zzze}
\zeta=\zeta(r)\in C_{\mathrm{c}}^\infty([0,1]),\quad \zeta_r\leq 0,\quad  \zeta=1 \ \  \text{on $[0,\frac{1}{2}]$},\quad \zeta=0 \ \  \text{on $[\frac{3}{4},1]$},\quad \zeta^\sharp:=1-\zeta.
\end{equation}
Then it follows from integration by parts that
\begin{equation}\label{G1*}
\begin{aligned}
G^*_1&=-A\int_0^\infty \zeta r^m(\rho^\gamma)_r|u|^{p-2}u\,\mathrm{d}r-A\int_0^\infty \zeta^\sharp r^m(\rho^\gamma)_r|u|^{p-2}u\,\mathrm{d}r\\
&=A\int_0^\infty \zeta r^m \rho^\gamma |u|^{p-2}\big((p-1)u_r+\frac{m}{r}u\big)\,\mathrm{d}r+A\int_0^\infty \zeta_r r^m \rho^\gamma |u|^{p-2}u\,\mathrm{d}r\\
&\quad -A\int_0^\infty \zeta^\sharp r^m(\rho^\gamma)_r|u|^{p-2}u\,\mathrm{d}r:=\sum_{i=1}^3 G^*_{1,i}.
\end{aligned}
\end{equation}

For $G^*_{1,1}$, we obtain from the H\"older and Young inequalities that 
\begin{equation*} 
\begin{aligned}
G^*_{1,1}&\leq C(p)\big|(r^m\rho^\delta)^{\frac12}|u|^{\frac{p-2}{2}}\mathrm{D}_r u\big|_2\big|\chi_1^\flat r^{\frac{m}{2}}\rho^{\gamma-\frac{\delta}{2}}|u|^{\frac{p-2}{2}}\big|_2\\
&\leq  \frac{a_1 c_p}{32}\big|(r^m\rho^\delta)^{\frac12}|u|^{\frac{p-2}{2}}\mathrm{D}_r u\big|_2^2+C(p)\big|\chi_1^\flat r^\frac{p+m-2}{p\gamma-p\delta+\delta} \rho\big|_{p\gamma-p\delta+\delta}^\frac{2p\gamma-2p\delta+2\delta}{p}\big|(r^{m-2}\rho^\delta)^\frac{1}{p}u\big|_{p}^{p-2}\\
&\leq \frac{a_1 c_p}{16}\big|(r^m\rho^\delta)^{\frac12}|u|^{\frac{p-2}{2}}\mathrm{D}_r u\big|_2^2+C(p) \underline{\big|\chi_1^\flat r^\frac{p+m-2}{p\gamma-p\delta+\delta} \rho\big|_{p\gamma-p\delta+\delta}^{p\gamma-p\delta+\delta}}_{=I_{3,1}},
\end{aligned}
\end{equation*}
where $I_{3,1}$ is defined in \eqref{J1}. Hence, following the same argument as in Steps 4.1--4.2 in the proof of Lemma \ref{lem-u-lp} when $\bar{\rho}=0$, we obtain that, for any $\varepsilon\in (0,1)$,
\begin{equation} 
G^*_{1,1}\leq \frac{a_1 c_p}{8}\big|(r^m\rho^\delta)^{\frac12}|u|^{\frac{p-2}{2}}\mathrm{D}_r u\big|_2^2+\varepsilon \big|(r^m\rho^{\gamma-\delta+1})^\frac{1}{p}v\big|_p^p+C(p,\varepsilon).
\end{equation}

For $G^*_{1,2}$--$G^*_{1,3}$, we see from \eqref{V-expression}, Lemma \ref{lemma991}, and the H\"older and Young inequalities that
\begin{equation*} 
\begin{aligned}
G^*_{1,2}&\leq C_0\int_\frac{1}{2}^\frac{3}{4}  r^m \rho^\gamma |u|^{p-1}\,\mathrm{d}r\leq C(p)|\chi_\frac{1}{2}^\sharp\rho|_\infty^{\gamma-1+\frac{1}{p}}\big|(r^m\rho)^\frac{1}{p}u\big|_p^{p-1}\leq C(p)+\big|(r^m\rho)^\frac{1}{p}u\big|_p^p,\\
G^*_{1,3}&\leq C_0 \int_\frac{1}{2}^\infty\!\! r^m \rho^{\gamma-\delta+1}|v-u||u|^{p-1}\mathrm{d}r \leq  C_0|\chi_\frac{1}{2}^\sharp \rho|_\infty^{\gamma-\delta} \big|(r^m\rho)^\frac{1}{p}(u,v)\big|_p^p\leq  C_0 \big|(r^m\rho)^\frac{1}{p}(u,v)\big|_p^p.
\end{aligned}
\end{equation*}

Collecting \eqref{4.172}, \eqref{G1*}, and the above, we obtain that, for any $p\in [2,\tilde{p}_m(\delta))$ and $\varepsilon\in (0,1)$,
\begin{equation}\label{dtdtbu}
\begin{aligned}
&\,\frac{1}{p}\frac{\mathrm{d}}{\mathrm{d}t}\big|(r^m\rho)^\frac{1}{p}u\big|_p^p+ a_1 c_p\big|(r^m \rho^\delta)^\frac{1}{2} |u|^\frac{p-2}{2} \mathrm{D}_r u\big|_2^2 \\
&\leq C(p,\varepsilon)\big(\big|(r^m\rho)^\frac{1}{p}(u,v)\big|_p^p+1\big) +\varepsilon \big|(r^m\rho^{\gamma-\delta+1})^\frac{1}{p}v\big|_p^p.   
\end{aligned}
\end{equation}

\smallskip
\textbf{2.} Repeating the same proof of Lemma \ref{lem-v-lp}, we can still obtain \eqref{dt-v-p} when $\bar{\rho}>0$. Therefore, combining \eqref{dtdtbu} with \eqref{dt-v-p},  choosing  $\varepsilon>0$ sufficiently small, and then applying the Gr\"onwall inequality to the resulting inequality, we obtain the desired conclusion.
\end{proof}

\subsubsection{Global uniform bound of $v$}\label{subsub2}

Once we prove Lemma \ref{rho u-L2} when $\bar{\rho}>0$, the $L^\infty(I)$-estimate for $v$ can be derived by following the same method in \S \ref{section-effective} without further modifications. Hence, we only sketch some revisions to the proof of Lemma \ref{rho u-L2} when $\bar\rho>0$.
\begin{proof}[Proof of Lemma \ref{rho u-L2} when $\bar{\rho}>0$ in \eqref{e1.3}]
The main difference is in the estimate for the following integral:
\begin{equation*}
\mathrm{G}_2^*:=-A\int_0^\infty (\rho^\gamma)_r u\,\mathrm{d}r,
\end{equation*}
where we have used integration by parts to treat it in the proof of Lemma \ref{rho u-L2} in \S \ref{section-effective}; see \eqref{503}.

Here, to overcome the estimate away from the origin, we employ the cut-off function $\zeta$ in \eqref{zzze} and rewrite $\mathrm{G}_2^*$ as
\begin{equation*}
|\mathrm{G}_2^* |\le A\Big|\int_0^\infty \zeta(\rho^\gamma)_r u\,\mathrm{d}r\Big|+A\Big|\int_0^\infty \zeta^\sharp (\rho^\gamma)_r u\,\mathrm{d}r\Big|.    
\end{equation*}
The first term on the right-hand side of the above can be handled by using integration by parts and the same argument as in the proof of Lemma \ref{rho u-L2} in \S \ref{section-effective}:
\begin{equation*}
A\Big|\int_0^\infty \zeta(\rho^\gamma)_r u\,\mathrm{d}r\Big|
\leq C(\varepsilon,T)+\varepsilon|\rho^\frac{\delta}{2}u_r|_2^2.
\end{equation*}
For the second one, it follows from Lemma \ref{lemma991} that
\begin{equation*}
\begin{aligned}
&\,A\Big|\int_0^\infty \zeta^\sharp  (\rho^\gamma)_r u\,\mathrm{d}r\Big| 
\leq C_0\int_\frac{1}{2}^\infty  \rho^{\gamma-1}|\rho_r| u\,\mathrm{d}r\\
&\leq C_0|\chi_\frac{1}{2}^\sharp r^{-m}|_\infty|\chi_\frac{1}{2}^\sharp \rho|_\infty^\frac{\gamma-\delta}{2}\big|r^\frac{m}{2}\rho^\frac{\gamma+\delta-3}{2}\rho_r\big|_2|(r^m\rho)^\frac{1}{2}u|_2\leq C_0+\big|r^\frac{m}{2}\rho^\frac{\gamma+\delta-3}{2}\rho_r\big|_2^2.  
\end{aligned}
\end{equation*}

Collecting the above estimates, we arrive at
\begin{equation*}
|\mathrm{G}_2^*|\leq C(\varepsilon,T)+\varepsilon|\rho^\frac{\delta}{2}u_r|_2^2+\big|r^\frac{m}{2}\rho^\frac{\gamma+\delta-3}{2}\rho_r\big|_2^2,
\end{equation*}
and we can finally close the estimate to complete the proof of Lemma \ref{rho u-L2} when $\bar{\rho}>0$ by using the $L^2([0,T];L^2(I))$-estimate of $r^\frac{m}{2}\rho^\frac{\gamma+\delta-3}{2}\rho_r$ in Lemma \ref{lemma991}.
\end{proof}

\subsubsection{Global uniform lower bound of $\rho$}\label{subsub3}

Since we already show the global uniform $L^\infty(I)$-estimates for $(\rho,v)$, we can follow the same method provided in \S \ref{sub-refine} to obtain Lemmas \ref{cru4}--\ref{ele} when $\bar{\rho}>0$. Here, we summarize some related estimates for use in Lemma \ref{lemma-lowerbound} below. 
\begin{lem}\label{lemmauseful}
When $\bar{\rho}>0$ in \eqref{e1.3}, there exists a constant $C(T)>0$ such that
\begin{equation*}
|(\rho,v)(t)|_\infty+|u(t)|_2+\int_0^t \big(|\rho^\frac{\delta-1}{2}\mathrm{D}_r u|_2^2+|u|_\infty^4\big)\,\mathrm{d}s\leq C(T).
\end{equation*}
\end{lem}

Now, based on Lemma \ref{lemmauseful}, we establish the uniform lower bound for $\rho$.
\begin{lem}\label{lemma-lowerbound}
Let $\bar{\rho}>0$ in \eqref{e1.3}.
For any $(t,r)\in[0,T]\times I$,
\begin{equation*}
\rho(t,r)\geq C(T)^{-1},
\end{equation*}
where $C(T)\geq 1$ is a constant depending only on $(T,\bar\rho,\rho_0,u_0,n,a_1,\delta,\gamma,A)$.
\end{lem}

\begin{proof}
Set $\hat\rho=\rho/\bar\rho$ and $\hat\rho_0=\hat\rho|_{t=0}$. First, it follows from \eqref{V-expression} and Lemmas \ref{lemmauseful} and \ref{ale1} that
\begin{align*}
\begin{aligned}
|\log\hat\rho|_\infty^3
&\leq C_0\big\||\log\hat\rho|^3\big\|_{1,1}\leq C_0\int_0^\infty\big(|\log\hat\rho|^3+|\log\hat\rho|^2\rho^{1-\delta}(|u|+|v|)\big)\,\mathrm{d}r\\
&\leq C_0|\log\hat\rho|_\infty|
\log\hat\rho|_2^2+C_0|\rho|_\infty^{1-\delta}|v|_\infty|
\log\hat\rho|_2^2 +C_0 |\rho|_\infty^{1-\delta}|\log\hat\rho|_\infty|\log\hat\rho|_2|u|_2\\
&\leq C(T)\big((|\log\hat\rho|_\infty+1)
|\log\hat\rho|_2^2+ |\log\hat\rho|_\infty|\log\hat\rho|_2\big),
\end{aligned}
\end{align*}
which, along with the Young inequality, leads to
\begin{equation}\label{wuqiong-log-42}
|\log\hat\rho|_\infty\leq C(T)\big(|\log\hat\rho|_2+1\big).   
\end{equation}

Next, rewrite $\eqref{e1.5hh-}_1$ as
\begin{equation*} 
(\log\hat\rho)_t+u(\log\rho)_r+\big(u_r+\frac{m}{r}u\big)=0.
\end{equation*}
Multiplying the above by $\log\hat\rho$ and integrating over $I$, we obtain from \eqref{V-expression}, Lemma \ref{lemmauseful}, and the H\"older and Young inequalities that
\begin{equation*}
\begin{aligned}
\frac{\mathrm{d}}{\mathrm{d}t}|\log\hat\rho|_2^2&\leq |u|_2|(\log\rho)_r|_\infty|\log\hat\rho|_2
+C_0|\mathrm{D}_r u|_2|\log\hat\rho|_2\\
&\leq C_0\big(|u|_2^2|\rho|_\infty^{2-2\delta}|(u,v)|_\infty^2+|\rho|_\infty^{1-\delta}|\rho^\frac{\delta-1}{2}\mathrm{D}_r u|_2^2\big)+|\log\hat\rho|_2^2\\
&\leq C(T)\big(1+|u|_\infty^2+|\rho^\frac{\delta-1}{2}\mathrm{D}_r u|_2^2\big)+|\log\hat\rho|_2^2,
\end{aligned}
\end{equation*}
which, along with the Gr\"onwall inequality, leads to
\begin{equation}\label{324-po}
|\log\hat\rho|_2\leq C(T)\big(|\log\hat\rho_0|_2+1\big).
\end{equation}

For the $L^2(I)$-boundedness of $\log\hat\rho_0$, observing that $\rho_0(r)\to \bar\rho$ as $r\to \infty$, one can find a sufficiently large $R_0>0$, depending only on $\bar\rho$, such that, for all $r\in [R_0,\infty)$, 
\begin{equation*} 
\chi_{R_0}^\sharp|\log\hat\rho_0|
=\chi_{R_0}^\sharp|\log(1+(\hat\rho_0-1))|\leq 2\chi_{R_0}^\sharp|\hat\rho_0-1|
\leq C_0\chi_{R_0}^\sharp|\rho_0-\bar\rho|.
\end{equation*}
For such $R_0>0$, Lemma \ref{lemma-initial}, together with $0<\inf_{r\in I}\rho_0\leq \rho_0\leq |\rho_0|_\infty$, implies 
\begin{equation*} 
\begin{aligned}
|\log\hat\rho_0|_2&\leq |\chi_{R_0}^\flat\log\hat\rho_0|_2
+|\chi_{R_0}^\sharp\log\hat\rho_0|_2\\
&\leq C_0\sqrt{R_0}\big(|\log\rho_0|_\infty+|\log\bar\rho|\big)+C_0|\chi_{R_0}^\sharp (\rho_0-\bar\rho)|_2\\
&\leq C_0+C_0|r^\frac{m}{2} (\rho_0-\bar\rho)|_2\leq C_0+C_0\|\rho_0-\bar\rho\|_{L^2}\leq C_0.
\end{aligned}
\end{equation*}

Thus, substituting above into \eqref{324-po}, together with \eqref{wuqiong-log-42}, gives
\begin{equation*}
|\log\hat\rho|_\infty\leq C(T)\big(|\log\hat\rho|_2 +1\big)\leq C(T)\big(|\log\hat\rho_0|_2+1\big)\leq C(T),
\end{equation*}
which implies the desired result.
\end{proof}

\section{Local Well-Posedness  of Regular Solutions with Far-Field Vacuum}
\label{section-local-regular}

This section is devoted to establishing the local well-posedness of regular solutions of the Cauchy problem \eqref{eq:1.1}--\eqref{kelaoxiusi} with \eqref{eqs:CauchyInit}--\eqref{e1.3}  when $\bar\rho=0$. We first give the proof of Theorem \ref{th1} and,  at the end of this section, we show that this theorem indeed implies Theorem \ref{thm-loc}.

\subsection{Linearization}
Let $T>0$. In order to solve the nonlinear problem  \eqref{eq:cccq}{\rm--}\eqref{sfanb1}, we consider the following  linearized problem for $(\phi^{\epsilon,\eta}, \boldsymbol{u}^{\epsilon,\eta},  h^{\epsilon,\eta})$  in $[0,T]\times \mathbb{R}^n$ ($n=2$, $3$):
\begin{equation}\label{ln}
\begin{cases}
\phi^{\epsilon,\eta}_t+\boldsymbol{w}\cdot\nabla\phi^{\epsilon,\eta}+(\gamma-1)\phi^{\epsilon,\eta} \diver\boldsymbol{w}=0,\\[3pt]
\boldsymbol{u}^{\epsilon,\eta}_t+\boldsymbol{w}\cdot \nabla \boldsymbol{w}+\nabla\phi^{\epsilon,\eta}+a \sqrt {(h^{\epsilon,\eta})^2+\epsilon^2} L\boldsymbol{u}^{\epsilon,\eta}=\boldsymbol{\psi}^{\epsilon,\eta} \cdot Q(\boldsymbol{w}),\\[3pt]
h^{\epsilon,\eta}_t+\boldsymbol{w}\cdot \nabla h^{\epsilon,\eta}+(\delta-1)g\diver\boldsymbol{w}=0,\\[3pt]
(\phi^{\epsilon,\eta},\boldsymbol{u}^{\epsilon,\eta},h^{\epsilon,\eta})|_{t=0}=(\phi^\eta_0,\boldsymbol{u}^\eta_0,h^\eta_0)=(\phi_0+\eta,\boldsymbol{u}_0,(\phi_0+\eta)^{2\iota})\qquad  \text{for $\boldsymbol{x}\in\mathbb{R}^n$},\\[3pt]
(\phi^{\epsilon,\eta},\boldsymbol{u}^{\epsilon,\eta},h^{\epsilon,\eta})\to (\eta,\boldsymbol{0},\eta^{2\iota}) \qquad \text{as $|\boldsymbol{x}|\to \infty$} \quad \text{for $t\in [0,T]$},
\end{cases}
\end{equation}
where $\epsilon$ and $\eta$ are  positive constants,
\begin{equation}\label{psih}
\boldsymbol{\psi}^{\epsilon,\eta}=\frac{a\delta}{\delta-1}\nabla h^{\epsilon,\eta},
\end{equation}
$\boldsymbol{w}=(w_1,\cdots\!,w_n)^{\top}$ is a given spherically symmetric vector function of the form: $
\boldsymbol{w}(t,\boldsymbol{x})=w(t,|\boldsymbol{x}|)\frac{\boldsymbol{x}}{|\boldsymbol{x}|}$,
and  $g>0$ is a given spherically symmetric  function. Moreover, $\boldsymbol{w}$ and $g$ are independent of ($\epsilon, \eta$) and satisfy  
\begin{equation*}
\begin{aligned}
&\boldsymbol{w}(0,\boldsymbol{x})=\boldsymbol{u}_0(\boldsymbol{x}),\quad g(0,\boldsymbol{x})=h_0(\boldsymbol{x}),\quad g\in L^\infty([0,T]\times \mathbb{R}^n)\cap C([0,T]\times \mathbb{R}^n),\\
& \nabla g\in C([0,T]; D^{1,n}(\mathbb{R}^n)\cap D^2(\mathbb{R}^n)),\quad  g_t\in C([0,T];H^2(\mathbb{R}^n)),\\
&\nabla g_{tt}\in L^2([0,T];L^2(\mathbb{R}^n)),\quad  \boldsymbol{w}\in C([0,T];H^3(\mathbb{R}^n))\cap L^2([0,T];H^4(\mathbb{R}^n)),\\ 
&t^{\frac{1}{2}}\boldsymbol{w}\in L^\infty([0,T];D^4(\mathbb{R}^n)),\quad \boldsymbol{w}_t\in C([0,T];H^1(\mathbb{R}^n))\cap L^2([0,T];D^2(\mathbb{R}^n)),\\
& \boldsymbol{w}_{tt}\in L^2([0,T];L^2(\mathbb{R}^n)),\quad 
t^{\frac{1}{2}}\boldsymbol{w}_t\in L^\infty([0,T];D^2(\mathbb{R}^n))\cap L^2([0,T];D^3(\mathbb{R}^n)),\\
&t^{\frac{1}{2}}\boldsymbol{w}_{tt}\in L^\infty([0,T];L^2(\mathbb{R}^n))\cap L^2([0,T];D^1(\mathbb{R}^n)).
\end{aligned}
\end{equation*}
For the sake of clarity, we declare that the functions $(\phi_0,\boldsymbol{u}_0)$  shown above in problem \eqref{ln} are exactly the ones in  \eqref{sfana1}.

It follows from the classical theory \cite{evans, oar,amj}, at least when $\eta$
and $\epsilon$ are positive, that the following global well-posedness of \eqref{ln} in $[0,T]\times \mathbb{R}^n$ holds.
\begin{lem}\label{ls}
Let \eqref{canshu} hold, $\epsilon>0$, and $\eta>0$. 
If  the initial data  $(\phi_0, \boldsymbol{u}_0)(\boldsymbol{x})$ are spherically symmetric and satisfy \eqref{th78qq}{\rm--}\eqref{th78zxq}, then, for any  $T>0$,  there exists a unique classical solution $(\phi^{\epsilon,\eta},\boldsymbol{u}^{\epsilon,\eta},h^{\epsilon,\eta})$ in $[0,T]\times \mathbb{R}^n$ of the Cauchy problem $\eqref{ln}$ such that
\begin{align}
&\inf_{[0,T]\times \mathbb{R}^n}\phi(t,x)>0,\quad (\nabla \phi^{\epsilon,\eta},\phi^{\epsilon,\eta}_t)\in C([0,T];H^2(\mathbb{R}^n)), \notag\\
&h^{\epsilon,\eta}\in L^\infty([0,T]\times \mathbb{R}^n)\cap C([0,T]\times \mathbb{R}^n),\quad \nabla h^{\epsilon,\eta}\in C([0,T];H^2(\mathbb{R}^n)),\notag\\
&h^{\epsilon,\eta}_t\in C([0,T];H^2(\mathbb{R}^n)),\quad \boldsymbol{u}^{\epsilon,\eta}\in C([0,T];H^3(\mathbb{R}^n))\cap L^2([0,T];H^4(\mathbb{R}^n)),\label{linearregularity}\\
& \boldsymbol{u}^{\epsilon,\eta}_t\in C([0,T];H^1(\mathbb{R}^n))\cap L^2([0,T];D^2(\mathbb{R}^n)),\quad \boldsymbol{u}^{\epsilon,\eta}_{tt}\in L^2([0,T];L^2(\mathbb{R}^n)),\notag\\
&t^{\frac{1}{2}}\boldsymbol{u}^{\epsilon,\eta}\in L^\infty([0,T];D^4(\mathbb{R}^n)), \quad t^{\frac{1}{2}}\boldsymbol{u}^{\epsilon,\eta}_t\in L^\infty([0,T];D^2(\mathbb{R}^n))\cap L^2([0,T];D^3(\mathbb{R}^n)),\notag\\
& t^{\frac{1}{2}}\boldsymbol{u}^{\epsilon,\eta}_{tt}\in L^\infty([0,T];L^2(\mathbb{R}^n))\cap L^2([0,T];D^1(\mathbb{R}^n)).\notag
\end{align}
Moreover, $(\phi^{\epsilon,\eta},\boldsymbol{u}^{\epsilon,\eta},h^{\epsilon,\eta},\boldsymbol{\psi}^{\epsilon,\eta})$ is  spherically symmetric and takes the form
\begin{equation}\label{2..}
\begin{aligned}
(\phi^{\epsilon,\eta},h^{\epsilon,\eta})(t, \boldsymbol{x})&=(\phi^{\epsilon,\eta},\phi^{\epsilon,\eta})(t,|\boldsymbol{x}|),\quad 
(\boldsymbol{u}^{\epsilon,\eta},\boldsymbol{\psi}^{\epsilon,\eta})(t, \boldsymbol{x})=(u^{\epsilon,\eta}, \psi^{\epsilon,\eta})(t,|\boldsymbol{x}|) \frac{\boldsymbol{x}}{|\boldsymbol{x}|},
\end{aligned}
\end{equation} 
where  $\psi^{\epsilon,\eta}(t,r)=\frac{a\delta}{\delta-1}(h^{\epsilon,\eta})_r(t,r)$.
\end{lem}
\begin{proof} 
The well-posedness follows from the classical theories for the transport equation and the parabolic equations when $\eta$
and $\epsilon$ are positive, which can be found in  \cite{oar,amj}. We only show that the solution $(\phi^{\epsilon,\eta},\boldsymbol{u}^{\epsilon,\eta},h^{\epsilon,\eta},\boldsymbol{\psi}^{\epsilon,\eta})$ is spherically symmetric and takes form \eqref{2..}. For simplicity, in the rest of the proof, we drop the superscript $\epsilon$ and $\eta$ in $(\phi^{\epsilon,\eta},\boldsymbol{u}^{\epsilon,\eta},h^{\epsilon,\eta},\boldsymbol{\psi}^{\epsilon,\eta})$, and let  $(\phi,\boldsymbol{u},h)$ be the unique solution of  \eqref{ln} in $[0,T]\times\mathbb{R}^n$ and $(\boldsymbol{\psi},h)$ satisfy \eqref{psih}. 

Indeed, since $(\phi_0^\eta,\boldsymbol{u}_0^\eta,h_0^\eta)$ is spherically symmetric, then
\begin{equation*} 
\phi_0^\eta(\boldsymbol{x})=\phi_0^\eta(\mathcal{O}\boldsymbol{x}), \ \  \boldsymbol{u}_0^\eta(\boldsymbol{x})= \mathcal{O}^\top\boldsymbol{u}_0^\eta(\mathcal{O}\boldsymbol{x}), \ \ h_0^\eta(\boldsymbol{x})= h_0^\eta(\mathcal{O}\boldsymbol{x})\qquad\text{for any  $\mathcal{O}\in \mathrm{SO}(n)$}.
\end{equation*}
Thus, denoting
\begin{equation*}
\hat\phi(t,\boldsymbol{x})=\phi(t,\mathcal{O}\boldsymbol{x}),\quad \hat{\boldsymbol{u}}(t,\boldsymbol{x})=\mathcal{O}^{\top}\boldsymbol{u}(t,\mathcal{O}\boldsymbol{x}),\quad \hat{h}(t,\boldsymbol{x})=h(t,\mathcal{O}\boldsymbol{x}),
\end{equation*}
we obtain from  $|\mathcal{O}\boldsymbol{x}|=|\boldsymbol{x}|$ and \eqref{e1.3} that
\begin{equation*}
\begin{aligned}
&(\hat\phi,\hat{\boldsymbol{u}},\hat{h})|_{t=0}=(\hat\phi_0,\hat{\boldsymbol{u}}_0,\hat{h}_0)=(\phi_0^\eta,\boldsymbol{u}_0^\eta,h_0^\eta),\\
&\hat{\boldsymbol{u}}(t,\boldsymbol{x})|_{|\boldsymbol{x}|=0} =\mathcal{O}^{\top}\boldsymbol{u}(t, \mathcal{O}\boldsymbol{x})|_{|\boldsymbol{x}|=0} =\boldsymbol{0} \qquad \text{for $t\in [0,T]$},\\
&(\hat\phi,\hat{\boldsymbol{u}},\hat h)\to (\eta,\boldsymbol{0},\eta^{2\iota}) \qquad 
 \text{as $|\boldsymbol{x}|\to \infty$} \hspace{7mm}  \text{for $t\in [0,T]$}.
\end{aligned}
\end{equation*}
 
Next, we show that $(\hat\phi,\hat{\boldsymbol{u}},\hat{h})$ is also a solution of \eqref{ln}. For convenience, in what follows, we adopt the Einstein summation convention that 
any index that appears exactly twice in a term is summed over.
For any $\mathcal{O}=(\mathcal{O}_{lk})_{n\times n}\in \mathrm{SO}(n)$, let $\boldsymbol{x}=(x_1,\cdots\!, x_n)^\top$ and $\boldsymbol{y}=(y_1,\cdots\!, y_n)^\top$ satisfying  $\boldsymbol{y}=\mathcal{O}\boldsymbol{x}$. Clearly, $y_l=\mathcal{O}_{lk}x_k$, $\mathcal{O}_{ki}\mathcal{O}_{kj}=\delta_{ij}$, and  $\boldsymbol{w}(t,\boldsymbol{x})=\mathcal{O}^\top \boldsymbol{w}(t,\boldsymbol{y})$. 

To show that $\hat\phi$ satisfies $\eqref{ln}_1$, by the following direct calculations: 
\begin{align*}
(\hat\phi_t+\boldsymbol{w}\cdot\nabla\hat\phi)(t,\boldsymbol{x})&=(\phi_t+\mathcal{O}_{ki}\mathcal{O}_{li}w_l\partial_{y_k}\phi) (t,\boldsymbol{y})=(\phi_t+\boldsymbol{w}\cdot\nabla_{\boldsymbol{y}}\phi)(t,\boldsymbol{y}),\\
(\hat\phi\diver\boldsymbol{w})(t,\boldsymbol{x})&=(\phi\mathcal{O}_{kj}\mathcal{O}_{lj} \partial_{y_k}w_l)(t,\boldsymbol{y})=(\phi\diver_{\boldsymbol{y}}\boldsymbol{w})(t,\boldsymbol{y}),
\end{align*}
we find that $\hat\phi$ satisfies  $\eqref{ln}_1$ indeed. Clearly, repeating the same calculation, we can also show that $\hat h$ satisfies $\eqref{ln}_3$.

Next, we show that $(\hat\phi,\hat{\boldsymbol{u}},\hat h,\hat{\boldsymbol{\psi}})$ satisfies $\eqref{ln}_2$. 
Notice that
\begin{align*}
&\begin{aligned}
(\hat{\boldsymbol{u}}_t+\boldsymbol{w}\cdot\nabla\boldsymbol{w})_i(t,\boldsymbol{x})
&=(\mathcal{O}_{qi} (\partial_tu_q+\mathcal{O}_{lj}\mathcal{O}_{kj} w_l \partial_{y_k}w_q)) (t,\boldsymbol{y}) \\
&=(\mathcal{O}_{qi}(\partial_t u_q+ w_k \partial_{y_k}w_q)) (t,\boldsymbol{y})=(\mathcal{O}^\top (\boldsymbol{u}_t+\boldsymbol{w}\cdot\nabla_{\boldsymbol{y}} \boldsymbol{w}))_i(t,\boldsymbol{y}),\\
\partial_i\hat\phi(t,\boldsymbol{x})&=(\mathcal{O}_{ki}  \partial_{y_k}\phi) (t,\boldsymbol{y})=(\mathcal{O}^\top\nabla_{\boldsymbol{y}}\phi)_i(t,\boldsymbol{y}),\\
\Delta \hat{u}_i(t,\boldsymbol{x})&=(\mathcal{O}_{lj} \mathcal{O}_{kj} \mathcal{O}_{qi}\partial_{y_ly_k}u_q)(t,\boldsymbol{y}) =(\mathcal{O}^\top \Delta_{\boldsymbol{y}} \boldsymbol{u})_i(t,\boldsymbol{y}),
\end{aligned}\\
&  \, \,\,\, \, \,\begin{aligned}
\partial_i\diver\hat{\boldsymbol{u}}(t,\boldsymbol{x})&=(\mathcal{O}_{li}\mathcal{O}_{kj}\mathcal{O}_{qj}\partial_{y_ly_k}u_q)(t,\boldsymbol{y}) =(\mathcal{O}^\top \nabla_{\boldsymbol{y}}\diver_{\boldsymbol{y}}\boldsymbol{u})_i (t,\boldsymbol{y}),\\
2(\hat{\boldsymbol{\psi}}\cdot D(\boldsymbol{w}))_i(t,\boldsymbol{x})
&=(\mathcal{O}_{kj} \psi_k (\mathcal{O}_{lj}\mathcal{O}_{pi} \partial_{y_l}w_p+\mathcal{O}_{li}\mathcal{O}_{pj} \partial_{y_l}w_p))(t,\boldsymbol{y})\\
&=(\mathcal{O}_{li}\psi_k (\partial_{y_k}w_l+ \partial_{y_l}w_k))(t,\boldsymbol{y})=2(\mathcal{O}^\top(\boldsymbol{\psi}\cdot D_{\boldsymbol{y}}(\boldsymbol{w})))_i(t,\boldsymbol{y}),
\end{aligned}\\
&\quad \ \ \, \, \, \,\begin{aligned}
(\hat{\psi}_i\diver\boldsymbol{w})(t,\boldsymbol{x})
&=(\mathcal{O}_{kj} \psi_k (\mathcal{O}_{li}\mathcal{O}_{pi} \partial_{y_l}w_p))(t,\boldsymbol{y}) =(\mathcal{O}^\top(\boldsymbol{\psi}\diver_{\boldsymbol{y}}\boldsymbol{w}))_i(t,\boldsymbol{y}).
\end{aligned}
\end{align*}
Hence, the above identities yield that $(\hat\phi,\hat{\boldsymbol{u}},\hat h,\hat{\boldsymbol{\psi}})$ satisfies $\eqref{ln}_2$.

Now, since we have shown that $(\hat \phi,\hat{\boldsymbol{u}},\hat{h})$  is also a solution of the linearized problem \eqref{ln}, which takes the initial data $(\phi_0^\eta,\boldsymbol{u}_0^\eta,h_0^\eta)$, the uniqueness of the solution implies
\begin{equation*} 
\phi(t,\mathcal{O}\boldsymbol{x})=\phi(t,\boldsymbol{x}), \qquad (\mathcal{O}^{\top}\boldsymbol{u})(t,\mathcal{O}\boldsymbol{x})=\boldsymbol{u}(t,\boldsymbol{x}), \qquad h(t,\mathcal{O}\boldsymbol{x})=h(t,\boldsymbol{x}).
\end{equation*}
Therefore,  $(\phi,h)(t,\boldsymbol{x})=(\phi,h)(t,|\boldsymbol{x}|)$ with some functions $(\phi,h)(t,r)$ defined on $[0,T]\times I$. Moreover, it follows from Lemma \ref{rmk31} in Appendix \ref{appA} that   $\boldsymbol{u}(t,\boldsymbol{x})
=u(t,|\boldsymbol{x}|)\frac{\boldsymbol{x}}{|\boldsymbol{x}|}$ with some function $u(t,r)$ defined on $[0,T]\times I$, and $\boldsymbol{u}(t,\boldsymbol{x})|_{|\boldsymbol{x}|=0}=\boldsymbol{0}$. Finally, the spherical symmetry of $\boldsymbol{\psi}$ follows directly from that of $h$ and \eqref{psih}. 
\end{proof}
 
Now we are going to derive the uniform \textit{a priori} estimates, independent of $(\epsilon,\eta)$, for solutions $(\phi^{\epsilon,\eta},\boldsymbol{u}^{\epsilon,\eta},h^{\epsilon,\eta})$ of $\eqref{ln}$ obtained in Lemma \ref{ls}. In the rest of \S \ref{section-local-regular},  $C\geq 4$  denotes a generic constant depending only on  $(A, a_1,a_2, \gamma, \delta, T)$ and may be different at each occurrence. 

\subsection{\textit{A Priori} Estimates Independent of $(\epsilon,\eta)$}  \label{priorilinear}
Assume that  $(\phi_0,\boldsymbol{u}_0)$  satisfy \eqref{th78qq}--\eqref{th78zxq} and $(\phi^\eta_0,\boldsymbol{u}^\eta_0)$ are given in $\eqref{ln}_4$.

First, it follows from Lemmas \ref{ale1}, \ref{initial3}, and \ref{Hk-Ck-vector} (in Appendix \ref{improve-sobolev}), the facts that $\delta>1-\frac{1}{n}$ and $\phi_0\rightarrow 0$ as $|\boldsymbol{x}|\to \infty$,
and the initial assumption \eqref{th78qq} on $(\phi_0,\boldsymbol{u}_0,\boldsymbol{\psi}_0)$ (or the equivalent one \eqref{id1-high} on $(\rho_0,\boldsymbol{u}_0)$) that 
\begin{equation*}
(\phi_0,\ \boldsymbol{\psi_0}=\frac{a\delta}{\delta-1}\nabla \phi^{2\iota}_0,\ \nabla \phi^{\iota}_0)\in L^\infty(\mathbb{R}^n) \ \ \text{for $n=2,3$}, \qquad  \phi_0\in L^6(\mathbb{R}^3).
\end{equation*}

Next, for any fixed $\eta\in(0,1]$, denote 
\begin{equation}\label{etainitial}
\mathcal{G}^\eta_1=(\phi^\eta_0)^{\iota}\nabla \boldsymbol{u}^\eta_0,\qquad \mathcal{G}^\eta_2=(\phi^\eta_0)^{\iota}\nabla((\phi^\eta_0)^{2\iota}L\boldsymbol{u}^\eta_0),\qquad 
\boldsymbol{g}^\eta_*=(\phi^\eta_0)^{2\iota}L\boldsymbol{u}^\eta_0.
\end{equation}
 Since
\begin{align}
&\begin{aligned}
\|\phi_0^\iota\nabla^2\phi_0 \|_{L^2}&\leq  C\big(\|\phi_0\|_{L^{n^*}}\|\phi_0\|^{-\iota}_{L^\infty}\|\nabla\boldsymbol{\psi}_0\|_{L^n}+\|\nabla \phi^\iota_0\|_{L^{2n}}\|\nabla \phi_0\|_{L^{\frac{2n}{n-1}}}\big),\notag\\
\|(\phi^\eta_0)^\iota\nabla^2\phi^\eta_0\|_{L^2}&= \big\|\phi_0^\iota\nabla^2\phi_0 \frac{\phi_0^{-\iota}}{(\phi_0+\eta)^{-\iota}}\big\|_{L^2}\leq  \|\phi_0^\iota\nabla^2\phi_0 \|_{L^2},\notag
\end{aligned}\\
&\hspace{13mm}\begin{aligned}
\|\mathcal{G}^\eta_1\|_{L^2}&=\big\|\frac{\phi_0^{-\iota}}{(\phi_0+\eta)^{-\iota}}\mathcal{G}_1\big\|_{L^2}\leq \|\mathcal{G}_1\|_{L^2},\\ 
\|\mathcal{G}^\eta_2\|_{L^2}
&=\big\|\frac{\phi_0^{-3\iota}}{(\phi_0+\eta)^{-3\iota}}(\mathcal{G}_2-\frac{\eta\nabla\phi_0^{2\iota}}{\phi_0+\eta}\phi_0^\iota Lu_0)\big\|_{L^2}\\
&\leq C(\|\mathcal{G}_2\|_{L^2}+\|\boldsymbol{\psi}_0\|_{L^\infty}\|\phi_0\|^{-\iota}_{L^\infty}\|\boldsymbol{g}^\eta_*\|_{L^2}),\\
\|\boldsymbol{g}_*^\eta\|_{L^2}
&=\big\|\frac{\phi_0^{-2\iota}}{(\phi_0+\eta)^{-2\iota}}\boldsymbol{g}_*\big\|_{L^2}\leq \|\boldsymbol{g}_*\|_{L^2},
\end{aligned}\label{initialapproximation}\\
&\hspace{3mm}\begin{aligned}
\|\nabla^2(\phi^\eta_0)^{2\iota}\|_{L^n}
&\leq C\big\|\nabla \big( \frac{\phi_0^{-2\iota+1}}{(\phi_0+\eta)^{-2\iota+1}}\nabla \phi_0^{2\iota}\big)\big\|_{L^n}
\leq  C\big(\|\nabla \boldsymbol{\psi}_0\|_{L^n}+\|\nabla\phi^{\iota}_0 \|^2_{L^{2n}}\big),\notag\\
\|\nabla^3(\phi^\eta_0)^{2\iota}\|_{L^2}
&\leq C\big\|\nabla^2 \big( \frac{\phi_0^{-2\iota+1}}{(\phi_0+\eta)^{-2\iota+1}}\nabla \phi_0^{2\iota}\big)\big\|_{L^2}\notag\\
&\leq  C\big(\|\nabla^2 \boldsymbol{\psi}_0 \|_{L^2}+\|\phi_0\|^{-\iota}_{L^\infty}(\delta_{3n}\|\nabla\phi^{\iota}_0 \|^3_{L^{2n}}+\delta_{2n}\|\nabla\phi^{\iota}_0 \|_{L^{\infty}}\|\nabla\phi^{\iota}_0 \|^2_{L^{2n}})\big)\\
&\quad +C\|\phi_0\|^{-\iota}_{L^\infty}\|\nabla\phi^{\iota}_0 \|_{L^{n^*}}\|\nabla \boldsymbol{\psi}_0 \|_{L^n},\notag    
\end{aligned}
\end{align}
we can check that there exists a constant $c_0\geq 1$ independent of $\eta$ such that
\begin{equation}\label{2.14}
\begin{aligned}
&\|\nabla\phi^\eta_0\|_{H^2}+\|\phi^\eta_0\|_{ L^\infty}+\delta_{3n}\|\phi^\eta_0-\eta\|_{ L^6}+\|\boldsymbol{u}^\eta_0\|_{H^3}+\|\nabla h^\eta_0\|_{L^\infty\cap D^{1,n}\cap D^2}+\|\nabla (h^\eta_0)^{\frac{1}{2}}\|_{L^{2n}}
\\
&+\|(\phi^\eta_0)^\iota\nabla^2\phi^\eta_0\|_{L^2}+\|(h^\eta_0)^{-1}\|_{L^\infty}+\|\mathcal{G}^\eta_1\|_{L^2}+\|\mathcal{G}^\eta_2\|_{L^2}+\|\boldsymbol{g}^\eta_*\|_{L^2}+2+\eta\leq c_0.
\end{aligned}
\end{equation}

\begin{rk}\label{r1}
Based on \eqref{2.14}, we also have
\begin{equation}\label{incc}
\begin{aligned}
\|(\phi^\eta_0)^{2\iota}\boldsymbol{u}^\eta_0\|_{D^2}&\leq C(\|\boldsymbol{g}^\eta_*\|_{L^2}+\|G(\psi^\eta_0,\boldsymbol{u}^\eta_0)\|_{L^2})\leq Cc^2_0,\\
\|(\phi^\eta_0)^{2\iota}\nabla^2\boldsymbol{u}^\eta_0\|_{L^2}&\leq  C(\|(\phi^\eta_0)^{2\iota}\boldsymbol{u}^\eta_0\|_{D^2}+\|\nabla\boldsymbol{\psi}^\eta_0\|_{L^n}\|\boldsymbol{u}^\eta_0\|_{L^{n^*}}+\|\boldsymbol{\psi}^\eta_0\|_{L^\infty}\|\nabla \boldsymbol{u}^\eta_0\|_{L^2})\leq Cc^2_0.
\end{aligned}
\end{equation}
Indeed, it follows from  the definition of $\boldsymbol{g}^\eta_*$, $\phi^\eta_0>\eta$, and  $\eqref{ln}_5$  that 
\begin{equation*} 
L((\phi^\eta_0)^{2\iota}\boldsymbol{u}^\eta_0)=\boldsymbol{g}^\eta_*-\frac{\delta-1}{a\delta}G(\boldsymbol{\psi}^\eta_0,\boldsymbol{u}^\eta_0),\qquad \text{with \ $
(\phi^\eta_0)^{2\iota}\boldsymbol{u}^\eta_0\to \boldsymbol{0}$ as $|\boldsymbol{x}|\to \infty$},
\end{equation*}
where $\boldsymbol{\psi}^\eta_0=\frac{a\delta}{\delta-1}\nabla(\phi^\eta_0)^{2\iota}=\frac{a\delta}{\delta-1}\nabla h^\eta_0$ and 
\begin{equation}\label{Gdingyi}
G(\boldsymbol{f}_1,\boldsymbol{f}_2):=a_1\boldsymbol{f}_1\cdot\nabla \boldsymbol{f}_2+a_1  \diver(\boldsymbol{f}_2\otimes\boldsymbol{f}_1)+(a_1+a_2)(\boldsymbol{f}_1\,\diver \boldsymbol{f}_2+\boldsymbol{f}_1\cdot\nabla \boldsymbol{f}_2+\boldsymbol{f}_2\cdot\nabla\boldsymbol{f}_1).
\end{equation}
Then \eqref{incc} follows directly from \eqref{2.14} and the classical elliptic theory in {\rm Lemma \ref{df3}}.
\end{rk}

\medskip
Now let $T>0$ be a fixed constant, and  assume that there exist $T^*\in(0,T]$ and constants $c_i \ (i=1,\cdots\!,5)$ such that
\begin{equation*} 
1<c_0\leq c_1\leq c_2\leq c_3\leq c_4\leq c_5,
\end{equation*}
and
\begin{align}
\sup_{t\in [0,T^*]}\|\nabla g\|^2_{L^\infty\cap D^{1,n}\cap D^2}\leq c_1^2, \ \ 
0<g^{-1}\leq c_1,\ \  \sup_{t\in [0,T^*]}\big(\|g_t\|^2_{D^1}+\|g^{-1}g_t\|^2_{L^\infty}\big)\leq c^2_4,\notag\\
\sup_{t\in [0,T^*]}\|\boldsymbol{w}\|^2_{H^1}+\int^{T^*}_0\big(\|\boldsymbol{w}\|^2_{D^2}+\|\boldsymbol{w}_t\|^2_{L^2}\big)\,\mathrm{d}t\leq c_2^2,\notag\\
\sup_{t\in [0,T^*]}\big(\|\boldsymbol{w}\|^2_{D^2}+\|\boldsymbol{w}_t\|^2_{L^2}+\|g\nabla^2\boldsymbol{w}\|^2_{L^2}\big)+\int^{T^*}_0\big(\|\boldsymbol{w}\|_{D^3}^2+\|\boldsymbol{w}_t\|^2_{D^1}\big)\,\mathrm{d}t\leq c_3^2,\notag\\
\sup_{t\in [0,T^*]}\big(\|\boldsymbol{w}\|^2_{D^3}+\|\sqrt{g}\nabla \boldsymbol{w}_t\|^2_{L^2}+\|\nabla \boldsymbol{w}_t\|^2_{L^2}\big)+\int^{T^*}_0\big(\|\boldsymbol{w}\|^2_{D^4}+\|\boldsymbol{w}_t\|^2_{D^2}\big)\,\mathrm{d}t\leq c_4^2,\label{2.16}\\
\sup_{t\in [0,T^*]}\|g\nabla^2\boldsymbol{w}\|^2_{D^1}+\int^{T^*}_0\big(\|(g\nabla^2\boldsymbol{w})_t\|^2_{L^2}+\|g\nabla^2\boldsymbol{w}\|^2_{D^2}+\|\boldsymbol{w}_{tt}\|^2_{L^2}\big)\,\mathrm{d}t\leq c_4^2,\notag\\
\sup_{t\in [0,T^*]}\big(t\|\boldsymbol{w}\|^2_{D^4}+t\|\nabla^2\boldsymbol{w}_t\|^2_{L^2}+t\|g\nabla^2\boldsymbol{w}_t\|^2_{L^2}\big)+\int^{T^*}_0\|g_{tt}\|^2_{D^1}\,\mathrm{d}t\leq c^2_5,&\notag\\ 
\sup_{t\in [0,T^*]}t\|\boldsymbol{w}_{tt}\|^2_{L^2}+\int^{T^*}_0t\big(\|\boldsymbol{w}_{tt}\|^2_{D^1}+\|\sqrt{g}\boldsymbol{w}_{tt}\|_{D^1}^2+\|\boldsymbol{w}_t\|^2_{D^3}\big)\,\mathrm{d}t\leq c_5^2.\notag
\end{align}
Here, $T^*$ and $c_i$ $(i=1,\cdots\!,5)$ will be determined later, and depend only on $c_0$ and the fixed constants $(A, a_1,a_2,\gamma,\delta, T)$.

In  the rest of \S \ref{priorilinear}, without causing ambiguity,
we simply drop superscripts $\epsilon$ and $\eta$ in 
$(\phi^\eta_0,\boldsymbol{u}^\eta_0,h^\eta_0,\boldsymbol{\psi}^\eta_0)$, 
$(\phi^{\epsilon,\eta},\boldsymbol{u}^{\epsilon,\eta},h^{\epsilon,\eta},\boldsymbol{\psi}^{\epsilon,\eta})$, and $(\mathcal{G}^\eta_1,\mathcal{G}^\eta_2,\boldsymbol{g}_*^\eta)$, and we let  $(\phi,\boldsymbol{u},h)$ be the unique classical solution of $\eqref{ln}$ in $[0,T]\times\mathbb{R}^n$ obtained in Lemma \ref{ln}. 

\smallskip
\subsubsection{The  estimates for $\phi$.}

First, we give the estimates for $\phi$.
\begin{lem}\label{phiphi}
For any $0\leq t\leq T_1:=\min\{T^*,(1+Cc_4)^{-6}\}$,
\begin{equation*} 
\begin{aligned}
&\|\phi(t)-\eta\|_{L^\infty\cap D^1\cap D^3}\leq Cc_0,\qquad\|\phi_t(t)\|_{L^2}\leq Cc_0c_2,\qquad\|\phi_t(t)\|_{D^1}\leq Cc_0c_3,\\
&\|\phi_t(t)\|_{D^2}\leq Cc_0c_4,\qquad\|\phi_{tt}(t)\|_{L^2}\leq Cc_4^3,\qquad\int^t_0\|\phi_{tt}\|^2_{H^1} \,\mathrm{d}s\leq Cc_0^2c_4^2.
\end{aligned}
\end{equation*}
\end{lem}
\begin{proof} 
First, let   $X(t;\boldsymbol{x})$  be the particle path defined by
\begin{equation}\label{characteristic}
\frac{\mathrm{d}}{\mathrm{d}t}X(t;\boldsymbol{x})=\boldsymbol{w}(t,X(t; \boldsymbol{x})) \ \ \text{for $t\in [0,T]$},\qquad\text{with $X(0;\boldsymbol{x})=\boldsymbol{x} \in \mathbb{R}^n$}.
\end{equation}
Then, integrating $\eqref{ln}_1$ along $X(t; \boldsymbol{x})$, we see that, 
for $0\leq t\leq T_1:=\min\{T^*,(1+Cc_4)^{-6}\}$,
\begin{equation*} 
\|\phi\|_{L^\infty}\leq C\|\phi_0\|_{L^\infty}e^{ \int_0^tC\|\boldsymbol{w}\|_{H^3} \,\mathrm{d}s} \leq Cc_0.         
\end{equation*}

Next, it follows from the energy estimates  for the transport equation 
and \eqref{2.16} that
\begin{equation*}
\|\phi-\eta\|_{D^1\cap D^3}\leq C\Big(\|\phi_0-\eta\|_{D^1\cap D^3}+\eta\int^t_0\|\nabla \boldsymbol{w}\|_{H^3}\,\mathrm{d}s\Big)\exp\Big(\int^t_0C\|\boldsymbol{w}\|_{H^4} \,\mathrm{d}s\Big)
\leq Cc_0
\end{equation*}
for $t\in[0, T_1]$, which, together with $\eqref{ln}_1$ and \eqref{2.16},  yields the desired estimates.
\end{proof}

\smallskip
\subsubsection{The  estimates for $\boldsymbol{\psi}$.}
 It follows from  $\boldsymbol{\psi}=\frac{a\delta}{\delta-1}\nabla h$  and   $\eqref{ln}_3$ that 
\begin{equation}\label{psieq}
\boldsymbol{\psi}_t+ \boldsymbol{w}\cdot \nabla\boldsymbol{\psi}+B^* \boldsymbol{\psi}+a\delta ( g\nabla \diver \boldsymbol{w} +\nabla g \diver \boldsymbol{w} ) =0,
\end{equation}
where $B^*=B^*(\boldsymbol{w})=(\nabla \boldsymbol{w})^\top$. Now we give the estimates for $\boldsymbol{\psi}$.

\begin{lem}\label{psi} 
For any $t\in [0, T_1]$, 
\begin{equation*}
\|\boldsymbol{\psi}(t)\|^2_{L^\infty\cap D^{1,n}\cap D^2}\leq Cc_0^2,\quad\|\boldsymbol{\psi}_t(t)\|_{L^2}\leq Cc_3^2,
\quad \|\boldsymbol{\psi}_t(t)\|^2_{D^1}+\int^t_0\|\boldsymbol{\psi}_{tt}\|^2_{L^2}\,\mathrm{d}s\leq Cc_4^4.
\end{equation*}
\end{lem}
\begin{proof} 
First,  integrating \eqref{psieq} along $X(t; \boldsymbol{x})$ (see \eqref{characteristic}), we obtain that, for $t\in [0, T_1]$,
\begin{equation}\label{inftypsi}
\begin{aligned}
\|\boldsymbol{\psi}\|_{L^\infty}&\leq C\Big(\|\boldsymbol{\psi}_0\|_{L^\infty}+\int_0^t\big(\|g\nabla^2 \boldsymbol{w}\|_{L^\infty}+\|\nabla g\|_{L^\infty} \|\nabla \boldsymbol{w}\|_{L^\infty}\big) \,\mathrm{d}s\Big)e^{ \int_0^tC\|\boldsymbol{w}\|_{H^3} \,\mathrm{d}s} \\
& \leq C(c_0+c_4t^{\frac{1}{2}}+c_1c_4 t)e^{Cc_4t}\leq Cc_0,         
\end{aligned}
\end{equation}
where the $L^\infty$-boundedness for $\boldsymbol{\psi}_0$ is due to $\boldsymbol{\psi}_0\in D^{1,n}(\mathbb{R}^n)$ and Lemma \ref{Hk-Ck-vector}.

Next, let $\boldsymbol{\varsigma}=(\varsigma_1,\cdots\!,\varsigma_n)$ be the multi-index. Applying  $n|\partial_{\boldsymbol{x}}^{\boldsymbol{\varsigma}} \boldsymbol{\psi}|^{n-2}\partial_{\boldsymbol{x}}^{\boldsymbol{\varsigma}} \boldsymbol{\psi}\partial_{\boldsymbol{x}}^{\boldsymbol{\varsigma}}$ ($|\boldsymbol{\varsigma}|=1$) and $2\partial_{\boldsymbol{x}}^{\boldsymbol{\varsigma}} \boldsymbol{\psi}\partial_{\boldsymbol{x}}^{\boldsymbol{\varsigma}}$ ($|\boldsymbol{\varsigma}|=2$) to \eqref{psieq} respectively, and then integrating these two resulting equalities over $\mathbb{R}^n$, we obtain
\begin{equation}\label{2.26}
\begin{aligned}
\frac{\mathrm{d}}{\mathrm{d}t}\|\partial_{\boldsymbol{x}}^{\boldsymbol{\varsigma}}  \boldsymbol{\psi}\|^n_{L^n}
&\leq  C\|\nabla \boldsymbol{w}\|_{L^\infty}\|\partial_{\boldsymbol{x}}^{\boldsymbol{\varsigma}}  \boldsymbol{\psi}\|^n_{L^n}+C\|\Theta_{\boldsymbol{\varsigma}} \|_{L^n}\|\partial_{\boldsymbol{x}}^{\boldsymbol{\varsigma}}  \boldsymbol{\psi}\|^{n-1}_{L^n}&& \ \ \text{with $|\boldsymbol{\varsigma}|=1$},\\
\frac{\mathrm{d}}{\mathrm{d}t}\|\partial_{\boldsymbol{x}}^{\boldsymbol{\varsigma}}  \boldsymbol{\psi}\|^2_{L^2}
&\leq C\|\nabla \boldsymbol{w}\|_{L^\infty}\|\partial_{\boldsymbol{x}}^{\boldsymbol{\varsigma}}  \boldsymbol{\psi}\|^2_{L^2}+C\|\Theta_{\boldsymbol{\varsigma}}\|_{L^2}\|\partial_{\boldsymbol{x}}^{\boldsymbol{\varsigma}}  \boldsymbol{\psi}\|_{L^2}&& \ \ \text{with $|\boldsymbol{\varsigma}|=2$}, 
\end{aligned}
\end{equation}
where
\begin{equation*}
\begin{aligned}
\Theta_{\boldsymbol{\varsigma}}=\partial_{\boldsymbol{x}}^{\boldsymbol{\varsigma}} (B^*\boldsymbol{\psi})-B^*\partial_{\boldsymbol{x}}^{\boldsymbol{\varsigma}}  \boldsymbol{\psi}+\sum_{k=1}^{n}\big(\partial_{\boldsymbol{x}}^{\boldsymbol{\varsigma}} (\boldsymbol{w}_{k} \partial_k \boldsymbol{\psi})-\boldsymbol{w}_{k} \partial_k\partial_{\boldsymbol{x}}^{\boldsymbol{\varsigma}}  \boldsymbol{\psi}\big)+ a\delta \partial_{\boldsymbol{x}}^{\boldsymbol{\varsigma}}\big(g\nabla \diver \boldsymbol{w}+\nabla g \diver \boldsymbol{w}\big).
\end{aligned}
\end{equation*}

Then we can directly check that, for $|\boldsymbol{\varsigma}|=1$,
\begin{equation}\label{zhen2}
\|\Theta_{\boldsymbol{\varsigma}} \|_{L^n}\leq C\big(\|\nabla^2 \boldsymbol{w}\|_{L^n}\|(\boldsymbol{\psi},\nabla g)\|_{L^\infty} + \|\nabla \boldsymbol{w}\|_{L^\infty} \|(\nabla\boldsymbol{\psi},\nabla^2 g)\|_{L^n}
+\|\nabla(g\nabla^2 \boldsymbol{w})\|_{L^n}\big),
\end{equation}
and, for $|\boldsymbol{\varsigma}|=2$, 
\begin{equation}\label{zhen2*}
\begin{aligned}
\|\Theta_{\boldsymbol{\varsigma}} \|_{L^2}&\leq C\big(\|\nabla \boldsymbol{w}\|_{L^\infty}\|(\nabla^2\boldsymbol{\psi},\nabla^3 g)\|_{L^2}+\|\nabla^2 \boldsymbol{w}\|_{L^6}\|(\nabla\boldsymbol{\psi},\nabla^2 g)\|_{L^3}\big)\\
&\quad +C\big(\|\nabla^3 \boldsymbol{w}\|_{L^2} \|(\boldsymbol{\psi},\nabla g)\|_{L^\infty}+C\|g\nabla \diver\boldsymbol{w}\|_{D^2}\big).
\end{aligned}
\end{equation}
Hence, it follows from $\eqref{2.26}$--$\eqref{zhen2*}$ and the Gr\"onwall inequality that
\begin{equation*}
\|\boldsymbol{\psi}(t)\|_{D^{1,n}\cap D^2}\leq  \Big(c_0+Cc^2_4t+C\int_0^t \|g\nabla \diver\boldsymbol{w}\|_{D^2} \,\mathrm{d}s\Big) e^{Cc_4t}\leq Cc_0\qquad \text{for $t\in [0,T_1]$},
\end{equation*}
which, along with \eqref{2.16} and \eqref{psieq}--\eqref{inftypsi}, implies 
\begin{equation*}
\|\boldsymbol{\psi}_t(t)\|_{L^2}\leq Cc^2_3,\quad 
\|\boldsymbol{\psi}_t(t)\|_{D^1} \leq Cc^2_4,\quad  \int_0^t \|\boldsymbol{\psi}_{tt}\|^2_{L^2} \,\mathrm{d}s\leq Cc^4_4\qquad\text{for $t\in [0,T_1]$}.
\end{equation*}
\end{proof}

\subsubsection{\textit{A priori} estimates for auxiliary variables related to $h$}  We define the quantities:
\begin{equation*}
(\varphi,\varphi_0)=(h^{-1},h_0^{-1}).
\end{equation*}
\begin{lem}\label{varphi} 
For any $t\in [0, T_1]$ and $\boldsymbol{x}\in \mathbb{R}^n$, 
\begin{equation*} 
\begin{aligned}
&C^{-1}\leq (gh^{-1})(t,\boldsymbol{x})\leq C,\quad 
h(t,\boldsymbol{x}) \geq (2c_0)^{-1},\quad\|\varphi(t)\|_{L^\infty}+\|\nabla\sqrt{h}(t)\|_{L^{2n}}\leq Cc_0, \\
&\|h^{-1}h_t(t)\|^2_{L^\infty}\leq Cc_2^{\frac{9}{2}}c^{\frac{3}{2}}_4,\qquad \int^t_0\|h^{-1}h_{tt}\|^2_{L^{2n}}\,\mathrm{d}s\leq Cc_4^2.
\end{aligned}
\end{equation*}
\end{lem}
\begin{proof}
We divide the proof into four steps.

\smallskip
\textbf{1.} Set $\Lambda=\Lambda(t,\boldsymbol{x})=gh^{-1}$. We then see from $\eqref{ln}_3$ that $\Lambda^{-1}$ satisfies
\begin{equation*}
(\Lambda^{-1})_t + (g^{-1}g_t)\Lambda^{-1}=-(g^{-1}\boldsymbol{w}\cdot \nabla h+(\delta-1)\diver \boldsymbol{w}),\qquad \Lambda^{-1}(0,\boldsymbol{x}) =1.
\end{equation*}
Hence,  the standard argument for ODE gives that, for $(t,\boldsymbol{x})\in [0,T_1]\times \mathbb{R}^n$,
\begin{equation*}
\Lambda^{-1} (t,\boldsymbol{x})= e^{-\int_0^t(g^{-1}g_t)(\tau, \boldsymbol{x})\mathrm{d}\tau}\Big(1-\!\int_0^t\! \big(g^{-1}\boldsymbol{w}\cdot \nabla h+(\delta-1)\diver \boldsymbol{w}\big)(\tau, \boldsymbol{x})e^{\int_0^\tau(g^{-1}g_t)(s, \boldsymbol{x})\mathrm{d}s}\mathrm{d}\tau\Big),
\end{equation*}
which, along with  Lemma \ref{psi} and \eqref{2.16}, yields 
\begin{equation}\label{1008}
\Lambda^{-1}\in [C^{-1},C]\iff \Lambda \in [C^{-1},C].
\end{equation} 

\smallskip
\textbf{2.} Note that 
\begin{equation*} 
\varphi_t+\boldsymbol{w}\cdot \nabla \varphi-(\delta-1)\Lambda\varphi\diver \boldsymbol{w}=0.
\end{equation*}
Then the standard argument for ODE implies that, for $(t,\boldsymbol{x})\in [0,T_1]\times \mathbb{R}^n$, 
\begin{equation}\label{2.34d}
\frac{2}{3}\eta^{-2\iota}<\varphi(t,\boldsymbol{x})<2\|\varphi_0\|_{L^\infty}\leq 2c_0\implies h(t,\boldsymbol{x})>(2c_0)^{-1}.
\end{equation}

\smallskip
\textbf{3.} It follows from  $\eqref{ln}_3$ that 
\begin{equation*} 
(\sqrt{h})_t+\boldsymbol{w}\cdot\nabla\sqrt{h}+\frac{1}{2\sqrt{h}}(\delta-1)g\diver\boldsymbol{w}=0.
\end{equation*}
Applying  $\nabla$ to the above,
multiplying  by $2n|\nabla\sqrt{h}|^{2n-2}\nabla\sqrt{h}$, 
and  integrating  over $\mathbb{R}^n$ yield
\begin{equation*}
\begin{aligned}
\frac{\mathrm{d}}{\mathrm{d}t}\|\nabla\sqrt{h}\|_{L^{2n}} 
\leq&\, C\|\nabla \boldsymbol{w}\|_{L^\infty}\|\nabla\sqrt{h}\|_{L^{2n}}+C\big \|\varphi\|^{\frac{1}{2}}_{L^\infty}\big(\|g\nabla^2\boldsymbol{w}\|_{L^{2n}}+\|\nabla g\|_{L^\infty}\|\nabla \boldsymbol{w}\|_{L^{2n}}\big) \\
&\, +C\|gh^{-1}\|_{L^\infty}\|\varphi\|^{\frac{1}{2}}_{L^\infty}\|\diver\boldsymbol{w}\|_{L^{2n}}\|\boldsymbol{\psi}\|_{L^\infty} ,
\end{aligned}
\end{equation*}
which, along with the Gr\"onwall inequality, \eqref{2.16}, \eqref{2.34d}, and Lemma \ref{psi},  leads to 
\begin{equation*} 
\|\nabla\sqrt{h}(t)\|_{L^{2n}}\leq Cc_0\qquad \text{for $t\in [0,T_1]$}.
\end{equation*}

\smallskip
\textbf{4.}
Finally, it follows from $\eqref{ln}_3$, \eqref{2.16}, \eqref{1008}--\eqref{2.34d}, and Lemmas \ref{psi} and \ref{GN-ineq} that
\begin{equation*} 
\begin{aligned}
\|h^{-1}h_t\|_{L^\infty}&\leq  C\big(\|\boldsymbol{w}\|_{L^\infty}\|\boldsymbol{\psi}\|_{L^\infty}\|\varphi\|_{L^\infty}+\|gh^{-1}\|_{L^\infty}\|\mathrm{div}\,\boldsymbol{w}\|_{L^\infty}\big)\\
&\leq C\big(c^3_2+\|\nabla \boldsymbol{w}\|^{\frac{4-n}{4}}_{L^2}\|\nabla^3\boldsymbol{w}\|^{\frac{n}{4}}_{L^2}\big)\leq Cc^2_2c^{\frac{4-n}{4}}_2c^{\frac{n}{4}}_4,\\
\|h^{-1}h_{tt}\|_{L^{2n}}&\leq C\big(\|\varphi\|_{L^\infty}\|\psi\|_{L^\infty}\| \boldsymbol{w}_t\|_{L^{2n}}+\|\varphi\|_{L^\infty}\|\boldsymbol{w}\|_{L^\infty}\| \boldsymbol{\psi}_t\|_{L^{2n}}\big)\\
&\quad +C\|gh^{-1}\|_{L^\infty}\big(\|g^{-1}g_t\|_{L^\infty}\| \nabla \boldsymbol{w}\|_{L^{2n}}+ \| \nabla \boldsymbol{w}_t\|_{L^{2n}}\big) \leq C\big(c^4_4+\| \nabla^2 \boldsymbol{w}_t\|_{L^2}\big),
\end{aligned}
\end{equation*}
which, along with \eqref{2.16}, yields the rest of estimates.
\end{proof} 

\smallskip
\subsubsection{The estimates for $\boldsymbol{u}$}
Now we derive the estimates for $\boldsymbol{u}$ in the following several lemmas.
For simplicity, we denote 
\begin{equation*}
\begin{aligned}
\mathcal{H}&:=-\boldsymbol{u}_t-\boldsymbol{w}\cdot\nabla \boldsymbol{w}-\nabla\phi+\boldsymbol{\psi}\cdot Q(\boldsymbol{w}),\quad \mathcal{H}_0:=\mathcal{H}+\boldsymbol{u}_t,\notag\\
\mathcal{H}_1&:=\mathcal{H}-aG(\nabla\sqrt{h^2+\epsilon^2},\boldsymbol{u}),\quad 
\mathcal{H}_2:=\mathcal{H}_t -a\frac{h}{\sqrt{h^2+\epsilon^2}} L\boldsymbol{u}
h_t-aG(\nabla\sqrt{h^2+\epsilon^2},\boldsymbol{u}_t),\notag
\end{aligned}
\end{equation*}
where $G(\cdot,\cdot)$ is defined by \eqref{Gdingyi}.

\begin{lem}\label{HH} 
For any $t\in [0, T_1]$, we have the following estimates{\rm:}
\begin{itemize}
\item[\rm(i)] Estimates on $\mathcal{H}${\rm :}
\begin{equation*}
\|\mathcal{H}\|_{L^2}\leq C\big(\|u_t\|_{L^2}+ c_2^{\frac{3}{2}}c_3^{\frac{1}{2}}\big),\ \  \|\mathcal{H}\|_{D^1}\leq C\big(\|u_t\|_{D^1}+ c_3^2\big),\ \  \|\mathcal{H}\|_{D^2}\leq C\big(\|\boldsymbol{u}_t\|_{D^2}+c_4^2\big);
\end{equation*}
\item[\rm(ii)] Estimates on $\mathcal{H}_0${\rm :}
\begin{equation*}
\begin{aligned}
&\|\mathcal{H}_0\|_{L^2} \leq Cc_2^{\frac{3}{2}}   c_3^{\frac{1}{2}}, \qquad 
\|(\mathcal{H}_0)_t\|_{L^2}\leq  Cc^3_4,\notag\\
&\|(\mathcal{H}_0)_{tt}\|_{L^2}\leq C \big(c^3_4+c_4^2\|\nabla^2 \boldsymbol{w}_t\|_{L^2}+c_4\|\boldsymbol{w}_{tt}\|_{H^1}+\|\nabla\phi_{tt}\|_{L^2}+c_4\|\boldsymbol{\psi}_{tt}\|_{L^2}\big);\notag
\end{aligned}
\end{equation*}
\item[\rm(iii)] Estimates on $\mathcal{H}_1${\rm :}
\begin{equation*}
\begin{aligned}
&\|\mathcal{H}_1\|_{L^2}\leq C\big(\|\boldsymbol{u}_t\|_{L^2}+c_0^{3}\| \boldsymbol{u}\|_{L^2}+ c_0^{\frac{7}{2}}\|\sqrt{h}\nabla \boldsymbol{u}\|_{L^2}+c_2^{\frac{3}{2}}c_3^{\frac{1}{2}}\big),\notag\\[-2pt]
&\|\mathcal{H}_1\|_{D^1}\leq C\big(\|\boldsymbol{u}_t\|_{D^1}+c^4_0\| \boldsymbol{u}\|_{H^2}+c_3^2\big);\notag
\end{aligned}
\end{equation*}
\item[\rm(iv)] Estimates on $\mathcal{H}_2${\rm :}
\begin{equation*}
\begin{aligned}
&\|\mathcal{H}_2\|_{L^2}\leq C\big(\|\boldsymbol{u}_{tt}\|_{L^2}+c^3_4\|h\nabla^2\boldsymbol{u}\|_{L^2}+c^3_0\| \boldsymbol{u}_t\|_{H^1}+c^3_4\big),\notag\\
&\|\mathcal{H}_2\|_{D^1}\leq C\big(\|\boldsymbol{u}_{tt}\|_{D^1 }+c_0^4\|\boldsymbol{u}_t\|_{H^2}+c_4^4\big(\|h\nabla^3 \boldsymbol{u}\|_{L^2} +\|\nabla^2 \boldsymbol{u}\|_{H^1} \big)+c_3\|\nabla^2 \boldsymbol{w}_t\|_{L^2}+c_4^3\big).
\end{aligned}
\end{equation*}
\end{itemize}
\end{lem}
\begin{proof}
It follows from the H\"older inequality, \eqref{2.16}, and Lemmas \ref{phiphi}--\ref{varphi} and \ref{GN-ineq} that 
\begin{equation}\label{hl2}
\begin{aligned}
\|\nabla \boldsymbol{w}\|_{L^3} &\leq  \|\nabla \boldsymbol{w}\|^{\frac{1}{2}}_{L^2}\|\nabla \boldsymbol{w}\|^{\frac{1}{2}}_{L^6} \leq  C\|\nabla \boldsymbol{w}\|^{\frac{1}{2}}_{L^2}\|\nabla \boldsymbol{w}\|^{\frac{1}{2}}_{H^1}\leq Cc_2^{\frac{1}{2}}c_3^{\frac{1}{2}},\\
\|\nabla^2\sqrt{h^2+\epsilon^2}\|_{L^3} &\leq   C\big(\|\nabla \sqrt{h}\|^2_{L^{6}}+\|\nabla \boldsymbol{\psi}\|_{L^3}\big)\\
&\leq C\big(\delta_{3n}\|\nabla \sqrt{h}\|^2_{L^{6}}\!+\delta_{2n}\|\nabla \sqrt{h}\|^{\frac{4}{3}}_{L^{4}}\|\boldsymbol{\psi}\|^{\frac{2}{3}}_{L^\infty} \|\varphi\|^{\frac{1}{3}}_{L^\infty}\!+\|\nabla \boldsymbol{\psi}\|_{L^3}\big)\leq Cc^3_0,\\
 \|\nabla^3\sqrt{h^2+\epsilon^2}\|_{L^2} 
&\leq C\big(\delta_{3n}\|\varphi\|^{\frac{1}{2}}_{L^\infty}\|\nabla \sqrt{h}\|^3_{L^{2n}}+
\delta_{2n}\|\varphi\|_{L^\infty}\|\boldsymbol{\psi}\|_{L^\infty}\|\nabla \sqrt{h}\|^2_{L^{2n}}\big)\\
&\quad +C\big(\|\varphi\|^{\frac{1}{2}}_{L^\infty}\|\nabla \boldsymbol{\psi}\|_{L^n}\|\nabla \sqrt{h}\|_{L^{n^*}}+\|\nabla^2 \boldsymbol{\psi}\|_{L^2}\big)\leq Cc^{4}_0. 
\end{aligned}
\end{equation}

Then, based on \eqref{hl2}, we obtain from the H\"older inequality, \eqref{2.16}, and Lemmas  \ref{phiphi}--\ref{varphi}, \ref{ale1}--\ref{GN-ineq}, and \ref{lemma-L6}--\ref{Hk-Ck-vector} that
\begin{align}
&\begin{aligned}
\|\mathcal{H}_0\|_{L^2}
&\leq C\big(\|\boldsymbol{w}\|_{L^6}\|\nabla \boldsymbol{w}\|_{L^3}+\|\nabla\phi\|_{L^2}+\|\boldsymbol{\psi}\|_{L^\infty}\|\nabla \boldsymbol{w}\|_{L^2}\big)\leq C c_2^{\frac{3}{2}}   c_3^{\frac{1}{2}},\notag\\
\|(\mathcal{H}_0)_t\|_{L^2}&\leq  C\big(\|\boldsymbol{w}\|_{H^2}\| \boldsymbol{w}_t\|_{H^1}+\| \phi_t\|_{H^1}+\|\boldsymbol{\psi}_t\|_{L^2}\|\nabla \boldsymbol{w}\|_{L^\infty}+\|\boldsymbol{\psi}\|_{L^\infty}\|\nabla \boldsymbol{w}_t\|_{L^2}\big)\leq Cc^3_4,\notag \\
\|(\mathcal{H}_0)_{tt}\|_{L^2}&\leq C \big(\|\nabla \boldsymbol{w}_t\|_{L^6}\|\boldsymbol{w}_t\|_{L^3}+\|\nabla \boldsymbol{w}\|_{L^\infty}\|\boldsymbol{w}_{tt}\|_{L^2}+\|\boldsymbol{w}\|_{L^\infty}\|\nabla \boldsymbol{w}_{tt}\|_{L^2}+\|\nabla\phi_{tt}\|_{L^2}\big)\notag\\
&\quad + C\big(\|\boldsymbol{\psi}_{tt}\|_{L^2}\|\nabla \boldsymbol{w}\|_{L^\infty}+\|\boldsymbol{\psi}\|_{L^\infty}\|\nabla \boldsymbol{w}_{tt}\|_{L^2}
+\|\boldsymbol{\psi}_t\|_{L^3}\|\nabla \boldsymbol{w}_t\|_{L^6}\big)\notag\\
&\leq C \big(c^3_4+c_4^2\|\nabla^2 \boldsymbol{w}_t\|_{L^2}+c_4\|\boldsymbol{w}_{tt}\|_{H^1} +\|\nabla\phi_{tt}\|_{L^2}+c_4\|\boldsymbol{\psi}_{tt}\|_{L^2}\big)\notag,
\end{aligned}\\
&\hspace{6.5mm}\begin{aligned}
\|\mathcal{H}\|_{L^2}&\leq C\big(\|\mathcal{H}_0\|_{L^2}+\|\boldsymbol{u}_t\|_{L^2}\big)\leq C\big(\|\boldsymbol{u}_t\|_{L^2}+ c_2^{\frac{3}{2}}c_3^{\frac{1}{2}}\big),\notag\\
\|\mathcal{H}\|_{D^1}
&\leq C\big(\|\boldsymbol{u}_t\|_{D^1}+\|\boldsymbol{w}\|_{H^2}^2+\|\nabla\phi\|_{H^1}+\|\boldsymbol{\psi}\|_{D^{1,n}}\|\boldsymbol{w}\|_{H^2}\big)\leq C\big(\|\boldsymbol{u}_t\|_{D^1}+c_3^2\big),\notag\\
\|\mathcal{H}\|_{D^2}&\leq C\big(\|\boldsymbol{u}_t\|_{D^2}+\|\boldsymbol{w}\|_{H^3}^2+\|\nabla \phi\|_{H^2} +\|\boldsymbol{\psi}\|_{ D^{1,n}\cap D^2}\|\boldsymbol{w}\|_{H^3}\big) \leq  C\big(\|\boldsymbol{u}_t\|_{D^2}+c_4^2\big),\notag 
\end{aligned}\\
&\hspace{5.5mm}\begin{aligned}
\|\mathcal{H}_1\|_{L^2}&\leq  C\big(\|\mathcal{H}\|_{L^2}+\|\boldsymbol{\psi}\|_{L^\infty} \|\nabla \boldsymbol{u}\|_{L^2}
+\|\nabla^2 \sqrt{h^2+\epsilon^2}\|_{L^3} 
\|\boldsymbol{u}\|_{L^6}\big)\notag\\
& \leq  C\big(\|\boldsymbol{u}_t\|_{L^2}+c_0^{3}\| \boldsymbol{u}\|_{L^2}+ c_0^{\frac{7}{2}}\|\sqrt{h}\nabla \boldsymbol{u}\|_{L^2}+c_2^{\frac{3}{2}}c_3^{\frac{1}{2}}\big),\notag
\end{aligned}\\
&\hspace{5mm}\begin{aligned}
\|\mathcal{H}_1\|_{D^1}
&\leq  C\big(\|\mathcal{H}\|_{D^1} +\|\boldsymbol{\psi}\|_{L^\infty} \|\nabla^2 \boldsymbol{u}\|_{L^2}\big)  +C \|\nabla \sqrt{h^2+\epsilon^2}\|_{D^{1,3}\cap D^2} \|\boldsymbol{u}\|_{H^2} \notag\\
&\leq C\big(\|\boldsymbol{u}_t\|_{D^1}+c^4_0\| \boldsymbol{u}\|_{H^2}+c_3^2\big),\notag
\end{aligned}\\
&\hspace{5.5mm}\begin{aligned}
\|\mathcal{H}_2\|_{L^2}&\leq  C\big(\|((\mathcal{H}_0)_t,\boldsymbol{u}_{tt})\|_{L^2}  +\|h^{-1}h_t\|_{L^\infty}\|h\nabla^2\boldsymbol{u}\|_{L^2}\big)\notag\\
&\quad +C\big(\|\nabla\sqrt{h^2+\epsilon^2}\|_{L^\infty} \|\nabla \boldsymbol{u}_t\|_{L^2}+\|\nabla^2 \sqrt{h^2+\epsilon^2}\|_{L^3} \|\boldsymbol{u}_t\|_{L^{6}}\big)\notag\\
&\leq   C\big(\|\boldsymbol{u}_{tt}\|_{L^2}+c^3_4\|h\nabla^2\boldsymbol{u}\|_{L^2}+c^3_0\| \boldsymbol{u}_t\|_{H^1}+c^3_4\big),\notag    
\end{aligned}\\
&\hspace{5mm}\begin{aligned}
\|\mathcal{H}_2\|_{D^1}
&\leq  C\big(\|\boldsymbol{u}_{tt}\|_{D^1 }+\|\boldsymbol{w}\|_{H^3}\| \boldsymbol{w}_t\|_{H^1}+\|\boldsymbol{w}\|_{L^\infty}\|\nabla^2 \boldsymbol{w}_t\|_{L^2}+ \| \phi_t\|_{H^2}\big)\notag\\
&\quad +C\big(\|\boldsymbol{\psi}\|_{L^\infty \cap  D^{1,3}}\|\nabla \boldsymbol{w}_t\|_{H^1}+\|\boldsymbol{\psi}_t\|_{H^1}\| \boldsymbol{w}\|_{H^3}+\|h^{-1}h_t\|_{L^\infty}  \|h\nabla^3 \boldsymbol{u}\|_{L^2}\big)\notag\\
&\quad +C\big(\|h^{-1}h_t\|_{L^\infty}  \|\boldsymbol{\psi}\|_{L^\infty}\|\nabla^2 \boldsymbol{u}\|_{L^2}+\|\boldsymbol{\psi}_t\|_{L^3}\|\nabla^2 \boldsymbol{u}\|_{L^6}\big)\notag\\
&\quad +C\big(\|\boldsymbol{\psi}\|_{L^\infty}+\|\nabla \sqrt{h^2+\epsilon^2}\|_{D^{1,3}\cap D^2}\big)\|\boldsymbol{u}_t\|_{H^2} \notag\\
&\leq  C\big(\|\boldsymbol{u}_{tt}\|_{D^1 }+c_0^4\|\boldsymbol{u}_t\|_{H^2}+c_4^4\big(\|h\nabla^3 \boldsymbol{u}\|_{L^2} +\|\nabla^2 \boldsymbol{u}\|_{H^1} \big)+c_3\|\nabla^2 \boldsymbol{w}_t\|_{L^2}+c_4^3\big).\notag
\end{aligned}
\end{align}
\end{proof}

\begin{lem}\label{lu}
Let $T_2:=\min\{T_1,(1+Cc_4)^{-20}\}$.
Then, for any $t\in [0, T_2]$, 
\begin{equation*} 
\begin{aligned}
\|\boldsymbol{u}(t)\|_{H^1}^2+\|\sqrt{h}\nabla \boldsymbol{u}(t)\|^2_{L^2}  +\int^t_0\big(\|\nabla^2\boldsymbol{u}\|^2_{L^2}+\|\boldsymbol{u}_t\|^2_{L^2}+\|h\nabla^2\boldsymbol{u}\|^2_{L^2}\big)\,\mathrm{d}s
&\leq Cc^4_0,\\
\|\boldsymbol{u}_t(t)\|_{L^2}^2+
 \|h\nabla^2 \boldsymbol{u}(t)\|^2_{L^2}+ \|\boldsymbol{u}(t)\|^2_{D^2}&\leq Cc_2^{10}c_3,\\
\int_0^t \big(\|\sqrt{h}\nabla \boldsymbol{u}_t\|^2_{L^2}+\|\nabla \boldsymbol{u}_t\|^2_{L^2}+\|h\nabla^3\boldsymbol{u}\|^2_{L^2}+ \|\nabla^3\boldsymbol{u}\|^2_{L^2}\big)\,\mathrm{d}s
&\leq  Cc^6_0.
\end{aligned}
\end{equation*}
\end{lem}
\begin{proof}
We divide the proof into three steps.

\smallskip
\textbf{1. $L^2(\mathbb{R}^n)$-estimate on $\boldsymbol{u}$.}
Multiplying $\eqref{ln}_2$ by $\boldsymbol{u}$ and integrating over $\mathbb{R}^n$, we obtain from the H\"older inequality and the Young inequality that
\begin{align}
&\,\frac{1}{2}\frac{\mathrm{d}}{\mathrm{d}t}\|\boldsymbol{u}\|_{L^2}^2+aa_1\|(h^2+\epsilon^2)^{\frac{1}{4}}\nabla \boldsymbol{u}\|^2_{L^2}+a(a_1+a_2)\|(h^2+\epsilon^2)^{\frac{1}{4}}\diver \boldsymbol{u}\|^2_{L^2} \notag\\
& =\int_{\mathbb{R}^n} \big(\mathcal{H}_0
-a\nabla\sqrt{h^2+\epsilon^2}\cdot Q(\boldsymbol{u})\big)\cdot \boldsymbol{u}\,\mathrm{d}\boldsymbol{x}\leq C\big(\|\mathcal{H}_0\|_{L^2}+\|\boldsymbol{\psi}\|_{L^\infty}\|\nabla \boldsymbol{u}\|_{L^2}\big)\|\boldsymbol{u}\|_{L^2}\notag\\
&\leq C\big(c_0^{3}\|\boldsymbol{u}\|^2_{L^2}+c_3^4\big)+\frac{1}{2}a a_1\|\sqrt{h}\nabla \boldsymbol{u}\|^2_{L^2},\notag
\end{align}
which, along with the Gr\"onwall inequality, yields 
\begin{equation}\label{ul2}
\begin{aligned}
\|\boldsymbol{u}\|_{L^2}^2+\int^t_0\|\sqrt{h}\nabla \boldsymbol{u}\|^2_{L^2}\,\mathrm{d}s
\leq  C\big(\|\boldsymbol{u}_0\|_{L^2}^2+c_3^4t\big)e^{Cc_0^{3}t}\leq Cc^2_0\qquad \text{for $t\in [0,T_2]$}.
\end{aligned}
\end{equation}

\smallskip
\textbf{2. $D^1(\mathbb{R}^n)$-estimate on $\boldsymbol{u}$.} Multiplying $\eqref{ln}_2$ by $\boldsymbol{u}_t$ and  integrating over $\mathbb{R}^n$, we obtain from the H\"older and Young inequalities, and Lemmas \ref{phiphi}--\ref{varphi} that
\begin{equation}
\begin{aligned}
&\,\frac{1}{2}\frac{\mathrm{d}}{\mathrm{d}t}\big(aa_1\|(h^2+\epsilon^2)^{\frac{1}{4}}\nabla \boldsymbol{u}\|^2_{L^2}+a(a_1+a_2)\|(h^2+\epsilon^2)^{\frac{1}{4}}\diver \boldsymbol{u}\|^2_{L^2}\big)+\|\boldsymbol{u}_t\|^2_{L^2},\notag\\
& =\int_{\mathbb{R}^n}\Big( \big(\mathcal{H}_0-a\nabla\sqrt{h^2+\epsilon^2}\cdot Q(\boldsymbol{u})
\big)\cdot \boldsymbol{u}_t\,\mathrm{d}\boldsymbol{x} +\frac{a}{2}\frac{ hh_t}{\sqrt{h^2+\epsilon^2}}\big(a_1|\nabla \boldsymbol{u}|^2+(a_1+a_2)|\diver\boldsymbol{u}|^2\big)\Big)\,\mathrm{d}\boldsymbol{x}\notag\\
&\leq C\big(\|\mathcal{H}_0\|_{L^2}+\|\boldsymbol{\psi}\|_{L^\infty}\|\varphi\|^{\frac{1}{2}}_{L^\infty}\|\sqrt{h}\nabla \boldsymbol{u}\|_{L^2}\big)\|\boldsymbol{u}_t\|_{L^2} +C\|h_t h^{-1}\|_\infty\|\sqrt{h}\nabla \boldsymbol{u}\|^2_{L^2}\notag\\
&\leq C\big(c_4^3\|\sqrt{h}\nabla \boldsymbol{u}\|_{L^2}^2+c_3^4\big)+\frac{1}{2}\|\boldsymbol{u}_t\|^2_{L^2},\notag
\end{aligned}
\end{equation}
which, along with the Gr\"onwall inequality and \eqref{2.14}, implies that, for $t\in [0,T_2]$,
\begin{equation}\label{2.61f}
\|\sqrt{h}\nabla \boldsymbol{u}\|^2_{L^2}+\int^t_0\|\boldsymbol{u}_t\|^2_{L^2}\,\mathrm{d}s
\leq C\big(c^2_0+c_3^4t\big)e^{ Cc_4^3t}\leq  Cc^2_0,\qquad \|\nabla \boldsymbol{u}\|^2_{L^2}\leq Cc^3_0.
\end{equation}

By the definitions of the Lam\'e operator $L$ and $\boldsymbol{\psi}$ and equation $\eqref{ln}_2$, we have
\begin{equation}\label{2.61h}
aL(\sqrt{h^2+\epsilon^2}\boldsymbol{u}) =a\sqrt{h^2+\epsilon^2}L\boldsymbol{u}-aG(\nabla\sqrt{h^2+\epsilon^2},\boldsymbol{u})=\mathcal{H}_1.
\end{equation}
Then it follows from \eqref{hl2}--\eqref{2.61f}, and  Lemmas \ref{df3} and  \ref{psi}--\ref{HH} that 
\begin{equation}\label{2.61i}
\begin{aligned}
&\,\|h\nabla^2\boldsymbol{u}\|_{L^2} \leq\|\sqrt{h^2+\epsilon^2}\nabla^2\boldsymbol{u}\|_{L^2}\\
&\leq C\big(\|\sqrt{h^2+\epsilon^2}\boldsymbol{u}\|_{D^2} +\|\nabla^2\sqrt{h^2+\epsilon^2}\|_{L^3}\|\boldsymbol{u}\|_{L^6}+\|\nabla \sqrt{h^2+\epsilon^2}\|_{L^\infty}\|\nabla \boldsymbol{u}\|_{L^2}\big)\\
&\leq  C \big(\|\mathcal{H}_1\|_{L^2} +c_2^{4}c_3^{\frac{1}{2}}\big) \leq  C\big(\|\boldsymbol{u}_t\|_{L^2}+c_2^{4}c_3^{\frac{1}{2}}\big),
\end{aligned}
\end{equation}
which, along with \eqref{2.61f} and Lemma \ref{varphi}, implies 
\begin{equation*} 
\int_0^{T_2} \|h\nabla^2\boldsymbol{u}\|^2_{L^2}\,\mathrm{d}s
\leq  Cc^2_0,\qquad \int_0^{T_2} \|\nabla^2\boldsymbol{u}\|^2_{L^2}\,\mathrm{d}s
\leq  Cc^4_0.
\end{equation*}

\smallskip
\textbf{3. $D^2(\mathbb{R}^n)$-estimate on $\boldsymbol{u}$.} Applying $\partial_t$ to $\eqref{ln}_2$ yields
\begin{equation}\label{2.61j}
\boldsymbol{u}_{tt}+a\sqrt{h^2+\epsilon^2}L\boldsymbol{u}_t= (\mathcal{H}_0)_t-a\frac{h}{\sqrt{h^2+\epsilon^2}} h_tL\boldsymbol{u}.
\end{equation}
Multiplying \eqref{2.61j} by $\boldsymbol{u}_t$ and integrating over $\mathbb{R}^n$  lead to 
\begin{equation*} 
\begin{aligned}
&\,\frac{1}{2}\frac{\mathrm{d}}{\mathrm{d}t}\|\boldsymbol{u}_{t}\|^2_{L^2}+aa_1\|(h^2+\epsilon^2)^{\frac{1}{4}}\nabla \boldsymbol{u}_t\|_{L^2}^2+a(a_1+a_2)\|(h^2+\epsilon^2)^{\frac{1}{4}}\diver \boldsymbol{u}_t\|_{L^2}^2\\
&=\int_{\mathbb{R}^n} (\mathcal{H}_0)_t\cdot \boldsymbol{u}_t\,\mathrm{d}\boldsymbol{x}-\int_{\mathbb{R}^n} a\frac{h}{\sqrt{h^2+\epsilon^2}} h_tL\boldsymbol{u}\cdot \boldsymbol{u}_t\,\mathrm{d}\boldsymbol{x}-\int_{\mathbb{R}^n} a\nabla\sqrt{h^2+\epsilon^2}\cdot Q(\boldsymbol{u}_t)\cdot \boldsymbol{u}_t\,\mathrm{d}\boldsymbol{x}\\
&\leq C\big(\|(\mathcal{H}_0)_t\|_{L^2}+\|h^{-1}h_t\|_{L^\infty}\|h\nabla^2\boldsymbol{u}\|_{L^2}+\|\varphi\|^{\frac{1}{2}}_{L^\infty}\|\boldsymbol{\psi}\|_{L^\infty}\|\sqrt{h}\nabla \boldsymbol{u}_t\|_{L^2}\big)\|\boldsymbol{u}_t\|_{L^2}\\
&\leq C\big(c^6_4\|\boldsymbol{u}_{t}\|^2_{L^2}+c^9_4\big)+\frac{aa_1}{2}\|\sqrt{h}\nabla \boldsymbol{u}_t\|_{L^2}^2.
\end{aligned}
\end{equation*}
Integrating the above over $(\tau,t)$ with $\tau\in(0,t)$,  we obtain by using \eqref{2.16}, Lemmas \ref{phiphi}--\ref{varphi}  and the Young inequality that
\begin{equation}\label{2.61n}
\|\boldsymbol{u}_t(t)\|^2_{L^2}+aa_1\int^t_\tau\|\sqrt{h}\nabla \boldsymbol{u}_t\|^2_{L^2}\,\mathrm{d}s\leq  \|\boldsymbol{u}_t(\tau)\|^2_{L^2}+Cc_4^{6}\int^t_0\|\boldsymbol{u}_t\|^2\,\mathrm{d}s+Cc_4^9t.
\end{equation}
Thanks to $\eqref{ln}_2$, we have
\begin{equation*} 
\|\boldsymbol{u}_t(\tau)\|_{L^2}\leq  C\big(\|\boldsymbol{w}\|_{L^\infty}\|\nabla \boldsymbol{w}\|_{L^2}+\|\nabla\phi\|_{L^2}+\|(h+\epsilon)L\boldsymbol{u}\|_{L^2}+\|\psi\|_{L^\infty}\|\nabla \boldsymbol{w}\|_{L^2}\big)(\tau),
\end{equation*}
which, along with \eqref{2.14}--\eqref{incc} and Lemma \ref{ls}, yields 
\begin{equation*}\label{2.611p}\begin{aligned}
\limsup_{\tau\to 0}\|\boldsymbol{u}_t(\tau)\|_{L^2}&\leq C\big(\|(\boldsymbol{u}_0,\boldsymbol{\psi}_0)\|_{L^\infty}\|\nabla \boldsymbol{u}_0\|_{L^2}+\|\nabla\phi_0\|_{L^2}+\|\boldsymbol{g}_*\|_{L^2}+\|L\boldsymbol{u}_0\|_{L^2}\big)\leq Cc^2_0.
\end{aligned}
\end{equation*}
Letting $\tau\to 0$ in \eqref{2.61n} and using the Gr\"onwall inequality give 
\begin{equation*} 
\|\boldsymbol{u}_t(t)\|_{L^2}^2+\int^t_0 \|\sqrt{h}\nabla \boldsymbol{u}_t \|^2_{L^2} \,\mathrm{d}s
\leq C \big(c_4^9t+c^4_0\big)e^{Cc_4^{6}t}\leq Cc^4_0\qquad\text{for $t\in [0,T_2]$},
\end{equation*}
which, along with  \eqref{2.61i}, yields 
\begin{equation}\label{2.61s}
\|h\nabla^2 \boldsymbol{u}(t)\|_{L^2}\leq Cc_2^{4}c_3^{\frac{1}{2}},\quad \|\boldsymbol{u}(t)\|_{D^2}\leq Cc_2^{5}c_3^{\frac{1}{2}},\quad \int^{T_2}_0\|\nabla \boldsymbol{u}_t \|^2_{L^2}\,\mathrm{d}t
\leq Cc^5_0.
\end{equation}

Next, we see from \eqref{2.61f} and  Lemma \ref{HH} that
\begin{equation}\label{2.61t}
\|\mathcal{H}_1\|_{D^1}\leq C\big(\|\boldsymbol{u}_t\|_{D^1}+c_3^{10}\big).
\end{equation}
Then it follows from \eqref{2.61h}, \eqref{2.61t}, and Lemmas \ref{psi}, \ref{HH}, and \ref{df3} that
\begin{equation}\label{2.61tkkk}\begin{aligned}
\|h\nabla^3\boldsymbol{u}\|_{L^2}
&\leq\|\sqrt{h^2+\epsilon^2}\nabla^3\boldsymbol{u}\|_{L^2}
\leq C\big( \|\sqrt{h^2+\epsilon^2}\boldsymbol{u}\|_{D^3} +\|\nabla^3\sqrt{h^2+\epsilon^2}\|_{L^2}\|\boldsymbol{u}\|_{L^\infty}\big)\\
&\quad +C\big(\|\nabla^2 \sqrt{h^2+\epsilon^2}\|_{L^3}\|\nabla \boldsymbol{u}\|_{L^6}+\|\nabla \sqrt{h^2+\epsilon^2}\|_{L^\infty}\|\nabla^2 \boldsymbol{u}\|_{L^2}\big)\\
&\leq  C \big(\|\mathcal{H}_1\|_{D^1} +c_3^{10}\big) \leq  C\big(\|\boldsymbol{u}_t\|_{D^1}+c_3^{10}\big),\end{aligned}
\end{equation}
which, along with  \eqref{2.61s}, \eqref{2.61tkkk}, and Lemma \ref{varphi}, implies 
\begin{equation*} 
\begin{aligned}
\int_0^{T_2} \|h\nabla^3\boldsymbol{u}\|^2_{L^2}\,\mathrm{d}t
\leq  Cc^5_0,\qquad \int_0^{T_2} \|\nabla^3\boldsymbol{u}\|^2_{L^2}\,\mathrm{d}t
\leq  Cc^6_0.
\end{aligned}
\end{equation*}

This completes the proof.
\end{proof}

Next, we establish the $D^3(\mathbb{R}^n)$-estimates for $\boldsymbol{u}$.
\begin{lem}\label{hu}
Let $T_3:=\min\{T_2,(1+Cc_4)^{-28}\}$. Then, for any $t\in [0, T_3]$, 
\begin{equation*} 
\begin{aligned}
\|\sqrt h\nabla \boldsymbol{u}_t(t)\|_{L^2}^2+\int^t_0\|\boldsymbol{u}_{tt}\|^2_{L^2}\,\mathrm{d}s\leq Cc^7_0,\qquad  \|\nabla \boldsymbol{u}_t(t)\|^2_{L^2}&\leq Cc^8_0,\\
\|h\nabla^3\boldsymbol{u}(t)\|_{L^2}+\|h\nabla^2\boldsymbol{u}(t)\|_{D^1}+\|\nabla^3\boldsymbol{u}(t)\|_{L^2}&\leq Cc^{11}_3,\\
\int^{t}_0\big(\|h\nabla^2\boldsymbol{u}_t\|^2_{L^2}+\|\boldsymbol{u}_t\|^2_{D^2}+\|\boldsymbol{u}\|^2_{D^4}+\|h\nabla^2\boldsymbol{u}\|_{D^2}^2+\|(h\nabla^2\boldsymbol{u})_t\|^2_{L^2}\big)\,\mathrm{d}s&\leq  Cc^{11}_0.
\end{aligned}
\end{equation*}
\end{lem}
\begin{proof}
We divide the proof into two steps.

\smallskip
\textbf{1.} Multiplying \eqref{2.61j} by $\boldsymbol{u}_{tt}$ and integrating over $\mathbb{R}^n$ lead to 
\begin{align*} 
&\begin{aligned}[b]
&\, \frac{1}{2}\frac{\mathrm{d}}{\mathrm{d}t}\big(aa_1\|(h^2+\epsilon^2)^{\frac{1}{4}}\nabla \boldsymbol{u}_t\|^2_{L^2}+a(a_1+a_2)\|(h^2+\epsilon^2)^{\frac{1}{4}}\mathrm{div}\,\boldsymbol{u}_t\|^2_{L^2}\big)+\|\boldsymbol{u}_{tt}\|_{L^2}^2\\
&=\int_{\mathbb{R}^n} \Big((\mathcal{H}_0)_t-a \frac{h}{\sqrt{h^2+\epsilon^2}} h_tL\boldsymbol{u}
\Big)\cdot \boldsymbol{u}_{tt}\,\mathrm{d}\boldsymbol{x} -\int_{\mathbb{R}^n} a\nabla\sqrt{h^2+\epsilon^2}\cdot Q(\boldsymbol{u}_t)\cdot \boldsymbol{u}_{tt}\,\mathrm{d}\boldsymbol{x}\\
&\quad +\frac{a}{2}\int_{\mathbb{R}^n} \frac{h}{\sqrt{h^2+\epsilon^2}}h_t\big(a_1|\nabla \boldsymbol{u}_t|^2+(a_1+a_2)|\mathrm{div}\, \boldsymbol{u}_t|^2\big)\,\mathrm{d}\boldsymbol{x}
\end{aligned}\\
&\begin{aligned}
&\leq C\big(\|(\mathcal{H}_0)_t\|_{L^2}+\|h^{-1}h_t\|_{L^\infty}\|h\nabla^2\boldsymbol{u}\|_{L^2}\big)\|\boldsymbol{u}_{tt}\|_{L^2},\\
&\quad +
C\big(\|\varphi\|^{\frac{1}{2}}_{L^\infty}\|\boldsymbol{\psi}\|_{L^\infty}\|\sqrt{h}\nabla \boldsymbol{u}_t\|_{L^2}\|\boldsymbol{u}_{tt}\|_{L^2}
+\|h^{-1}h_t\|_{L^\infty}\|\sqrt{h}\nabla \boldsymbol{u}_t\|^2_{L^2}\big)\notag\\
&\leq Cc^3_4\|\sqrt{h}\nabla \boldsymbol{u}_t(t)\|^2_{L^2}+\frac{1}{2}\|\boldsymbol{u}_{tt}\|^2_{L^2}+Cc^{17}_4.
\end{aligned}
\end{align*}
Integrating the above over $(\tau,t)$ yields that, for $t\in [0,T_3]$,
\begin{equation}\label{2.62b}
\begin{aligned}
&\,\|\sqrt{h}\nabla \boldsymbol{u}_t(t)\|^2_{L^2}+\int^t_\tau\|\boldsymbol{u}_{tt}\|^2_{L^2}\,\mathrm{d}s\\
&\leq C\|(h^2+\epsilon^2)^{\frac{1}{4}}\nabla \boldsymbol{u}_t(\tau)\|^2_{L^2}+Cc_4^3\int^t_0\|\sqrt{h}\nabla \boldsymbol{u}_t\|^2_{L^2}\,\mathrm{d}s +Cc_4^{17}t.
\end{aligned}
\end{equation}

Due to  $\eqref{ln}_2$, we have
\begin{equation*} 
\|\sqrt{h}\nabla \boldsymbol{u}_t(\tau)\|_{L^2} \leq \big(\|\sqrt{h}\nabla(\boldsymbol{w}\cdot \nabla \boldsymbol{w}+\nabla\phi+a\sqrt {h^2+\epsilon^2} L\boldsymbol{u} -\boldsymbol{\psi}\cdot Q(\boldsymbol{w}))\|_{L^2}\big)(\tau). 
\end{equation*}
Then it follows from \eqref{2.14}, Lemma \ref{ls}, and Remark \ref{r1} that 
\begin{equation*}
\begin{aligned}
\lim\sup_{\tau\to 0}\|\sqrt{h}\nabla \boldsymbol{u}_t(\tau)\|_{L^2}
& \leq  C\big(\|\sqrt{h_0}\nabla(\boldsymbol{u}_0\cdot\nabla \boldsymbol{u}_0)\|_{L^2}+\|\sqrt{h_0}\nabla^2\phi_0\|_{L^2}\big)\\
&\quad +C\big(\|\sqrt{h_0}\nabla(\boldsymbol{\psi}_0 \cdot  Q(\boldsymbol{u}_0))\|_{L^2} +\|\mathcal{G}_2\|_{L^2}+ \|\sqrt{\varphi_0}\nabla^3 u_0\|_{L^2}\big)
 \leq Cc^{3}_0,
\end{aligned}
\end{equation*}
where we have used the fact that 
\begin{equation*}
\sqrt{h_0}\nabla(\sqrt{h_0^2+\epsilon^2} L\boldsymbol{u}_0)=\Big(\frac{h_0}{\sqrt{h_0^2+\epsilon^2}}\mathcal{G}_2+\epsilon^2\nabla L\boldsymbol{u}_0\frac{\sqrt{h_0}}{\sqrt{h_0^2+\epsilon^2}}\Big).
\end{equation*}

Letting $\tau\to 0$ in \eqref{2.62b}, we see from the Gr\"onwall inequality that, for $t\in [0, T_3]$,
\begin{equation}\label{2.62m}\begin{aligned}
&\|\sqrt h\nabla \boldsymbol{u}_t(t)\|_{L^2}^2+\int^t_0\|\boldsymbol{u}_{tt}\|^2_{L^2}\,\mathrm{d}s\leq Cc^7_0,\qquad  \|\nabla \boldsymbol{u}_t(t)\|^2_{L^2}\leq Cc^8_0,
\end{aligned}
\end{equation}
which, along with \eqref{2.61tkkk} and Lemmas \ref{psi}--\ref{varphi}, yields  
\begin{equation}\label{2.63}
\|h\nabla^3\boldsymbol{u}\|_{L^2}+\|h\nabla^2\boldsymbol{u}\|_{D^1}+\|\nabla^3\boldsymbol{u}\|_{L^2}\leq Cc^{11}_3.
\end{equation}

\smallskip
\textbf{2.} We give the estimates for $(\nabla^2\boldsymbol{u}_t,\nabla^4 \boldsymbol{u})$. First, it follows from Lemmas \ref{HH}--\ref{lu} and \eqref{2.63} that 
\begin{equation}\label{2.61ccvv}
\|\mathcal{H}_2\|_{L^2}\leq  C\big(\|\boldsymbol{u}_{tt}\|_{L^2}+c_4^{9}\big).
\end{equation}
Moreover, recall from the definition of $\mathcal{H}_2$ and \eqref{2.61j} that 
\begin{equation}\label{2.64}
aL(\sqrt{h^2+\epsilon^2}\boldsymbol{u}_t)=\mathcal{H}_2.
\end{equation}
Thus, based on  \eqref{hl2},  \eqref{2.61t}, \eqref{2.62m}, 
and \eqref{2.61ccvv}--\eqref{2.64}, 
we obtain from the classical theory for elliptic equations in Lemma \ref{df3}, and  Lemmas \ref{psi}--\ref{lu} that
\begin{equation}\label{2.65}\begin{aligned}
\displaystyle
\|h\nabla^2\boldsymbol{u}_t\|_{L^2}&\leq C\big(\|\sqrt{h^2+\epsilon^2}\boldsymbol{u}_t\|_{D^2}+\|\nabla \boldsymbol{u}_t\|_{L^2}\|\boldsymbol{\psi}\|_{L^\infty}+\|\nabla^2\sqrt{h^2+\epsilon^2}\|_{L^3}\|\boldsymbol{u}_t\|_{L^6}\big)\\
&\displaystyle
\leq  C\big(\|\mathcal{H}_2\|_{L^2}+c_0^5+c_0^7\big)\leq C\big(\|\boldsymbol{u}_{tt}\|_{L^2}+c_4^{9}\big),\\[2pt]
\displaystyle
\|(h\nabla^2\boldsymbol{u})_t\|_{L^2}&\leq C\big(\|h\nabla^2\boldsymbol{u}_t\|_{L^2}+\|h^{-1}h_t\|_{L^\infty}\|h\nabla^2 \boldsymbol{u}\|_{L^2}\big)
\leq  C\big(\|\boldsymbol{u}_{tt}\|_{L^2}+c_4^{9}\big),\\[2pt]
\displaystyle
\|\boldsymbol{u}\|_{D^4}&\leq C\|(h^2+\epsilon^2)^{-\frac{1}{2}}\mathcal{H}\|_{D^2}
\leq  C\big(c^2_0\|\boldsymbol{u}_{tt}\|_{L^2}+c_4^{11}\big),
\end{aligned}
\end{equation}
where, in the last inequality of \eqref{2.65}, we have used the fact that
\begin{equation}\label{H=Lu}
\mathcal{H}=a\sqrt{h^2+\epsilon^2}L\boldsymbol{u}.
\end{equation}

Next, notice from \eqref{H=Lu} that, for multi-index $\boldsymbol{\xi}=(\xi_1,\cdots\!,\xi_n)$ with $|\boldsymbol{\xi}|=2$,  
\begin{equation*} 
\begin{aligned}
aL(\sqrt{h^2+\epsilon^2}\nabla^{\boldsymbol{\xi}} \boldsymbol{u})=\sqrt{h^2+\epsilon^2}\nabla^{\boldsymbol{\xi}}\big( (\sqrt{h^2+\epsilon^2})^{-1}\mathcal{H}\big)
-aG(\nabla\sqrt{h^2+\epsilon^2},\nabla^{\boldsymbol{\xi}}\boldsymbol{u}).
\end{aligned}
\end{equation*}
Then this, together with \eqref{hl2}, \eqref{2.61t}, \eqref{2.62m}--\eqref{H=Lu}, the classical theory for elliptic equations in Lemma \ref{df3}, and  Lemmas \ref{psi}--\ref{lu},  implies 
\begin{equation*} 
\begin{aligned}
\|\sqrt{h^2+\epsilon^2}\nabla^2\boldsymbol{u} \|_{D^2}
&\leq C\big\|\sqrt{h^2+\epsilon^2}\nabla^{\boldsymbol{\xi}}\big((\sqrt{h^2+\epsilon^2})^{-1}\mathcal{H}\big)\big\|_{L^2}\\
&\quad+C\big(\|\boldsymbol{\psi}\|_{L^\infty}\|\nabla^3\boldsymbol{u}\|_{L^2}+\|\nabla^2\sqrt{h^2+\epsilon^2}\|_{L^3}\|\nabla^2\boldsymbol{u}\|_{L^6}\big)\\
&\leq  C\big(\|\boldsymbol{u}_{t}\|_{D^2}+c_4^{14}\big) \leq  C\big(c_0\|\boldsymbol{u}_{tt}\|_{L^2}+c_4^{14}\big).
\end{aligned}
\end{equation*}

Finally, it follows from the above, \eqref{2.16}, \eqref{2.62m}, \eqref{2.65}, and Lemma \ref{varphi} that  
\begin{equation*} 
\int^{T_3}_0\big(\|h\nabla^2\boldsymbol{u}_t\|^2_{L^2}+\|\boldsymbol{u}_t\|^2_{D^2}+\|\boldsymbol{u}\|^2_{D^4}+\|h\nabla^2\boldsymbol{u}\|_{D^2}^2+\|(h\nabla^2\boldsymbol{u})_t\|^2_{L^2}\big)\,\mathrm{d}t\leq  Cc^{11}_0.
\end{equation*}

This completes the proof.
\end{proof}

We now derive the time-weighted estimates for velocity $\boldsymbol{u}$.
\begin{lem}\label{hu2}
Let $T_4:=\min\{T_3,(1+Cc_5)^{-30}\}$. Then, for any $t\in [0,T_4]$, 
\begin{equation*} 
\begin{aligned}
t\big(\|\boldsymbol{u}_t\|^2_{D^2}+\|h\nabla^2\boldsymbol{u}_t\|^2_{L^2}+ \|\boldsymbol{u}_{tt}\|^2_{L^2}+ \|\boldsymbol{u}\|^2_{D^4}\big)(t)&\leq  Cc^{24}_4\\
\int^t_0s\big(\|\boldsymbol{u}_{tt}\|_{D^1}^2+\|\boldsymbol{u}_t\|^2_{D^3}+\|\sqrt{h} \boldsymbol{u}_{tt}\|_{D^1}^2\big)\,\mathrm{d}s &\leq  Cc^{22}_4.
\end{aligned}
\end{equation*}
\end{lem}
\begin{proof}
We divide the proof into two steps.

\smallskip
\textbf{1.} Applying  $\boldsymbol{u}_{tt}\partial_t$ to \eqref{2.61j} and integrating over $\mathbb{R}^n$ give
\begin{equation}\label{2.72}
\frac{1}{2}\frac{\mathrm{d}}{\mathrm{d}t}\|\boldsymbol{u}_{tt}\|^2_{L^2}+aa_1\|(h^2+\epsilon^2)^{\frac{1}{4}}\nabla \boldsymbol{u}_{tt}\|^2_{L^2}+a(a_1+a_2)\|(h^2+\epsilon^2)^{\frac{1}{4}}\mathrm{div}\,\boldsymbol{u}_{tt}\|^2_{L^2}=\sum^{2}_{i=1}J_i.
\end{equation}
Here, $J_1$ and $J_2$ are given as follows and can be estimated by using Lemmas \ref{HH} and \ref{ale1}, and the H\"older and Young inequities:
\begin{align}
J_1&=\int_{\mathbb{R}^n} (\mathcal{H}_0)_{tt}\cdot \boldsymbol{u}_{tt}\,\mathrm{d}\boldsymbol{x}\leq \|(\mathcal{H}_0)_{tt}\|_{L^2}\|\boldsymbol{u}_{tt}\|_{L^2}\notag\\
&\leq C \big(c_5^4\|\boldsymbol{u}_{tt}\|^2_{L^2}+\|\nabla^2\boldsymbol{w}_{t}\|^2_{L^2}+\|\boldsymbol{w}_{tt}\|^2_{L^2}+c^{-2}_5\|\nabla \boldsymbol{w}_{tt}\|^2_{L^2}+\|\nabla\phi_{tt}\|^2_{L^2}+\|\boldsymbol{\psi}_{tt}\|^2_{L^2}+c^6_4\big),\notag\\
J_2&=-\int_{\mathbb{R}^n} \Big(a\frac{h} {\sqrt{h^2+\epsilon^2}}\nabla h\cdot Q(\boldsymbol{u}_{tt})+a\frac{h}{\sqrt{h^2+\epsilon^2}}h_{tt} L\boldsymbol{u}\notag\\
&\quad\quad\qquad+a\frac{\epsilon^2h^2_t}{(h^2+\epsilon^2)^{\frac{3}{2}}} L\boldsymbol{u}+2a\frac{h}{\sqrt{h^2+\epsilon^2}}h_{t}  L\boldsymbol{u}_t\Big) \cdot \boldsymbol{u}_{tt}\,\mathrm{d}\boldsymbol{x}\notag\\
&\leq C
\big(\|\boldsymbol{\psi}\|_{L^\infty}\|\sqrt{h}\nabla \boldsymbol{u}_{tt}\|_{L^2}\|\varphi\|^{\frac{1}{2}}_{L^\infty}
+\|h^{-1}h_{tt}\|_{L^{2n}}\|h\nabla^2\boldsymbol{u}\|_{L^{\frac{2n}{n-1}}}\big)\|\boldsymbol{u}_{tt}\|_{L^2}\notag\\[1pt]
&\quad +C\big(\|h^{-1}h_t\|^2_{L^\infty}\|\varphi\|_{L^\infty}\|\nabla^2 \boldsymbol{u}\|_{L^2}
+\|h^{-1}h_t\|_{L^\infty}\|h\nabla^2 \boldsymbol{u}_{t}\|_{L^2}\big)\|\boldsymbol{u}_{tt}\|_{L^2}.\notag
\end{align}
Multiplying $\eqref{2.72}$ by $t$ and integrating over $(\tau,t)$ for $0<\tau<t \leq T_4$, together with the above estimates on $J_1$ and $J_2$, \eqref{2.16} and Lemmas \ref{phiphi}--\ref{hu}, yield
\begin{equation}\label{2.73}
t\|\boldsymbol{u}_{tt}(t)\|^2_{L^2}+\frac{aa_1}{4}\int^t_\tau s\|\sqrt{h}\nabla \boldsymbol{u}_{tt}\|^2_{L^2}\,\mathrm{d}s
\leq  \tau\|\boldsymbol{u}_{tt}(\tau)\|^2_{L^2}+Cc^{20}_4+Cc^{4}_5\int^t_\tau s\|\boldsymbol{u}_{tt}\|^2_{L^2}\,\mathrm{d}s.
\end{equation}
To handle the right-hand side of \eqref{2.73}, we see from \eqref{2.62m} and Lemma \ref{bjr} that  there exists a sequence $s_k$ such that $s_k\to 0$ and $s_k\|\boldsymbol{u}_{tt}(s_k,\boldsymbol{x})\|^2_{L^2}\to 0$ as $k\to \infty$. Taking $\tau=s_k$ and letting $k\to \infty$ in $\eqref{2.73}$, we obtain from the Gr\"onwall inequality that
\begin{equation}\label{2.74}
t\|\boldsymbol{u}_{tt}(t)\|^2_{L^2}+\int^t_0 s\big(\|\sqrt{h}\nabla \boldsymbol{u}_{tt}\|^2_{L^2}+\|\sqrt{h}\boldsymbol{u}_{tt}\|^2_{D^1}+\|\nabla \boldsymbol{u}_{tt}\|^2_{L^2}\big)\,\mathrm{d}s\leq Cc^{20}_4\quad \text{for $t\in [0,T_4]$}.
\end{equation}

\smallskip
\textbf{2.} It follows from \eqref{2.65}  and \eqref{2.74}  that
\begin{equation}\label{2.75}
t^{\frac{1}{2}}\|\nabla^2\boldsymbol{u}_t(t)\|_{L^2}+t^{\frac{1}{2}}\|h\nabla^2\boldsymbol{u}_t(t)\|_{L^2}+ t^{\frac{1}{2}}\|\nabla^4\boldsymbol{u}(t)\|_{L^2}\leq  Cc^{12}_4.
\end{equation}

Next, it follows from   Lemmas \ref{HH}--\ref{hu} that 
\begin{equation}\label{2.61ccvvmmmm}
\begin{aligned}
\|\mathcal{H}_2\|_{D^1}&\leq C\big(\|\nabla \boldsymbol{u}_{tt}\|_{L^2}+c^4_4\|\boldsymbol{w}_t\|_{D^2}+c^4_4\|\boldsymbol{u}_t\|_{D^2}+c^{15}_4\big).
\end{aligned}
\end{equation}
Hence, \eqref{hl2}, \eqref{2.64}, \eqref{2.61ccvvmmmm}, and Lemmas \ref{df3} and \ref{psi}--\ref{hu} yield that,  for $t\in [0,T_4]$,
\begin{equation*}
\begin{aligned}
\|h\nabla^3\boldsymbol{u}_t(t)\|_{L^2}& \leq C\big(\|\sqrt{h^2+\epsilon^2}\boldsymbol{u}_t\|_{D^3}+\|\boldsymbol{\psi}\|_{L^\infty}\|\nabla^2\boldsymbol{u}_t\|_{L^2}\big)\\
&\quad +C\big(\|\nabla^2\sqrt{h^2+\epsilon^2}\|_{L^3}\|\nabla\boldsymbol{u}_t\|_{L^6}+\|\nabla^3\sqrt{h^2+\epsilon^2}\|_{L^2}\|\boldsymbol{u}_t\|_{L^\infty}\big)\\
&\leq  C\big(\|\mathcal{H}_2\|_{D^1}+c^4_0\|\boldsymbol{u}_t\|_{D^2}+c_4^{8}\big),
\end{aligned}
\end{equation*}
which, along with \eqref{2.74}--\eqref{2.61ccvvmmmm}, and Lemmas  \ref{varphi} and \ref{hu}, implies 
\begin{equation*}\label{2.78}
\int^{T_4}_0s\big(\|\sqrt{h^2+\epsilon^2}\boldsymbol{u}_t\|_{D^3}^2+\|h\nabla^3\boldsymbol{u}_t\|_{L^2}^2+\|\nabla^3\boldsymbol{u}_t\|^2_{L^2}\big)\,\mathrm{d}s\leq Cc^{22}_4.
\end{equation*}

This completes the proof of Lemma \ref{hu2}.
\end{proof}
It then  follows from Lemmas \ref{phiphi}--\ref{hu2} that, for $0\leq t\leq T_4=\min\{T^*,(1+Cc_5)^{-30}\}$,
\begin{align*}
\begin{aligned}
\big(\|\phi\|^2_{L^\infty\cap D^1\cap D^3}+\|\phi_t\|^2_{H^2}+
\|\phi_{tt}\|^2_{L^2}+\|\boldsymbol{\psi}_t\|^2_{D^1}\big)(t)+\int^t_0\big(\|\phi_{tt}\|^2_{H^1}+\|\boldsymbol{\psi}_{tt}\|^2_{L^2}\big)\,\mathrm{d}s&\leq Cc_4^6,\\
\big(\|\nabla\sqrt{h}\|^2_{L^{2n}}+
\|\boldsymbol{\psi}\|^2_{L^\infty\cap D^{1,n}\cap D^2}\big)(t)\leq Cc^2_0,\ \ \big(\|h^{-1}h_t\|_{L^\infty}^2+
\|\boldsymbol{\psi}_t\|^2_{L^2}\big)(t)&\leq 
Cc_3^{\frac{9}{2}}c^{\frac{3}{2}}_4,\\
\frac{2}{3}\eta^{-2\iota}<h^{-1}(t,\boldsymbol{x})<2c_0, \  
\big(\|\sqrt{h}\nabla \boldsymbol{u}\|^2_{L^2}+\|\boldsymbol{u}\|^2_{H^1}\big)(t)+\int^t_0\big(\|\nabla \boldsymbol{u}\|^2_{H^1}+\|\boldsymbol{u}_t\|^2_{L^2}\big)\,\mathrm{d}s
&\leq  Cc^4_0,\\
\big(\|\boldsymbol{u}\|_{D^2}^2+\|h\nabla^2\boldsymbol{u}\|^2_{L^2}+\|\boldsymbol{u}_t\|^2_{L^2}\big)(t)+\int^t_0\big(\|\boldsymbol{u}\|^2_{D^3}+\|h\nabla^2\boldsymbol{u}\|_{D^1}^2+\|\boldsymbol{u}_t\|^2_{D^1}\big)\,\mathrm{d}s&\leq  Cc_2^{10}c_3,\\
\big(\|\boldsymbol{u}_t\|^2_{D^1}+\|\sqrt{h}\nabla \boldsymbol{u}_t\|^2_{L^2}+\|\boldsymbol{u}\|^2_{D^3}+\|h\nabla^2\boldsymbol{u}\|^2_{D^1}\big)(t)+\int^t_0\|\boldsymbol{u}_t\|^2_{D^2}\,\mathrm{d}s&\leq   Cc^{22}_3,\\
\int^t_0\big(\|\boldsymbol{u}_{tt}\|^2_{L^2}+\|\boldsymbol{u}\|^2_{D^4}+\|h\nabla^2\boldsymbol{u}\|^2_{D^2}+\|(h\nabla^2\boldsymbol{u})_t\|^2_{L^2}\big)\,\mathrm{d}s&\leq Cc^{22}_3,\\
t\big(\|\boldsymbol{u}_t\|^2_{D^2}+\|h\nabla^2\boldsymbol{u}_t\|^2_{L^2}+\|\boldsymbol{u}_{tt}\|^2_{L^2}+\|\boldsymbol{u}\|^2_{D^4}\big)(t) &\leq  Cc^{24}_4,
\end{aligned}\\
\begin{aligned}
\int^t_0\big(s\|\boldsymbol{u}_{tt}\|_{D^1}^2+s\|\boldsymbol{u}_t\|^2_{D^3}+s\|\sqrt{h} \boldsymbol{u}_{tt}\|_{D^1}^2\big)\,\mathrm{d}s  &\leq  Cc^{24}_4.
\end{aligned}
\end{align*}
Therefore, defining the time ($C\geq 4$: the generic constant)
\begin{equation*} 
T^*=
\min\{T,(1+C^{1501}c^{4920}_0)^{-30}\},
\end{equation*}
and constants
\begin{equation*}
c_1^2=Cc^2_0,\quad  c_2^2=Cc^4_0,\quad c_3^2=C^{6}c^{20}_0,\quad
c_4^2=C^{250}c^{820}_0,\quad c_5^2=C^{3001}c^{9840}_0,
\end{equation*}
we arrive at  the following desirable estimates  for $t\in [0,T^*]$:
\begin{align}
\big(\|\phi\|^2_{D^1\cap D^3}+\|\phi_t\|^2_{H^2}+
\|\phi_{tt}\|^2_{L^2}+\|\boldsymbol{\psi}_t\|^2_{D^1}\big)(t)+\int^t_0\big(\|\phi_{tt}\|^2_{H^1}+\|\boldsymbol{\psi}_{tt}\|^2_{L^2}\big)\,\mathrm{d}s\leq  c_5^2,\notag\\
\big(\|\nabla\sqrt{h}\|^2_{L^{2n}}+
\|\boldsymbol{\psi}\|^2_{D^{1,n}\cap D^2}\big)(t)\leq c^2_1,\quad \big(\|h_t\|_{L^\infty}^2+\|\boldsymbol{\psi}_t\|^2_{L^2}\big)(t)\leq  c_4^2,\notag\\
\frac{2}{3}\eta^{-2\iota}\!<h^{-1}(t,\boldsymbol{x})<\!c_1, \  
\big(\|\sqrt{h}\nabla \boldsymbol{u}\|^2_{L^2}+\|\boldsymbol{u}\|^2_{H^1}\big)(t)+\int^t_0\big(\|\nabla \boldsymbol{u}\|^2_{H^1}+\|\boldsymbol{u}_t\|^2_{L^2}\big)\,\mathrm{d}s\leq c_2^2,\notag\\
\big(\|\boldsymbol{u}\|_{D^2}^2+\|h\nabla^2\boldsymbol{u}\|^2_{L^2}+\|\boldsymbol{u}_t\|^2_{L^2}\big)(t)+\int^t_0\big(\|\boldsymbol{u}\|^2_{D^3}+\|h\nabla^2\boldsymbol{u}\|_{D^1}^2+\|\boldsymbol{u}_t\|^2_{D^1}\big)\,\mathrm{d}s\leq  c_3^2,\label{key1kk}\\
\big(\|\boldsymbol{u}_t\|^2_{D^1}+\|\sqrt{h}\nabla \boldsymbol{u}_t\|^2_{L^2}+\|\boldsymbol{u}\|^2_{D^3}+\|h\nabla^2\boldsymbol{u}\|^2_{D^1}\big)(t)+\int^t_0\|\boldsymbol{u}_t\|^2_{D^2}\,\mathrm{d}s\leq  c_4^2,\notag\\
\int^t_0\big(\|\boldsymbol{u}_{tt}\|^2_{L^2}+\|\boldsymbol{u}\|^2_{D^4}+\|h\nabla^2\boldsymbol{u}\|^2_{D^2}+\|(h\nabla^2\boldsymbol{u})_t\|^2_{L^2}\big)\,\mathrm{d}s\leq  c_4^2,\quad t\|h\nabla^2\boldsymbol{u}_t(t)\|^2_{L^2}\leq  c_5^2,\notag\\
t\big(\|\boldsymbol{u}_t\|^2_{D^2}+\|\boldsymbol{u}_{tt}\|^2_{L^2}+\|\boldsymbol{u}\|^2_{D^4}\big)(t)+
\int^t_0s\big(\|\boldsymbol{u}_{tt}\|_{D^1}^2+\|\boldsymbol{u}_t\|^2_{D^3}+\|\sqrt{h} \boldsymbol{u}_{tt}\|_{D^1}^2\big)\,\mathrm{d}s\leq  c_5^2.\notag
\end{align}

\subsection{Inviscid Limit of the Artificial Viscosity}

Based on the uniform estimates in \eqref{key1kk}, now we can obtain the local well-posedness of problem \eqref{ln} with $\epsilon=0$.

\begin{lem}\label{epsilon0}
Let \eqref{canshu} hold, $\eta>0$, and $(\phi_0,\boldsymbol{u}_0)$ satisfy \eqref{th78qq}--\eqref{th78zxq}. Then there exists $T^*>0$ independent of $\eta$ and a unique classical solution $(\phi^\eta,\boldsymbol{u}^\eta, h^\eta)$ in $[0,T^*]\times\mathbb{R}^n$ of  problem \eqref{ln} with  $\epsilon=0$ satisfying \eqref{linearregularity} with $T$ replaced by $T^*$, and $\boldsymbol{\psi}^\eta=\frac{a\delta}{\delta-1}\nabla h^\eta$. Moreover, the uniform estimates in \eqref{key1kk} that are independent of $\eta$ still hold for $(\phi^\eta,\boldsymbol{u}^\eta,h^\eta)$, and $C^{-1}\leq (gh^{-1})(t,\boldsymbol{x})\leq C$ in $[0,T^*]\times\mathbb{R}^n$.
\end{lem}
\begin{proof} 
We divide the proof into two steps.

\smallskip
\textbf{1. Existence}. First, since $(\phi_0,\boldsymbol{u}_0)$ satisfy \eqref{th78qq}--\eqref{th78zxq}, it follows from  \eqref{etainitial}--\eqref{initialapproximation} that  there exists a constant $c_0\geq 1$ independent of $\eta$ such that \eqref{2.14} holds. Then, according to  Lemmas \ref{ls}--\ref{hu2}, for any fixed $\epsilon>0$ and $\eta>0$, there exist $T^*>0$ independent of $(\epsilon,\eta)$ and a unique classical solution $(\phi^{\epsilon,\eta}, \boldsymbol{u}^{\epsilon,\eta}, h^{\epsilon,\eta})(t,\boldsymbol{x})$ in $[0,T^*]\times \mathbb{R}^n$  of  problem \eqref{ln} satisfying  the estimates in   \eqref{key1kk}  which, along with  $\eqref{ln}_3$, yields that, for $t\in [0,T^*]$,
\begin{equation}\label{related}
\|h^{\epsilon,\eta}(t)\|_{L^\infty}+\|\nabla h^{\epsilon,\eta}(t)\|_{L^2}+ \|h^{\epsilon,\eta}_t(t)\|_{L^2}\leq C(A, \eta, a_1, a_2, \gamma, \delta, T^*, \phi_0,\boldsymbol{u}_0).
\end{equation}

Thus, it follows from  \eqref{key1kk}--\eqref{related}  and  the Aubin--Lions Lemma \ref{aubin-lions}   that, for any $R> 0$,  there exists a subsequence of solutions (still denoted by) $(\phi^{\epsilon,\eta}, \boldsymbol{u}^{\epsilon,\eta},  h^{\epsilon,\eta} )$ converging  to  a limit $(\phi^\eta,  \boldsymbol{u}^\eta,  h^\eta) $ as $\epsilon \to 0$:
\begin{equation}\label{ert1}
(\phi^{\epsilon,\eta}, \boldsymbol{u}^{\epsilon,\eta},  h^{\epsilon,\eta}) \to (\phi^\eta,  \boldsymbol{u}^\eta,  h^\eta ) \quad \text{ in $C([0,T^*];H^2(B_R))$ \ as $\epsilon \to 0$},
\end{equation}
where $B_R$ is a ball centered at the origin with radius $R$, such that 
\begin{align}\label{ruojixian}
\begin{aligned}
\boldsymbol{u}^{\epsilon,\eta}\rightharpoonup   \boldsymbol{u}^\eta \qquad &\text{weakly* \ in } L^\infty([0,T^*];H^3(\mathbb{R}^n)),\\ 
(\nabla  \phi^{\epsilon,\eta},  \phi^{\epsilon,\eta}_t,\boldsymbol{\psi}^{\epsilon,\eta},h^{\epsilon,\eta}_t)\rightharpoonup (\nabla \phi^\eta, \phi^\eta_t,\boldsymbol{\psi}^\eta,h^{\eta}_t) \qquad &\text{weakly* \ in } L^\infty([0,T^*];H^2(\mathbb{R}^n)),\\
\boldsymbol{u}^{\epsilon,\eta}_t \rightharpoonup   \boldsymbol{u}^\eta_t \qquad &\text{weakly* \ in } L^\infty([0,T^*];H^1(\mathbb{R}^n)),\notag
\end{aligned}\hspace{0.8mm}\\
\begin{aligned}
t^{\frac{1}{2}}(\nabla^2\boldsymbol{u}_t^{\epsilon,\eta},\nabla^4\boldsymbol{u}^{\epsilon,\eta}) \rightharpoonup  t^{\frac{1}{2}}(\nabla^2\boldsymbol{u}_t^{\eta}, \nabla^4\boldsymbol{u}^{\eta}) \qquad &\text{weakly* \ in } L^\infty([0,T^*];L^2(\mathbb{R}^n)),\\
(\phi^{\epsilon,\eta}_{tt}, t^{\frac{1}{2}}\boldsymbol{u}^{\epsilon,\eta}_{tt}) \rightharpoonup ( \phi^{\eta}_{tt}, t^{\frac{1}{2}} \boldsymbol{u}^{\eta}_{tt}) \qquad &\text{weakly* \ in } L^\infty([0,T^*];L^2(\mathbb{R}^n)),\\
(h^{\epsilon,\eta},(h^{\epsilon,\eta})^{-1}h^{\epsilon,\eta}_t) \rightharpoonup  (h^{\eta},(h^{\eta})^{-1}h^{\eta}_t) \qquad &\text{weakly* \ in } L^\infty([0,T^*];L^\infty(\mathbb{R}^n)),\\
\nabla\sqrt{h^{\epsilon,\eta}} \rightharpoonup  \nabla \sqrt{h^{\eta}}\qquad  &\text{weakly* \ in } L^\infty([0,T^*];L^{2n}(\mathbb{R}^n)),\\
\boldsymbol{u}^{\epsilon,\eta}_{t} \rightharpoonup  \boldsymbol{u}^{\eta}_{t}\qquad &\text{weakly \hspace{2mm}  in } L^2([0,T^*];H^2(\mathbb{R}^n)), 
\end{aligned}\\
\begin{aligned}
\phi^{\epsilon,\eta}_{tt} \rightharpoonup    \phi^{\eta}_{tt} \qquad &\text{weakly \hspace{2mm}  in } L^2([0,T^*];H^1(\mathbb{R}^n)),\\
(\boldsymbol{\psi}^{\epsilon,\eta}_{tt},\nabla^4 \boldsymbol{u}^{\epsilon,\eta},
\boldsymbol{u}^{\epsilon,\eta}_{tt}) \rightharpoonup  (\boldsymbol{\psi}^{\eta}_{tt},\nabla^4 \boldsymbol{u}^{\eta} ,\boldsymbol{u}^{\eta}_{tt})\qquad &\text{weakly \hspace{2mm}  in } L^2([0,T^*];L^2(\mathbb{R}^n)),\\
t^{\frac{1}{2}}(\nabla \boldsymbol{u}^{\epsilon,\eta}_{tt},\nabla^3\boldsymbol{u}^{\epsilon,\eta}_t) \rightharpoonup  t^{\frac{1}{2}}(\nabla \boldsymbol{u}^{\eta}_{tt},\nabla^3\boldsymbol{u}^{\eta}_t) \qquad &\text{weakly \hspace{2mm}  in } L^2([0,T^*];L^2(\mathbb{R}^n)),\notag
\end{aligned}\hspace{2.5mm}
\end{align}
which implies that $(\phi^\eta,  \boldsymbol{u}^\eta, h^\eta) $  satisfies  \eqref{key1kk}--\eqref{related} except the weighted estimates on $\boldsymbol{u}^\eta$.

Now,  the uniform estimates on $(\phi^\eta,  \boldsymbol{u}^\eta,  h^\eta) $ obtained above  and  the    convergences in \eqref{ert1}--\eqref{ruojixian} imply
\begin{equation}\label{ruojixian2}
\begin{aligned}
\sqrt{h^{\epsilon,\eta}}( \nabla \boldsymbol{u}^{\epsilon,\eta}, \nabla \boldsymbol{u}^{\epsilon,\eta}_t) \rightharpoonup \sqrt{h^{\eta}} ( \nabla \boldsymbol{u}^{\eta},\nabla \boldsymbol{u}^{\eta}_t) \quad &\text{weakly* in } L^\infty([0,T^*];L^2(\mathbb{R}^n)),\\
h^{\epsilon,\eta}\nabla^2 \boldsymbol{u}^{\epsilon,\eta} \rightharpoonup  h^{\eta}\nabla^2 \boldsymbol{u}^{\eta} \quad &\text{weakly* in } L^\infty([0,T^*];H^1(\mathbb{R}^n)),\\
(h^{\epsilon,\eta}\nabla^4 \boldsymbol{u}^{\epsilon,\eta},
 (h^{\epsilon,\eta}\nabla^2 \boldsymbol{u}^{\epsilon,\eta})_t)
\rightharpoonup (h^{\eta}\nabla^4 \boldsymbol{u}^{\eta},
(h^{\eta}\nabla^2 \boldsymbol{u}^{\eta})_t) 
\quad &\text{weakly \hspace{2mm} in } L^2([0,T^*]; L^2(\mathbb{R}^n)),
\end{aligned}
\end{equation}
which implies that $(\phi^\eta,  \boldsymbol{u}^\eta,  h^\eta) $  satisfies also the uniform weighted  estimates of   $ \boldsymbol{u}^{\eta}$ in \eqref{key1kk}.

Now we are ready to show that $(\phi^\eta,  \boldsymbol{u}^\eta, h^\eta) $ is a weak solution in the sense of distributions to  \eqref{ln} with  $\epsilon=0$.
First, multiplying $\eqref{ln}_1$ by  any given  $X(t,\boldsymbol{x})\in C^\infty_c ([0,T^*)\times \mathbb{R}^n)$ on both sides and integrating over $[0,T^*)\times \mathbb{R}^n$, we have
\begin{equation*} 
\int_0^{T^*} \int_{\mathbb{R}^n}  \Big(\phi^{\epsilon,\eta} X_t-\boldsymbol{w}\cdot\nabla\phi^{\epsilon,\eta}X-(\gamma-1)\phi^{\epsilon,\eta} \diver\boldsymbol{w}X\Big)\,\mathrm{d}\boldsymbol{x}\mathrm{d}t+ \int_{\mathbb{R}^n}  \phi_0^{\eta} X(0,\boldsymbol{x})\,\mathrm{d}\boldsymbol{x}=0.
\end{equation*}
According to   the  uniform estimates obtained above and  the convergences shown in \eqref{ert1}--\eqref{ruojixian}, we  can pass to the  limit  $\epsilon \to 0$ in the above to obtain 
\begin{equation*}
\int_0^{T^*} \int_{\mathbb{R}^n}  \Big(\phi^{\eta} X_t-\boldsymbol{w}\cdot\nabla\phi^{\eta}X-(\gamma-1)\phi^{\eta} \diver\boldsymbol{w}X\Big)\,\mathrm{d}\boldsymbol{x}\mathrm{d}t+ \int_{\mathbb{R}^n}  \phi_0^{\eta} X(0,\boldsymbol{x})\,\mathrm{d}\boldsymbol{x}=0.
\end{equation*}
Similarly, we can show that $(\boldsymbol{u}^\eta,h^\eta)$ also satisfies equations $\eqref{ln}_2$--$\eqref{ln}_3$ and the initial data in the sense of distributions. Therefore, $(\phi^\eta,  \boldsymbol{u}^\eta,  h^\eta) $ is a weak solution of  problem \eqref{ln} with  $\epsilon=0$ in the sense of distributions, satisfying  
\begin{equation*}\label{zheng}
\begin{aligned}
&\nabla \phi^\eta\in L^\infty([0,T^*]; H^2(\mathbb{R}^n)), \quad 0<h^\eta\in L^\infty([0,T^*]\times \mathbb{R}^n),\quad \nabla(\phi^\eta)^\iota\in L^\infty([0,T^*]; L^{2n}(\mathbb{R}^n)),\\
&(\nabla h^\eta, h^\eta_t)\in L^\infty([0,T^*];H^2(\mathbb{R}^n)), \  \ \ \boldsymbol{u}^\eta\in L^\infty([0,T^*]; H^3(\mathbb{R}^n))\cap L^2([0,T^*]; H^4(\mathbb{R}^n)), \\
&\boldsymbol{u}^\eta_t \in L^\infty([0,T^*]; H^1(\mathbb{R}^n))\cap L^2([0,T^*]; D^2(\mathbb{R}^n)),\ \   \boldsymbol{u}^\eta_{tt}\in L^2([0,T^*];L^2(\mathbb{R}^n)),\\
&t^{\frac{1}{2}}\boldsymbol{u}^\eta\in L^\infty([0,T^*];D^4(\mathbb{R}^n)),\quad   t^{\frac{1}{2}}\boldsymbol{u}^\eta_t\in L^\infty([0,T^*];D^2(\mathbb{R}^n))\cap L^2([0,T^*]; D^3(\mathbb{R}^n)),\\
&t^{\frac{1}{2}}\boldsymbol{u}^\eta_{tt}\in L^\infty([0,T^*];L^2(\mathbb{R}^n))\cap L^2([0,T^*];D^1(\mathbb{R}^n)).
\end{aligned}
\end{equation*}
This implies that $(\phi^\eta,\boldsymbol{u}^\eta,h^\eta)$  satisfies equations $\eqref{ln}_1$--$\eqref{ln}_3$  with $\epsilon=0$ pointwisely.
\smallskip

\textbf{2. Uniqueness and time continuity.}  Due to  $h^\eta>\frac{1}{2c_0}$,  the uniqueness and the time continuity of the classical solution obtained above can be given by the same arguments used in Lemma \ref{ls}. Here we omit the details.
\end{proof}

\subsection{Nonlinear Approximation Solutions away from Vacuum}
Now we  give the local well-posedness of the classical solution to the following nonlinear problem when   $\eta>0$:
\begin{equation}\label{nl}
\begin{cases}
\phi^{\eta}_t+\boldsymbol{u}^{\eta}\cdot\nabla\phi^{\eta}+(\gamma-1)\phi^{\eta} \mathrm{div}\boldsymbol{u}^{\eta}=0,\\[2pt]
\displaystyle\boldsymbol{u}^{\eta}_t+\boldsymbol{u}^{\eta}\cdot \nabla \boldsymbol{u}^{\eta}+\nabla\phi^{\eta}+a h^{\eta} L\boldsymbol{u}^{\eta}= \boldsymbol{\psi}^{\eta} \cdot Q(\boldsymbol{u}^{\eta}),\\[3pt]
h^{\eta}_t+\boldsymbol{u}^\eta\cdot\nabla h^\eta+(\delta-1) h^\eta \mathrm{div} \boldsymbol{u}^{\eta}=0,\\[3pt]
(\phi^{\eta},\boldsymbol{u}^{\eta},h^{\eta})|_{t=0}=(\phi^\eta_0,\boldsymbol{u}^\eta_0,h_0^\eta)
=(\phi_0+\eta,\boldsymbol{u}_0,(\phi_0+\eta)^{2\iota})\quad \text{in $\mathbb{R}^n$},\\[3pt]
(\phi^{\eta},\boldsymbol{u}^{\eta},h^{\eta})\to (\eta,\boldsymbol{0},\eta^{2\iota}) \quad \text{as $|\boldsymbol{x}|\to \infty$} \qquad \text{for $t\geq 0$},
\end{cases}
\end{equation}
where $\boldsymbol{\psi}^\eta=\frac{a\delta}{\delta-1}\nabla h^\eta$.

\begin{thm}\label{nltm}
Let \eqref{canshu} hold, $\eta>0$, and  $(\phi_0,\boldsymbol{u}_0)$ satisfy \eqref{th78qq}--\eqref{th78zxq}. Then there exist $T_*>0$ independent of $\eta$ and a unique classical solution 
$(\phi^\eta, \boldsymbol{u}^\eta, h^\eta)$ in $[0,T_*]\times\mathbb{R}^n$ of problem \eqref{nl} satisfying \eqref{linearregularity}, where $T_*$ is independent of $\eta$. Moreover, the $\eta$-uniform estimates in \eqref{key1kk} still hold for $(\phi^\eta,\boldsymbol{u}^\eta,h^\eta)$ with $T^*$ replaced by $T_*$. 
\end{thm}
\begin{proof}
We divide the proof into three steps.  

\smallskip
\textbf{1.} Since $(\phi_0,\boldsymbol{u}_0)$ satisfy \eqref{th78qq}--\eqref{th78zxq}, it follows from  \eqref{etainitial}--\eqref{initialapproximation} that  there exists a constant $c_0\geq 1$ independent of $\eta$ such that \eqref{2.14} holds. Next, denote $\mathrm{X}=(\Phi,\mathrm{H})^\top$, and  let $(\phi^0,\boldsymbol{u}^0,h^0)$ be the solution of the following Cauchy problem in $(0,\infty)\times\mathbb{R}^n$:
\begin{equation*} 
\begin{cases}
\displaystyle \mathrm{X}_t+\boldsymbol{u}_0\cdot \nabla \mathrm{X}=0, \\
\mathbf{U}_t-\mathrm{H}\Delta \mathbf{U}=\boldsymbol{0}, \\
(\Phi,\mathbf{U},\mathrm{H})|_{t=0}=(\phi^\eta_0,\boldsymbol{u}^\eta_0,h^\eta_0)\quad \text{in $\mathbb{R}^n$},\\
(\Phi,\mathbf{U},\mathrm{H})\to (\eta,\boldsymbol{0},\eta^{2\iota}) \quad \quad  \ \ \ \ \text{as $|\boldsymbol{x}|\to \infty$} \quad \text{for $t\geq 0$}.
\end{cases}
\end{equation*}

Choose $\bar{T}\in (0,T^*]$ small enough such that
\begin{align}
\begin{aligned}
0<(h^0)^{-1}\leq c_1,\ \
\sup_{t\in[0,\bar{T}]}\|\nabla h^0(t)\|^2_{L^\infty\cap D^{1,n}\cap D^2}&\leq c^2_1,\notag\\
\sup_{t\in[0,\bar{T}]}\|\boldsymbol{u}^0(t)\|^2_{H^1}
+\int^{\bar{T}}_0\big(\|\boldsymbol{u}^0\|^2_{D^2}+\|\boldsymbol{u}^0_t\|^2_{L^2}\big)\,\mathrm{d}t&\leq c_2^2,\notag
\end{aligned}\\
\begin{aligned}\label{it0}
\sup_{t\in[0,\bar{T}]}\big(\|\boldsymbol{u}^0\|_{D^2}^2+\|h^0\nabla^2\boldsymbol{u}^0\|^2_{L^2}+\|\boldsymbol{u}^0_t\|^2_{L^2}\big)(t)
+\int^{\bar{T}}_0\big(\|\boldsymbol{u}^0\|^2_{D^3}+\|\boldsymbol{u}^0_t\|^2_{D^1}\big)\,\mathrm{d}t\leq c_3^2,\\
\sup_{t\in[0,\bar{T}]}\big(\|\boldsymbol{u}^0_t\|^2_{D^1}+\|h^0_t\|^2_{D^1}+\|\boldsymbol{u}^0\|^2_{D^3}\big)(t)
+\int^{\bar{T}}_0\big(\|\boldsymbol{u}^0_t\|^2_{D^2}+\|\boldsymbol{u}^0\|^2_{D^4}+\|\boldsymbol{u}^0_{tt}\|^2_{L^2}\big)\,\mathrm{d}t
\leq c_4^2,\\[-2pt]
\sup_{t\in[0,\bar{T}]}\|(h_0)^{-1}h^0_t(t)\|^2_{L^{\infty}}
+\int^{\bar{T}}_0\big(\|(h^0\nabla^2\boldsymbol{u}^0)_t\|^2_{L^2}+\|h^0\nabla^2\boldsymbol{u}^0\|^2_{D^2}\big)\,\mathrm{d}t\leq c_4^2,\\[4pt]
\sup_{t\in[0,\bar{T}]}\big(\|\sqrt{h^0}\nabla \boldsymbol{u}^0_t\|^2_{L^2}+\|h^0\nabla^2\boldsymbol{u}^0\|^2_{D^1}\big)(t)\leq c^2_4,
\end{aligned}\\[1pt]
\begin{aligned} 
\sup_{t\in[0,\bar{T}]}\big(t\|\boldsymbol{u}^0_t(t)\|^2_{D^2}+t\|h^0\nabla^2\boldsymbol{u}^0_t(t)\|^2_{L^2}+t\|\boldsymbol{u}^0_{tt}(t)\|^2_{L^2}+t\|\boldsymbol{u}^0(t)\|^2_{D^4}\big)\leq  c_5^2,\\
\int^t_0\big(s\|\boldsymbol{u}^0_{tt}\|_{D^1}^2+s\|\boldsymbol{u}^0_t\|^2_{D^3}+s\|\sqrt{h^0} \boldsymbol{u}^0_{tt}\|_{D^1}^2+\|h^0_{tt}\|^2_{D^1}\big)\,\mathrm{d}s\leq  c_5^2.\notag
\end{aligned}
\end{align}

\smallskip
\textbf{2. Existence}. Let $(\boldsymbol{w},g)=(\boldsymbol{u}^0, h^0)$ in problem  \eqref{ln} with $\epsilon=0$. Then, according to Lemma \ref{epsilon0},
we  can obtain a unique  classical solution $(\phi^1,\boldsymbol{u}^1,h^1)$, which solves this problem locally in time. Inductively, we construct approximate sequences $(\phi^{k+1},\boldsymbol{u}^{k+1},h^{k+1})$ as follows:
given $(\boldsymbol{u}^k, h^k)$ for $k\geq 1$, define  $(\phi^{k+1},\boldsymbol{u}^{k+1},h^{k+1})$ by solving the  linear problem:
\begin{equation}\label{k+1}
\begin{cases}
\displaystyle
\phi^{k+1}_t+\boldsymbol{u}^k\cdot\nabla\phi^{k+1}+(\gamma-1)\phi^{k+1}\mathrm{div}\boldsymbol{u}^k=0,\\[2pt]
\displaystyle
\boldsymbol{u}_t^{k+1}+\boldsymbol{u}^k\cdot\nabla \boldsymbol{u}^k+\nabla\phi^{k+1}+ah^{k+1}L\boldsymbol{u}^{k+1}
=\boldsymbol{\psi}^{k+1}\cdot Q(\boldsymbol{u}^k),\\[2pt]
\displaystyle
h^{k+1}_t+\boldsymbol{u}^k\cdot \nabla h^{k+1}+(\delta-1)h^k\mathrm{div}\boldsymbol{u}^k=0,\\[2pt]
\displaystyle
(\phi^{k+1},\boldsymbol{u}^{k+1},h^{k+1})|_{t=0}=(\phi^\eta_0,\boldsymbol{u}^\eta_0,h^\eta_0)
 \qquad \text{in $\mathbb{R}^n$}, \\[2pt]
\displaystyle
(\phi^{k+1},\boldsymbol{u}^{k+1},h^{k+1})\to (\eta,\boldsymbol{0},\eta^{2\iota}) \quad\quad \ \qquad \text{as $|\boldsymbol{x}|\to \infty$} \quad \text{for $t\geq 0$},
\end{cases}
\end{equation}
where $\boldsymbol{\psi}^{k+1}=\frac{a\delta}{\delta-1}\nabla h^{k+1}$.
It follows from Lemma \ref{epsilon0} that the solution  $(\phi^{k+1},\boldsymbol{u}^{k+1},h^{k+1})$ $(k\in \mathbb{N})$ satisfies  the uniform estimates \eqref{key1kk}  and $C^{-1}\leq (h^k(h^{k+1})^{-1})(t,\boldsymbol{x})\leq C$ in $[0,\bar{T}]\times\mathbb{R}^n$.
To show the strong convergence of $(\phi^k,\boldsymbol{u}^k,h^k)$, we set
\begin{equation*}
\begin{aligned}
&\bar{\phi}^{k+1}=\phi^{k+1}-\phi^k,\ \ \bar{\boldsymbol{u}}^{k+1}=\boldsymbol{u}^{k+1}-\boldsymbol{u}^k, \ \  \bar{\boldsymbol{\psi}}^{k+1}=\boldsymbol{\psi}^{k+1}-\boldsymbol{\psi}^k,\ \  \bar{h}^{k+1}=h^{k+1}-h^k.
\end{aligned}
\end{equation*}
Then $(\phi^k,\boldsymbol{u}^k,h^k)$ solves the following problem due to \eqref{k+1}: 
\begin{equation}\label{k+3}
\begin{cases}
\displaystyle
\bar{\phi}^{k+1}_t+\boldsymbol{u}^k\cdot\nabla\bar{\phi}^{k+1}+\bar{\boldsymbol{u}}^k\cdot\nabla\phi^k+(\gamma-1)\big(\bar{\phi}^{k+1}\mathrm{div}\boldsymbol{u}^k+\phi^k\mathrm{div}\bar{\boldsymbol{u}}^k\big)=0,\\[4pt]
\displaystyle
\bar{\boldsymbol{u}}_t^{k+1}+\boldsymbol{u}^k\cdot\nabla \bar{\boldsymbol{u}}^k+\bar{\boldsymbol{u}}^k\cdot\nabla \boldsymbol{u}^{k-1}+\nabla\bar{\phi}^{k+1}
+ah^{k+1}L\bar{\boldsymbol{u}}^{k+1}+a\bar{h}^{k+1}L\boldsymbol{u}^{k}\\[2pt]
\displaystyle
\quad =\bar{\boldsymbol{\psi}}^{k+1}\cdot Q(\boldsymbol{u}^k)+\boldsymbol{\psi}^{k}\cdot Q(\bar{\boldsymbol{u}}^k),\\[2pt]
\displaystyle
\bar{h}^{k+1}_t+\boldsymbol{u}^k\cdot\nabla \bar{h}^{k+1}+\bar{\boldsymbol{u}}^k\cdot\nabla h^k+(\delta-1)\big(h^k\mathrm{div}{\boldsymbol{u}}^k-{h}^{k-1}\mathrm{div}\boldsymbol{u}^{k-1}\big)=0,\\[2pt]
\displaystyle
(\bar{\phi}^{k+1},\bar{\boldsymbol{u}}^{k+1},\bar{h}^{k+1})|_{t=0}=(0,\boldsymbol{0},0) \qquad\text{in $\mathbb{R}^n$},
\\[4pt]
\displaystyle
(\bar{\phi}^{k+1},\bar{\boldsymbol{u}}^{k+1},\bar{h}^{k+1})\to (0,\boldsymbol{0},0) \ \ \ \ \qquad\text{as $|\boldsymbol{x}|\to \infty$} \quad \text{for $t\geq 0$}.
\end{cases}
\end{equation}

Besides, by the same argument as in the proof of \cite[Lemma 3.11]{dxz}, we can obtain the following regularity for $(\bar{\phi}^{k+1},\bar{\boldsymbol{\psi}}^{k+1},\bar{h}^{k+1})$:
\begin{equation*} 
(\bar{\phi}^{k+1},\ \bar{\boldsymbol{\psi}}^{k+1}, \ \bar{h}^{k+1}) \in L^\infty([0,\bar{T}];H^2(\mathbb{R}^n)) \qquad \text{for $k\in \mathbb{N}$}.
\end{equation*}
Moreover, it follows from the relation $\boldsymbol{\psi}^{k+1}=\frac{a\delta}{\delta-1}\nabla h^{k+1}$ and $\eqref{k+1}_3$ that 
\begin{equation}\label{psibar1}
\begin{aligned}
&\,\bar{\boldsymbol{\psi}}^{k+1}_t+\nabla\big(\boldsymbol{u}^k\cdot \bar{\boldsymbol{\psi}}^{k+1}+\bar{\boldsymbol{u}}^k\cdot\boldsymbol{\psi}^k\big)+(\delta-1)\big(\bar{\boldsymbol{\psi}}^k\mathrm{div}\boldsymbol{u}^k+\boldsymbol{\psi}^{k-1}\mathrm{div}\bar{\boldsymbol{u}}^k\big)\\[2pt]
&+a\delta\big(h^k\nabla\diver \bar{\boldsymbol{u}}^k+\bar{h}^k\nabla\mathrm{div}\boldsymbol{u}^{k-1}\big)=\boldsymbol{0}.
\end{aligned}
\end{equation}

\textbf{2.1. Estimates for $\varphi^{k+1}\bar{h}^{k+1}$.} Denote $\varphi^{k}=(h^{k})^{-1}$. Then, according to $\eqref{k+1}_3$, 
\begin{equation*}
\varphi^{k+1}_t+\boldsymbol{u}^k\cdot \nabla \varphi^{k+1}+(1-\delta)h^k (\varphi^{k+1})^2\diver\boldsymbol{u}^k=0.
\end{equation*}
Multiplying $\eqref{k+3}_3$ by $2\overline{h}^{k+1}(\varphi^{k+1})^2$ and integrating over $\mathbb{R}^n$ yield
\begin{equation}\label{go64aah}
\frac{\rm d}{{\rm d}t}\|\varphi^{k+1}\bar{h}^{k+1}\|^2_{L^2}\\ 
\leq \frac{C}{\sigma}\|\varphi^{k+1}\bar{h}^{k+1}\|^2_{L^2} + \sigma \big(\|\bar{\boldsymbol{\psi}}^{k+1}\|^2_{L^2}
+ \|(\bar{\boldsymbol{u}}^k,\, \sqrt{h^k}\diver\bar{\boldsymbol{u}}^k,\,\varphi^k\bar{h}^k)\|_{L^2}^2\big).
\end{equation}

\textbf{2.2. Estimates for $(\bar{\phi}^{k+1},\bar{\boldsymbol{u}}^{k+1},\bar{\boldsymbol{\psi}}^{k+1})$.} First, it follows from $\eqref{k+3}_2$ that
\begin{equation}\label{hubarr}
\begin{aligned}
\|h^k\nabla^2\bar{\boldsymbol{u}}^k\|_{L^2}& \leq C\big(\|(\bar{\boldsymbol{u}}^k,\,\sqrt{h^k}\nabla \bar{\boldsymbol{u}}^k,\,\bar{\boldsymbol{u}}^k_t)\|_{L^2}+\|\bar{\boldsymbol{u}}^{k-1}\|_{L^2}\big)\\
&\quad +C\big(\|\sqrt{h^{k-1}}\nabla\bar{\boldsymbol{u}}^{k-1}\|_{L^2}+\|(\bar{\boldsymbol{\psi}}^{k},\,\varphi^{k}\bar{h}^{k},\,\nabla \bar{\phi}^{k})\|_{L^2}\big),
\end{aligned}
\end{equation}
where we have used the fact that $\nabla \varphi^k=-(\varphi^k)^2\nabla h^k$ and the following estimates:
\begin{equation}\label{2Dxinfaxian}
\begin{aligned}
\|\varphi^k\bar{h}^{k}\|_{L^3}
&\leq C\big(\|\varphi^k\bar{h}^{k}\|_{L^2}+\|\nabla(\varphi^k\bar{h}^{k})\|_{L^2}\big)\leq C\|(\bar{\boldsymbol{\psi}}^{k},\,\varphi^{k}\bar{h}^{k})\|_{L^2}\big),\\
\|\bar{h}^{k}L\boldsymbol{u}^{k-1}\|_{L^2}&=\|\varphi^k\bar{h}^{k} h^k L\boldsymbol{u}^{k-1}\|_{L^2}=\|(\varphi^k\bar{h}^{k}) (h^{k-1} L\boldsymbol{u}^{k-1})((h^{k-1})^{-1}h^k)\|_{L^2}\\
& \leq C\|\varphi^k\bar{h}^{k}\|_{L^3}\|h^{k-1} L\boldsymbol{u}^{k-1}\|_{L^6}\|(h^{k-1})^{-1}h^k\|_{L^\infty}.
\end{aligned}
\end{equation}

Next, multiplying \eqref{psibar1} by $2\bar{\boldsymbol{\psi}}^{k+1}$ and integrating over $\mathbb{R}^n$, combined with \eqref{2Dxinfaxian}, give that, for any $\sigma\in (0,1)$,
\begin{equation}\label{psibar}
\begin{aligned}
\frac{\mathrm{d}}{\mathrm{d}t}\|\bar{\boldsymbol{\psi}}^{k+1}\|^2_{L^2}
&\leq C\big(\|(\bar{\boldsymbol{u}}^{k},\sqrt{h^k}\nabla \bar{\boldsymbol{u}}^{k},h^k\nabla^2\bar{\boldsymbol{u}}^k,\bar{\boldsymbol{\psi}}^k,\bar{\boldsymbol{\psi}}^{k+1})\|_{L^2}+\|\varphi^k\bar{h}^{k}\|_{L^3}\big)\|\bar{\boldsymbol{\psi}}^{k+1}\|_{L^2}\\
&\leq  \frac{C}{\sigma}\|\bar{\boldsymbol{\psi}}^{k+1}\|^2_{L^2}
+\sigma \big(\|(\bar{\boldsymbol{u}}^k,\,\sqrt{h^k}\nabla \bar{\boldsymbol{u}}^k,\,\bar{\boldsymbol{u}}^k_t)\|^2_{L^2}+\|\bar{\boldsymbol{u}}^{k-1}\|^2_{L^2}\big)\\
&\quad +\sigma\big(\|\sqrt{h^{k-1}}\nabla\bar{\boldsymbol{u}}^{k-1}\|^2_{L^2}
+\|(\bar{\boldsymbol{\psi}}^{k},\,\varphi^{k}\bar{h}^{k},\,\nabla \bar{\phi}^{k})\|^2_{L^2}\big).
\end{aligned}
\end{equation}

For the estimates of $\bar{\phi}^{k+1}$, multiplying $\eqref{k+3}_1$ by $2\bar{\phi}^{k+1}$ and integrating over $\mathbb{R}^n$ give 
\begin{equation}\label{phibar1}
\frac{\mathrm{d}}{\mathrm{d}t}\|\bar{\phi}^{k+1}\|^2_{L^2}\leq C\|\bar{\phi}^{k+1}\|^2_{L^2}
+C\|(\bar{\boldsymbol{u}}^{k},\,\sqrt{h^k}\nabla\bar{\boldsymbol{u}}^{k})\|_{L^2}\|\bar{\phi}^{k+1}\|_{L^2}.
\end{equation}
Furthermore, applying $2\partial^{\boldsymbol{\varsigma}}_{\boldsymbol{x}}\bar{\phi}^{k+1}\partial^{\boldsymbol{\varsigma}}_{\boldsymbol{x}}$ $(|\boldsymbol{\varsigma}|=1)$ to $\eqref{k+3}_1$  and  integrating over $\mathbb{R}^n$, we obtain 
\begin{equation*}
\frac{\mathrm{d}}{\mathrm{d}t}\|\partial^{\boldsymbol{\varsigma}}_{\boldsymbol{x}}\bar{\phi}^{k+1}\|^2_{L^2}\leq C\|\bar{\phi}^{k+1}\|^2_{H^1}+C\|(\bar{\boldsymbol{u}}^{k},\,\sqrt{h^k}\nabla\bar{\boldsymbol{u}}^{k},\,h^k\nabla^2\bar{\boldsymbol{u}}^{k})\|_{L^2}\|\nabla\bar{\phi}^{k+1}\|_{L^2},
\end{equation*}
which, along with \eqref{hubarr} and \eqref{phibar1}, implies that, for $t\in [0,\bar{T}]$,
\begin{equation}\label{phibar2}
\frac{\mathrm{d}}{\mathrm{d}t}\|\bar{\phi}^{k+1}\|^2_{H^1}\leq \frac{C}{\sigma}\|\bar{\phi}^{k+1}\|^2_{H^1
}+\sigma\|(\bar{\boldsymbol{u}}^{k},\,\sqrt{h^k}\nabla\bar{\boldsymbol{u}}^{k},\,h^k\nabla^2\bar{\boldsymbol{u}}^{k})\|^2_{L^2}.
\end{equation}

We now estimate  $\bar{\boldsymbol{u}}^{k+1}$. Multiplying $\eqref{k+3}_2$ by $2\bar{\boldsymbol{u}}^{k+1}$ and integrating over $\mathbb{R}^n$ yield 
\begin{equation}\label{ubar}
\begin{aligned}
\frac{\mathrm{d}}{\mathrm{d}t}\|\bar{\boldsymbol{u}}^{k+1}\|^2_{L^2}\!+\!aa_1\|\sqrt{h^{k+1}}\nabla\bar{\boldsymbol{u}}^{k+1}\|^2_{L^2}&\leq \frac{C}{\sigma}\|\bar{\boldsymbol{u}}^{k+1}\|^2_{L^2}+C\big(\|\bar{\phi}^{k+1}\|^2_{H^1}+\|\bar{\boldsymbol{\psi}}^{k+1}\|^2_{L^2}\big)\\
&\quad +\sigma\big(\|(\bar{\boldsymbol{u}}^{k},\,\sqrt{h^{k}}\nabla\bar{\boldsymbol{u}}^{k})\|^2_{L^2}\!+\!\|\varphi^{k+1}\bar{h}^{k+1}\|^2_{L^2}\big),
\end{aligned}
\end{equation}
where we have used \eqref{2Dxinfaxian} and the fact that
\begin{equation*}
\|h^{k+1}L\boldsymbol{u}^k\|_{L^6}\leq C\|(h^k)^{-1}h^{k+1}\|_{L^\infty}\|h^{k}L\boldsymbol{u}^k\|_{L^6}.    
\end{equation*} 

Multiplying $\eqref{k+3}_2$ by $2\bar{\boldsymbol{u}}_t^{k+1}$ and integrating over $\mathbb{R}^n$ yield 
\begin{equation*} 
\begin{aligned}
&\,\frac{\rm d}{{\rm d}t}\big(a a_1\|\sqrt{h^{k+1}}\nabla\bar{\boldsymbol{u}}^{k+1}\|^2_{L^2}
+a(a_1+a_2)\|\sqrt{h^{k+1}}\mathrm{div}\bar{\boldsymbol{u}}^{k+1}\|^2_{L^2}\big)+2\|\bar{\boldsymbol{u}}^{k+1}_t\|^2_{L^2}\\
&\leq C\big\|(\bar{\boldsymbol{u}}^k,\sqrt{h^k}\nabla\bar{\boldsymbol{u}}^k,\sqrt{h^{k+1}}\nabla\bar{\boldsymbol{u}}^{k+1},\bar{\boldsymbol{\psi}}^{k+1},\varphi^{k+1}\bar{h}^{k+1})\big\|_{L^2}\|\bar{\boldsymbol{u}}^{k+1}_t\|_{L^2}\\
&\quad +C\|\bar{\phi}^{k+1}\|_{H^1} \|\bar{\boldsymbol{u}}^{k+1}_t\|_{L^2} +C\|\sqrt{h^{k+1}}\nabla\bar{\boldsymbol{u}}^{k+1}\|_{L^2}^2,
\end{aligned}
\end{equation*}
which, along with the Young inequality, leads to
\begin{equation}\label{hubar}
\begin{aligned}
&\,\frac{\rm d}{{\rm d}t}\big(a a_1\|\sqrt{h^{k+1}}\nabla\bar{\boldsymbol{u}}^{k+1}\|^2_{L^2}
+a(a_1+a_2)\|\sqrt{h^{k+1}}\mathrm{div}\bar{\boldsymbol{u}}^{k+1}\|^2_{L^2}\big)+\|\bar{\boldsymbol{u}}^{k+1}_t\|^2_{L^2}\\
&\leq C\big(\big\|(\bar{\boldsymbol{u}}^k,
\sqrt{h^k}\nabla\bar{\boldsymbol{u}}^k,\sqrt{h^{k+1}}\nabla\bar{\boldsymbol{u}}^{k+1},\bar{\boldsymbol{\psi}}^{k+1},\varphi^{k+1}\bar{h}^{k+1})\big\|_{L^2}^2+\|\bar{\phi}^{k+1}\|_{H^1}^2 \big).
\end{aligned}
\end{equation}

\textbf{2.3. Closing the estimates.} Combining \eqref{psibar} with \eqref{phibar2}--\eqref{hubar} implies 
\begin{equation}\label{allbar}
\begin{aligned}
&\,\frac{\mathrm{d}}{\mathrm{d}t}\varGamma^{k+1}(t,\upsilon_1) +a a_1\|\sqrt{h^{k+1}}\nabla\bar{\boldsymbol{u}}^{k+1}\|^2_{L^2}+\upsilon_1\|\bar{\boldsymbol{u}}_t^{k+1}\|^2_{L^2}\\
&\leq \frac{C}{\sigma}\varGamma^{k+1}(t,\upsilon_1) \!+\!C\sigma\|\bar{\boldsymbol{u}}^k_t\|_{L^2}^2\! +\!C(\sigma+\upsilon_1)\big(\|\bar{\boldsymbol{u}}^{k-1}\|^2_{L^2}\!+\!\|\sqrt{h^{k-1}}\nabla\bar{\boldsymbol{u}}^{k-1}\|^2_{L^2}\!+\!\varGamma^{k}(t,\upsilon_1)\big),
\end{aligned}
\end{equation}
where $\upsilon_1\in (0,1)$ is a constant to be determined later and
\begin{equation*}
\begin{aligned}
\varGamma^{k+1}(t,\upsilon_1)&=  \|\bar\phi^{k+1}\|_{H^1}^2+ \|\bar{\boldsymbol{\psi}}^{k+1}\|_{L^2}^2+ \|\bar{\boldsymbol{u}}^{k+1}\|_{L^2}^2+\|\varphi^{k+1}\bar{h}^{k+1}\|^2_{L^2}\\
&\quad+\upsilon_1 a a_1\|\sqrt{h^{k+1}}\nabla\bar{\boldsymbol{u}}^{k+1}\|^2_{L^2}
+\upsilon_1 a(a_1+a_2)\|\sqrt{h^{k+1}}\mathrm{div}\bar{\boldsymbol{u}}^{k+1}\|^2_{L^2}.
\end{aligned}
\end{equation*}

Set $\sigma=\upsilon_1^{\frac{3}{2}}$ and define
\begin{equation*}
\mathcal{E}^{k+1}(t,\upsilon_1):=\sup_{s\in [0,t]}\varGamma^{k+1}(s,\upsilon_1)+\int^t_0\big(aa_1\|\sqrt{h^{k+1}}\nabla\bar{\boldsymbol{u}}^{k+1}\|^2_{L^2}+\upsilon_1\|\bar{\boldsymbol{u}}^{k+1}_t\|^2_{L^2}\big)\,\mathrm{d}s.
\end{equation*}
Then it follows from \eqref{allbar}  and the Gr\"onwall inequality  that 
\begin{equation}\label{Gamma}
\mathcal{E}^{k+1}(t,\upsilon_1)\leq C\upsilon^{\frac{1}{2}}_1 e^{C\upsilon_1^{-\frac{3}{2}}t}(t+1)\big(\mathcal{E}^{k}(t,\upsilon_1)+\mathcal{E}^{k-1}(t,\upsilon_1)\big).
\end{equation}
Since $\bar{T}\in (0,1)$, we can first choose   $\upsilon_1=\bar{\upsilon}\in (0,1)$ such that
$C\bar{\upsilon}^{\frac{1}{2}}\leq \frac{1}{64}$, and then choose  $T_*\in (0,\bar{T}]$ such that
$(1+T_*)e^{C\bar{\upsilon}^{-\frac{3}{2}}T_*}\leq 4$. Then \eqref{Gamma} becomes
\begin{equation*}
\mathcal{E}^{k+1}(T_*,\bar\upsilon)\leq \frac{1}{16}\big(\mathcal{E}^{k}(T_*,\bar\upsilon)+\mathcal{E}^{k-1}(T_*,\bar\upsilon)\big),    
\end{equation*}
which yields $\sum_{k=1}^{\infty}\mathcal{E}^{k}(T_*,\bar\upsilon)<\infty$.

This, combined with the $k$-uniform estimates \eqref{key1kk}, gives that, for any $s'\in [1,3)$,
\begin{equation*}
\|(\bar\phi^{k},\bar {\boldsymbol{u}}^{k})\|_{H^{s'}}+\|(\bar{\boldsymbol{\psi}}^{k},\bar h^{k})\|_{L^\infty}\to 0\qquad \text{as $k\to\infty$},
\end{equation*}
which implies that there exist a subsequence (still denoted by $(\phi^k,\boldsymbol{u}^k,h^k ,\boldsymbol{\psi}^k)$) and limit functions  $(\phi^\eta,\boldsymbol{u}^\eta,h^\eta, \boldsymbol{\psi}^\eta)$ such that
\begin{equation*} 
\begin{aligned}
\phi^k\to \phi^\eta\qquad &\text{in $L^\infty([0,T_*];L^\infty(\mathbb{R}^n)\cap D^1(\mathbb{R}^n)\cap D^2(\mathbb{R}^n))$},\\
\boldsymbol{u}^k\to \boldsymbol{u}^\eta \qquad   &\text{in $L^\infty([0,T_*];H^{s'}(\mathbb{R}^n))$},\\
(\boldsymbol{\psi}^k,h^k)\to (\boldsymbol{\psi}^\eta,h^\eta) \qquad  &\text{in $L^\infty([0,T_*];L^\infty(\mathbb{R}^n))$}.
\end{aligned}
\end{equation*}

On the other hand,    the local estimates  \eqref{key1kk} independent of $k$ yield that there exists a subsequence (still denoted by) $(\phi^k,\boldsymbol{u}^k, h^k, \boldsymbol{\psi}^k)$ converging to the limit $(\phi^\eta,\boldsymbol{u}^\eta, h^\eta, \boldsymbol{\psi}^\eta)$ in the weak or weak* sense.
According to the lower semi-continuity of norms,  the corresponding estimates in \eqref{key1kk} still hold for  $(\phi^\eta,\boldsymbol{u}^\eta, h^\eta,\boldsymbol{\psi}^\eta)$ with $T^*$ replaced by $T_*$. Thus, we can check that  $(\phi^\eta,\boldsymbol{u}^\eta, h^\eta)$ is a weak solution of \eqref{nl} in the sense of distributions and satisfies the equations in \eqref{nl} pointwise.

\smallskip
\textbf{3. Uniqueness and time-continuity.}  Let $(\phi_1,  \boldsymbol{u}_1, h_1)$ and $(\phi_2, \boldsymbol{u}_2, h_2)$ be two classical solutions to the   Cauchy problem \eqref{nl} satisfying the  estimates in \eqref{key1kk}.

Set 
$\boldsymbol{\psi}_i=\frac{a\delta}{\delta-1}\nabla h_i$ ($i=1,2$), and 
\begin{equation*}
\begin{aligned}
\bar{h}=h_1-h_2,\qquad \bar{\phi}= \phi_1-\phi_2,\qquad \bar{\boldsymbol{u}}=\boldsymbol{u}_1-\boldsymbol{u}_2, \qquad  \bar{\boldsymbol{\psi}}=\boldsymbol{\psi}_1-\boldsymbol{\psi}_2.
\end{aligned}
\end{equation*}
Then it follows from the equations in  \eqref{nl} that
\begin{equation}\label{eq:1.2wcvb}
\begin{cases}
\displaystyle
\bar{\phi}_t+\boldsymbol{u}_1\cdot \nabla\bar{\phi}+\bar{\boldsymbol{u}} \cdot\nabla\phi_2+(\gamma-1)(\bar{\phi}\mathrm{div}\boldsymbol{u}_1 +\phi_2\mathrm{div}\bar{\boldsymbol{u}})=0,\\[3pt]
 \displaystyle
\bar{\boldsymbol{u}}_t+ \boldsymbol{u}_1\cdot\nabla \bar{\boldsymbol{u}}+\nabla \bar{\phi}+ a h_1L\bar{\boldsymbol{u}} \\[3pt]
\displaystyle
\quad =- \bar{\boldsymbol{u}} \cdot \nabla \boldsymbol{u}_2- a \bar h L\boldsymbol{u}_2+ \boldsymbol{\psi}_1 \cdot Q( \bar{\boldsymbol{u}})+\bar{\boldsymbol{\psi}} \cdot Q(\boldsymbol{u}_2),\\[3pt]
\bar{h}_t+\boldsymbol{u}_1\cdot \nabla\bar{h}+\bar{\boldsymbol{u}}\cdot\nabla h_2+(\delta-1)(\bar{h} \mathrm{div}\boldsymbol{u}_2 +h_1\mathrm{div}\bar{\boldsymbol{u}})=0,\\[3pt]
\displaystyle
(\bar{\phi},\bar{\boldsymbol{u}},\bar{h})|_{t=0}=(0,\boldsymbol{0},0) \qquad \text{in $\mathbb{R}^n$},\\[3pt]
\displaystyle
(\bar{\phi},\bar{\boldsymbol{u}},\bar{h})\to (0,\boldsymbol{0},0) \qquad\quad \ \text{as $|\boldsymbol{x}|\to \infty$} \quad \text{for $t\geq 0$}.
\end{cases}
\end{equation}

Set 
$\varGamma(t)=\|\bar\phi\|_{H^1}^2
+\|(h_1)^{-1}\bar{h}\|_{L^2}+\|(\bar{\boldsymbol{\psi}},\, \sqrt{aa_1h_1}\nabla\bar{\boldsymbol{u}},\,\sqrt{a(a_1+a_2)h_1}\diver\bar{\boldsymbol{u}},\bar{\boldsymbol{u}})\|^2_{L^2}$. 
In a similar way for  the derivation of \eqref{Gamma}, we obtain
\begin{equation*}
\varGamma'(t)+C^{-1}\|(\nabla \bar{\boldsymbol{u}},\,\bar{\boldsymbol{u}}_t)\|^2_{L^2}
\leq \bar{H}(t)\varGamma(t),
\end{equation*}
with a continuous  function $\bar{H}(t)\in L^1(0,T_*)$. Hence, it follows from the Gr\"onwall inequality that
$\bar{\phi}=0$ and $\bar{\boldsymbol{\psi}}=\bar{\boldsymbol{u}}=\boldsymbol{0}$.

Finally, the time-continuity follows from the same procedure as in Lemma \ref{ls}. Thus, the proof of Theorem \ref{nltm} is completed.
\end{proof}

\subsection{Passing to the Limit $\eta\to 0$}
Based on the uniform estimates in \eqref{key1kk}, we are ready to prove Theorem  \ref{th1}.

\begin{proof}[Proof of {\rm Theorem \ref{th1}}] 
We divide the proof into four steps.

\smallskip
\textbf{1. Local uniform lower bound of $\phi$}. For any $\eta\in (0,1]$, set 
\begin{equation*}
\phi_0^\eta=\phi_0+\eta,\qquad \boldsymbol{\psi}^\eta_0=\frac{a\delta}{\delta-1}\nabla (\phi_0+\eta)^{2\iota},\qquad h_0^\eta=(\phi_0+\eta)^{2\iota}.
\end{equation*}
First, since \eqref{th78qq}--\eqref{th78zxq} hold, it follows from  \eqref{etainitial}--\eqref{initialapproximation} that  there exists a constant $c_0\geq 1$ independent of $\eta$ such that \eqref{2.14} holds.
Therefore, it follows from Theorem \ref{nltm} that, for the initial data $(\phi^\eta_0,\boldsymbol{u}^\eta_0,h^\eta_0)$, problem \eqref{nl} admits a unique classical solution  $(\phi^\eta,\boldsymbol{u}^\eta, h^\eta)$ in $[0,T_*]\times \mathbb{R}^n$ satisfying the  estimates in \eqref{key1kk}, where $T_*$ is independent of $\eta$ and $\boldsymbol{\psi}^\eta=\frac{a\delta}{\delta-1}\nabla h^\eta$.

Moreover, $\phi^\eta$ is locally uniformly positive, as shown below, which can be proved by the same argument for Lemma 3.12 in \cite{dxz}.
\begin{lem}\label{phieta}
For any $R_0>0$ and $\eta\in (0,1)$, there exists a constant $C(R_0)$ independent of $\eta$ such that
\begin{equation*}
\phi^\eta(t,x)\geq C(R_0)^{-1}>0 \qquad \text{for all $(t,\boldsymbol{x})\in [0,T_*]\times B_{R_0}$}.
\end{equation*}
\end{lem}

\smallskip
\textbf{2. Passing to the limit $\eta\to 0$}. First, it follows from  the $\eta$-independent estimates in \eqref{key1kk} and Lemma \ref{aubin-lions}  that,
 for any $R>0$, there exists a subsequence (still denoted by ($\phi^\eta,\boldsymbol{u}^\eta, h^\eta,\boldsymbol{\psi}^\eta$)) converging to a limit  $(\phi,\boldsymbol{u}, h,\boldsymbol{\psi})$ as follows:
\begin{equation} \label{nonlinearstrongconvergence}
\begin{aligned}
\phi^\eta-\eta\to \phi \ \ &\text{in} \ \ C([0,T_*];D^1(B_R)), \quad &&\boldsymbol{\psi}^\eta\to \boldsymbol{\psi}\ \ \text{in} \ \ C([0,T_*];D^{1,n}(B_R)),\\
\boldsymbol{u}^\eta\to \boldsymbol{u} \ \ &\text{in} \ \ C([0,T_*]; H^2(B_R)),  
\quad &&\hspace{1.6mm} h^\eta\to h\ \ \text{in} \ \ C([0,T_*];H^2(B_R)).
\end{aligned}
\end{equation}
Moreover, since $\boldsymbol{\psi}^\eta=\frac{a\delta}{\delta-1}\nabla h^\eta$, we see that $\boldsymbol{\psi}=\frac{a\delta}{\delta-1}\nabla h$.

Second, the $\eta$-independent   estimate \eqref{key1kk} also yields that 
there exists a subsequence (still denoted by) $(\phi^\eta,\boldsymbol{u}^\eta, \nabla \sqrt{h^\eta},\boldsymbol{\psi}^\eta)$ converging to $(\phi,\boldsymbol{u}, \boldsymbol{\tilde{h}},\boldsymbol{\psi})$ as follows:
\begin{equation}\label{key}
\begin{aligned}
\phi^\eta-\eta\rightharpoonup \phi \qquad &\text{weakly}^* \ \ \!\text{in } L^\infty([0,T_*];D^1(\mathbb{R}^n)\cap D^3(\mathbb{R}^n)),\\
\boldsymbol{u}^\eta\rightharpoonup \boldsymbol{u} \qquad &\text{weakly} \ \,\, \hspace{1mm} \text{in } L^2([0,T_*];H^4(\mathbb{R}^n)),\\
\boldsymbol{\psi}^\eta\rightharpoonup \boldsymbol{\psi}  \qquad &\text{weakly}^* \ \ \!\text{in } L^\infty([0,T_*];L^\infty(\mathbb{R}^n)\cap D^{1,n}(\mathbb{R}^n)\cap D^2(\mathbb{R}^n)),\\
\nabla \sqrt{h^\eta}\rightharpoonup \boldsymbol{\tilde{h}}  \qquad &\text{weakly}^* \ \ \!\text{in } L^\infty([0,T_*];L^{2n}(\mathbb{R}^n)),\\
\phi_t^\eta\rightharpoonup \phi_t \qquad &\text{weakly}^* \ \ \!\text{in } L^\infty([0,T_*];H^2(\mathbb{R}^n)),\\
(\boldsymbol{u}_t^\eta,\boldsymbol{\psi}_t^\eta)\rightharpoonup (\boldsymbol{u}_t,\boldsymbol{\psi}_t) \qquad &\text{weakly}^* \ \ \!\text{in } L^\infty([0,T_*];H^1(\mathbb{R}^n)).
\end{aligned}
\end{equation}
Then, by the lower semi-continuity of weak convergences, $(\phi,\boldsymbol{u},\boldsymbol{\psi})$ satisfies the corresponding estimates as in \eqref{key1kk} except the estimate on $\nabla \sqrt{h}$ and the  $h$-weighted ones on $\boldsymbol{u}$.
In fact, it follows from  the $\eta$-independent estimate \eqref{key1kk} for $(\phi^\eta,\boldsymbol{u}^\eta, \boldsymbol{\psi}^\eta)$,  the facts that $h\geq (2c_0)^{-1}$ and $h^\eta\geq (2c_0)^{-1}$, and  the convergences shown in \eqref{nonlinearstrongconvergence}--\eqref{key} that $\boldsymbol{\tilde{h}}=\nabla \sqrt{h}$  and
\begin{equation*} 
\begin{aligned}
\sqrt{h^
\eta}(\nabla \boldsymbol{u}^\eta,\nabla \boldsymbol{u}^\eta_t)\rightharpoonup \sqrt{h}(\nabla \boldsymbol{u},\nabla \boldsymbol{u}_t)\quad &\text{weakly}^*\ \ \!\text{in}\ \ L^\infty([0,T_*];L^2(\mathbb{R}^n)),\\
h^\eta\nabla^2 \boldsymbol{u}^\eta\rightharpoonup h\nabla^2\boldsymbol{u}\quad &\text{weakly}^*\ \ \!\text{in}\ \ L^\infty([0,T_*];H^1(\mathbb{R}^n)),\\
h^\eta\nabla^2\boldsymbol{u}^\eta\rightharpoonup  h\nabla^2\boldsymbol{u} \quad &\text{weakly}\ \ \hspace{1.5mm} \!\text{in}\ \ L^2([0,T_*];D^1(\mathbb{R}^n) \cap D^2(\mathbb{R}^n)),\\
(h^\eta\nabla^2 \boldsymbol{u}^\eta)_t\rightharpoonup (h\nabla^2\boldsymbol{u})_t\quad &\text{weakly}\ \ \hspace{1.5mm} \!\text{in}\ \ L^2([0,T_*];L^2(\mathbb{R}^n)).
\end{aligned}
\end{equation*}
Hence, the corresponding weighted estimates for $u$ in \eqref{key1kk} still hold for the limit functions, and it is direct to show that $(\phi,\boldsymbol{u},h)$ satisfy the equations in $\eqref{eq:cccq}_1$ and 
\begin{equation*}
h_t+\boldsymbol{u}\cdot \nabla h+(\delta-1)h\mathrm{div}\,\boldsymbol{u}=0\qquad \text{holds pointwise}.
\end{equation*}
Then, denoting $\phi^*=h-\phi^{2\iota}$, we have
\begin{equation*}
\phi^*_t+\boldsymbol{u}\cdot \nabla \phi^*+(\delta-1)\phi^*\diver\boldsymbol{u}=0,\qquad \phi^*(0,\boldsymbol{x})=0 \ \  \text{for $\boldsymbol{x}\in \mathbb{R}^n$},
\end{equation*}
which, along with the characteristic method, implies that $\phi^*=0$. 
Therefore, $h=\phi^{2\iota}$.

Hence, $(\phi,\boldsymbol{u},\boldsymbol{\psi})$ is a weak solution of  problem \eqref{eq:cccq}{\rm--}\eqref{sfanb1}  in the sense of distributions and satisfies system \eqref{eq:cccq} pointwisely. 

\smallskip
\textbf{3. Proof for $\phi^\frac{1}{\gamma-1}\in C([0,T];L^1(\mathbb{R}^n))$.}
Let $f(s)\in C^\infty([0,\infty))$ be a function such that
\begin{equation*}
f(s)>0,\qquad f'(s)\leq 0,\qquad f(s)= 1 \quad\text{for $s\in [0,\frac{1}{2}]$}, \qquad f(s)=e^{-s} \quad\text{for $s\in [1,\infty)$}.
\end{equation*}
Clearly, there exists a generic constant $C >0$ such that  $|f'(s)| \le C f(s)$. Moreover, for any $R>1$, define $f_R(\boldsymbol{x})=f (\tfrac{|\boldsymbol{x}|}{R})$.

Since equation $\eqref{eq:cccq}_1$ holds pointwisely everywhere, we can multiply it  by $f_R(\boldsymbol{x})\phi^\frac{2-\gamma}{\gamma-1}$ and integrate over $\mathbb{R}^n$ to obtain
\begin{equation*} 
\begin{aligned}
\frac{\mathrm{d}}{\mathrm{d}t} \int_{\mathbb{R}^n} \phi^\frac{1}{\gamma-1}  f_R\,\mathrm{d}\boldsymbol{x}&=-\int_{\mathbb{R}^n} \diver (\phi^\frac{1}{\gamma-1} \boldsymbol{u})  f_R\,\mathrm{d}\boldsymbol{x}= \int_{\mathbb{R}^n} \phi^\frac{1}{\gamma-1} \boldsymbol{u} \cdot \boldsymbol{x} f' (\frac{|\boldsymbol{x}|}{R}) \frac{1 }{R|\boldsymbol{x}|}\,\mathrm{d}\boldsymbol{x}\\
&\leq \frac{C}{R}\|\boldsymbol{u}\|_{L^\infty} \int_{\mathbb{R}^n}\phi^\frac{1}{\gamma-1}   f_R\,\mathrm{d}\boldsymbol{x}\leq C\int_{\mathbb{R}^n}\phi^\frac{1}{\gamma-1}   f_R\,\mathrm{d}\boldsymbol{x},    
\end{aligned}
\end{equation*}
which,  along with the Gr\"onwall inequality, yields 
\begin{align*}
\sup_{t\in [0,T]}\int_{\mathbb{R}^n} \phi^\frac{1}{\gamma-1}  f_R\,\mathrm{d}\boldsymbol{x}  \leq C \int_{\mathbb{R}^n} \phi^\frac{1}{\gamma-1}_0  f_R\,\mathrm{d}\boldsymbol{x}\leq C\|\phi^\frac{1}{\gamma-1}_0\|_{L^1} 
\end{align*}
with $C$ a generic constant independent of $R$. 
Note that $\phi^\frac{1}{\gamma-1} f_R \to \phi^\frac{1}{\gamma-1}$ as $R\to \infty$ for all $\boldsymbol{x}\in \mathbb{R}^n$. Thus, by the Fatou lemma (Lemma \ref{Fatou}), we have 
\begin{equation*} 
\sup_{t\in [0,T]} \int_{\mathbb{R}^n} \phi^\frac{1}{\gamma-1}\,\mathrm{d}\boldsymbol{x}   \leq \sup_{t\in [0,T]}\liminf_{R \to \infty}\int_{\mathbb{R}^n} \phi^\frac{1}{\gamma-1}  f_R \,\mathrm{d}\boldsymbol{x}  \leq C\|\phi^\frac{1}{\gamma-1}_0\|_{L^1},
\end{equation*}
which implies that $\phi^\frac{1}{\gamma-1}\in L^\infty([0,T];L^1(\mathbb{R}^n))$. 

To show the time-continuity, we see from $\eqref{eq:cccq}_1$ that
\begin{equation*}
(\phi^\frac{1}{\gamma-1})_t=-\phi^\frac{1}{\gamma-1}\diver\boldsymbol{u}-(a
\delta)^{-1}\phi^\frac{2-\delta}{\gamma-1} \boldsymbol{u}\cdot\boldsymbol{\psi}.  
\end{equation*}
Then, by the regularity properties of $(\phi,\boldsymbol{\psi},\boldsymbol{u})$, we obtain that $(\phi^\frac{1}{\gamma-1})_t\in L^\infty([0,T];L^1(\mathbb{R}^n))$, so that $\phi^\frac{1}{\gamma-1}\in C([0,T];L^1(\mathbb{R}^n))$.

\smallskip
\textbf{4. Uniqueness and time-continuity}. The uniqueness follows from the same procedure as in the proof of Theorem \ref{nltm}, and the time-continuities of $(\phi,\boldsymbol{\psi})$ can be obtained by a similar argument as in Lemma \ref{ls}. Moreover, since the vector $\boldsymbol{\psi}$ is spherically symmetric, we see from  Lemma \ref{Hk-Ck-vector} that
\begin{equation}\label{psi-hi}
(\boldsymbol{\psi},\nabla \phi^{2\iota}) \in C([0,T_*];C(\overline{\mathbb{R}^n})).
\end{equation}

For velocity $\boldsymbol{u}$,  the \textit{a priori} estimates obtained above and Lemmas \ref{ale1} and \ref{triple} imply 
\begin{equation}\label{zheng1}
\begin{aligned}
\boldsymbol{u}\in C([0,T_*]; H^2(\mathbb{R}^n))\cap  L^\infty([0,T_*]; H^3(\mathbb{R}^n)), \qquad \phi^{\iota}\nabla \boldsymbol{u}\in  C([0,T_*]; L^2(\mathbb{R}^n)).
 \end{aligned}
\end{equation}
It then follows from $\eqref{eq:cccq}_2$ that
\begin{equation*}
\phi^{-2\iota} \boldsymbol{u}_t \in L^2([0,T_*];H^2(\mathbb{R}^n)),\quad (\phi^{-2\iota} \boldsymbol{u}_t)_t \in L^2([0,T_*];L^2(\mathbb{R}^n)),
\end{equation*}
which, along with Lemma \ref{triple}, implies that $\phi^{-2\iota} \boldsymbol{u}_t \in C([0,T_*];H^1(\mathbb{R}^3))$. This and the classical elliptic estimates in Lemma \ref{df3} for
\begin{equation*}
\begin{aligned}
aL\boldsymbol{u}&=-\phi^{-2\iota}(\boldsymbol{u}_t+\boldsymbol{u}\cdot\nabla \boldsymbol{u} +\nabla\phi- \boldsymbol{\psi}  \cdot Q(\boldsymbol{u}))
\end{aligned}
\end{equation*}
yield that $\boldsymbol{u}\in C([0,T_*]; H^{3}(\mathbb{R}^n))$ immediately.

Finally, note that 
\begin{equation*}
\phi^{2\iota}\nabla^2 \boldsymbol{u} \in L^\infty([0,T_*]; H^1(\mathbb{R}^n))\cap L^2([0,T_*] ; D^2(\mathbb{R}^n)),\qquad (\phi^{2\iota}\nabla^2 \boldsymbol{u})_t \in  L^2([0,T_*] ; L^2(\mathbb{R}^n)).
\end{equation*}
Thus, Lemma \ref{triple} gives $\phi^{2\iota}\nabla^2 \boldsymbol{u}\in C([0,T_*]; H^1(\mathbb{R}^3))$ which, combined with the facts that $\nabla \phi^{2\iota} \in C([0,T_*];C(\overline{\mathbb{R}^n}))$ and $\boldsymbol{u}\in C([0,T_*]; H^{3}(\mathbb{R}^n))$, yields 
\begin{equation*}
(\phi^{2\iota}\nabla^2 \boldsymbol{u},\phi^{2\iota}\nabla^3 \boldsymbol{u})\in C([0,T_*]; L^2(\mathbb{R}^n)).
\end{equation*}
Finally, $\boldsymbol{u}_t\in C([0,T];H^1(\mathbb{R}^n))$ follows directly from the above, \eqref{psi-hi}--\eqref{zheng1}, the time continuities of $(\phi,\boldsymbol{\psi})$, and $\boldsymbol{u}\in C([0,T_*]; H^{3}(\mathbb{R}^n))$. 
\end{proof}

\subsection{The Proof for Theorem \ref{thm-loc}}
Based on Theorem \ref{th1}, now we are ready to establish the local well-posedness of the regular solution of the original Cauchy problem \eqref{eq:1.1}--\eqref{kelaoxiusi} with \eqref{eqs:CauchyInit}--\eqref{e1.3}  shown in Theorem \ref{thm-loc}.
\begin{proof}[Proof for Theorem \ref{thm-loc}] 
It follows from the initial assumptions  \eqref{id1-high}--\eqref{th78zx} and Theorem \ref{th1} that there exists $T_{*}> 0$ such that problem \eqref{eq:cccq}--\eqref{sfanb1} has a unique regular solution $(\phi,\boldsymbol{u},\boldsymbol{\psi})$ satisfying the regularity property \eqref{er2-high} with $T$ replaced by $T_*$, which implies 
\begin{equation*} 
\phi\in C^1([0,T_{*}]\times \mathbb{R}^n), \qquad (\boldsymbol{u}, \nabla \boldsymbol{u}) \in   C([0,T_{*}]\times \mathbb{R}^n).
\end{equation*}
Set $\rho=(\frac{\gamma-1}{A\gamma}\phi)^{\frac{1}{\gamma-1}}>0$ with $\rho(0,\boldsymbol{x})=\rho_0$. Due to $\boldsymbol{\psi}=\frac{a\delta}{\delta-1}\nabla h$ and $h=\phi^{2\iota}$, we see that  $\boldsymbol{\psi}=\frac{\delta}{\delta-1}\nabla \rho^{\delta-1}$. Hence, based on the above regularity property and relations of solution $(\phi,\boldsymbol{u},\boldsymbol{\psi})$,  multiplying $\eqref{eq:cccq}_1$ by
\begin{equation*}
\frac{\partial \rho}{\partial \phi}(t,\boldsymbol{x})=\frac{1}{\gamma-1}\Big(\frac{\gamma-1}{A\gamma}\Big)^{\frac{1}{\gamma-1}}\phi^{\frac{2-\gamma}{\gamma-1}}(t,\boldsymbol{x})
\end{equation*}
yields equation $\eqref{eq:1.1}_1$, while multiplying $\eqref{eq:cccq}_2$ by $\rho$ gives equation  \eqref{eq:1.1}.

Therefore, we have shown that $(\rho,\boldsymbol{u})$ satisfies problem \eqref{eq:1.1}--\eqref{kelaoxiusi} with \eqref{eqs:CauchyInit}--\eqref{e1.3}  in the sense of distributions and has the regularity properties shown in Definition \ref{cjk} and \eqref{er2-high}. 
\end{proof}

\appendix

\section{Some Basic Lemmas}\label{appA}

We present some useful lemmas in this appendix, which are used frequently in the previous sections. The first lemma is on the classical  Sobolev embedding theorem.  

\begin{lem}[\cite{af,CZZ1}]\label{ale1}
Assume that $\Omega \subset \mathbb{R}^n$ $(n\in \mathbb{N}^*)$ is a domain with smooth boundary. Let  $f\in W^{k,p}(\Omega)$ for some $k\in \mathbb{N}^*$ and $p\in [1,\infty]$. 
\begin{enumerate}
\item[$\mathrm{(i)}$] 
Let $kp\leq n$. Then $W^{k,p}(\Omega)\hookrightarrow L^s(\Omega)$ for any $s\in\big[p,\frac{np}{n-kp}\big]$ if $kp<n$ and any $s\in[p,\infty)$ if $kp=n$, and there exists  $C_1>0$ depending only on $(k,p,s,n,\Omega)$ such that 
\begin{equation*}
\|f\|_{L^s(\Omega)}\leq C_1\|f\|_{W^{k,p}(\Omega)}.
\end{equation*}
In particular, if $\Omega=\mathbb{R}^n$, $kp<n$, and $f\in D^{k,p}(\mathbb{R}^n)\cap L^{\frac{np}{n-kp}}(\mathbb{R}^n)$, or $f\in D^{k,p}(\mathbb{R}^n)$  and $f\to 0$ as $|\boldsymbol{x}|\to \infty$, then 
\begin{equation*} 
\|f\|_{L^{\frac{np}{n-kp}}(\mathbb{R}^n)}\leq C_1\|\nabla^k f\|_{L^p(\mathbb{R}^n)}.
\end{equation*}
\item[$\mathrm{(ii)}$] 
Let $(k,p)=(n,1)$. Then $W^{n,1}(\Omega)\hookrightarrow C(\overline{\Omega})$ and  there exists $C_2>0$ depending only on $(n,\Omega)$ such that
\begin{equation*}
\|f\|_{L^\infty(\Omega)}\leq C_2\|f\|_{W^{n,1}(\Omega)}. 
\end{equation*}
\item[$\mathrm{(iii)}$] 
Let $kp>n$. Then $W^{k,p}(\Omega)\hookrightarrow C^\ell(\overline\Omega)$ for all $\ell\in \mathbb{N}$ and $0\leq \ell<k-n/p$,  and  there exists $C_3>0$ depending only on $(k,p,\ell,n,\Omega)$ such that
\begin{equation*}
\max_{0\leq j\leq \ell}\|\nabla ^j f\|_{L^\infty(\Omega)}\leq C_3\|f\|_{W^{k,p}(\Omega)}, 
\end{equation*}
where $C^\ell(\overline\Omega)$ $(\ell\in \mathbb{N})$ denotes the space of all functions $f$ for which $\nabla^j f$ $(0\leq j\leq \ell)$ are bounded and uniformly continuous in $\Omega$.  In particular, the following inequality holds for $f=f(r)\in H^1(0,R)$ $(R>0)${\rm :}
\begin{equation*} 
\|f\|_{L^\infty(0,R)}^2 \leq (1+R^{-1})\|f\|_{L^2(0,R)}^2+\|f_r\|_{L^2(0,R)}^2.
\end{equation*}
\end{enumerate}
Moreover, if $\Omega=\mathbb{R}^n$, then the above constants $(C_1,C_2,C_3)$ are independent of $\Omega$.
\end{lem}

The second lemma concerns the well-known Gagliardo--Nirenberg inequality.  

\begin{lem}[\cite{ln}]\label{GN-ineq}
Assume that  $f\in L^{p}(\mathbb{R}^n)\cap D^{\ell,q}(\mathbb{R}^n)$ for   $1 \leq p,  q \leq \infty$.  Let real numbers $(\vartheta,s)$ and  natural numbers $(n,\ell,j)$ satisfy
\begin{equation*}
\frac{1}{{s}} = \frac{j}{n} + \Big( \frac{1}{q} - \frac{\ell}{n} \Big) \vartheta + \frac{1 - \vartheta}{p}, \qquad \frac{j}{\ell} \leq \vartheta \leq 1.
\end{equation*}
Then $f\in D^{j,{s}}(\mathbb{R}^n)$ and  there exists $C>0$ depending only on $(\ell,n,j,p,q,\vartheta)$ such that
\begin{equation*} 
\| \nabla^{j} f \|_{L^{{s}}} \leq C \| \nabla^{\ell} f \|_{L^{q}}^{\vartheta} \| f \|_{L^{p}}^{1 - \vartheta}.
\end{equation*}
Moreover, if $j = 0$, $\ell q < n$, and $p = \infty$, then it is necessary to make the additional assumption that either f tends to zero at infinity or that f lies in $L^{\tilde{s}}(\mathbb{R}^n)$ for some finite $\tilde{s} > 0${\rm;} if $1 < q < \infty$ and $\ell -j -n/q$ is a non-negative integer, then it is necessary to assume also that $\vartheta \neq 1$.
\end{lem}

The third lemma is on the fundamental theorem of calculus. This lemma plays an important role in justifying the procedure of integration by parts in spherical coordinates.
\begin{lem}[\cite{realrudin}]\label{calculus}
The following statements hold{\rm:}
\begin{enumerate}
\item[{\rm(i)}] Let $l_1\in (0,\infty)$ and $l_2\in [0,\infty)$. Assume that $f(r)\in W^{1,1}(0,l_1)$ and $g(r)\in L^p(l_2,\infty)\cap D^{1,1}(l_2,\infty)$ for some $p\in [1,\infty)$.  Then, for any $r_0,r_1\in [0,l_1]$ and $r_2\in [l_2,\infty)$,
\begin{equation}\label{newton}
f(r_1)=f(r_0)+\int_{r_0}^{r_1} f_r\,\mathrm{d}r,\qquad g(r_2)=-\int_{r_2}^{\infty} g_r\,\mathrm{d}r.
\end{equation}
In particular, $g(r)\to 0$ as $r\to \infty$ and the following estimates hold{\rm:}
\begin{equation}\label{newton2}
\|f\|_{L^\infty(0,l_1)}\leq \|f_r\|_{L^1(0,l_1)} \quad \text{if $f(0)=0$},\qquad\,\, \|g\|_{L^\infty(l_2,\infty)}\leq \|g_r\|_{L^1(l_2,\infty)}.
\end{equation}
\item[{\rm(ii)}] If $f$ is a function such that $r^m(f_r,\frac{f}{r}) \in L^1(I)$, then
\begin{equation}
\int_0^\infty (r^m f)_r\,\mathrm{d}r=0.
\end{equation}
\end{enumerate}
\end{lem}

\begin{proof}
We divide the proof into two steps.

\smallskip
\textbf{1. Proof for (i).} \eqref{newton2} is a direct consequence of \eqref{newton}, and $\eqref{newton}_1$ is the classical fundamental theorem of calculus; see \cite[Chapter 7]{realrudin}. Hence, it suffices to prove $\eqref{newton}_2$.

We first see from Lemma \ref{ale1} that $g\in C([l_2,l_3])$ for any $l_3\in [l_2,\infty)$. Thus, $g(r)$ is well-defined on $[l_2,\infty)$.  Next, for any $g\in L^p(l_2,\infty)\cap D^{1,1}(l_2,\infty)$, by the fundamental theorem of calculus, we obtain that, for any fixed $r_2,r_3\in [l_2,\infty)$,
\begin{equation}\label{r3}
g(r_3)=g(r_2)+\int_{r_2}^{r_3} g_r\,\mathrm{d}r.
\end{equation}
Taking the limit as $r_3\to \infty$ gives
\begin{equation*} 
L:=\lim_{r_3\to\infty} |g(r_3)|\leq |g(r_2)|+\int_{r_2}^{\infty} |g_r|\,\mathrm{d}r\leq |g(r_2)|+\|g_r\|_{L^1(l_2,\infty)}<\infty.
\end{equation*}

We claim that the limit $L=0$. Otherwise, if $L>0$, we can find a constant $M>0$ such that, for any $r>M$, $|g(r)|\geq \frac{L}{2}$. This yields
\begin{equation*}
\int_{l_2}^\infty |g|^p\,\mathrm{d}r\geq \int_{M}^\infty |g|^p\,\mathrm{d}r\geq \int_{M}^\infty \big(\frac{L}{2}\big)^p\,\mathrm{d}r=\infty,
\end{equation*}
which contradicts the fact that $g\in L^p (l_2,\infty)$.

Therefore, $L=0$, and we obtain from \eqref{r3} that
\begin{equation*} 
g(r_2)=\lim_{r_3\to\infty} \Big(g(r_3)-\int_{r_2}^{r_3} g_r\,\mathrm{d}r\Big)=-\int_{r_2}^{\infty} g_r\,\mathrm{d}r. 
\end{equation*}

\smallskip
\textbf{2. Proof for (ii).} Define $\boldsymbol{f}(\boldsymbol{x})=f(r)\frac{\boldsymbol{x}}{r}$. Then $\boldsymbol{f}$ is a spherically symmetric vector function defined on $\mathbb{R}^n$, and Lemmas \ref{lemma-initial} and \ref{lemma-L6} (see Appendices \ref{appb}--\ref{improve-sobolev}) yield $\boldsymbol{f}\in L^\frac{n}{n-1}(\mathbb{R}^n)\cap D^{1,1}(\mathbb{R}^n)$. Moreover, via the coordinate transformation, we have
\begin{equation*}
\int_0^\infty (r^mf)_r\,\mathrm{d}r=\int_0^\infty r^m\big(f_r+\frac{m}{r}f\big)\,\mathrm{d}r=\frac{1}{\omega_n}\int_{\mathbb{R}^n}\diver \boldsymbol{f}\,\mathrm{d}\boldsymbol{x},
\end{equation*}
where $\omega_n$ denotes the surface area of the $n$-D sphere, satisfying $\omega_2=2\pi$ and $\omega_3=4\pi$. 

Next, define $B_R:=\{\boldsymbol{x}:|\boldsymbol{x}|<R\}$. Let $\varphi(\boldsymbol{x})\in C^\infty_{\mathrm{c}}(B_2)$ be a cut-off function such that $\varphi\equiv 1$ in $B_1$ and $|\nabla \varphi|\leq C$ for some constant $C>0$, and let $\varphi_R(\boldsymbol{x}):=\varphi(\frac{\boldsymbol{x}}{R})$. Then we obtain from the H\"older inequality that
\begin{equation*}
\begin{aligned}
\Big|\int_{\mathbb{R}^n}\diver \boldsymbol{f}\,\mathrm{d}\boldsymbol{x}\Big|&=\lim_{R\to \infty} \Big|\int_{\mathbb{R}^n}(\diver \boldsymbol{f})\varphi_R\,\mathrm{d}\boldsymbol{x}\Big|=\lim_{R\to \infty}\Big|\frac{1}{R}\int_{\mathbb{R}^n}\boldsymbol{f}\cdot(\nabla\varphi)\big(\frac{\boldsymbol{x}}{R}\big)\,\mathrm{d}\boldsymbol{x}\Big|\\
&\leq \lim_{R\to \infty} \frac{C}{R}\int_{B_{2R}\backslash B_R}|\boldsymbol{f}|\,\mathrm{d}\boldsymbol{x}\leq C \lim_{R\to \infty} \Big(\int_{B_{2R}\backslash B_R}|\boldsymbol{f}|^\frac{n}{n-1}\,\mathrm{d}\boldsymbol{x}\Big)^\frac{n-1}{n}=0.
\end{aligned}
\end{equation*}

This completes the proof.
\end{proof}

The fourth lemma is on the Hardy inequality.
\begin{lem}[\cite{brown,CZZ1,opic}]\label{hardy}  
Let $q\in [2,\infty)$, $b>0$, and let $f=f(r)$ be a function defined on $[0,b]$ such that $r^{\ell +1+\frac{1}{p}-\frac{1}{q}}(f,f_r)\in L^q(0,b)$ for some $p\in [q,\infty]$ and $\ell>-\frac{1}{p}$ $(\ell>0$ if $p=\infty)$. Then $r^\ell f\in L^p(0,b)$ and the following inequalities hold{\rm :}
\begin{equation*}
\begin{aligned}
\|r^\ell f\|_{L^p(0,b)}&\leq C_1\big\|r^{\ell+1+\frac{1}{p}-\frac{1}{q}}(f,f_r)\big\|_{L^q(0,b)} &&\quad \text{if }p\in [q,\infty),\\
\|r^\ell f\|_{L^\infty(0,b)}&\leq C_2\big\|r^{\ell +1-\frac{1}{q}}(f,f_r)\big\|_{L^q(0,b)}  &&\quad \text{if }p=\infty,
\end{aligned}
\end{equation*}
where $C_1$ and $C_2$ are positive constants depending only on $(\ell,p,q,b)$ and  $(\ell,q,b)$, respectively. Moreover, if $r^{\ell +1-\frac{1}{q}}(f,f_r)\in L^q(0,b)$ for some $\ell>0$, then $r^\ell f\in C([0,b])$. 
\end{lem}

The fifth lemma is the well-known Fatou lemma which can be found in \cite{realrudin}.
\begin{lem}[\cite{realrudin}]\label{Fatou}
Let  $\{f_n\}$ be a sequence of measurable non-negative functions $f_n$ on $\mathbb{R}^n$. Define $f(\boldsymbol{x}):= \liminf_{n\to \infty} f_n(\boldsymbol{x})$ for {\it a.e.} $\boldsymbol{x}\in \mathbb{R}^n$. Then $f$ is measurable and
\begin{equation*}
\int_{\mathbb{R}^n} f \,\mathrm{d}\boldsymbol{x} \leq \liminf_{n\to \infty} \int_{\mathbb{R}^n} f_n \,\mathrm{d}\boldsymbol{x} \qquad \text{for {\it a.e.} $\boldsymbol{x}\in \mathbb{R}^n$}.
\end{equation*}
\end{lem}

The sixth lemma is used to obtain the time-weighted estimates of the velocity.
\begin{lem}[\cite{bjr,CZZ1}]\label{bjr}
Let $E\subset \mathbb{R}^n$ $(n\in \mathbb{N}^*)$ be any set and $f\in L^2([0,T]; L^2(E))$. Then there exists a sequence $\{t_k\}_{k=1}^\infty$ such that
\begin{equation*}
t_k\to 0, \quad\ t_k \|f(t_k)\|^2_{L^2(E)}\to 0 \qquad\,\, \text{as $k\to\infty$}.
\end{equation*}
\end{lem}

The seventh lemma gives an equivalent statement on the spherically symmetric vector fields.
\begin{lem}\label{rmk31}
Let $\boldsymbol{f}=\boldsymbol{f}(\boldsymbol{x})$ be a spherically symmetric  continuous vector function on $\mathbb{R}^n$. Then $\boldsymbol{f}$ takes the form{\rm :} $\boldsymbol{f}(\boldsymbol{x})=f(|\boldsymbol{x}|)\frac{\boldsymbol{x}}{|\boldsymbol{x}|}$ if and only if 
\begin{equation}\label{dengjia}
\mathcal{O}\boldsymbol{f}(\boldsymbol{x})=\boldsymbol{f}(\mathcal{O}\boldsymbol{x}) \qquad \text{for all $\boldsymbol{x}\in \mathbb{R}^n$ and $\mathcal{O}\in \mathrm{SO}(n)$}.
\end{equation}
In particular, any spherically symmetric  vector  function $\boldsymbol{f}$ satisfies $\boldsymbol{f}(\boldsymbol{0})=\boldsymbol{0}$.
\end{lem}
\begin{proof}
Clearly, if $\boldsymbol{f}(\boldsymbol{x})=f(|\boldsymbol{x}|)\frac{\boldsymbol{x}}{|\boldsymbol{x}|}$, then \eqref{dengjia} holds.

Conversely, if \eqref{dengjia} holds, we take the 3-D case as an example. Let $\boldsymbol{x}_0\in \mathbb{R}^n$ be any fixed displacement vector, $\boldsymbol{e}_1=\frac{\boldsymbol{x}_0}{|\boldsymbol{x}_0|}$, and let $\mathcal{O}_1\in \mathrm{SO}(n)$ be a rotation by 180 degrees about an axis parallel to $\boldsymbol{x}_0$, \textit{i.e.}, $\mathcal{O}_1\boldsymbol{x}_0=\boldsymbol{x}_0$. Then \eqref{dengjia} implies
\begin{equation}\label{cls'}
\mathcal{O}_1\boldsymbol{f}(\boldsymbol{x}_0)=\boldsymbol{f}(\boldsymbol{x}_0).
\end{equation}
Next, suppose that $\{\boldsymbol{e}_2,\boldsymbol{e}_3\}$ are unit vectors that, together with $\boldsymbol{e}_1$, form an orthonormal basis for $\mathbb{R}^3$. 
Then there exist constants $\alpha_i=\alpha_i(\boldsymbol{x}_0)\in \mathbb{R}$ $(i=1,2,3)$, depending only on $\boldsymbol{x}_0$, such that $\boldsymbol{f}(\boldsymbol{x}_0)=\alpha_1\boldsymbol{e}_1+\alpha_2\boldsymbol{e}_2+\alpha_3\boldsymbol{e}_3$. This, combined with \eqref{cls'}, gives 
\begin{equation*}
\alpha_1\boldsymbol{e}_1+\alpha_2\boldsymbol{e}_2+\alpha_3\boldsymbol{e}_3=\alpha_1\boldsymbol{e}_1-\alpha_2\boldsymbol{e}_2-\alpha_3\boldsymbol{e}_3\implies \alpha_2\boldsymbol{e}_2+\alpha_3\boldsymbol{e}_3=\boldsymbol{0}\implies \alpha_2=\alpha_3=0.
\end{equation*}
Hence, for any fixed $\boldsymbol{x}_0\in \mathbb{R}^3$,
\begin{equation*}
\boldsymbol{f}(\boldsymbol{x}_0)=\alpha_1\boldsymbol{e}_1=\alpha_1(\boldsymbol{x}_0)\frac{\boldsymbol{x}_0}{|\boldsymbol{x}_0|}.
\end{equation*}
By \eqref{dengjia}, we see that $\alpha_1(\mathcal{O}\boldsymbol{x}_0)=\alpha_1(\boldsymbol{x}_0)$ for all $\mathcal{O}\in \mathrm{SO}(n)$, that is, $\alpha_1(\boldsymbol{x}_0)=\alpha_1(|\boldsymbol{x}_0|)$. 
Thus, letting $f(r):=\alpha_1(r)$ implies that $\boldsymbol{f}$ takes the form: $\boldsymbol{f}(\boldsymbol{x})=f(|\boldsymbol{x}|)\frac{\boldsymbol{x}}{|\boldsymbol{x}|}$.

Finally, from \eqref{dengjia}, we can take $\boldsymbol{x}=\boldsymbol{0}$ to obtain
\begin{equation}\label{s0}
\boldsymbol{f}(\boldsymbol{0})=\mathcal{O}\boldsymbol{f}(\boldsymbol{0}) \qquad \text{for all $\mathcal{O}\in \mathrm{SO}(n)$}.
\end{equation}
Then, choosing $\mathcal{O}=\mathcal{O}_2$ by 180 degrees with respect to 
an axis perpendicular to $\boldsymbol{f}(\boldsymbol{0})$, that is, $\mathcal{O}_2\boldsymbol{f}(\boldsymbol{0})=-\boldsymbol{f}(\boldsymbol{0})$, we obtain from \eqref{s0} that $\boldsymbol{f}(\boldsymbol{0})=-\boldsymbol{f}(\boldsymbol{0})$, so that $\boldsymbol{f}(\boldsymbol{0})=\boldsymbol{0}$. 
\end{proof}

The following lemma is on the evolution triple embedding.
\begin{lem}[\cite{evans}]\label{triple}
Let $T>0$, $n\in\mathbb{N}$, and $n\geq 2$. 
Suppose that $f\in L^2([0,T];H^1(\mathbb{R}^n))$ and  $f_t\in L^2([0,T];H^{-1}(\mathbb{R}^n))$. 
Then $f\in C([0,T];L^2(\mathbb{R}^n))$, and the mapping{\rm :} $t\mapsto |f(t)|_2^2$ is absolutely continuous 
with
\begin{equation*}
\frac{\mathrm{d}}{\mathrm{d}t} \|f(t)\|_{L^2(\mathbb{R}^n)}^2=2\left<f_t, f\right>_{H^{-1}(\mathbb{R}^n)\times H^1(\mathbb{R}^n)} \qquad \text{for {\it a.e.} $t\in (0,T)$}.
\end{equation*}
Moreover, if additionally $f\in L^\infty([0,T];H^1(\mathbb{R}^n))$, then $f\in C([0,T];L^q(\mathbb{R}^n))$ for $q\in [2,\infty)$ if $n=2$ and for $q\in [2,\frac{2n}{n-2})$ if $n\geq 3$.
\end{lem}

The following lemma is on the Aubin--Lions lemma.
\begin{lem}\cite{jm}\label{aubin-lions}
Let $X_0 \subset X \subset X_1$ be three Banach spaces. Suppose that $X_0$ is compactly embedded in $X$, and $X$ is continuously embedded in $X_1$. Then the following statements hold{\rm:}
\begin{enumerate}
\item[{\rm(i)}] If $F$ is bounded in $L^p([0,T];X_0)$ for $1\leq p<\infty$, and $F_t$ is bounded in $L^1([0,T];X_1)$, then $F$ is relatively compact in $L^p([0,T];X)${\rm;}
\item[{\rm(ii)}] If $F$ is bounded in $L^\infty([0,T];X_0)$, and $F_t$ is bounded in $L^p([0,T];X_1)$ for $p>1$, then $F$ is relatively compact in $C([0,T];X)$.
\end{enumerate}
\end{lem}

The following auxiliary lemma is used to show some equivalent norms for a function $f$ 
satisfying $f\in L^1(\mathbb{R}^n)$ and $\nabla f^\alpha\in D^{1,n}(\mathbb{R}^n)\cap D^2(\mathbb{R}^n)$. 

\begin{lem}\label{initial3}
Let $n=2,3$, and let $f>0$ be a spherically symmetric scalar function defined on $\mathbb{R}^n$. If $f\in L^1(\mathbb{R}^n)$ and $\nabla f^{\alpha}\in D^{1,n}(\mathbb{R}^n)\cap D^2(\mathbb{R}^n)$ for some $\alpha\in (-\frac{1}{n},0)$, then  
\begin{enumerate}
\item[$\mathrm{(i)}$] $f\in L^p(\mathbb{R}^n)$ for all $p\in (1, \infty]${\rm ;}
\smallskip
\item[$\mathrm{(ii)}$] $\nabla f^{\alpha}\in  L^\infty(\mathbb{R}^n)$ and $\nabla f^\beta\in  H^2(\mathbb{R}^n)$  for all $\beta\in[\alpha+\frac{1}{2},\infty)$.
\end{enumerate}
\end{lem}
\begin{proof}
We divide the proof into two steps.

\smallskip
\textbf{1.} Since $f$ is a spherically symmetric scalar function, then $\nabla f^\alpha$ is also a spherically symmetric vector function.
\smallskip
First,  we obtain from Lemma \ref{Hk-Ck-vector} in Appendix \ref{improve-sobolev} that
\begin{equation}\label{A13}
\|\nabla f^\alpha\|_{L^\infty}\leq C\|\nabla f^\alpha\|_{D^{1,n}}\leq C.
\end{equation}

Next, let $n^*$ be defined as in \S\ref{othernote}. Then it follows from \eqref{A13}, Lemma \ref{GN-ineq}, and the H\"older inequality that
\begin{equation*}
\begin{aligned}
\|f\|_{L^\infty}&\leq C\|f\|_{L^1}^\frac{4-n}{n+4}\|\nabla^2 f\|_{L^2}^\frac{2n}{n+4} \leq C\|f\|_{L^1}^\frac{4-n}{n+4}\big(\|f^{1-2\alpha}|\nabla f^\alpha|^2\|_{L^2}+\|f^{1-\alpha}\nabla^2 f^\alpha\|_{L^2}\big)^\frac{2n}{n+4}\\
&\leq C\|f\|_{L^1}^\frac{4-n}{n+4}\big(\|f\|_{L^{2-4\alpha}}^\frac{2n(1-2\alpha)}{n+4}\|\nabla f^\alpha\|_{L^\infty}^\frac{4n}{n+4}+ \|f\|_{L^{n^*(1-\alpha)}}^\frac{2n(1-\alpha)}{n+4}\|\nabla^2 f^\alpha\|_{L^n}^\frac{2n}{n+4}\big)\notag\\
&\leq C\|f\|_{L^1}^\frac{4}{n+4}\|f\|_{L^\infty}^\frac{n(1-4\alpha)}{n+4}\|\nabla f^\alpha\|_{L^\infty}^\frac{4n}{n+4}+C\|f\|_{L^1}^\frac{2}{n+4}\|f\|_{L^\infty}^\frac{n+2-2n\alpha}{n+4}\|\nabla^2 f^\alpha\|_{L^n}^\frac{2n}{n+4},
\end{aligned}
\end{equation*}
which, along with the Young inequality, implies
\begin{equation}\label{lemb2-1}
\|f\|_{L^\infty}\leq C\|f\|_{L^1}^\frac{1}{1+n\alpha}\|\nabla f^\alpha\|_{L^\infty}^\frac{n}{1+n\alpha}+C\|f\|_{L^1}^\frac{1}{1+n\alpha}\|\nabla^2 f^\alpha\|_{L^n}^\frac{n}{1+n\alpha}\leq C.
\end{equation}

Therefore,  $f\in L^\infty(\mathbb{R}^n)$, which, along with $f\in L^1(\mathbb{R}^n)$, gives (i).

\smallskip
\textbf{2.} It remains to show that $\nabla f^\beta\in  H^2(\mathbb{R}^n)$  for all $\beta\in[\alpha+\frac{1}{2},\infty)$. A direct calculation gives 
\begin{equation*}
\begin{aligned}
(f^\beta)_{x_i} &= \frac{\beta}{\alpha} f^{\beta-\alpha} (f^\alpha)_{x_i},\qquad (f^\beta)_{x_ix_j}=\frac{\beta}{\alpha}\Big(\frac{\beta-\alpha}{\alpha} f^{\beta-2\alpha} (f^\alpha)_{x_i} (f^\alpha)_{x_j}+ f^{\beta-\alpha}(f^\alpha)_{x_ix_j}\Big),\\
(f^\beta)_{x_ix_jx_k}&=\frac{\beta}{\alpha}\Big(\frac{\beta-\alpha}{\alpha}\frac{\beta-2\alpha}{\alpha}f^{\beta-3\alpha} (f^\alpha)_{x_i} (f^\alpha)_{x_j}(f^\alpha)_{x_k}+\frac{\beta-\alpha}{\alpha} f^{\beta-2\alpha} (f^\alpha)_{x_ix_k} (f^\alpha)_{x_j}\Big)\\
&\quad+ \frac{\beta}{\alpha}\frac{\beta-\alpha}{\alpha}f^{\beta-2\alpha}\Big( (f^\alpha)_{x_i} (f^\alpha)_{x_jx_k}+ (f^\alpha)_{x_k} (f^\alpha)_{x_ix_j}\big)+ \frac{\beta}{\alpha} f^{\beta-\alpha} (f^\alpha)_{x_ix_jx_k}\Big). 
\end{aligned}
\end{equation*}
Then it follows from $\beta\in [\alpha+\frac{1}{2},\infty)$, \eqref{A13}--\eqref{lemb2-1}, and  the H\"older inequality that 
\begin{align*}
\|\nabla f^\beta \|_{L^2}&\leq C\|f\|^{\beta-\alpha}_{L^{2\beta-2\alpha}} \|\nabla f^\alpha\|_{L^{\infty}}\leq C,\\
\|\nabla^2 (f^\beta)\|_{L^2}&\leq C \|f\|_{L^{2\beta-4\alpha}}^{\beta-2\alpha}\|\nabla f^\alpha\|_{L^{\infty}}^2  +C \|f\|_{L^{n^*(\beta-\alpha)}}^{\beta-\alpha} \|\nabla^2 f^\alpha\|_{L^{n}} \leq C,\\
\|\nabla^3 (f^\beta)\|_{L^2}&\leq C \|f\|_{L^{2\beta-6\alpha}}^{\beta-3\alpha}\|\nabla f^\alpha\|_{L^{\infty}}^3 +C \|f\|_{L^{\infty}}^{\beta-\alpha} \|\nabla^3 f^\alpha\|_{L^{2}} \\
&\quad +C \|f\|_{L^{n^*(\beta-2\alpha)}}^{\beta-2\alpha}\|\nabla f^\alpha\|_{L^{\infty}} \|\nabla^2 f^\alpha\|_{L^{n}}\leq C.
\end{align*}

This completes the proof.
\end{proof}

Next, to establish the $L^p(\mathbb{R}^n)$-estimates for the transport equation:
\begin{equation}\label{diff}
g_t+\diver(\boldsymbol{w} g)=f,
\end{equation}
we need the following lemma to justify the energy equality.
\begin{lem}[\cite{CZZ1,lions1}]\label{lemma-lions}
Let $n\in\mathbb{N}^*$ and $T>0$. Assume that
\begin{equation*}
f\in L^1([0,T];L^p(\mathbb{R}^n)),\quad \boldsymbol{w}\in L^1([0,T];W^{1,\infty}(\mathbb{R}^n)),\quad g\in L^\infty([0,T];L^p(\mathbb{R}^n)),
\end{equation*}
for some $p\in[2,\infty)$, satisfying \eqref{diff} in the sense of distributions. Then 
\begin{equation*} 
\frac{\mathrm{d}}{\mathrm{d}t}\|g\|^p_{L^p}=-(p-1)\int_{\mathbb{R}^n}|g|^p\diver \boldsymbol{w}\,\mathrm{d}\boldsymbol{x}+p\int_{\mathbb{R}^n}|g|^{p-2}gf\,\mathrm{d}\boldsymbol{x}  \qquad \text{for {\it a.e.} $t\in (0,T)$}.
\end{equation*}
\end{lem}

Finally, we need to the following estimates on the equation for the Lam\'e operator $L$:
\begin{equation}\label{aue}
L\boldsymbol{f}=-a_1\Delta\boldsymbol{f} -(a_1+a_2)\nabla\diver\boldsymbol{f} =\boldsymbol{g} \qquad\text{in $\mathbb{R}^n$}.
\end{equation}
\begin{lem}[\cite{zhangzhifei}]\label{df3}
Let $q\in (1,\infty)$, and let $\boldsymbol{f}\in D^{1,q}(\mathbb{R}^n)$ 
be a weak solution to \eqref{aue} with the asymptotic condition $\boldsymbol{f}\to \boldsymbol{0}$ as $|\boldsymbol{x}|\to\infty$. 
Then, if $\boldsymbol{g}\in L^q(\mathbb{R}^n)$, there exists a constant $C>0$, depending only on $(n,a_1,k,q)$ and independent of $(\boldsymbol{f},\boldsymbol{g})$, such that
\begin{equation*}
\|\nabla^{2}\boldsymbol{f}\|_{L^q(\mathbb{R}^n)} \leq C\|\boldsymbol{g}\|_{L^q(\mathbb{R}^n)}.
\end{equation*} 
\end{lem}

\medskip
\section{Conversion of Sobolev Spaces for Spherically Symmetric Functions}\label{appb}

This appendix is devoted to showing the conversion of some Sobolev spaces between the M-D Eulerian coordinate $\boldsymbol{x}$ and the spherical coordinate $r$ for spherically symmetric functions. 
Let $n$ be the spatial dimensions and $m=n-1$.

\begin{lem}[\cite{CZZ1}]\label{lemma-initial}
Let $q\in [1,\infty]$, $0\leq a<b\leq \infty$, $\Omega:=\{\boldsymbol{x}\in \mathbb{R}^n:\, a\leq |\boldsymbol{x}|< b\}$, and $r\in J:=[a,b)$ with $r=|\boldsymbol{x}|$. 
Then
\begin{enumerate}
\item[$\mathrm{(i)}$] For spherically symmetric function $f\in W^{3,q}(\Omega)$ with $f(\boldsymbol{x})=f(r)$,  
\begin{equation*} 
\|f\|_{L^q(\Omega)} \sim 
\|r^\frac{m}{q}f\|_{L^q(J)},\qquad \|\nabla^j f\|_{L^q(\Omega)} \sim 
\|r^\frac{m}{q}\mathrm{D}_r^{j-1}f_r\|_{L^q(J)} \quad \text{for $j=1,2,3$};
\end{equation*}
\item[$\mathrm{(ii)}$] For spherically symmetric vector function $\boldsymbol{f}\in W^{4,q}(\Omega)$ with $\boldsymbol{f}(\boldsymbol{x})=f(r)\frac{\boldsymbol{x}}{r}$,
\begin{equation*}
\|\nabla^j\boldsymbol{f}\|_{L^q(\Omega)} \sim \|r^\frac{m}{q}\mathrm{D}_r^jf\|_{L^q(J)} \quad \text{for $j=0,1,2,3,4$}.
\end{equation*}
\end{enumerate}
Here, $E\sim F$ denotes $C^{-1}E\leq F\leq CE$ for some constant $C\geq 1$ depending only on $n$.
\end{lem}

\section{Sobolev Embeddings for Spherically Symmetric Functions}\label{improve-sobolev}

In this appendix, we give several improved Sobolev embeddings of the type: $D^{1,p}(\mathbb{R}^n)\hookrightarrow L^{q}(\mathbb{R}^n)$ for spherically symmetric vector functions when $p\in [1,n]$. The proofs of the following two lemmas can be found in \cite[Appendix C]{CZZ1}.

\begin{lem}[\cite{CZZ1}]\label{lemma-L6}
Let $\boldsymbol{f}(\boldsymbol{x})=f(r)\frac{\boldsymbol{x}}{r}$ be any spherically symmetric vector function defined in $\mathbb{R}^n$ {\rm(}$n\geq 2${\rm)}. If $\boldsymbol{f}\in D^{1,p}(\mathbb{R}^n)$ for some $p\in [1,n)$, then $\boldsymbol{f}\in L^\frac{np}{n-p}(\mathbb{R}^n)$, and there exists a constant $C(n,p)>0$ depending only on $(n,p)$ such that 
\begin{equation}\label{ineq-L6}
\|\boldsymbol{f}\|_{L^\frac{np}{n-p}(\mathbb{R}^n)}\leq C(n,p)\|\nabla \boldsymbol{f}\|_{L^p(\mathbb{R}^n)}.
\end{equation}
\end{lem}

\begin{lem}[\cite{CZZ1}]\label{Hk-Ck-vector}
Let $\boldsymbol{f}(\boldsymbol{x})=f(r)\frac{\boldsymbol{x}}{r}$ be any spherically symmetric vector function defined in $\mathbb{R}^n$ $(n\geq 2)$. If $\boldsymbol{f}\in D^{1,n}(\mathbb{R}^n)$, then $\boldsymbol{f}\in C(\overline{\mathbb{R}^n})$, and there exists a uniform constant $C(n)>0$ depending only on $n$ such that
\begin{equation}\label{vectorC3}
\|\boldsymbol{f}\|_{L^\infty(\mathbb{R}^n)}\leq C(n)\|\nabla\boldsymbol{f}\|_{L^n(\mathbb{R}^n)}.
\end{equation}
\end{lem}

\begin{rk}\label{remc1}
{\rm Lemmas \ref{lemma-L6}--\ref{Hk-Ck-vector}} do not hold for general vector functions or scalar functions.

For {\rm Lemma \ref{lemma-L6}}, such examples include  
$\boldsymbol{f}=(1,\cdots\!,1)^\top$ or $f=1$. This is mainly owing to the fact that any spherically symmetric constant vector function must vanish. 

For {\rm Lemma \ref{Hk-Ck-vector}}, consider a function $f=f(z)$ defined on $[0,\infty)$ satisfying
\begin{equation*} 
f(0)=0,\qquad f(z)= |\log z|^\frac{1}{3}\quad \text{on $(0,e^{-1}]$},\qquad f(z)=(ez)^{-2}\quad  \text{on $(e^{-1},\infty)$}.
\end{equation*}
Define $g(\boldsymbol{x})\!=\!f(|\boldsymbol{x}|)$ and $\boldsymbol{h}(\boldsymbol{x})$ with $h_i(\boldsymbol{x})=f(|x_i|)$ $(1\leq i\leq n)$. 
It can be checked that $(g,\boldsymbol{h})$ admits the weak derivatives $(\partial_jg,\partial_j\boldsymbol{h})\in L^n(\mathbb{R}^n)$ $(1\leq j\leq n)$, while $(g,\boldsymbol{h})\notin L^\infty(\mathbb{R}^n)$.
\end{rk}

\section{Some Auxiliary Calculations}\label{appd}
This section is devoted to providing the detailed calculations of \eqref{6-3} and \eqref{uell}. To this end, we only need to show the following lemma.
\begin{lem}
Let $\ell=2,3,4,5$. Then the following equality holds{\rm:}
\begin{align*}
&\,\Big(\rho u_t+\rho uu_r-2a_1\delta \big(\rho^\delta u_r+\frac{m\rho^\delta u}{r}\big)_r+2a_1m(\rho^\delta)_r\frac{u}{r}\Big)\times (\rho^{-\alpha}|u|^{\ell-2}u)\\
&=\big(\frac{1}{\ell}\rho^{1-\alpha}|u|^\ell\big)_t+2(\ell-1)a_1\delta\rho^{\delta-\alpha}|u|^{\ell-2}\Big(u_r^2-m\frac{1-\delta}{\delta-\alpha}u_r\frac{u}{r}+\frac{m}{\ell}\frac{1-\alpha}{\delta-\alpha}\frac{u^2}{r^2}\Big)\notag\\
&\quad -\Big(a_1\delta\rho^{\delta-\alpha}|u|^{\ell-2}\big(2uu_r+2m\frac{\ell\delta-\ell\alpha- \ell+1}{\ell(\delta-\alpha)}\frac{u^{2}}{r}\big) -\frac{(\ell-1)\alpha+\ell+1}{\ell(\ell+1)}\rho^{1-\alpha}u|u|^\ell\Big)_r\\
&\quad -\rho^{1-\alpha}vu|u|^{\ell-2}\Big(\alpha u_r-\frac{m(1-\alpha)}{\ell}\frac{u}{r}\Big)-\frac{(\ell-1)\alpha(1-\alpha)}{2\ell(\ell+1)a_1\delta}\rho^{2-\delta-\alpha}(v-u)u|u|^{\ell}.
\end{align*}
\end{lem}
\begin{proof}
First, it follows from \eqref{V-expression} and \eqref{mass-alpha} that
\begin{align*}
&\,(\rho u_t+\rho uu_r)\times (\rho^{-\alpha}|u|^{\ell-2}u)\\
&= \big(\frac{1}{\ell}\rho^{1-\alpha}  |u|^\ell\big)_t+\Big(\frac{1}{\ell}\rho^{1-\alpha} |u|^\ell u\Big)_r-\frac{1}{\ell}|u|^\ell\big((\rho^{1-\alpha})_t+(\rho^{1-\alpha}u)_r\big) \\
&= \big(\frac{1}{\ell}\rho^{1-\alpha}  |u|^\ell\big)_t+\Big(\frac{1}{\ell}\rho^{1-\alpha} |u|^\ell u\Big)_r-\frac{\alpha}{\ell} \rho^{1-\alpha}|u|^\ell u_r +\frac{m(1-\alpha)}{\ell} \frac{\rho^{1-\alpha}|u|^\ell u}{r}\\
&= \big(\frac{1}{\ell}\rho^{1-\alpha}  |u|^\ell\big)_t+\Big(\frac{1}{\ell}\rho^{1-\alpha} |u|^\ell u\Big)_r-\frac{\alpha}{\ell(\ell+1)} \rho^{1-\alpha}(|u|^\ell u)_r +\frac{m(1-\alpha)}{\ell} \frac{\rho^{1-\alpha}|u|^\ell u}{r}\\
&= \big(\frac{1}{\ell}\rho^{1-\alpha}  |u|^\ell\big)_t+\Big(\frac{\ell+1-\alpha}{\ell(\ell+1)}\rho^{1-\alpha} |u|^\ell u\Big)_r +\frac{\alpha(1-\alpha)}{\ell(\ell+1)} \rho^{-\alpha}\rho_r |u|^\ell u +\frac{m(1-\alpha)}{\ell} \frac{\rho^{1-\alpha}|u|^\ell u}{r}\\
&= \big(\frac{1}{\ell}\rho^{1-\alpha}  |u|^\ell\big)_t+\Big(\frac{\ell+1-\alpha}{\ell(\ell+1)}\rho^{1-\alpha} |u|^\ell u\Big)_r
+\frac{\alpha(1-\alpha)}{2\ell(\ell+1)a_1\delta} \rho^{2-\delta-\alpha}(v-u)u|u|^\ell\\
&\quad +\boxed{\frac{m(1-\alpha)}{\ell} \frac{\rho^{1-\alpha}|u|^\ell u}{r}}.
\end{align*}

Next, using \eqref{V-expression} again, we have
\begin{align*}
&\,\Big(2a_1\delta \big(\rho^\delta u_r+\frac{m\rho^\delta u}{r}\big)_r-2a_1m(\rho^\delta)_r\frac{u}{r}\Big)\times (\rho^{-\alpha}|u|^{\ell-2}u)\\
&= \Big(2a_1\delta \rho^{\delta-\alpha} |u|^{\ell-2}\big(uu_r+m \frac{u^2}{r}\big)\Big)_r-2a_1\delta \rho^\delta \big(u_r+m\frac{u}{r}\big)(\rho^{-\alpha}|u|^{\ell-2}u)_r-2a_1\delta m\rho^{\delta-\alpha-1}\rho_r\frac{|u|^\ell}{r}\\
&= \Big(2a_1\delta \rho^{\delta-\alpha} |u|^{\ell-2}\big(uu_r+m \frac{u^2}{r}\big)\Big)_r-2(\ell-1)a_1\delta \rho^{\delta-\alpha}|u|^{\ell-2}\big(u_r^2+m u_r\frac{u}{r}\big) \\
&\quad+2a_1\delta\alpha \rho^{\delta-\alpha-1}\rho_r |u|^{\ell-2}uu_r -2a_1\delta m(1-\alpha) \rho^{\delta-\alpha-1}\rho_r \frac{|u|^\ell}{r}\\
&=\Big(2a_1\delta \rho^{\delta-\alpha} |u|^{\ell-2}\big(uu_r+m \frac{u^2}{r}\big)\Big)_r-2(\ell-1)a_1\delta \rho^{\delta-\alpha}|u|^{\ell-2}\big(u_r^2+m u_r\frac{u}{r}\big) \\
&\quad+ \alpha \rho^{1-\alpha}v |u|^{\ell-2}uu_r-\alpha \rho^{1-\alpha}|u|^{\ell} u_r\\
&\quad  - \frac{2(\ell-1)}{\ell}a_1\delta m\frac{1-\alpha}{\delta-\alpha} (\rho^{\delta-\alpha})_r \frac{|u|^\ell}{r}- \frac{2}{\ell}a_1\delta m(1-\alpha)\rho^{\delta-\alpha-1}\rho_r \frac{|u|^\ell}{r}\\
&=\Big(2a_1\delta \rho^{\delta-\alpha} |u|^{\ell-2}\big(uu_r+m \frac{u^2}{r}\big)\Big)_r-2(\ell-1)a_1\delta \rho^{\delta-\alpha}|u|^{\ell-2}\big(u_r^2+m u_r\frac{u}{r}\big) \\
&\quad+ \alpha \rho^{1-\alpha}v |u|^{\ell-2} uu_r-\frac{\alpha}{\ell+1} \rho^{1-\alpha}(|u|^{\ell}u)_r\\
&\quad - \frac{2(\ell-1)}{\ell} a_1\delta m\frac{1-\alpha}{\delta-\alpha} \Big(\rho^{\delta-\alpha}\frac{|u|^\ell}{r}\Big)_r +  \frac{2(\ell-1)}{\ell} a_1\delta m\frac{1-\alpha}{\delta-\alpha} \rho^{\delta-\alpha} \big(\frac{|u|^\ell}{r}\big)_r\\
&\quad - \frac{m(1-\alpha)}{\ell}\frac{\rho^{1-\alpha}v|u|^\ell}{r}+ \frac{m(1-\alpha)}{\ell}\frac{\rho^{1-\alpha}u|u|^\ell}{r}\\
&= \Big( a_1\delta \rho^{\delta-\alpha} |u|^{\ell-2}\big(2uu_r+ 2m \frac{\ell\delta-\alpha-(\ell-1)}{\ell(\delta-\alpha)} \frac{u^2}{r}\big)-\frac{\alpha}{\ell+1}  \rho^{1-\alpha}|u|^\ell u\Big)_r\\
&\quad -2(\ell-1)a_1\delta \rho^{\delta-\alpha}|u|^{\ell-2}\big( u_r^2+m u_r\frac{u}{r}\big) \\
&\quad+ \rho^{1-\alpha}v u|u|^{\ell-2}\Big(\alpha  u_r-\frac{m(1-\alpha)}{\ell}\frac{u}{r}\Big)+\frac{m(1-\alpha)}{\ell}\frac{\rho^{1-\alpha}|u|^\ell u}{r} +\frac{\alpha(1-\alpha)}{\ell+1} \rho^{-\alpha}\rho_r|u|^\ell u \\
&\quad +\frac{2(\ell-1)}{\ell} a_1\delta m\frac{1-\alpha}{\delta-\alpha} \rho^{\delta-\alpha} |u|^{\ell-2}\big(\ell u_r\frac{u}{r}- \frac{u^2}{r^2}\big)\\
&= \Big( a_1\delta \rho^{\delta-\alpha} |u|^{\ell-2}\big(2uu_r+ 2m \frac{\ell\delta-\alpha-(\ell-1)}{\ell(\delta-\alpha)} \frac{u^2}{r}\big)-\frac{\alpha}{\ell+1}  \rho^{1-\alpha}|u|^\ell u\Big)_r\\
&\quad -2(\ell-1)a_1\delta\rho^{\delta-\alpha}|u|^{\ell-2}\Big(u_r^2-m\frac{1-\delta}{\delta-\alpha}u_r\frac{u}{r}+\frac{m}{\ell}\frac{1-\alpha}{\delta-\alpha}\frac{u^2}{r^2}\Big) +\boxed{\frac{m(1-\alpha)}{\ell}\frac{\rho^{1-\alpha}|u|^\ell u}{r}}\\
&\quad+ \rho^{1-\alpha}v u|u|^{\ell-2}\Big(\alpha u_r-\frac{m(1-\alpha)}{\ell}\frac{u}{r}\Big)+\frac{\alpha(1-\alpha)}{2(\ell+1)a_1\delta}  \rho^{2-\delta-\alpha}(v-u)u|u|^\ell.
\end{align*}

Therefore, combining the above two equalities, we arrive at the desired equality by canceling the above two framed terms.
\end{proof}

\bigskip
\noindent{\bf Acknowledgments:}  
This research is partially supported by National Key R$\&$D Program of China (No. 2022YFA1007300). The research of Gui-Qiang G. Chen was also supported in part by the UK Engineering and Physical Sciences Research
Council Award EP/V008854 and EP/V051121/1.
The research of Shengguo Zhu was also supported in part by 
the National Natural Science Foundation of China under the Grant  12471212 
and the Royal Society (UK)-Newton International 
Fellowships NF170015.

\bigskip
\noindent{\bf Conflict of Interest:} The authors declare  that they have no conflict of
interest. 
The authors also  declare that this manuscript has not been previously  published, 
and will not be submitted elsewhere before your decision.

\bigskip
\noindent{\bf Data availability:} Data sharing is not applicable to this article as no datasets were generated or analysed during the current study

\end{document}